\documentclass[a4paper,11pt,leqno]{amsart}
\usepackage{xcolor,array,a4wide}
\usepackage{cite}
\usepackage{hyperref}\usepackage{enumitem}
\usepackage{amsmath,amssymb,amsthm}
\usepackage{mathrsfs,mathtools}
\hypersetup{linktocpage,colorlinks, citecolor={magenta}}
\usepackage[english]{babel}
\usepackage[utf8]{inputenc}

\usepackage{accents}

\usepackage{xspace}
\usepackage{bm}				
\usepackage{bbm,upgreek}		
\usepackage[e]{esvect}		
\usepackage{esint}			
\usepackage{etoolbox}			
\usepackage{calc}				
\usepackage{eso-pic}			
\usepackage{datetime}			

\usepackage{lmodern}
\numberwithin{equation}{section}

\newtheorem{theo}{Theorem}
\newtheorem{prop}{Proposition}[section]
\newtheorem{lem}[prop]{Lemma}

\newtheorem{remark}[prop]{Remark}
\newtheorem{assum}[prop]{Assumption}

\theoremstyle{plain}

\theoremstyle{definition}
\newtheorem{defi}[prop]{Definition}

\def \t0{\rightarrow 0} 

\def \be{\begin{equation}}
	\def \ee{\end{equation}}

\def \hal{\frac{1}{2}}

\def \div{\mathrm{div} \,} 
\def \1{\mathbf{1}} 

\def \ep{\varepsilon}
\def\F{\mathsf{F}}
\def\mr{\mathbb{R}}

\def\U{\mathsf{U}}

\def\G{\mathsf{G}}
\def\R{\mathsf{R}}
\def\M{\mathsf{M}}

\def\A{\mathsf{A}}

\def\indic{\mathbf{1}}

\def\nab{\nabla}

\def\({\left(}
\def\){\right)}

\def\Xint#1{\mathchoice
	{\XXint\displaystyle\textstyle{#1}}%
	{\XXint\textstyle\scriptstyle{#1}}%
	{\XXint\scriptstyle\scriptscriptstyle{#1}}%
	{\XXint\scriptscriptstyle\scriptscriptstyle{#1}}%
	\!\int}
\def\XXint#1#2#3{{\setbox0=\hbox{$#1{#2#3}{\int}$}
		\vcenter{\hbox{$#2#3$}}\kern-.5\wd0}}
\def\dashint{\Xint-}

\def \XN{X_N}

\def\indic{\mathbf{1}}

\def\pex{\Phi^\eta_{x_i}}

\def\xg{\langle x\rangle^\gamma}

\makeatletter
\def\namedlabel#1#2{\begingroup
	#2%
	\def\@currentlabel{#2}%
	\phantomsection\label{#1}\endgroup
}
\makeatother

\begin{document}
	\title[Mean-field limits via multiscale mollification]{Singular mean-field limits via a multiscale mollification metric }
	
	\author[Q. H. Nguyen]{Quoc Hung Nguyen}
\address{State Key Laboratory of Mathematical Sciences\\ Academy of Mathematics
	and Systems Science, Chinese Academy of Sciences, Beijing, 100190, China.}
\email{qhnguyen@amss.ac.cn}
\thanks{}
\author[S. Serfaty]{Sylvia Serfaty}
\address{Sorbonne Universit\'e,  CNRS,  Laboratoire Jacques-Louis Lions (LJLL), F-75005 Paris,
		France \\ \& Institut Universitaire de France \&
		Courant Institute of Mathematical Sciences, New York University.}
	
	\email{serfaty@cims.nyu.edu}
	\maketitle

\begin{abstract}
	We consider a general class of first order ODE systems for the evolution of $N$ interacting particles (in Euclidean space $\mathbb{R}^d$)  in a mean-field  regime. The class of interactions treated includes singular interactions of inverse power type up to power $d+1$, attractive or repulsive, and not necessarily deriving from a potential -- unlike, for instance, the modulated energy method. We introduce a new method to prove  quantitative convergence of the discrete system to solutions of the mean-field equation. It relies on studying the evolution of a metric encoding a multiscale control of the difference between the  empirical measure and its limit, via mollification by heat kernels. 
	
		We prove that the desired convergence holds (i) up to the maximal time of existence of the smooth solution to the limiting equation if the singularity is sub-coulombic in any dimension, or coulombic in dimensions 1 and 2 (where, to do so, we introduce a notion of weak solution to the ODE system), or (ii) for short time in the case of Coulomb singularity in dimension 3 and above and  (iii) up to a short $N$-dependent timescale for super-coulombic  interactions in all dimensions. The latter two results are demonstrated to be optimal as we prove that  collisions  occur within the same timescale for a class of attractive interactions.
	\end{abstract}

\section{Introduction}
\subsection{Problem and context}
We are interested in first order  evolutions over $(\mathbb{R}^d)^N$ of the form
\begin{equation}\label{ode}
	\left\{ \begin{array}{ll}
		\dot{x}_i= \displaystyle \frac1N\sum_{\substack{j=1\\j \neq i}}^N \F(x_i,x_j) &
		\\
		x_i(0)=x_i^0, &
	\end{array}\right.
\end{equation}
where $\F$ is a possibly force field which can be singular, under assumptions below which include the Coulomb case and interactions even more singular than Coulomb. We prove that as $N\to \infty$, solutions to \eqref{ode} converge to solutions of the {\it mean-field equation}
\be\label{mflimit}
\partial_t \mu= -\operatorname{div} (( \F *\mu )\mu),\ee
where we denote
\be \label{defconvol}
\F*\varphi(x) := \int_{\mr^d} \F(x,y) \, d\varphi(y).\ee
The convergence is in the sense of convergence of the time-dependent empirical measure
\be\label{empmeasure}
\mu_N^t := \frac{1}{N} \sum_{i=1}^N \delta_{x_i^t}
\ee
to $\mu$ solving \eqref{mflimit}, in the sense of weak convergence of probability measures.

As is now standard, see for instance \cite{CD2021}, such a convergence implies {\it propagation of chaos} for the system \eqref{ode}: if the initial positions of the particles $x_i^0$ follow an iid law $\mu_0$, then the law at later times is well approximated by the tensorized state $\mu^t \otimes \dots \otimes \mu^t$ in the sense of the $k$-point marginals converging to $(\mu^t)^{\otimes k}$ for each fixed $k$.

We note that the form we are considering is quite general: $\F$ is not assumed to be symmetric, moreover taking $\F(x,y)= F(x,y)+ \frac{N}{N-1}V(x)$, where $F$ is a given force field, allows to include the case of dynamics with both pairwise force $F$ and external potential $V$. For that reason, we want to allow $\F$ to grow at infinity (at most linearly, because otherwise the system \eqref{ode} can blow up at any positive time, as seen at the level of the simple ODE $\dot{x}=|x|^\alpha x,~\alpha>0$). More precise assumptions will be given below in Section~\ref{sec:assump}.

We present here a new proof of convergence valid for a broad range of singular force fields $\F$ that include forces blowing up as $\frac{1}{|x-y|^s}$ when $x-y \to 0$, with $s< d+1$, where $d$ is the space dimension. Note that the Coulomb interaction corresponds to $s=d-1$. We will call the situation with $s<d-1$ the sub-Coulomb case, $s=d-1$ the Coulomb case, and $s>d-1$ the super-Coulomb case.
The case $s \ge d+1$  corresponds to a {\it hypersingular} interaction and is expected to  lead to a different  limit, hydrodynamic instead of mean-field limit. Indeed, in that regime  the interaction energy $\sum_{i\neq j} 
|x_i-x_j|^{-(s-1)}$ diverges for typical configurations, the leading contribution to the dynamics comes only for the nearby particles, and the correct scaling should use the factor $N^{-(s-1)/d}$ rather than $N^{-1}$. There are very few convergence results in that setting, and all are restricted to dimension 1, such as  \cite{oelschlager}.
  Here we do not prove convergence results in that regime, only small-time closeness of the empirical measure to the initial  data.

When force fields are regular, say Lipschitz, classical arguments dating back to the seminal studies  \cite{Mckean1967,Dobrushin1979,BH1977,Sznitman1991} (often in probabilistic, second-order, settings) allow to prove mean-field convergence.
In recent years, there has been much progress on the topic in going beyond the Lipschitz case, with new quantitative proofs proposed, with the first quantitative result for singular interactions in \cite{HJ2007}.  Most proofs rely on identifying a good metric in which to measure convergence of the empirical measure to the expected limit, for which one needs to show that  the time derivative of the metric along the evolution remains bounded by the metric itself, allowing to conclude by Gronwall's lemma.

For interactions that are regular or not too singular, Wasserstein distances have proven efficient metrics \cite{Hauray2009,HM2014,CFP2012,CC2021}, see also the reviews \cite{Golse2016ln,Golse2022ln,Jabin2014}, which include discussions of second-order dynamics. 

Another important progress has been  the introduction of the  modulated energy method in \cite{Duerinckx2016,Serfaty2020}, with follow-ups in \cite{NRS2021,Rosen2022,Rosen2022b,HcRS2024,HCRScx}  (see also \cite{rosenzweig2026commutator} for a recent survey). The method provides another metric  to metrize convergence, allowing to treat Coulomb and super-Coulomb interactions (here $s \in [d-1,d+1)$), provided they are {\it repulsive} and derive from a potential. The modulated energy  metric  is based on the interaction itself. It is akin to a negative Sobolev norm, and is a metric thanks to the repulsive nature of the interaction. 
The control of its derivative in time by itself is accomplished by a functional inequality of {\it commutator estimate} type proven in various versions in  \cite{Serfaty2020,NRS2021,HcRS2024,RS2026}. 
 The  method has also allowed to prove convergence of (second order) Newtonian dynamics  in the monokinetic case \cite[Appendix]{Serfaty2020}.

When noise is present in the form of added independent diffusions, the 
relative entropy method \cite{JW2018}  applies to forces that are not too singular (more precisely in $W^{-1,\infty}$), while the modulated free energy method of \cite{BJW2020}, incorporating the modulated energy method,  allows to again treat repulsive Coulomb and super-Coulomb interactions, as well as attractive sub-logarithmic interactions \cite{BJW2020,CdCRS2023,CdCRS2023a}. See also a method via Wasserstein gradient flows which functions when $d=1$ \cite{BO2019}.

Another recent direction is based on hierarchy methods. For bounded or
sufficiently regular interactions, sharp rates of convergence have been
obtained through relative entropy estimates along the BBGKY hierarchy
\cite{Lacker2023,LlF2023}. Related hierarchy-based methods have also led to
higher-order corrections and cluster expansions, using for instance
$\chi^2$- or $L^2$-based metrics of convergence \cite{hCR2023}. More recently,
weighted estimates for hierarchies of cumulants, and related estimates for
hierarchies of marginals, have been used to prove propagation of chaos for
some singular interactions, including certain Riesz-type interactions
\cite{BDJ24,BJS25,DJ25}. These methods are particularly well adapted to
positive-temperature or diffusive particle systems. They are quite different
from the approach developed here: our setting is deterministic and
zero-temperature, and our estimates give quantitative convergence of the
empirical measure itself. In particular, we do not rely on the BBGKY or
cumulant hierarchy, nor on an entropy structure.

Again, let us emphasize that while the modulated energy (resp.~modulated free energy) method is inherently energy-based, we make here no direct assumption that the system derives from an energy ($\F$ need not derive from a gradient) nor that the interaction is repulsive, it may on the contrary be attractive.  We do make assumptions of almost anti-symmetry of $\F$, modelled on the situation where $\F$ is a gradient, however our assumptions allow for quite large perturbations around this model case. 
 Note that all the versions of the modulated energy method made various (and different) structural assumptions on $\F$ and on its Fourier transform,  related to the repulsive nature of the interaction, which were always more rigid than those we make here -- in particular we need no assumption on the Fourier transform.
This allows us to treat more general singular kernels than all the previously known proofs mentioned above.
\subsection{The multiscale mollification metric and the proof method}

The present paper introduces a new metric to measure convergence, based on a multiscale mollification by the heat kernel.

We use the heat kernel to mollify the Diracs at a scale $\eta$ which depends on $N$ (think of $\eta$ as a negative power of $N$), and to estimate directly the distance of the mollified measures to the limit in both $L^\infty$ and in a weak sense.


Let us first define the natural microscale
\begin{equation}\label{defetastar}
	\eta_*:=N^{-\frac{1}{d}}.
\end{equation}

For $\eta \ge \eta_*$, let us consider   $\Phi^\eta(x)$ the standard heat kernel
\be \label{heat}\Phi^\eta(x) := \frac1{\eta^d} \Phi\Big( \frac{x}{\eta}\Big), \qquad \Phi(x):= \frac{1}{(2\pi)^{\frac{d}2}} e^{- \frac{|x|^2}{2}}\ee
and let
\be \label{heattrans}
\Phi^\eta_{x_i}(x):= \Phi^\eta(x-x_i).\ee
Using the heat kernel allows us to repeatedly use the semi-group property
\be\label{semigroup}
\Phi^{\eta}\star \Phi^{\eta'} = \Phi^{\sqrt{\eta^2+\eta'^2}}, \ee
where $\star$ denotes the usual convolution.

Given a configuration $\XN:= (x_1, \dots , x_N)$ of points $x_i \in \mathbb{R}^d$, we then define the empirical measure 
\be 
\mu_N[\XN]:=\frac1N \sum_{i=1}^N \delta_{x_i}\ee and the 
mollified empirical measure
\be \label{fneta}
\mu_{N, \eta}[\XN]:= \frac1N\sum_{i=1}^N \Phi^\eta_{x_i}= \mu_N[\XN]\star \Phi^\eta.\ee
For the sake of lightness, we will drop the $[\XN]$ dependence in the notation.

By \eqref{semigroup},
we have
\be \label{semigroupmu}
\mu_{N,\eta'}=\Phi^{\sqrt{\eta'^2-\eta^2}} \star \mu_{N,\eta}, \qquad \text{for} \ \eta'\ge \eta.
\ee
To show that $\mu_N \to \mu$, where $\mu$ solves \eqref{mflimit},
it will suffice to show that $\mu_{N,\eta}-\mu$ converges weakly to zero, which will be done quantitatively  via the  multiscale metrics that we next introduce.

Throughout the paper, we use the notation
\be\langle x\rangle = \sqrt{1+|x|^2}.\ee
Let  $\gamma>d+2$ be fixed. The main multiscale distance from $\mu_N$ to any function $f$ is based on the following weighted $L^\infty$ distance from $\mu_{N,\eta}$ to $f$:  
\be\label{defM}
\M(\XN, f, \eta):=\sup_{x \in \mr^d} \Big( \langle x\rangle^\gamma\left|\mu_{N, \eta}[\XN](x) -f(x)\right|\Big)= \left\|\langle \cdot \rangle^\gamma \Big( \frac{1}{N} \sum_{i=1}^N \delta_{x_i} \star \Phi^\eta - f\Big) \right\|_{L^\infty}.
\ee
The metric $\M(\XN, f, \eta)$ measures the closeness of the empirical measure $\mu_N$ to $f$, at the scale $\eta$.
The weight $ \langle x\rangle^\gamma$ is only added to effectively confine the system and prevent the supremum from being achieved at infinity.
We will place assumptions of decay of $\mu$ that guarantee that the maximum is achieved (note that $\mu_{N,\eta}$ decays exponentially at infinity).

 We will be able to show that $\M(\XN, \mu, \eta)$ becomes (quantitatively) small when $\eta>\eta_*$, which is the best that one may hope for: empirical measures smoothed at scales $\eta \le \eta_*$ can never  approach a smooth distribution in $L^\infty$ distance. On the other hand, if $f$ is regular enough, we can find  $\XN$, a configuration of $N$ points, whose empirical measure approximates $f$ in such a way that for any $\eta\geq \eta_*$, we have
\begin{equation}\label{examf}
	\M(X_N, f, \eta)=\sup_{x \in \mr^d} \Bigg( \langle x\rangle^\gamma\left| \frac1N\sum_{i=1}^N \Phi^\eta_{x_i}(x) -f(x)\right|\Bigg)\leq C\Big(\frac{1}{N^{1/d}\eta}+\eta^2\Big).
\end{equation}
More generally, by Morrey's characterization of H\"older spaces (with $\|f\|_{C^\alpha} \sim \sup_{\eta}\eta^{-\alpha}\|f\star \Phi^\eta-f\|_{L^\infty}$), if $\M(\XN, f, \eta) \le C \eta^\alpha$ for $\alpha\in (0,2)$, then this means roughly that ``$\mu_N $ is $C^\alpha$-close to $f$ at lengthscale $\eta$."  

To optimize on the convergence rate, we will also consider the weaker metric
\begin{align}\nonumber
	\mathsf{W}(\XN, f, \eta):&=\sup_{\|(\phi,\nab \phi)\|_{L^1(\mathbb{R}^d)} \leq 1}\left|\int_{\mathbb{R}^d} \langle x\rangle^{\gamma-1}\left(\mu_{N, \eta}[\XN](x) -f(x)\right)\phi(x)\, dx\right|\\&= \left\|\langle \cdot \rangle^{\gamma-1} \Big( \frac{1}{N} \sum_{i=1}^N \delta_{x_i} \star \Phi^\eta - f\Big) \right\|_{(W^{1,1})^*}.	\label{defM-1}
\end{align}
An optimization of error terms will lead us to a control of $\mathsf{W}$ at scale $\eta_*^{1/2}$.
\smallskip

When differentiating the metric $\M$ along the flow, i.e.~computing $\partial_t \M(\XN^t, \mu^t, \eta)$ (where $\XN^t$ solves \eqref{ode} and $\mu^t$ solves \eqref{mflimit}), it is {\it not} possible to control this derivative by $ \M(\XN^t, \mu^t, \eta)$: the closeness over small scales deteriorates in time.
 However, the use of the heat kernel allows to take advantage of its semi-group property to use cancellations and control information at one scale by information at smaller scales. This leads us to 
 controls of $ 
\partial_t \M(\XN^t, \mu^t, \eta)$ that involve the smaller scale metric $ \M(\XN^t, \mu^t, \eta/2).$
In order to hope to close the estimate, we thus define a {\it multiscale metric}
summing the $\M$'s at dyadically larger scales.

Let $\sigma\in [\frac{1}{2},1]$ (we will take $\sigma= \hal$ in the proof). Given $m$ integer, we consider the  multiscale metric  defined as
\be\label{defMN}
\M_m(\XN,f):= \sum_{n=0}^{\mathbf{n}-m} 2^{\sigma n} \M(\XN,f,2^{n+m}\eta_*).
\ee
It encodes the metrics at scale $2^m \eta_*$ to  $ 2^{\mathbf{n}}\eta_*$, with smaller scales discounted.  This, as well as the use of the heat kernel to encode regularity, is also reminiscent of the quantities used in  stochastic homogenization in \cite{armstrongkuusi}   where similar multiscale norms encoding regularity at all scales above a microscale are used (the ``minimal radius" playing the role of $\eta_*$).

In our proof, the integer $m$ will be chosen independent of $N$, and $\mathbf{n}$ will be chosen so that $2^{\mathbf{n}}\eta_*=\eta_*^{1/3}$.
The proof will proceed by first showing that $\M_0(\XN^t, \mu^t)$ can be controlled on an interval of time $[0, \tau]$, for $\tau>0$ independent of $N$. Then we will discard the smallest scale, and using that 
\be \label{bobvious}\M_m \le  2^{-\sigma m}\M_0, \ee
get a better control of $\M_1$ at time $\tau$, then iterate the argument to obtain that $\M_1$ is controlled over $[\tau, 2\tau]$, etc. This will allow to control $\M_m$ on $[0,T]$ for any given $T>0$ independent of $N$, with $m=T/\tau$. 

An estimate for $\M_0$ on $[0,T]$ can then be retrieved from the estimate on $\M_m$ from the observation that  for every integer $m$ satisfying $0\leq m\leq\mathbf n$,
\be \label{retourMM}\M_0 \lesssim 2^{dm}\M_m +2^{dm}\A(f),\ee where $\A(f)$ is a norm defined in \eqref{defA}, 
see Lemma  \ref{lem2.3} for a proof of this relation. 
Since by definition of $\M_0$, for $\eta \in [\eta_*, \eta_*^{1/3}]$, 
\be\label{Mmeta} \M(\XN, \mu, \eta)\le \(\frac{\eta_*}{\eta}\)^\sigma \M_0(\XN, \mu),
\ee a control of $\M_0$ by, say, a constant, 
implies that $\M(\XN, \mu, \eta)$ tends to $0$ algebraically in $\eta_*/\eta$, yielding the convergence of $\mu_N$ to $\mu$.
Inserting this estimate back into the Gronwall argument then allows to reduce to a linear equation and to 
 almost obtain the optimal rate $\eta_*/\eta$.

Our task will thus be to control the growth of $\M_0(\XN^t, \mu^t)$ by Gronwall on some (possibly small) interval $[0, \tau]$. In the case $s \le 1$  where $|x-y|\F(x,y)$ remains bounded when $x$ approaches $y$, then this is possible directly. This is borderline for the force to be in the Besov space $B_{\infty,\infty}^{-1}$, and  analogous to the assumption $F \in W^{-1,\infty}$ in \cite{JW2018}.

	Otherwise, in the general sub-Coulomb case $s <d-1$, we are still able to conclude by using an additional ingredient consisting  in coupling 
 the time evolution of $\M_m(\XN^t,\mu^t)$ with that of a
microscale interaction quantity similar to
$$H(t)= \frac1N \sup_{i\in [1,N]} \sum_{j=1, j\neq i}^N \frac{\indic_{|x_i-x_j|\le \eta_*}}{|x_i-x_j|^{d-\theta}},$$
for some appropriate $\theta\ge 0$. 

Achieving this coupled control requires propagating time estimates for the metric for the whole range of parameters $\eta$ at once. 
Note that these controls borrow from the regularity of the solution to \eqref{mflimit} and require  the initial data to be  sufficiently ``well prepared", by which we mean that $ \M(\XN^0, \mu^0, \eta)$ are small and  the particles do not accumulate at small scales.

In the $ d \ge 3 $ Coulomb case, respectively in the super-Coulomb cases, we prove that collisions may occur in finite, respectively short time, which shows that our convergence results cannot be improved: the Gronwall argument above cannot be iterated, and we only obtain convergence for finite, resp. short time, consistent with the possibility of breakdown of classical solutions to the  ODE system.

\subsection{Assumptions on $\F$}\label{sec:assump}
The assumptions below can be read keeping in mind the model case where $\F$ is, up to an additive smooth term, antisymmetric in $x$ and $y$ and derives from a potential, for instance $\F(x,y)=\mathbb{M} \nab |x-y|^{-s+1}$ for some constant matrix $\mathbb{M}$, or more generally where $\F(x,y)=\mathbb{M}(x,y) \nab |x-y|^{-s+1}$ with $\mathbb{M}$ smooth. They include that case and are meant to generalize it.
First, without loss of generality, we may assume that the singularity power $s$ is not too small, more precisely we will assume that $s\geq (d-\frac{5}{4})_+$ (more regular cases can of course be covered by that one). 
We note that the case $s=1$ encompasses the logarithmic case $\F(x,y)=\mathbb{M}\nab \log |x-y|$.\\
We assume that $\F = \F^{\mathrm{sing}} + \F^{\mathrm{reg}}$, where $\F^{\mathrm{reg}}$ corresponds to the regular part of $\F$ and may grow linearly at infinity, while $\F^{\mathrm{sing}}$ is the singular part that blows-up as $|x-y|^{-s}$ near $x=y$. 

If $\mu$ is smooth and sufficiently decaying, then the convolution $
\F^{\mathrm{sing}}*\mu$
is well defined without using any cancellation whenever $s<d$. However, if $\F^{\mathrm{sing}}$ enjoys a cancellation property of the form
\[
\left|
\int_{|x-y|=r}\F^{\mathrm{sing}}(x,y)dy
\right|
\lesssim 1+r^{-s+1},
\qquad 0<r\leq 1,
\]
then the convolution $\F^{\mathrm{sing}}*\mu$ is well defined in the larger range $s<d+1$. This condition is satisfied, for instance, by kernels of the form
\[
\F^{\mathrm{sing}}(x,y)
=
\mathbb{M}\nabla_y |x-y|^{-s+1}
+
O(|x-y|^{-s+1}),
\qquad |x-y|\leq 1.
\]
In this paper, we impose such a cancellation assumption.

When $s>d-1$, the limiting equation \eqref{mflimit} may in general be ill-posed. One way to see this is to consider an $N$-dependent regularized interaction  and the equation
\begin{equation} \label{120}
	\partial_t \mu= -\operatorname{div} (( \F_{\eta_*} *\mu )\mu),
\end{equation}
 where $\F_{\eta_*}$ is defined by $\F_\eta (x,y):= (\F(x, \cdot)\star\Phi^\eta)(y)$.
For each fixed cutoff parameter $\eta_*$, this equation admits a smooth solution $\mu=\mu^{\eta_*}$ on a time interval $[0,T_{\eta_*}]$, where $T_{\eta_*}$ may depend on $\eta_*$.  Formally, for $s<d+1$, one may write
\begin{equation}
	\partial_t \mu= -\operatorname{div} (( \F_{\eta_*} *\mu )\mu)=-\operatorname{div} (( \F*\mu )\mu)+O(\eta_*^2).
\end{equation}
Thus the regularized dynamics may be viewed as an approximation of the singular limiting equation as long as the singular convolution is meaningful.

For $s\geq d+1$, however, even with the above cancellation, the quantity $\F^{\mathrm{sing}}*\mu$ is no longer well defined in general. This motivates the introduction of the cutoff in the definition \eqref{deF}, namely
\begin{equation*}
	\F^{\eta_*}(x,y)=	\F^{\mathrm{reg}}(x,y)+	\F^{\mathrm{sing}}(x,y)(1-\chi)(\frac{|x-y|}{\eta_{*}}).
\end{equation*}
Thus the interaction is removed only at distances of order $\eta_*$. This cutoff will be used throughout the proof. In this strongly singular regime, the limiting equation is instead frozen at the initial datum $\mu^0$.\\

The assumptions (regular part, singular part, almost antisymmetry, almost isotropy) are phrased as follows:
that there exist $s\geq (d-\frac{5}{4})_+$ and $0<s'<\min \{s+1, d\}$, such that 
 for all $m_1,m_2=0,1,2,3$, the following hold :
\begin{align}	\label{assumpFre}
	&\text{for all } \ x\neq y, \qquad |\nab_x^{m_1}\nab_y^{m_2} \F^{\mathrm{reg}}(x,y)| \le C\Big( 1+(|x|+|y|) \mathbf{1}_{m_1=m_2=0}\Big);\\
	\label{assumpFsing}
	&\text{for all } \ x\neq y, \qquad |\nab_x^{m_1}\nab_y^{m_2} \F^{\mathrm{sing}}(x,y)| \le C\Big( \frac{1}{|x-y|^{s+m_1+m_2}}+1\Big);
\end{align}

\begin{align}
	\label{assumpF2}
	& \text{for all } \ x\neq y \qquad | \F^{\mathrm{sing}}(x,y)+ \F^{\mathrm{sing}}(y,x) |\le C\Big( \frac{1}{|x-y|^{s-1}}+1\Big);\\
\label{assumpF2G}  &\text{for } s<d+1, \ \text{for } m=1,2,3,4\ \text{and for any function} \ f,\\ \notag
&
\left|\nab_x^{m}\F^{\mathrm{sing}}*f(x)-\int_{\mathbb{R}^d} \F^{\mathrm{sing}}(x,y)\nab^{m}f(y)\right|  \leq C\int_{\mathbb{R}^d}\left( 1+\frac{1}{|x-y|^{s'}}\right)\sum_{k=0}^{m-1}|\nab^k f|(y) dy;
\end{align}
\begin{align}\nonumber&\text{for } s\ge d+1,m=1,2,3,4,~~\varepsilon\in [\frac{1}{10}\eta_{*},1)\\& \nonumber\Bigg|\nab_x^{m}\left(\int_{\mathbb{R}^d} \F^{\mathrm{sing}}(x,y)(1-\chi)\(\frac{|x-y|}{\ep}\)f(y)dy\right)\\
	& \qquad \qquad \qquad \qquad\qquad \qquad\nonumber-\int_{\mathbb{R}^d} \F^{\mathrm{sing}}(x,y)(1-\chi)\(\frac{|x-y|}{\ep}\)\nab^{m}f(y)\Bigg|\\  &  \qquad \qquad \qquad \qquad\qquad   \leq C\int_{\mathbb{R}^d}\left( 1+\frac{1}{|x-y|^{s-1}}\right)\mathbf{1}_{|x-y|\geq \frac{\ep}{2}}\sum_{k=0}^{m-1}|\nab^k f|(y) dy,\label{assumpF>=d+1}
\end{align}
where $\chi$ is a smooth cut-off equal $1$ in $[0,1]$ and vanishing outside $[0,2]$;\footnote{For $s<d+1$, the relevant integration-by-parts convention is encoded by \eqref{assumpF2G}. For $s\geq d+1$, the corresponding cutoff commutator is encoded by \eqref{assumpF>=d+1}. The spherical cancellation used near the diagonal is given by \eqref{assumpF2'}.}
\begin{align}
	\label{assumpF3c}
	&\text{for all } \ x\neq y+z,m=0,1,2, \qquad | \nab_y^m\F^{\mathrm{sing}}(x,y+z)+(-1)^m \nab_x^m\F^{\mathrm{sing}}(y,x-z) |\\ \notag &\quad\quad\quad\quad\quad\quad\quad\quad\quad\quad\quad\quad\quad\quad\quad\quad\quad\quad\quad\quad\quad\quad\quad\quad\quad \le C\Big( \frac{1}{|x-y-z|^{s-1+m}}+1\Big);
\end{align}
\begin{align}
	\label{assumpF2'}
	&\text{for all } x \in \mathbb{R}^d, \quad \left|\nab_x^m \dashint_{|z|=r}\F^{\mathrm{sing}}(x,x+z)d\mathcal H^{d-1}(z)\right| \le C\Big( r^{-s+1-m}+1  \Big)\quad \text{for } m=0,1,2.
\end{align}
With the gradient-flow Coulomb case in mind, we also assume
\begin{align} \label{coudet}
	\text{if} \  s=d-1, \quad 	\|\operatorname{div}_x \F^{\mathrm{sing}}*\phi \|_{L^\infty}\lesssim \|\langle\cdot\rangle^{\gamma_0} \phi\|_{L^\infty}~~\forall \phi\in W^{2,\infty}(\mathbb{R}^d),
\end{align}
for $\gamma_0\in (d,d+1)$. Here $\operatorname{div}_x(\F^{\mathrm{sing}}*\phi)$ is understood
distributionally, initially for
$\phi\in C_c^\infty(\mathbb R^d)$, and then by extension to the
weighted class appearing on the right-hand side.
\begin{remark}
We note that, unlike \cite{JW2018}, we do not require the product $|x-y|\F(x,y)$ to be bounded.
In that paper, the additional assumption 
$\div \F\in W^{-1,\infty}$ is placed. By contrast, in \cite{BJW2020}, the modulated free energy method can treat the two-dimensional logarithmic   case for which  $\div \F =\delta_0$.
\end{remark}

\subsection{Main results}
In all the paper we use the notation $A\lesssim B$ to denote $A\le C B$ for some constant $C>0$ independent of $N$, and $A\sim B$ if $A\lesssim B$ and $A \gtrsim  B$.

Before presenting the main results, we need to discuss solutions to \eqref{ode} and \eqref{mflimit}.
First, let us  define, for any function $f$,
\be\label{defA} \A(f):= \sum_{m=0}^{3} \|\langle x\rangle^{\gamma+1}\nab^m f\|_{L^\infty}.\ee 

In Appendix~\ref{app:existence}, we will show  by a fixed point argument that for $s \le d-1$, if $\mu_0$ satisfies $\A(\mu_0)+\mathbf{1}_{s\ge d}\|\langle x\rangle^{\gamma+1}\nab^4 \mu_0\|_{L^\infty}<\infty$, then \eqref{mflimit} has a unique regular solution at least for short time.
Our only assumption on the solution will be that $\A(\mu^t) +\mathbf{1}_{s\ge d}\|\langle x\rangle^{\gamma+1}\nab^4 \mu_0\|_{L^\infty}$ is bounded over the time interval we are interested in.\\

When $s> d-1$, the singularity of $\F$ makes \eqref{mflimit} ill-posed in general as mentioned above.
We have to work with either the solution of a regularized equation for $s\in (d-1,d)$, or the constant function $\mu^0$ for $s \ge d$.

Most of our results concern classical solutions of the particle system. We show that the convergence estimates developed in this paper, measured in the $\M$-norm, can also be used to construct such solutions. More precisely, we apply these estimates to a regularized version of \eqref{ode} and then pass to the limit in the regularization parameter. This yields the existence of classical solutions to \eqref{ode}
\begin{itemize}
	\item globally in time if $s<d-1$;
	\item up to the first collision time otherwise.
\end{itemize}
We also show that collisions may occur on the timescale $\eta_*^{s-d+1}$, see Theorems~\ref{cou} and~\ref{supercou}. To the best of our knowledge, this provides the first use of a mean-field-limit argument to prove the existence of classical solutions for a singular interacting particle system.

In the one and two-dimensional Coulomb cases, even though collisions may occur, it is still possible to continue giving a meaning to the dynamics beyond them, via a notion of weak solutions at the level of empirical measures that we introduce in Appendix~\ref{app1-weak}  and for which we prove existence of weak solutions, but not uniqueness. These are essentially solutions once tested against smooth test-functions.
We will be able in Theorem~\ref{th1-Coulomb} to provide a result in the one and two-dimensional Coulomb cases, which applies to such weak solutions of \eqref{ode} and retrieves convergence for arbitrary fixed times.
\begin{theo}[The  Coulomb and sub-Coulomb cases in dimensions  $1$ and $2$]\label{th1-Coulomb}
	Assume that $d=1 $ or $d=2$, and that the force field $\F$ satisfies
	\eqref{assumpFre}--\eqref{coudet} for $s=d-1$  \footnote{This of course allows to treat $s\le d-1$ as well.}. Let $\mu_N^t$ be a weak particle continuation of the ODE system	\eqref{ode} on $[0,T]$, in the sense of	Definition~\ref{def-weak-particle-solution}, obtained through the
	vanishing-regularization procedure.
		Let $\mu$ be a regular solution of \eqref{mflimit} on a time interval $[0,T]$ such that
	\[
	\A(\mu^t)<\infty \quad \text{for all } t\in[0,T],
	\]
	where $\A$ is defined in \eqref{defA} and let $\A_T:= \sup_{t\in [0,T]} \A(\mu^t) <\infty$.
	Assume that at $t=0$, we have
	\begin{align}\label{Cou-initial}
		\sup_{\eta\in [ \eta_*, \eta_*^{1/3}]}\eta \M( X_N^0, \mu^0,\eta)\leq C\eta_{*}
	\end{align}
Then there exist
$
N_0=N_0(d,s,C,T,\A_T)
$ and  $C'=C'(d,s,C)$
such that, for all $N\ge N_0$, the following conclusions hold. For every
$\eta\in[\eta_*,\eta_*^{1/3}]$
and every $t\in[0,T]$, we have
	\begin{align}
		\label{conclth1-cou} \M( X_N^t, \mu^t,\eta) \leq \frac{\eta_*}{\eta}  \exp\left[C'(1+\A_T)\sqrt{t \log \frac{\eta}{\eta_*}}+C'(1+\A_T)^2(1+T)\right]
	\end{align}
	and
		\begin{align}
		\label{conclth1-wn-cou} \mathsf{W}( X_N^t, \mu^t,\eta_*^{1/2}) &\le  \exp \big(C'(T+1)\A_T\big) \mathsf{W}\Big( X_N^0, \mu^0,\frac{1}{2}\eta_*^{1/2}\Big)\\&\quad+\eta_*(\log N) \exp\left[C'(1+\A_T)\sqrt{t \log N}+C'(1+\A_T)^2(1+T)\right].\nonumber
	\end{align}
	
	In particular, for every weak particle continuation selected by the
vanishing-regularization procedure, we have
\[
\mu_N^t \rightharpoonup \mu^t
\qquad
\text{as } N\to\infty,
\]
in the sense of probability measures, uniformly for $t\in[0,T]$ in the above
quantitative sense 	\end{theo}
\begin{remark}
Since the mean-field equation has a unique regular
solution on $[0,T]$, the possible non-uniqueness of the weak particle
continuation for fixed $N$ disappears in the mean-field limit.
\end{remark}

Thanks to this result, we can treat the two-dimensional attractive log case  (which is included in the $s=1$ case) all the way to the blow-up time of the limiting PDE, without needing any noise. This is in contrast with \cite{BJW2020,CdCRS2023a} which crucially require positive temperature.

The factor $\eta_*/\eta$ in the convergence rate is essentially optimal. The logarithmic factor and exponential in time factor are the price for the multiscale Gronwall cascades from fine to coarse scales: each scale adds a constant, leading to the $\sqrt{\log (\eta/\eta_*)}$ factor.

\begin{theo}[The sub-Coulomb case, $d\geq 3$]\label{th1}
	Assume that the force field $\F$ satisfies
	\eqref{assumpFre}--\eqref{assumpF2'} for $
	\Bigl(d-\frac54\Bigr)_+ \le s<d-1.
	$
	Let $\mu$ be a regular solution of \eqref{mflimit} on a time interval $[0,T]$,
	such that
	\[
	\A(\mu^t)<\infty \quad \text{for all } t\in[0,T],
	\]
	where $\A$ is defined in \eqref{defA} and let $\A_T:= \sup_{t\in [0,T]} \A(\mu^t) <\infty$.
		Assume moreover that
	 the initial particle configuration satisfies
	\be
	\label{condt02} \frac{1}{N}\sup_{i\in [1,N]}\sum_{ j=1,j\neq i}^N\frac{\mathbf{1}_{|x_i^0-x_j^0|\leq 2\eta_*}}{|x_i^0-x_j^0|^{d-\frac{\theta}{2}}}\le C,\quad \theta:=\frac{( d-1-s)_+}{3d}>0.
	\ee
	Assume that at $t=0$, we have
	\begin{align}\label{Cou-initial0sub}
		\sup_{\eta\in [ \eta_*, \eta_*^{1/3}]}\eta \M( X_N^0, \mu^0,\eta)\leq C\eta_{*}.
	\end{align}
	Then there exists
	$
	N_0=N_0(d,s,C,T,\A_T)
	$
	such that, for all $N\ge N_0$, the particle system \eqref{ode} admits a unique classical solution $X_N(t)$ on $[0,T]$.  Moreover, there exists $C'=C'(d,s,C)$ such that  for every
	$\eta\in[\eta_*,\eta_*^{1/3}]$ 	and every $t\in[0,T]$, we have
	\begin{align}
		\label{conclth1} \M( X_N^t, \mu^t,\eta) \leq \frac{\eta_*}{\eta} \exp(C'\A_T^2(T+1)),
	\end{align}
	\begin{align}
		\label{conclth1-wn}\mathsf{W}( X_N^t, \mu^t,\eta_*^{1/2}) \leq  \exp (C'\A_T^2(T+1))\left( \mathsf{W}\Big( X_N^0, \mu^0,\frac{1}{2}\eta_*^{1/2}\Big)+ \eta_*\right),
	\end{align}
	and
	\begin{align}	\label{conclth1-con}
	\frac{1}{N}\sup_{i\in [1,N]}\sum_{ j=1,j\neq i}^N\frac{\mathbf{1}_{|x_i^t-x_j^t|\leq 2\eta_*}}{|x_i^t-x_j^t|^{d-\theta}}\le 1.
	\end{align}
	In particular, $\mu_N^t\to\mu^t$ as $N\to\infty$ in the sense of probability measures, uniformly for all $t\in[0,T]$.
	\end{theo}

Let us now turn to the case $s= d-1$, which now only needs to be considered for $d \ge 3$. 

\begin{theo}[The Coulomb case $s=d-1$, $d\geq 3$]\label{th2}
	Assume $d\ge 3$, and that the force field $\F$ satisfies
	\eqref{assumpFre}--\eqref{coudet} for $s=d-1$. 
	Let $\mu$ be a regular solution of \eqref{mflimit} on a time interval $[0,T]$ such that
	\[
	\A(\mu^t)<\infty \quad \text{for all } t\in[0,T],
	\] 
	where $\A$ is defined in \eqref{defA} and let $\A_T:= \sup_{t\in [0,T]} \A(\mu^t) <\infty$.
	
	Assume that at $t=0$, we have
	\begin{align} \label{condt01}
		&\qquad \inf_{i \neq j} |x_i^0-x_j^0|\ge \frac1C \eta_* ,\\
		&\label{Cou-initial-high}
			\sup_{\eta\in [ \eta_*, \eta_*^{1/3}]}\eta \M( X_N^0, \mu^0,\eta)\leq C\eta_{*},
	\end{align}
	for some $C>0$.
	Then there exist a time $T_0=T_0(d,C,\A_T)\in(0,T]$ and an integer
	$
	N_0=N_0(d,C,T_0,\A_{T})
	$ such that, for all $N\ge N_0$, the particle system \eqref{ode} admits a unique
	classical solution $X_N(t)$ on $[0,T_0]$.
Moreover, there exists a constant
$C'=C'(d,C,s)$
such that, for every $t\in[0,T_0]$ and every
$\eta\in[\eta_*,\eta_*^{1/3}]$,
	\begin{align}&
		\label{conclth2} \M( X_N^t, \mu^t,\eta) \le\frac{\eta_*}{\eta}\exp\Big(C'\A_{T}\sqrt{t\log(\eta/\eta_*)}+C'\A_{T}^2(1+T)\Big),\\&
		\label{conclth2-wn} \mathsf{W}\Big( X_N^t, \mu^t,\eta_*^{1/2}\Big) \le  \exp \big(C'\A_T(1+T)\big) \mathsf{W}\Big( X_N^0, \mu^0,\frac{1}{2}\eta_*^{1/2}\Big)\\& \qquad\qquad \qquad\qquad\quad + \eta_*(\log N)\exp(C'\A_{T}\sqrt{t(\log N)}+C'\A_{T}^2(1+T)),\nonumber
	\end{align}
	and
	\begin{align}
		\inf_{i \neq j} |x_i^t-x_j^t|\gtrsim\eta_*.
	\end{align}
 Thus $\mu_N^t\to \mu^t$ as $N \to \infty$ in the sense of probability measures for all $t \in [0,T_0]$.
\end{theo}
Here again, we are able to leverage on the short-time existence result for the limiting PDE to deduce the existence of classical solutions to the ODE system (see Appendix \ref{app1}).
Also, let us emphasize that $T_0$ is independent of $N$, but may be smaller than $T$, contrarily to the previous theorem. The convergence rate we obtain is stronger than that provided by the modulated energy method,  and  the $L^\infty$-based metric is stronger than the negative Sobolev-type metric of the modulated energy method, but this also requires a stronger assumption on the initial data.
\begin{remark}[Almost $C^2$-type approximation]
	At the scale $\eta=\eta_*^{1/3}$, the estimates
	\eqref{conclth1-cou}, \eqref{conclth1}, and \eqref{conclth2}
	give
	\[
\sup_{x\in\mathbb R^d}\langle x\rangle^\gamma
\bigl|\mu_{N,\eta}^t(x)-\mu^t(x)\bigr|
\;\le\;
C(\A,T)\,\eta^2
\exp\bigl(C\sqrt{t\log\eta}\bigr).
\]
	Thus, up to the subpolynomial correction $\exp\bigl(C\sqrt{t\log\eta}\bigr),$
	the mollified empirical measure converges to $\mu^t$ with the
	quadratic rate $\eta^2$. In this sense, the estimate may be viewed
	as an almost $C^2$-type approximation. This terminology refers only
	to the approximation rate and does not assert convergence in the
	$C^2$ topology.
\end{remark}
\subsection{Main results in the super-Coulomb regime}

We now discuss the super-Coulomb regime $s>d-1$.  In that regime, blow-up is possible in the case of an attractive interaction.
Formally, the mean-field limit of \eqref{ode}
is still \eqref{mflimit}, however in the super-Coulomb range, the singularity of $\F$ makes \eqref{mflimit}
ill-posed in general. 
For this reason, as discussed in \eqref{120},  in the regime $s\in (d-1,d+1)$ we will have to  work with the regularized mean-field equation
\begin{equation}
\label{mflimitsup}\partial_t \mu= - \div ((\F_{\eta_*} * \mu)\mu).\end{equation}
 We will prove (see Appendix \ref{app:existence})  that  for $s\in (d-1,d+1)$, \eqref{mflimitsup} is locally well-posed on the natural timescale $T \sim \eta_*^{s+1-d}$.

\begin{lem}[Local well-posedness for the regularized equation]\label{Le-localwp} Let $s\in (d-1,d+1)$ and $d\ge 2$. 	Assume that the force field $\F$ satisfies
	\eqref{assumpFre}--\eqref{assumpF2'},  and let the initial datum satisfy $\A(\mu^0)<\infty.$
	There exist constants $T_0>0$ and $C>0$ such that the equation \eqref{mflimitsup}
	admits a unique solution on the time interval $[0,T_0\eta_*^{s+1-d}]$.
	Moreover, the solution satisfies the a priori bound
	\[ \A(\mu^t) \le C \A(\mu^0),
	\qquad
	t\in[0,T_0\eta_*^{s+1-d}].
	\]
\end{lem}

\begin{remark}
	The ill-posedness discussed above is specific to the attractive super-Coulomb
	interaction. In contrast, the limiting equation \eqref{mflimit}
	is locally well-posed when the interaction kernel is repulsive or has sufficient
	cancellation. For instance, local well-posedness holds if
	\begin{equation}\label{eq:F_supercoulomb-gs2}
		\F(x,y)=\frac{\mathbb M(x-y)}{|x-y|^{s+1}},
	\end{equation}
	where $s<d+1$ and $\mathbb M$ is a constant skew-symmetric matrix, i.e.\ $\mathbb M\xi\cdot\xi=0$ for all $\xi\in\mathbb R^d$.
	In these cases, the leading singularity is divergence-free and does not produce
	concentration, which allows one to establish local well-posedness of the limiting equation.
\end{remark}
In the regime $s\ge d+1$, the term $\F_{\eta_*}*\mu$ is no longer even integrable, and we will need to use a cutoff of the interaction. 

\begin{theo}[Short-time mean-field limit in the super-Coulomb regime $s>d-1$]\label{th3'} Assume that $s>d-1$ and that the force field $\F$ satisfies \eqref{assumpFre}--\eqref{assumpF2'}. 
	Let $\mu$ be 
	\begin{itemize}
	\item 
	a solution of the regularized mean-field equation \eqref{mflimitsup}
	on the time interval $[0,\eta_*^{s+1-d}T]$ for some $T<\infty$ in the case  $s\in(d-1,d)$,
	\item 
$\mu^t\equiv\mu^0$ with $\tilde{\A}(\mu^0):= \A(\mu^0)+\|\langle x\rangle^{\gamma+1}\nab^4 \mu_0\|_{L^\infty}<\infty$ in the case $s\ge d$.
	\end{itemize}
Set
$
\A'_T
:=
\sup_{0\le t\le \eta_*^{s+1-d}T}\A(\mu^t)
+
\|\langle x\rangle^{\gamma+1}\nabla^4\mu^0\|_{L^\infty}\mathbf 1_{s\ge d}.
$
		Assume that the initial configuration satisfies 
\be\inf_{i \neq j}	|x_i^0-x_j^0|\ge \frac1C \eta_* \ee
and
\be			\sup_{\eta\in [ \eta_*, \eta_*^{1/3}]}\eta \M( X_N^0, \mu^0,\eta)\leq C\eta_{*} ,\ee
	for some $C>0$. 
Then there exist a rescaled time $T_0=T_0(d,s,C,\A'_T)\in(0,T]$ and an integer
$N_0=N_0(d,s,C,T_0,\A'_{T_0})$
such that, for all $N\ge N_0$, the particle system \eqref{ode} admits a unique
classical solution on the time interval
$
[0,\eta_*^{s+1-d}T_0].
$
Moreover,  there exists $C'=C'(d,s,C)$ such that for every $t\in[0,T_0]$ and every
$\eta\in[\eta_*,\eta_*^{1/3}]$, we have
	\begin{equation}\label{M_bound_supercoulomb}
		\M\!\left(
		X_N^{\eta_*^{s+1-d}t},\,
		\mu^{\eta_*^{s+1-d}t},\,
		\eta
		\right)
		\le
		 \frac{\eta_*}{\eta}
		\exp\Bigl(
		C'(\tilde \A_{T_*}^2 t+1)
		\Bigr),
	\end{equation}
	\begin{equation}\label{W_bound_supercoulomb}
		\mathsf{W}\left(
		X_N^{\eta_*^{s+1-d}t},\,
		\mu^{\eta_*^{s+1-d}t},\,
		\eta_*^{1/2}
		\right)
		\le
	\left(	\mathsf{W}\left(
		X_N^{0},\mu^{0},\frac{1}{2}\eta_*^{1/2}
		\right)
		+
		\eta_*\right)
		\exp \Bigl(
		C'(\tilde\A_{T_*}^2t+1)
		\Bigr)\,,
	\end{equation}
	and
	\begin{align}
		\inf_{i \neq j} |x_i^{\eta_*^{s+1-d}t}-x_j^{\eta_*^{s+1-d}t}|\gtrsim\eta_*.
	\end{align}
	\end{theo}
	This result will be shown to be sharp in the next section.
\subsection{Collision results}\label{sec:blowup}

We now show that collisions can occur for the attractive Coulomb and super-Coulomb cases. Remarkably, the mean-field convergence is maintained up to the first collision time, meaning that the discrete dynamics faithfully tracks the continuum solution even as particles collide.

We denote $x_{ij}:= x_i - x_j$.
The following theorem shows that Coulomb interactions may lead to particles collisions, even though weak 
 particle continuation solutions still exist in dimension $2$, as seen in Appendix~\ref{app1-weak}
(see also  Theorem \ref{th1-Coulomb}).\\\\
We assume that the singular part of the interaction kernel is attractive in the following sense: there exist
constants $c_0\geq 2$ and $\varepsilon_0>0$ such that
\begin{align}\label{att}
	c_0^{-1}|x-y|^{-s+1}
	\leq
	-\F^{\mathrm{sing}}(x,y)\cdot (x-y)
	\leq
	c_0 |x-y|^{-s+1},
\end{align}
for all $x,y\in\mathbb{R}^d$ satisfying $0<|x-y|\leq \varepsilon_0$.

\begin{assum}[Well-prepared colliding initial data]
	\label{ass:well-prepared-colliding}
	Let $\eta_*=N^{-1/d}$. Fix $c'\in[2,4]$ and set
	\begin{equation}\label{wp-mu0}
		\mu_0(x)=\Phi^{c'^2}(x).
	\end{equation}
	We assume that the initial configuration
	$X_N^0=(x_1^0,\ldots,x_N^0)$ is chosen as follows.
	
	First, we fix the background particles $(x_j^0)_{j=2,\ldots,N}$ with
	\begin{equation}\label{wp-background-grid}
		x_2^0=0,
		\qquad
		(x_j^0)_{j=2,\ldots,N}\subset (3\eta_*)\mathbb Z^d .
	\end{equation}
	Moreover, after removing the particle $x_2^0$, the remaining background
	particles are well prepared for $\mu_0$ in the sense that
	\begin{equation}\label{wp-background-empirical}
		\sup_{x\in\mathbb R^d}
		\langle x\rangle^{3d}
		\left|
		\frac1N\sum_{j=3}^N\Phi_{x_j^0}^{\eta}(x)
		-\mu_0(x)
		\right|
		\le
		C\frac{\eta_*}{\eta},
		\qquad
		\eta\in [N^{-1/d},N^{-1/(3d)}],
	\end{equation}
	where $C$ is independent of $N$ and $\eta$.
	
	Finally, we place the first particle close to $x_2^0$ by setting
	\begin{equation}\label{wp-close-pair}
		x_1^0=\varepsilon\eta_* e_1,
		\qquad
		0<\varepsilon\ll1.
	\end{equation}
	The parameter $\varepsilon$ is fixed sufficiently small, independent of
	$N$.
\end{assum}

\begin{theo}[Collisions for the Coulomb case]\label{cou}
	Let $d\ge 2$ and $s=d-1$. Assume that the force field $\F$ satisfies
	\eqref{assumpFre}--\eqref{coudet} and \eqref{att}. Let the initial data
	$X_N^0$ satisfy Assumption~\ref{ass:well-prepared-colliding}.
	Then, for sufficiently large $N$, the ODE system \eqref{ode} leads to a
	collision in finite time.
	
	More precisely, there exists a collision time $T^*>0$, independent of
	$N$, such that
	\begin{equation}\label{cou-collision-time}
		T^*
		\sim
		\left(
		\frac{\inf_{i\neq j}|x_{ij}^0|}{\eta_*}
		\right)^d
		\sim
		\varepsilon^d,
	\end{equation}
	and, for all $0\le t<T^*$,
	\begin{equation}\label{cou-collision-rate}
		\inf_{i\neq j}|x_{ij}^t|
		\sim
		\eta_*
		\left(1-\frac{t}{T^*}\right)^{1/d}.
	\end{equation}
	In particular,
	\begin{equation}
		\lim_{t\uparrow T^*}\inf_{i\neq j}|x_{ij}^t|=0.
	\end{equation}
	Moreover, the convergence stated in Theorem~\ref{th2}
	(resp. Theorem~\ref{th1-Coulomb} when $d=2$) holds on $[0,T^*)$.
\end{theo}
\begin{theo}[Collisions for the super-Coulomb case]\label{supercou}
	Let $s\in(d-1,\infty)$ and set $\eta_*=N^{-1/d}$.
	Assume that the force field $\F$ satisfies
	\eqref{assumpFre}--\eqref{assumpF2'} and \eqref{att}.
	Let the initial data $X_N^0=(x_1^0,\ldots,x_N^0)$ satisfy
	Assumption~\ref{ass:well-prepared-colliding}.
	
	Then, for $\varepsilon>0$ sufficiently small and $N$ sufficiently large,
	the ODE system \eqref{ode} leads to a collision in finite time.
	More precisely, there exists a rescaled collision time $T^*>0$,
	independent of $N$, such that
	\begin{equation}\label{supercou-rescaled-time}
		T^*
		\sim
		\left(
		\frac{\inf_{i\neq j}|x_{ij}^0|}{\eta_*}
		\right)^{s+1}
		\sim
		\varepsilon^{s+1}.
	\end{equation}
	The corresponding physical collision time $T_N$ satisfies
	\begin{equation}\label{supercou-physical-time}
		T_N=\eta_*^{s+1-d}T^*.
	\end{equation}
	Moreover, for all $0\le \tau<T^*$,
	\begin{equation}\label{supercou-collision-rate}
		\inf_{i\neq j}
		\left|
		x_{ij}^{\,\eta_*^{s+1-d}\tau}
		\right|
		\sim
		\eta_*
		\left(
		1-\frac{\tau}{T^*}
		\right)^{1/(s+1)}.
	\end{equation}
	In particular,
	\begin{equation}
		\lim_{\tau\uparrow T^*}
		\inf_{i\neq j}
		\left|
		x_{ij}^{\,\eta_*^{s+1-d}\tau}
		\right|
		=0.
	\end{equation}
	Finally, the convergence stated in Theorem~\ref{th3'} holds on the
	rescaled time interval $[0,T^*)$, equivalently on the physical time
	interval $[0,T_N)$.
\end{theo}

Note that  these results show that Theorems \ref{th2} and \ref{th3'} are sharp and that we provide here the collision rate $ \eta_*(1-t/T^*)^{1/(s+1)}$, where the balance is determined by the singular pairwise force in $r^{-s}$ and the background force in $O(1)$.
\subsection{Further comments about the case with noise}

	Consider the interacting particle system with independent Brownian noise
	\begin{equation}\label{qq}
		\left\{
		\begin{array}{ll}
			\displaystyle
			dx_i(t)
			=
			\frac1N
			\sum_{\substack{j=1\\ j\neq i}}^N
			\F(x_i(t),x_j(t))\,dt
			+
			\sqrt{2\sigma}\,dW_t^i,
			\\[0.5em]
			x_i(0)=x_i^0,
		\end{array}
		\right.
	\end{equation}
	where $(W_t^i)_{i=1}^N$ are independent standard Brownian motions and
	$\sigma>0$. At the level of laws, the joint density
	$f_N=f_N(t,x_1,\ldots,x_N)$ formally solves the Liouville--Fokker--Planck (or forward Kolmogorov)
	equation
	\[
	\partial_t f_N
	+
	\sum_{i=1}^N
	\nabla_{x_i}\cdot
	\left(
	f_N
	\frac1N
	\sum_{\substack{j=1\\ j\neq i}}^N
	\F(x_i,x_j)
	\right)
	=
	\sigma
	\sum_{i=1}^N
	\Delta_{x_i} f_N .
	\]
	Thus the addition of Brownian noise replaces the singular ODE dynamics
	by a parabolic equation with singular drift on the joint density $f_N$.
	
	The natural scaling of the linear parabolic equation
	\[
	\partial_t u+b\cdot\nabla u-\sigma\Delta u=0
	\]
	suggests that the critical spatial regularity for the drift is
	$B^{-1}_{\infty,\infty}$. Equivalently, one may measure this regularity
	through the heat semigroup by
	\[
	\|b\|_{B^{-1}_{\infty,\infty}}
	\sim
	\sup_{\tau>0}
	\tau^{1/2}
	\|e^{\tau\Delta}b\|_{L^\infty}.
	\]
	For the attractive singular interaction
	\[
	\F(x,y)\sim
	-\frac{x-y}{|x-y|^{s+1}},
	\]
	the singularity has size $|\F(x,y)|\sim |x-y|^{-s}$. Hence, for each
	fixed $y$, the condition
	$
	\F(\cdot,y)\in B^{-1}_{\infty,\infty}(\mathbb{R}^d)
	$
	is  exactly consistent with the threshold $
	s\leq 1 .
	$
	Indeed, near the singularity, the heat regularization satisfies
$	\|e^{\tau\Delta}|\cdot|^{-s}\|_{L^\infty}
	\sim
	\tau^{-s/2},$
for any $ \tau\in (0,1),$
	and therefore
	$\tau^{1/2}
	\|e^{\tau\Delta}|\cdot|^{-s}\|_{L^\infty}
	\sim
	\tau^{(1-s)/2}.$
	This quantity remains bounded as $\tau\to0$ if and only if $s\leq1$.
	
	This scaling suggests that Brownian noise may be compatible with  the
	critical singularity
	\[
	|\F^{\mathrm{sing}}(x,y)|
	\lesssim
	\frac1{|x-y|}+1,
	\qquad |x-y|\leq1 .
	\]
	This includes the two-dimensional Coulomb force, corresponding to
	$s=1$. For more singular attractive interactions, namely $s>1$, the
	drift is supercritical with respect to the parabolic scaling. Therefore,
	the fixed-particle scaling alone does not suggest  a general
	well-posedness theory for the $N$-particle Liouville--Fokker--Planck
	equation without additional structure, renormalization, or stronger
	assumptions.\\
	
	This criticality is consistent with several well-posedness theories for
	parabolic equations and stochastic differential equations with rough drifts.
	We refer, for instance, to Krylov-R\"ockner
	\cite{KrylovRockner2005} for SDEs with singular drifts under
	$L_t^qL_x^p$ assumptions, to Flandoli-Gubinelli-Priola
	\cite{FlandoliGubinelliPriola2010} and
	Beck-Flandoli-Gubinelli-Maurelli
	\cite{BeckFlandoliGubinelliMaurelli2014} for regularization by noise and
	stochastic transport/continuity equations with critical drift, and to
	Qian-Xi \cite{QianXi2019} for parabolic equations with singular
	divergence-free drifts in the critical space $BMO^{-1}$.\\
	
	Related questions for stochastic signed Coulomb systems have recently been
	studied by van Meurs--Peletier--Slangen
	\cite{vanMeursPeletierSlangen2025} and by van Meurs--Tardy
	\cite{vanMeursTardy2026}. These works show that collisions are a central
	feature of two-dimensional signed Coulomb dynamics. In particular, in the
	signed case, well-posedness beyond collisions requires a prescribed
	continuation mechanism, such as removing colliding pairs of opposite sign.
	
	These results support the relevance of the threshold $s=1$. However, the
	critical case $s=1$ lies at the borderline of the existing rough-drift
	well-posedness theories. The cited results do not directly yield a full
	well-posedness theory for the particle system \eqref{qq} at this critical
	singularity. It would therefore be interesting to develop a well-posedness
	theory for \eqref{qq} in the critical case $s=1$, possibly exploiting the
	special structure of the interacting-particle drift or an appropriate
	collision/selection mechanism.
\medskip

{\bf Acknowledgements.} The research of Quoc-Hung Nguyen was supported by the 
CAS Project for Young Scientists in Basic Research, Grant No.~YSBR-031; and the NSFC under Grant Nos.~1251101538 and 12595282. The research of Sylvia Serfaty  was supported by NSF Grant DMS-2247846, the Simons Foundation through the Simons Investigator program, the Institut Universitaire de France and the Fondation Sciences Mathématiques de Paris.

\section{Preliminary results and kernel estimates} \label{sec:prelim}

This section presents preliminaries on the kernel and their regularizations.
In \S\ref{sec:gaussest} we collect the Gaussian integral estimates and scale-comparison inequalities for the metric $\M$ that will be used throughout.
In \S\ref{sec:kernelest} we establish the key pointwise and integral bounds on the mollified force kernel  and its truncation, distinguishing systematically between the sub-Coulomb ($s<d-1$), Coulomb ($s=d-1$), and super-Coulomb ($s>d-1$) regimes.

\subsection{Kernel regularization and truncation}
Let us first recall some notation. We recall that we denote by $\star$ the usual convolution.
Given $\eta>0$, we split $\F$ into a mollified (smooth) part $\F_\eta$ and a remainder $\F_{<\eta}$ carrying the short-range singularity:
\be \label{defFeta}\F_\eta(x,y):=(\F(x,\cdot)\star \Phi^\eta)(y), \quad\F_{<\eta}(x,y):=\F(x,y)-\F_\eta(x,y).\ee
Intuitively, $\F_\eta$ is the force field seen at resolution $\eta$, and $\F_{<\eta}$ captures the scales below $\eta$.
We also define 
\begin{equation}\label{deF}
	\F^{\eta_*}(x,y):=\F^{\mathrm{reg}}(x,y)+\mathbf{1}_{s< d+1}\F^{\mathrm{sing}}(x,y)+\mathbf{1}_{s\geq d+1}\F^{\mathrm{sing}}(x,y)\big(1-\chi\big)\left(\frac{|x-y|}{\eta_*}\right).
\end{equation}
where, throughout the paper, $\chi$ denotes a fixed smooth nonnegative cutoff function supported in $[0,2]$ and equal to $1$ on $[0,1]$. In particular, $\F^{\eta_*}\equiv \F$ when $s<d+1$; the cutoff only modifies the kernel in the strongly singular regime $s\ge d+1$.\\
With this definition, the singularity of $\F$ at $x=y$ is removed at scale $\eta_*$
when $s\ge d+1$, while the kernel is left unchanged in the regime $s<d+1$. When
$d\leq s<d+1$, the corresponding singular convolution is understood
through the cancellation assumptions rather than by absolute
integrability.

\subsection{Gaussian estimates}\label{sec:gaussest}

Let us first gather a few preliminaries and computations that we will use repeatedly to control the error terms.
The following lemma says that Gaussian averages of polynomial weights $\langle x\rangle^b$ are equivalent to the weight at the center.
\begin{lem}\label{lem:gauss} For any $b\in \mr$  and $\varepsilon\in (0,\hal]$, we have
	\begin{equation}\label{z0}
		\int_{\mr^d}\langle x+\varepsilon z\rangle^b \Phi(z)\,dz\sim \langle x\rangle^b,
	\end{equation}
	where $\Phi$ denotes the Gaussian kernel.
	In particular,
		\begin{equation}\label{z0'}
		\int_{\mr^d}|z|^n\langle x+\varepsilon z\rangle^b \Phi^\eta(z)\,dz\sim \eta^n \langle x\rangle^b
	\end{equation}
	for any $b\in \mathbb{R},n\geq 0$ and $\varepsilon, \eta\in (0,1]$.
\end{lem}
\begin{proof}  We easily obtain a lower bound by observing that 
	\begin{equation*}
		\int_{\mr^d}\langle x+\varepsilon z\rangle^b \Phi(z)\,dz\gtrsim 	\int_{\mr^d}\mathbf{1}_{|x|+1\geq |z|}\langle x\rangle^b \Phi(z)\,dz\ge \int_{\mr^d}\mathbf{1}_{1\geq |z|}\langle x\rangle^b \Phi(z)\,dz
		\sim \langle x\rangle^b.	\end{equation*}
	On the other hand, we have
	\begin{align*}
		\int_{\mr^d}\langle x+\varepsilon z\rangle^b \Phi(z)\,dz&\lesssim \langle x\rangle^b+\int_{\mr^d} \mathbf{1}_{|x|\leq 2 \langle z\rangle} \langle x+\varepsilon z\rangle^b\exp\!\Big(-\frac{|z|^2}{16}\Big) \exp\!\Big(-\frac{|x|^2}{16}\Big)\,dz
		\\&	\lesssim \langle x\rangle^b+\int_{\mr^d}  \exp\!\Big(-\frac{|z|^2}{32}\Big) \exp\!\Big(-\frac{|x|^2}{32}\Big)\,dz,
	\end{align*}
which implies the upper bound. The relation \eqref{z0'} then follows by scaling and reducing to $|z|\ge \hal \eta$.
\end{proof}

The next lemma provides Hermite-type identities on the Gaussian.
\begin{lem}\label{lem:hermite}
The polynomial $|x_i-x|^m$ can be absorbed into the Gaussian at the cost of doubling the lengthscale:
	\be\label{prod2exp}
	\frac{|x_i-x|^m}{\eta^m}\Phi^{\eta}(x_i-x)\lesssim \Phi^{2\eta}(x_i-x),\ee
	Moreover, products of the form $(z-x)^{\otimes m}\partial^j\Phi_z^\eta(x)$ can be expressed as linear combinations of derivatives $\nabla^n\Phi_z^\eta(x)$. First, we have
	\begin{align}\label{z15}
		&	(z-x)_{j}\partial_{x_{k}}\Phi^\eta_z(x)=\eta^2\partial_{x_{j}}\partial_{x_{k}}\Phi^\eta_z(x)+\mathbf{1}_{j=k}\Phi^\eta_z(x).
	\end{align}
	In particular, for any matrix $A=(A_{jk})_{1\leq j,k\leq d}$,
	\begin{equation}\label{z15'}
		\sum_{j,k}	(z-x)_{j}A_{jk}\partial_{x_{k}}\Phi^\eta_z(x)=\eta^2 A:\nab^2\Phi^\eta_z(x)+\mathrm{Tr}(A)\Phi^\eta_z(x).
	\end{equation}
Let $I_d$ be the $d\times d$ identity matrix. For a tensor $\mathcal{T}$ of order $m$, define its normalized symmetrization by $ \operatorname{Sym}(\mathcal{T})_{i_1\ldots i_m} := \frac1{m!} \sum_{\sigma\in\mathfrak S_m} \mathcal{T}_{i_{\sigma(1)}\ldots i_{\sigma(m)}}.$ Then
\begin{align}\label{z15a}
	&(z-x)^{\otimes 2}\,\Phi_z^\eta(x)
	= \eta^4 \,\nabla_x^{\otimes 2}\Phi_z^\eta(x)
	+ \eta^2 I_d\,\Phi_z^\eta(x),\\&\label{z15b}
	(z-x)^{\otimes 3}\,\Phi_z^\eta(x)
	= \eta^6 \,\nabla_x^{\otimes 3}\Phi_z^\eta(x)
	+ 3\,\eta^4\,\operatorname{Sym}\bigl(I_d \otimes \nabla_x\Phi_z^\eta(x)\bigr),
\end{align}
	and more generally,	\begin{align}\label{z15''}
		(z-x)^{\otimes m} \Phi^\eta_z(x)=\eta^{2m}\nab^m\Phi^\eta_z(x)+\sum_{n=0}^{m-1}\eta^{m+n}A_{m-n}\otimes\nab^{n}\Phi^\eta_z(x),
	\end{align}
		where $A_{m-n}$ is a constant tensor of order $m-n$.
	\end{lem}

\begin{proof}
The first point comes from writing that for a constant $\lambda>0$ and any $m\ge 1$,
	$$\Big(\frac{|x_i-x|}{\sqrt{\lambda}\,\eta}\Big)^m\lesssim \exp\!\Big( \frac{|x_i-x|^2}{\lambda\eta^2}\Big),$$ thus
$$	\frac{|x_i-x|^m}{\eta^m}\Phi^{\eta}(x_i-x)\lesssim \frac{1}{\eta^d} \exp\!\Big( |x_i-x|^2\Big( - \frac{1}{4\eta^2}+ \frac{1}{\lambda\eta^2}\Big)\Big) \lesssim \Phi^{2\eta}(x_i-x),$$
	 after picking $\lambda$ appropriately.
	The identities \eqref{z15}--\eqref{z15''} are direct computations.

\end{proof}

\subsection{Comparison of scales}

We recall that $\mu_{N,\eta}$ is defined by \eqref{fneta}, and we denote
\be \tilde \mu_{N,\eta}:=\mu_{N,\eta}-\mu.\ee
With the semi-group property \eqref{semigroup}, we may relate the scales $\eta$ and $\eta/2$ by writing
\be \label{defc0}\mu_{N,\eta}= \Phi^{\sqrt{\eta^2 - \eta^2/4}} \star \mu_{N, \eta/2} = \Phi^{\kappa_0 \eta} \star \mu_{N, \eta/2}, \quad \kappa_0=\frac{\sqrt{3}}{2}.\ee
This allows us to express $\mu_{N,\eta}$ as a further smoothing of $\mu_{N,\eta/2}$, which is essential for separating the error $\tilde\mu_{N,\eta/2}$ from the smooth background $\mu$ in the estimates below.
Also, for any $\eta_1, \eta_2$,
\be
\label{eta1eta2}
\Phi^{\eta_1}\star \mu_{N,\eta_2}=\Phi^{\eta_2}\star \mu_{N,\eta_1}.\ee

We will repeatedly use the triangle inequality 
\be\label{triangleU}
  \M(\XN, 0, \eta)\le \M(\XN, \mu, \eta)+C \A(\mu),
\ee
which follows from the definition~\eqref{defM}.

\begin{lem}\label{lem2.3}
For $m=0,1$, and for any function $f$, we have
	\begin{equation}\label{Dphi}
		\big\|\xg \nabla^m (f-f\star \Phi^{\eta})\big\|_{L^\infty}
		\;\lesssim\; \eta^2\A(f),
	\end{equation}
	where $\A$ is as in \eqref{defA}.
		For any integer $m\ge 0$, letting $\mu_{N,\eta}$ be as in \eqref{fneta} and $\M$ as in \eqref{defM}, we have
	\be
	\label{d2m}
\|\xg \nab^m  \mu_{N,\eta}\|_{L^\infty} \lesssim \eta^{-m} \M(X_N, \mu,\eta/2)+ \|\xg \nab^m \mu\|_{L^\infty}.
	\ee
	The metric $\M$ is essentially monotone in $\eta$: if $0<\eta\le \eta'\leq 1$, then
	\begin{equation}\label{Mee}
		\M(X_N, 0,\eta')\lesssim  \M(X_N, 0,\eta)\leq\Big( \frac{\eta'}{\eta}\Big)^d \M(X_N, 0,\eta'),
	\end{equation}
	and
	\be \label{Meta1eta2}
	\M(X_N, \mu, \eta') \lesssim \M(X_N, \mu, \eta)+(\eta')^2 \A(\mu).\ee
Finally, \eqref{retourMM} holds.

\end{lem}

\begin{proof}

By Taylor expansion of $f$, we have 
	\begin{align}\notag
		f(x)-f\star \Phi^\eta(x)
		&=\int \big(f(x)-f(x+z)\big)\Phi^\eta(z)\,dz
		\\ \label{summu}
		&= - \int z\cdot\nab f(x)\Phi^\eta(z)\,dz-\int_{0}^{1}(1-\tau) \int z^{\otimes 2}:\nab^2f (x+\tau z)\Phi^\eta(z)\, dz\, d\tau.
	\end{align}
The first term vanishes by the oddness of $z\Phi^\eta(z)$.
Therefore,
	\begin{equation}
		|f(x)-f\star \Phi^\eta(x)| \le \int_{0}^{1}	\int |z|^2\langle x+\tau z\rangle^{-\gamma}\Phi^\eta(z)\, dz\, d\tau\, \|\xg \nab^2f\|_{L^\infty}\overset{\eqref{z0'}}{\lesssim} \langle x\rangle^{-\gamma}\eta^2 \|\xg \nab^2f\|_{L^\infty}.
	\end{equation}
The result for $m=0$ follows. The case $m=1$ is similar. \\
Using \eqref{defc0}, \eqref{semigroup}, we may write
	\begin{align}\notag
		\Big|\frac{1}{N}\sum_{i=1}^N      \nab^m  \Phi_{x_i}^{\eta}(x)\Big|=	| \nab^m \mu_{N,\eta} (x)|&=	| \nab^m \Phi^{\kappa_0\eta}\star \mu_{N,\eta/2}(x) |\\ \notag &\leq
			|\nab^m\Phi^{\kappa_0\eta}\star \tilde\mu_{N,\eta/2}(x) |+ 	|\nab^m\Phi^{\kappa_0\eta}\star \mu(x) |.\end{align}
			Thus,
				\begin{multline*}\xg\Big|\frac{1}{N}\sum_{i=1}^N      \nab^m  \Phi_{x_i}^{\eta}(x)\Big| \lesssim \M(X_N, \mu,\eta/2) \int_{\mr^d} |\nab_x^m\Phi^{\kappa_0\eta}|(x-y) \frac{\xg}{\langle y\rangle^\gamma}\,dy \\
		+\|\xg \nab^m \mu\|_{L^\infty} \int_{\mr^d} \Phi^{\kappa_0\eta}(x-y) \frac{\xg}{\langle y\rangle^\gamma}\,dy.\end{multline*}
	Using the Gaussian derivative bound  \eqref{prod2exp},  we obtain
		\begin{multline*}\xg\Big|\frac{1}{N}\sum_{i=1}^N      \nab^m  \Phi_{x_i}^{\eta}(x)\Big|\lesssim (\eta^{-m}\M(X_N, \mu,\eta/2)+\|\xg \nab^m \mu\|_{L^\infty} ) \int_{\mr^d} \Phi^{2\kappa_0\eta}(x-y) \frac{\xg}{\langle y\rangle^\gamma}\,dy,\end{multline*}
		The desired estimate \eqref{d2m} follows from \eqref{z0}.

For the left-hand inequality in \eqref{Mee}, by the semi-group property we have
\be \label{semigroup2}
\M(X_N^t, 0, \eta') \lesssim \M(X_N^t, 0, \eta)\quad \text{if} \ \eta'\ge \eta.\ee
Indeed, by \eqref{semigroupmu},
\begin{align}\notag
	 \xg\mu_{N, \eta'}(x)&= \xg\Phi^{\sqrt{\eta'^2-\eta^2}}\star\mu_{N, \eta}(x)\\&\leq \M(X_N, 0, \eta)\,\xg(\Phi^{\sqrt{\eta'^2-\eta^2}}\star\langle \cdot\rangle^{-\gamma})(x)\overset{\eqref{z0}}{\sim}   \M(X_N, 0, \eta).\notag
\end{align}
		The right-hand inequality in \eqref{Mee} follows from the fact that by monotonicity of the exponential, if $\eta \le \eta'$,  we have
		$$ e^{-|x|^2/(4\eta^2) } \le e^{-|x|^2/(4(\eta')^2)},$$
		thus
		$$\Phi^{\eta}(x)\le \Big(\frac{\eta'}{\eta}\Big)^d  \Phi^{\eta'}(x),$$
and $\mu_{N,\eta}\le (\eta'/\eta)^d \mu_{N,\eta'}$.

The relation \eqref{Meta1eta2} is proven by observing that, by \eqref{semigroup}, we have
	$$\mu_{N,\eta'}-\mu = \Phi^{\sqrt{{\eta'}^2-\eta^2}} \star (\mu_{N,\eta}-\mu) + \Phi^{\sqrt{{\eta'}^2-\eta^2}}\star \mu - \mu,$$
	 and using \eqref{Dphi}.	
	
	We finish with the proof of \eqref{retourMM}, i.e.~we need to prove that $ \M_0 \lesssim 2^{dm}\M_m +2^{dm}\A(f)$.
	Note that 
	$$\M_m(\XN,f)= \sum_{n=0}^{\mathbf{n}-m} 2^{\sigma n} \M(\XN,f,2^{n+m}\eta_*)=2^{-\sigma m}\sum_{n=m}^{\mathbf{n}} 2^{\sigma n} \M(\XN,f,2^{n}\eta_*).$$
For $m, m'$ integers with  $m'< m$, using \eqref{Mee} we have
\begin{align*}
\M(\XN,f,2^{m'}\eta_*)&\lesssim \M(\XN,0,2^{m'}\eta_*)+\|\langle\cdot\rangle^\gamma f\|_{L^\infty}\\&\lesssim 2^{d(m-m')}\M(\XN,0,2^{m}\eta_*)+\|\langle\cdot\rangle^\gamma f\|_{L^\infty}
\\&\lesssim 2^{d(m-m')}\M(\XN,f,2^{m}\eta_*)+2^{d(m-m')}\|\langle\cdot\rangle^\gamma f\|_{L^\infty}\\
& \lesssim 2^{d(m-m')}\M_m(\XN,f)+2^{d(m-m')}\|\langle\cdot\rangle^\gamma f\|_{L^\infty}.
\end{align*}
Multiplying by $2^{\sigma m'}$ and summing over $m'\le m-1$, we find 
$$\M_0(\XN, f) \lesssim 2^{\sigma m} \M_m(\XN, f) +  \sum_{m'=0}^{m-1} 2^{\sigma m'} 2^{d(m-m')}\M_m(\XN,f)+2^{\sigma m'+d(m-m')}\|\langle\cdot\rangle^\gamma f\|_{L^\infty}.$$
The relation \eqref{retourMM} follows.

\end{proof}

\begin{remark}
From \eqref{prod2exp} we easily get $\xg |\nab ^m \mu_{N,\eta}|\le C \eta^{-m} \M(X_N^t, 0, \eta/2)+ C \|\xg \nab^m \mu\|_{L^\infty}$; however, this is not sufficient because we need a control by the small quantity $\M(X_N^t, \mu^t,\eta/2)$.
\end{remark}

\subsection{Pointwise and integral kernel estimates}\label{sec:kernelest}

We now turn to the key analytic estimates on the mollified force kernel $\F_\eta^{\eta_*}$ and its truncation $\F_{<\eta}^{\eta_*}$. Throughout this subsection, the bounds are stated uniformly in the singularity parameter $s$, with indicator functions $\mathbf{1}_{s\leq d-1}$, $\mathbf{1}_{s=d+1}$, etc.\ selecting the relevant contributions in each regime. On a first reading, the reader may wish to focus on the sub-Coulomb case $s < d-1$, where many terms vanish and the estimates take their simplest form.

We  will use a compact notation for the  recurring prefactors 
\begin{align}
	&\label{defCeta}
	\mathcal C(\eta,\eta_*):= \mathbf{1}_{s= d+1}\!\log\frac{2\eta}{\eta_{*}}+\mathbf{1}_{s>d+1}\eta_*^{-s+d+1},\\&
	\label{defCeta'}	\mathcal C(\eta_*):=
		\mathbf{1}_{s<d+1}
		+ \mathbf{1}_{s=d+1}|\log\eta_*|
		+ \mathbf{1}_{s>d+1}\eta_*^{-s+d+1}.
\end{align}

\begin{lem}\label{lem2.4}
	Assume \eqref{assumpFre}--\eqref{assumpF2'}. Let $g$ satisfy $	\A(g)+
	\big\|\xg \nabla^4 g\big\|_{L^\infty}	\mathbf{1}_{s\geq d}<\infty.$ Then, for any $\eta\in[\eta_*,1]$, and any $m=0,1,2,3$, the following estimates hold.

	\begin{enumerate}[label=(\roman*)]
		\item \textbf{Pointwise truncated and regularized kernel bounds.} For all $x,y\in\mathbb R^d$, and $m \le 3$
		\begin{align}\label{proF0}
			\big|\nabla_x^m \F_{\eta}^{\eta_*}(x,y)\big|
			&\lesssim
			\frac{
				1
				+ \mathcal C(\eta, \eta_*)
			}{
				(\eta+|x-y|)^{\min\{s,d+1\}+m}
			}  \\
			&\quad
			+ \big(1+(|x|+|y|)\mathbf{1}_{m=0}\big), \nonumber
		\end{align}
		\item \textbf{Truncated kernel estimate.}
		\begin{align}\label{proF1}
			\big|\F_{<\eta}^{\eta_*}(x,y)\big|
			&\lesssim
			\eta^2
			+\frac{
				1
				+ \mathcal C(\eta, \eta_*)
			}{
				|x-y|^{\min\{s,d+1\}}
			}
			\min\!\Big(1,\frac{\eta}{|x-y|}\Big)^2,
		\end{align}

		\item \textbf{Symmetrized bound.}
		\begin{equation}\label{proF2}
			\big|
			\F_{<\eta}^{\eta_*}(x,y)
			+\F_{<\eta}^{\eta_*}(y,x)
			\big|
			\lesssim
	\eta^2
			+\frac{	1
				+ \mathcal C(\eta, \eta_*)}{|x-y|^{\min\{s-1,d\}}}
			\min\!\Big(1,\frac{\eta}{|x-y|}\Big)^2,
		\end{equation}
		\item \textbf{Convolution estimate.}
		\begin{align}\label{proF3}
			\big|
			\F_{<\eta}^{\eta_*}\ast g(x)
			\big|
			&\lesssim
			\eta^2\mathcal C(\eta_*)
			\A(g),
		\end{align}
		\item \textbf{Gradient convolution estimate.}
		\begin{align}\label{proF4}
			\big|
			\nabla \F_{<\eta}^{\eta_*}\ast g(x)
			\big|
			&\lesssim \eta^2\mathcal C(\eta_*) 	(	\A(g)+
			\big\|\xg \nabla^4 g\big\|_{L^\infty}	\mathbf{1}_{s\geq d}).
		\end{align} 
	\end{enumerate}
\end{lem}

\begin{proof}
For the whole proof, we let $r:=|x-y|$, and again let $\chi$ be a cutoff function equal to $1$ in $[0,1]$ and vanishing outside $[0,2]$.

We decompose $$	\F^{\eta_*}=\F^{\mathrm{reg}}+\left(\mathbf{1}_{s< d+1}\F^{\mathrm{sing}}+\mathbf{1}_{s\geq d+1}\F^{\mathrm{sing}}\big(1-\chi\big)\left(\frac{|x-y|}{\eta_*}\right)\right):=\F^{\mathrm{reg}}+\F^{\mathrm{sing},\eta_*}$$ as in Section~\ref{sec:assump} and treat each part separately. Using \eqref{assumpFsing}--\eqref{assumpFre} and \eqref{Dphi},  it is straightforward to verify that, for $m=0,1,2,3$,
\begin{align}\label{q2}
	\big|\nabla_x^m \F_{\eta}^{\mathrm{reg}}(x,y)\big|
	&\lesssim 1+(|x|+|y|)\mathbf{1}_{m=0}, \\
	\label{q3}
	\big|\nabla_x^m \F_{<\eta}^{\mathrm{reg}}(x,y)\big|
	&\lesssim \eta^2 .
\end{align}

\noindent
\textbf{Step 1: Proof of \eqref{proF0}.}\\Assume first that $r>10\eta$. We decompose the convolution defining $\F_\eta^{\mathrm{sing}}$ into two regions:
\begin{align*}
	\nabla_x^m \F_{\eta}^{\mathrm{sing},\eta_*}(x,y)
	&=\int \nabla_x^m\F^{\mathrm{sing},\eta_*}(x,y+z)\Phi^\eta(z)
	\chi\!\left(\frac{2|z|}{r}\right)\,dz \\
	&\quad+\int \nabla_x^m\F^{\mathrm{sing},\eta_*}(x,y+z)\Phi^\eta(z)
	\left(1-\chi\right)\!\left(\frac{2|z|}{r}\right)\,dz \\
	&=: I_1+I_2 .
\end{align*}
By \eqref{assumpFsing}, we have
\begin{align*}
	|I_1|
	&\lesssim \int_{|z|<r/2}
	\left(\frac{\mathbf{1}_{s<d+1}}{|x-y-z|^{s+m}}+\frac{\mathbf{1}_{s\geq d+1}}{(|x-y-z|+\eta_{*})^{s+m}}+1\right)\Phi^\eta(z)\,dz  \\
	&\lesssim \frac{1}{(r+\mathbf{1}_{s\geq d+1}\eta_{*})^{s+m}}+1,
\end{align*}
since $|x-y-z|\gtrsim r$ on the support of the integrand.

We next estimate $I_2$. Set
\[
\Psi(z):=\Phi^\eta(z)\left(1-\chi\right)\!\left(\frac{2|z|}{r}\right).
\]
Using \eqref{assumpF2G} and \eqref{assumpF>=d+1} together with integration by parts, we obtain
\begin{align*}
	|I_2|
	&\lesssim 
	\left|\int
	\left(\nabla_x^m+(-1)^{m+1}\nabla_y^m\right)
	\F^{\mathrm{sing},\eta_*}(x,y+z)\Psi(z)\,dz\right| \\
	&\quad+
	\left|\int
	\F^{\mathrm{sing},\eta_*}(x,y+z)\nabla_z^m\Psi(z)\,dz\right|  \\
	&\lesssim
	\mathbf{1}_{m\geq 1} 
	\int_{|z|\geq r/2}
	\left(1+\frac{\mathbf{1}_{s<d+1}}{|x-y-z|^{s'}}+\frac{\mathbf{1}_{s\geq d+1}}{(|x-y-z|+\eta_{*})^{s-1}}\right)
	\sum_{n=0}^{m-1}|\nabla_z^n\Psi(z)|dz \\
	&\quad+
	\mathbf{1}_{s<d}
	\int_{|z|\geq r/2}
	\left(1+\frac{1}{|x-y-z|^{s}}\right)|\nabla_z^m\Psi(z)|dz \\
	&\quad+
	\mathbf{1}_{s\geq d}
	\int_{|z|\geq r/2}
	\left(1+\frac{\mathbf{1}_{s<d+1}}{r^{s}}+\frac{\mathbf{1}_{s\geq d+1}}{(r+\eta_{*})^s}\right)|\nabla_z^m \Psi(z)|dz \\
	&\quad+
	\mathbf{1}_{s\geq d} I_2',
\end{align*}
where $0<s'<\min\{s+1,d\}$ and
\begin{align*}
	I_2'
	:=
	\left|
	\int \F^{\mathrm{sing},\eta_*}(x,y+z)
	\nabla_z^m\Psi(z)
	\chi\!\left(\frac{|x-y-z|}{2r}\right)\,dz
	\right|.
\end{align*}
The term $I_2'$ isolates the contribution near the singular set $z=x-y$. When $s\ge d$, the singularity is not locally integrable by absolute-value estimates alone, and we must therefore use the cancellation assumption on $\F^{\mathrm{sing},\eta_*}$ to control this term.\\
Since
\begin{align}\label{esgau1}
	|\nabla_z^n\Psi(z)|
	\lesssim 
	e^{-\frac{|z|^2}{4\eta^2}}
	\frac{1}{\eta^d}
	\left(\eta^{-n}+r^{-n}\right)\mathbf{1}_{|z|\geq r/2}\lesssim e^{-\frac{|z|^2}{8\eta^2}}
	r^{-d-n}\mathbf{1}_{|z|\geq r/2},
\end{align}
we deduce
\begin{align*}
	|I_2|
	&\lesssim 
	\mathbf{1}_{m\geq 1}
	\int_{|z|\geq r/2}
	\left(1+\frac{\mathbf{1}_{s<d+1}}{|x-y-z|^{s'}}+\frac{\mathbf{1}_{s\geq d+1}}{(|x-y-z|+\eta_{*})^{s-1}}\right)
e^{-\frac{|z|^2}{8\eta^2}}
r^{-d-m+1}dz \\
	&\quad+
	\mathbf{1}_{s<d}
	\int_{|z|\geq r/2}
	\left(1+\frac{1}{|x-y-z|^{s}}\right)
e^{-\frac{|z|^2}{8\eta^2}}
r^{-d-m}dz \\
	&\quad+
	\mathbf{1}_{s\geq d}	\int_{|z|\geq r/2}
		\left(1+\frac{\mathbf{1}_{s<d+1}}{r^{s}}+\frac{\mathbf{1}_{s\geq d+1}}{(r+\eta_{*})^s}\right)
	e^{-\frac{|z|^2}{8\eta^2}}
	r^{-d-m}dz \\
	&\quad+
	\mathbf{1}_{s\geq d}I_2' \\&
		\lesssim \frac{\mathbf{1}_{s<d+1}}{r^{s+m}}+\left(\eta_{*}^{d+1-s}\mathbf{1}_{s>d+1}+\log(\frac{2\eta}{\eta_{*}})\mathbf{1}_{s=d+1}\right)r^{-d-m+1}+ \frac{\mathbf{1}_{s\geq d+1}}{(r+\eta_{*})^sr^{m}}+1
	+\mathbf{1}_{s\geq d}I_2',
\end{align*}
where we used the fact that 
\begin{align*}
		\int_{ r/2\leq |z|\leq 2r} \frac{\mathbf{1}_{s\geq d+1}}{(|x-y-z|+\eta_{*})^{s-1}}e^{-\frac{|z|^2}{8\eta^2}}dz\lesssim \eta_{*}^{d+1-s}\mathbf{1}_{s>d+1}+\log(\frac{2\eta}{\eta_{*}})\mathbf{1}_{s=d+1}.
\end{align*}
It remains to estimate $I_2'$ in the case $s\geq d$. Using the coarea formula,
we write
\begin{align*}
	I_2'
	&=
	\left|
	\int_0^{2r}
	\int_{|x-y-z|=\rho}
	\F^{\mathrm{sing},\eta_*}(x,y+z)
	\nabla_z^m\Psi(z)
	\,d\mathcal{H}^{d-1}(z)
	\chi\!\left(\frac{\rho}{2r}\right)\,d\rho
	\right|.
\end{align*}
By the cancellation property \eqref{assumpF2'}, for every $\rho>0$,
\begin{align*}
	\left|
	\int_{|x-y-z|=\rho}
	\F^{\mathrm{sing},\eta_*}(x,y+z)
	\,d\mathcal{H}^{d-1}(z)
	\right|
	\lesssim \rho^{d-1}(\rho^{1-s}+1)(\mathbf{1}_{s<d+1}+\mathbf{1}_{\rho\geq \frac{\eta_{*}}{2}}\mathbf{1}_{s\geq d+1}).
\end{align*}
Hence,
\begin{align*}
	I_2'
	&\lesssim
	\int_0^{2r}
	\int_{|x-y-z|=\rho}
	|\F^{\mathrm{sing},\eta_*}(x,y+z)|
	\left|
	\nabla_z^m\Psi(z)-\nabla_z^m\Psi(x-y)
	\right|
	\,d\mathcal{H}^{d-1}(z)\,d\rho \\
	&\quad+
	\int_0^{2r}
\rho^{d-1}(\rho^{1-s}+1)(\mathbf{1}_{s<d+1}+\mathbf{1}_{\rho\geq \frac{\eta_{*}}{2}}\mathbf{1}_{s\geq d+1})
	|\nabla_z^m\Psi(x-y)|\,d\rho.
\end{align*}
By \eqref{esgau1}, we have
\begin{align*}
	\left|
	\nabla_z^m\Psi(z)-\nabla_z^m\Psi(x-y)
	\right|
	\lesssim
	\frac{	|x-y-z|}{r^{d+m+1}} e^{-\frac{r^2}{32\eta^2}}
\end{align*}
Consequently,
\begin{align*}
	I_2'
	&\lesssim
	e^{-\frac{r^2}{32\eta^2}}
	\int_0^{2r}
	\left[
	1+
	\rho^{-s}
	\left(
	\mathbf{1}_{s<d+1}
	+
	\mathbf{1}_{\rho\geq \eta_*/2}\mathbf{1}_{s\geq d+1}
	\right)
	\right]
	\frac{\rho^d}{r^{d+m+1}}
	d\rho \\
	&\quad+
	e^{-\frac{r^2}{32\eta^2}}
	\int_0^{2r}
(\rho^{d-s}+	\rho^{d-1})
	\left(
	\mathbf{1}_{s<d+1}
	+
	\mathbf{1}_{\rho\geq \eta_*/2}\mathbf{1}_{s\geq d+1}
	\right)
	r^{-d-m}\,d\rho.
\end{align*}
Evaluating these one-dimensional integrals gives
\begin{align*}
	I_2'
	&\lesssim 	e^{-\frac{r^2}{32\eta^2}}
	\left[
	r^{-m}+
	r^{-m-s}
	\mathbf{1}_{s<d+1}
	+
	r^{-d-m-1}(\mathbf{1}_{s= d+1}\log(2+\frac{r}{\eta_{*}})+
	\mathbf{1}_{s> d+1}\mathbf{1}_{r\geq \frac{\eta_{*}}{4}}\eta_*^{d+1-s})
	\right]
	\\
	&+
	e^{-\frac{r^2}{32\eta^2}}
	\left(r^{-m}+ r^{-m+1-s}
	\mathbf{1}_{s<d+1}
	+		r^{-d-m}(\mathbf{1}_{s= d+1}\log(2+\frac{r}{\eta_{*}})+
	\mathbf{1}_{s> d+1}\mathbf{1}_{r\geq \frac{\eta_{*}}{4}}\eta_*^{d+1-s})
	\right).
\end{align*}
Using also
\[
\log\!\left(2+\frac{r}{\eta_*}\right)
e^{-\frac{r^2}{100\eta^2}}
\lesssim
\log\!\left(2+\frac{\eta}{\eta_*}\right),
\]
we infer
\begin{align*}
	I_2'\lesssim r^{-m}+r^{-s-m}\mathbf{1}_{s<d+1}+ \log(\frac{2\eta}{\eta_{*}}) r^{-d-m-1}\mathbf{1}_{s=d+1}+ \eta_{*}^{d+1-s}r^{-d-m-1}\mathbf{1}_{r\geq \frac{\eta_{*}}{4}} \mathbf{1}_{s>d+1}.
\end{align*}
Combining the estimates for $I_1$, $I_2$, and $I_2'$, we conclude that, whenever
$r>10\eta$,
\begin{equation}\label{q4}
	|\nabla_x^m\F_{\eta}^{\mathrm{sing},\eta_*}(x,y)|
	\lesssim
	\frac{1+\mathcal{C}(\eta,\eta_{*})}{r^{\min\{s,d+1\}+m}}+1 .
\end{equation}
We now consider the complementary case $r\leq 10\eta$. Arguing as before, using the almost anti-symmetry assumption together with integration by parts, we obtain
\begin{align*}
	|\nabla_x^m\F_{\eta}^{\mathrm{sing},\eta_*}(x,y)|
	&\leq
	\left|
	\int
	\left(\nabla_x^m+(-1)^{m+1}\nabla_y^m\right)
	\F^{\mathrm{sing},\eta_*}(x,y+z)\Phi^\eta(z)\,dz
	\right| \\
	&\quad+
	\left|
	\int
	\F^{\mathrm{sing},\eta_*}(x,y+z)\nabla_z^m\Phi^\eta(z)\,dz
	\right|  \\
	&\lesssim
	\mathbf{1}_{m\geq 1}
	\int
	\left(1+\frac{\mathbf{1}_{s<d+1}}{|x-y-z|^{s'}}+\frac{\mathbf{1}_{s\geq d+1}}{(|x-y-z|+\eta_{*})^{s-1}}\right)
	\sum_{n=0}^{m-1}|\nabla_z^n\Phi^\eta(z)|\,dz \\
	&\quad+
	\mathbf{1}_{s<d}
	\int
	\left(1+\frac{1}{|x-y-z|^{s}}\right)
	|\nabla_z^m\Phi^\eta(z)|\,dz \\
	&\quad+
	\int
	\left(1+\frac{1}{\eta^{s}}\right)
	|\nabla_z^m\Phi^\eta(z)|\,dz
	+
	\mathbf{1}_{s\geq d} I_3 ,
\end{align*}
where $0<s'<\min\{s+1,d\}$ and
\begin{align*}
	I_3
	:=
	\left|
	\int
	\F^{\mathrm{sing},\eta_{*}}(x,y+z)
	\nabla_z^m\Phi^\eta(z)
	\chi\!\left(\frac{|x-y-z|}{2\eta}\right)\,dz
	\right|.
\end{align*}
Using the standard bounds on the derivatives of the mollifier, we infer
\begin{align}\nonumber
	|\nabla_x^m\F_{\eta}^{\mathrm{sing},\eta_*}(x,y)|
	&\lesssim
\mathbf{1}_{m\geq 1}	\int
	\left(1+\frac{\mathbf{1}_{s<d+1}}{|x-y-z|^{s'}}+\frac{\mathbf{1}_{s\geq d+1}}{(|x-y-z|+\eta_{*})^{s-1}}\right) 	\frac{1}{\eta^{d+m-1}}
	e^{-\frac{|z|^2}{4\eta^2}}dz\\& 
	+	\int
	\left(1+\frac{\mathbf{1}_{s<d}}{|x-y-z|^{s}}+\frac{1}{\eta^{s}}\right) 	\frac{1}{\eta^{d+m}}
	e^{-\frac{|z|^2}{4\eta^2}}dz +
	\mathbf{1}_{s\geq d} I_3 . \label{q7-pre}
\end{align}
Since $s'<d$ in the corresponding singular integral, while $r\leq 10\eta$, the Gaussian convolution gives
\begin{align}
	|\nabla_x^m\F_{\eta}^{\mathrm{sing},\eta_*}(x,y)|
	\lesssim
\frac{1+\mathcal{C}(\eta,\eta_{*})}{\eta^{\min\{s,d+1\}+m}}+1
	+
	\mathbf{1}_{s\geq d} I_3 .
	\label{q7}
\end{align}

It remains to estimate $I_3$ in the case 	$s\geq d$. As in the estimate of $I_2'$, using the coarea formula and the cancellation assumption \eqref{assumpF2'}, we get
\begin{align*}
	I_3
	&\lesssim
	\int_0^{2\eta}
\int_{|x-y-z|=\rho}
\left(1+\rho^{-s}(\mathbf{1}_{s<d+1}+\mathbf{1}_{\rho\geq \frac{\eta_{*}}{2}}\mathbf{1}_{s\geq d+1})\right)
\left|
\nabla_z^m\Phi^\eta(z)-\nabla_z^m\Phi^\eta(x-y)
\right|
\,d\mathcal{H}^{d-1}(z)\,d\rho \\
&\quad+
\int_0^{2\eta}
\rho^{d-1}(\rho^{1-s}+1)(\mathbf{1}_{s<d+1}+\mathbf{1}_{\rho\geq \frac{\eta_{*}}{2}}\mathbf{1}_{s\geq d+1})
|\nabla_z^m\Phi^\eta(x-y)|\,d\rho .
\end{align*}
Moreover, the smoothness and scaling of $\Phi^\eta$ imply
\begin{align*}
	\left|
	\nabla_z^m\Phi^\eta(z)
	-
	\nabla_z^m\Phi^\eta(x-y)
	\right|
	&\lesssim |x-y-z|\eta^{-d-m-1}, \\
	|\nabla_z^m\Phi^\eta(x-y)|
	&\lesssim \eta^{-d-m}.
\end{align*}
Therefore,
\begin{align*}
	I_3
	&\lesssim
	\int_0^{2\eta}
	\left(1+\rho^{-s}(\mathbf{1}_{s<d+1}+\mathbf{1}_{\rho\geq \frac{\eta_{*}}{2}}\mathbf{1}_{s\geq d+1})\right)\rho^d
\eta^{-d-m-1}d\rho \\
	&\quad+
	\int_0^{2\eta}
	\rho^{d-1}(\rho^{1-s}+1)(\mathbf{1}_{s<d+1}+\mathbf{1}_{\rho\geq \frac{\eta_{*}}{2}}\mathbf{1}_{s\geq d+1})
 \eta^{-d-m}d\rho\\&\lesssim  \frac{1+\mathcal{C}(\eta,\eta_{*})}{\eta^{\min\{s,d+1\}+m}}+\eta^{-m}.
\end{align*}
Substituting this bound into \eqref{q7}, we obtain
\begin{align}\label{q8'}
	|\nabla_x^m\F_{\eta}^{\mathrm{sing},\eta_*}(x,y)|
	\lesssim
\frac{1+\mathcal{C}(\eta,\eta_{*})}{\eta^{\min\{s,d+1\}+m}}+1.
\end{align}
Since $r\leq 10\eta$, this is consistent with the desired estimate at scale
$\max\{r,\eta\}$. Combining this case with \eqref{q2} and the estimate for
$r>10\eta$ proves \eqref{proF0}.\\

\noindent
\textbf{Step 2: Proof of \eqref{proF1}.}

We distinguish two regimes according to the relative size of $r$ and $\eta$.

\medskip
\noindent\emph{The near-field regime: $r\leq 10\eta$.}
By \eqref{q4}, \eqref{q8'} with $m=0$, and the pointwise assumption
\eqref{assumpFsing}, we directly obtain
\begin{equation}\label{z7}
	\left|\F_{<\eta}^{\mathrm{sing},\eta_*}(x,y)\right|
	\lesssim
	\frac{1+\mathcal C(\eta,\eta_*)}{r^{\min\{s,d+1\}}}+1 .
\end{equation}

\medskip
\noindent\emph{The far-field regime: $r>10\eta$.}
Arguing as in \eqref{summu}, Taylor's formula gives
\begin{align}\label{Fsingtaylor}
	\F_{<\eta}^{\mathrm{sing},\eta_*}(x,y)
	&=
	\F^{\mathrm{sing},\eta_*}(x,y)
	-
	\F_{\eta}^{\mathrm{sing},\eta_*}(x,y) \nonumber \\
	&=
	-\int_0^1(1-\tau)
	\int z^{\otimes 2}:
	\nabla_z^2\F^{\mathrm{sing},\eta_*}(x,y+\tau z)
	\Phi^\eta(z)\,dz\,d\tau .
\end{align}
Let $\chi$ be the cutoff used above. Splitting the integral according to the
region where $|x-y-\tau z|$ is comparable to $r$, and integrating by parts on
the complementary part, we obtain
\begin{align}\label{z8}
	\F_{<\eta}^{\mathrm{sing},\eta_*}(x,y)
	&=
	-\int_0^1(1-\tau)
	\int z^{\otimes 2}:
	\nabla_z^2\F^{\mathrm{sing},\eta_*}(x,y+\tau z)
	\Phi^\eta(z)
	\left(1-\chi\right)
	\left(\frac{4|x-y-\tau z|}{r}\right)
	\,dz\,d\tau
	\nonumber \\
	&\quad
	-\int_0^1(1-\tau)\tau^{-2}
	\int
	\F^{\mathrm{sing},\eta_*}(x,y+\tau z)
	\nabla_z^2
	\left[
	z^{\otimes 2}\Phi^\eta(z)
	\chi\left(\frac{4|x-y-\tau z|}{r}\right)
	\right]
	\,dz\,d\tau .
\end{align}
We shall use the elementary bounds
\begin{align}\label{cutoff-bound-1}
	&\big|\nab_z^2\big(z^{\otimes 2}\Phi^\eta(z)\chi\big(\tfrac{4|x-y-\tau z|}{r}\big)\big)\big|\lesssim \Big(1+\frac{|z|^4}{\eta^4}+|z|^2\frac{\tau^2}{r^2}\Big)\Phi^\eta(z)  \mathbf{1}_{|x-y-\tau z|\leq r/2},\\&
	\big|\nab_z^2\big[z^{\otimes 2}\Phi^\eta(z)\chi\big(\tfrac{4|x-y-\tau z|}{r}\big)\big]-\nab_z^2\big[z^{\otimes 2}\Phi^\eta(z)\big]\chi\big(\tfrac{4|x-y-\tau z|}{r}\big)\big|\lesssim \Big(1+\frac{|z|^4}{\eta^4}\Big)\Phi^\eta(z)  \mathbf{1}_{r/8\leq |x-y-\tau z|\leq r/2},\label{cutoff-bound-2}
\end{align}
Using \eqref{assumpFsing}, \eqref{cutoff-bound-1}, and
\eqref{cutoff-bound-2}, we infer that
\begin{align}
	\left|\F_{<\eta}^{\mathrm{sing},\eta_*}(x,y)\right|
	&\lesssim
	\int_0^1
	\int
	|z|^2
	\left(
	1+
	\frac{1}{
		\left(
		|x-y-\tau z|
		+
		\eta_* \mathbf 1_{\{s\geq d+1\}}
		\right)^{s+2}}
	\right)
	\Phi^\eta(z)
	\mathbf 1_{\{|x-y-\tau z|>r/8\}}
	\,dz\,d\tau
	\nonumber \\
	&\quad
	+
	\mathbf 1_{\{s<d\}}
	\int_0^1\tau^{-2}
	\int
	\left(
	1+\frac{1}{|x-y-\tau z|^s}
	\right)
	\left(
	1+\frac{|z|^4}{\eta^4}
	+\frac{|z|^2\tau^2}{r^2}
	\right)
	\Phi^\eta(z)
	\mathbf 1_{\{|x-y-\tau z|\leq r/2\}}
	\,dz\,d\tau
	\nonumber \\
	&\quad
	+
	\int_0^1\tau^{-2}
	\int
	\left(
	1+\frac{1}{r^s}
	\right)
	\left(
	1+\frac{|z|^4}{\eta^4}
	+\frac{|z|^2\tau^2}{r^2}
	\right)
	\Phi^\eta(z)
	\mathbf 1_{\{r/8\leq |x-y-\tau z|\leq r/2\}}
	\,dz\,d\tau
	\nonumber \\
	&\quad
	+
	\mathbf 1_{\{s\geq d\}} I_4,
	\label{pre-I4}
\end{align}
where
\begin{align}\label{def-I4}
	I_4
	:=
	\left|
	\int_0^1(1-\tau)\tau^{-2}
	\int
	\F^{\mathrm{sing},\eta_*}(x,y+\tau z)
	\nabla_z^2\left(z^{\otimes 2}\Phi^\eta(z)\right)
	\chi\left(\frac{4|x-y-\tau z|}{r}\right)
	\,dz\,d\tau
	\right|.
\end{align}

We first estimate the three explicit terms in \eqref{pre-I4}. For the first
one, we split the domain into the two regions
\[
|x-y-\tau z|\geq 4r
\qquad\text{and}\qquad
r/8\leq |x-y-\tau z|\leq 4r .
\]
In the first region, $\tau |z|\geq 3r$, and hence
\[
|x-y-\tau z|
\geq \tau |z|-r
\geq \frac12(\tau |z|+r).
\]
In the second region, $\tau |z|\leq 5r$, and therefore
\[
|x-y-\tau z|
\geq \frac r8
\geq \frac1{64}(r+\tau |z|).
\]
Consequently,
\begin{align}
	&\int_0^1
	\int
	|z|^2
	\left(
	1+
	\frac{1}{
		\left(
		|x-y-\tau z|
		+
		\eta_* \mathbf 1_{\{s\geq d+1\}}
		\right)^{s+2}}
	\right)
	\Phi^\eta(z)
	\mathbf 1_{\{|x-y-\tau z|>r/8\}}
	\,dz\,d\tau
	\nonumber \\
	&\qquad\lesssim
	\int_0^1
	\int
	|z|^2
	\left(
	1+
	\frac{1}{
		\left(
		r+\tau |z|+\eta_* \mathbf 1_{\{s\geq d+1\}}
		\right)^{s+2}}
	\right)
	\Phi^\eta(z)\,dz\,d\tau
	\nonumber \\
	&\qquad\lesssim
	\eta^2
	\left(
	1+\frac1{r^{s+2}}
	\right).
	\label{first-term-estimate}
\end{align}

For the second term, the constraint $|x-y-\tau z|\leq r/2$ implies
$\tau\geq r/(2|z|)$. Hence, for $s<d$,
\begin{align}
	&\int_0^1\tau^{-2}
	\int
	\left(
	1+\frac1{|x-y-\tau z|^s}
	\right)
	\left(
	1+\frac{|z|^4}{\eta^4}
	+\frac{|z|^2\tau^2}{r^2}
	\right)
	\Phi^\eta(z)
	\mathbf 1_{\{|x-y-\tau z|\leq r/2\}}
	\,dz\,d\tau
	\nonumber \\
	&\qquad\lesssim
	\int_0^1
	\int
	\frac{|z|^2}{r^2}
	\left(
	1+\frac1{|x-y-\tau z|^s}
	\right)
	\left(
	1+\frac{|z|^4}{\eta^4}
	\right)
	\Phi^\eta(z)
	\mathbf 1_{\{|x-y-\tau z|\leq r/2\}}
	\,dz\,d\tau
	\nonumber \\
	&\qquad\lesssim
	\frac{\eta^2}{r^2}
	\left(
	1+
	\int_0^1
	\int
	\frac1{|x-y-\tau z|^s}
	\frac1{\eta^d}
	e^{-\frac{|z|^2}{4\eta^2}}
	\,dz\,d\tau
	\right)
	\nonumber \\
	&\qquad\lesssim
	\frac{\eta^2}{r^2}
	\left(
	1+
	\int_0^1
	\big(|\cdot|^{-s}*\Phi^{\tau\eta}\big)(x-y)
	\,d\tau
	\right)
	\nonumber \\
	&\qquad\lesssim
	\frac{\eta^2}{r^2}
	\left(
	1+\frac1{r^s}
	\right).
	\label{second-term-estimate}
\end{align}
Similarly, the third term is bounded by
\begin{align}
	&\left(1+\frac1{r^s}\right)
	\int_0^1
	\int 
	\frac{|z|^2}{r^2}
	\left(
	1+\frac{|z|^4}{\eta^4}
	\right)
	\Phi^\eta(z)
	\mathbf 1_{\{r/8\leq |x-y-\tau z|\leq r/2\}}
	\,dz\,d\tau
	\nonumber \\
	&\qquad\lesssim
	\frac{\eta^2}{r^2}
	\left(
	1+\frac1{r^s}
	\right).
	\label{third-term-estimate}
\end{align}
Combining \eqref{pre-I4}--\eqref{third-term-estimate}, we obtain
\begin{equation}\label{z9}
	\left|\F_{<\eta}^{\mathrm{sing},\eta_*}(x,y)\right|
	\lesssim
	\eta^2
	\left(
	\frac{1+\mathcal C(\eta,\eta_*)}{r^{\min\{s,d+1\}+2}}+1
	\right)
	+
	\mathbf 1_{\{s\geq d\}} I_4 .
\end{equation}

It remains to estimate $I_4$. We keep the original variable $z$ and apply the coarea formula directly to the function$z\longmapsto |x-y-\tau z|.$
Since $ \left| \nabla_z|x-y-\tau z| \right| =\tau,$ we have \begin{align} I_4 &\leq \int_0^1 (1-\tau)\tau^{-3} \int_0^{r/2} \left| \int_{|x-y-\tau z|=\rho} \F^{\mathrm{sing},\eta_*}(x,y+\tau z)  \nabla_z^2 \left( z^{\otimes2}\Phi^\eta(z) \right)d\mathcal H^{d-1}(z) \right| \,d\rho\,d\tau, \label{I4-direct-coarea} \end{align}
As in the estimate of $I_2'$, we use the
cancellation assumption \eqref{assumpF2'}. For every $\rho>0$,
\begin{align}\label{sphere-cancel}
	\left|
	\int_{|x-y-\tau z|=\rho}
	\F^{\mathrm{sing},\eta_*}(x,y+\tau z)
	\,d\mathcal H^{d-1}(z)
	\right|
	\lesssim
	\tau^{-d+1}\rho^{d-1}(\rho^{1-s}+1)
	\left(
	\mathbf 1_{\{s<d+1\}}
	+
	\mathbf 1_{\{\rho\geq \eta_*/2\}}\mathbf 1_{\{s\geq d+1\}}
	\right).
\end{align}
Moreover,
\begin{align}
	\left|
	\left.
	\nabla_z^2
	\left(z^{\otimes 2}\Phi^\eta(z)\right)
	\right|_{z=(x-y)/\tau}
	\right|
	&\lesssim
	\eta^{-d}
	\exp\left(-\frac{r^2}{4\eta^2\tau^2}\right),
	\label{gaussian-bound-1}
	\\
	\mathbf 1_{\{\tau\geq r/(2|z|)\}}
	\left|
	\nabla_z^2
	\left(z^{\otimes 2}\Phi^\eta(z)\right)
	-
	\left.
	\nabla_z^2
	\left(z^{\otimes 2}\Phi^\eta(z)\right)
	\right|_{z=(x-y)/\tau}
	\right|
	&\lesssim
	\tau^{-1}
	|x-y-\tau z|
	\eta^{-d-1}
	\exp\left(-\frac{r^2}{16\eta^2\tau^2}\right).
	\label{gaussian-bound-2}
\end{align}
Since
\[
\int_{|x-y-\tau z|=\rho} d\mathcal H^{d-1}(z)
\sim
\rho^{d-1}\tau^{-d+1},
\]
the coarea formula, \eqref{sphere-cancel}, \eqref{gaussian-bound-1}, and
\eqref{gaussian-bound-2} yield
\begin{align}
	I_4
	&\lesssim
	\int_0^1\tau^{-d-2}
	\int_0^{r/2}
	\rho^d
	\left(
	1+
	\rho^{-s}
	\left(
	\mathbf 1_{\{s<d+1\}}
	+
	\mathbf 1_{\{\rho\geq \eta_*/2\}}\mathbf 1_{\{s\geq d+1\}}
	\right)
	\right)
	\eta^{-d-1}
	e^{-\frac{r^2}{16\eta^2\tau^2}}
	\,d\rho\,d\tau
	\nonumber \\
	&\quad
	+
	\int_0^1\tau^{-d-1}
	\int_0^{r/2}
	\rho^{d-1}(\rho^{1-s}+1)
	\left(
	\mathbf 1_{\{s<d+1\}}
	+
	\mathbf 1_{\{\rho\geq \eta_*/2\}}\mathbf 1_{\{s\geq d+1\}}
	\right)
	\eta^{-d}
	e^{-\frac{r^2}{16\eta^2\tau^2}}
	\,d\rho\,d\tau .
	\label{I4-coarea-estimate}
\end{align}
Therefore,
\begin{align}
	I_4
	&\lesssim
	\left(
	r^{d+1}
	+
	r^{d+1-s}\mathbf 1_{\{s<d+1\}}
	+
	\log\left(2+\frac r{\eta_*}\right)\mathbf 1_{\{s=d+1\}}
	+
	\eta_*^{d+1-s}\mathbf 1_{\{s>d+1\}}
	\right)
	\eta^{-d-1}
	\int_0^1\tau^{-d-2}
	e^{-\frac{r^2}{16\eta^2\tau^2}}
	\,d\tau
	\nonumber \\
	&\quad
	+
	\left(
	r^d
	+
	r^{d+1-s}\mathbf 1_{\{s<d+1\}}
	+
	\log\left(2+\frac r{\eta_*}\right)\mathbf 1_{\{s=d+1\}}
	+
	\eta_*^{d+1-s}\mathbf 1_{\{s>d+1\}}
	\right)
	\eta^{-d}
	\int_0^1\tau^{-d-1}
	e^{-\frac{r^2}{16\eta^2\tau^2}}
	\,d\tau .
	\label{I4-before-gaussian-time}
\end{align}
For every $m>0$,
\begin{equation}\label{time-gaussian}
	\int_0^1 \tau^{-m-1}
	\exp\left(-\frac{r^2}{16\eta^2\tau^2}\right)
	\,d\tau
	\lesssim
	\frac{\eta^m}{r^m}
	\exp\left(-\frac{r^2}{32\eta^2}\right).
\end{equation}
Applying \eqref{time-gaussian} in \eqref{I4-before-gaussian-time}, we get
\begin{align}
	I_4
	&\lesssim
	\left(
	1
	+
	r^{-s}\mathbf 1_{\{s<d+1\}}
	+
	r^{-d-1}
	\log\left(2+\frac r{\eta_*}\right)\mathbf 1_{\{s=d+1\}}
	+
	\eta_*^{d+1-s}r^{-d-1}\mathbf 1_{\{s>d+1\}}
	\right)
	\exp\left(-\frac{r^2}{32\eta^2}\right)
	\nonumber \\
	&\lesssim
	\left(
	\frac{1+\mathcal C(\eta,\eta_*)}{r^{\min\{s,d+1\}}}
	+
	1
	\right)
	\exp\left(-\frac{r^2}{32\eta^2}\right).
	\label{I4-final}
\end{align}
Since we are in the regime $r>10\eta$, the exponentially small factor in
\eqref{I4-final} is harmless and is absorbed by the right-hand side of
\eqref{z9}. Thus,
\begin{equation}\label{far-field-final}
	\left|\F_{<\eta}^{\mathrm{sing},\eta_*}(x,y)\right|
	\lesssim
	\eta^2
	\left(
	1+	\frac{1+\mathcal C(\eta,\eta_*)}{r^{\min\{s,d+1\}+2}}
	\right),
	\qquad r>10\eta .
\end{equation}
Finally, combining the near-field estimate \eqref{z7}, the far-field estimate
\eqref{far-field-final}, and \eqref{q3}, we obtain \eqref{proF1}.\\
\noindent
\textbf{Step 3: Proof of \eqref{proF2}.}
If $r\le 10\eta$, we have 
\begin{align*}|\F_{<\eta}^{\mathrm{sing},\eta_{*}}(x,y)+\F_{<\eta}^{\mathrm{sing},\eta_{*}}(y,x)|&\leq \left|\int(\F^{\mathrm{sing},\eta_{*}}(x,y+z)+\F^{\mathrm{sing},\eta_{*}}(y,x-z)) \Phi^\eta(z)dz\right|\\&+|\F^{\mathrm{sing},\eta_{*}}(x,y)+\F^{\mathrm{sing},\eta_{*}}(y,x)|. \end{align*}
By \eqref{assumpF3c}, we have 
\begin{align}\nonumber|\F_{<\eta}^{\mathrm{sing},\eta_{*}}(x,y)+\F_{<\eta}^{\mathrm{sing},\eta_{*}}(y,x)|&\nonumber\leq \int \frac{\Phi^\eta(z)dz}{(|x-y-z|+\eta_{*}\mathbf{1}_{s\geq d+1})^{s-1}} \\&\nonumber+1+\frac{1}{(r+\eta_{*}\mathbf{1}_{s\geq d+1})^{s-1}} \\&\lesssim 
	\frac{1+\mathcal C(\eta,\eta_*)}{r^{\min\{s-1,d\}}}+1.\label{q12}\end{align}
For $r>10\eta$,  using \eqref{Fsingtaylor}, we may write
\begin{align*}
&\F_{<\eta}^{\mathrm{sing},\eta_{*}}(x,y)+\F_{<\eta}^{\mathrm{sing},\eta_{*}}(y,x)
\\	&=
	-\int_0^1(1-\tau)
	\int z^{\otimes 2}:
(	(\nabla_y^2\F^{\mathrm{sing},\eta_*})(x,y+\tau z)+	(\nabla_x^2\F^{\mathrm{sing},\eta_*})(y,x-\tau z))
	\Phi^\eta(z)\\&\quad\quad\quad\quad\times
	\left(1-\chi\right)
	\left(\frac{4|x-y-\tau z|}{r}\right)
	\,dz\,d\tau
	\nonumber \\
	&\quad
	-\int_0^1(1-\tau)\tau^{-2}
	\int
\left(	\F^{\mathrm{sing},\eta_*}(x,y+\tau z)+	\F^{\mathrm{sing},\eta_*}(y,x-\tau z)\right)
\\&\quad\quad\quad\quad\quad\quad\times	\nabla_z^2
	\left[
	z^{\otimes 2}\Phi^\eta(z)
	\chi\left(\frac{4|x-y-\tau z|}{r}\right)
	\right]
	\,dz\,d\tau .
\end{align*}
By  \eqref{assumpF3c} and \eqref{cutoff-bound-1},  we get, 
\begin{align*}
	&|\F_{<\eta}^{\mathrm{sing},\eta_{*}}(x,y)+\F_{<\eta}^{\mathrm{sing},\eta_{*}}(y,x)|
	\\	&\lesssim
	\int_0^1
	\int |z|^2(1+\frac{1}{(|x-y-\tau z|+\eta_{*}\mathbf{1}_{s\geq d+1})^{s+1}})
	\Phi^\eta(z)\mathbf{1}_{|x-y-\tau z|\geq r/8}dzd\tau
	\nonumber \\
	&
	+\int_0^1\tau^{-2}
	\int (1+\frac{1}{(|x-y-\tau z|+\eta_{*}\mathbf{1}_{s\geq d+1})^{s-1}})\Big(1+\frac{|z|^4}{\eta^4}+|z|^2\frac{\tau^2}{r^2}\Big)\Phi^\eta(z)  \mathbf{1}_{|x-y-\tau z|\leq r/2}dzd\tau .
\end{align*}
Arguing as in the proof of \eqref{z9}, we obtain
\begin{align*}
	|\F_{<\eta}^{\mathrm{sing},\eta_{*}}(x,y)+\F_{<\eta}^{\mathrm{sing},\eta_{*}}(y,x)|\lesssim \eta^2	\left(
	1+	\frac{1+\mathcal C(\eta,\eta_*)}{r^{\min\{s-1,d\}+2}}
	\right).\end{align*}
This implies \eqref{proF2} for $r>10\eta$, hence the relation is proved in all cases.
\smallskip

\noindent
\textbf{Step 4: Proof of \eqref{proF3}.}
Steps~4 and~5 follow the same strategy as Step~2 (Taylor remainder to gain $\eta^2$), but with the test function $g$ playing the role of $\Phi^\eta$. The key difference is that $g$ is not Gaussian, so the regularity requirements on $g$ enter explicitly.
First, from \eqref{q3} we have
\begin{equation}\label{q9}
	|\F_{<\eta}^{\mathrm{reg}}*g(x)|\lesssim \eta^2 \|\xg \nab_x^2g\|_{L^\infty}.
\end{equation}
We have \begin{align*}
  \F_{<\eta}^{\mathrm{sing},\eta_{*}}*g(x)
  &= \int \F^{\mathrm{sing},\eta_{*}}(x,z')\,g(z')\,dz'
    - \iint \F^{\mathrm{sing}}(x,z'+z)g(z')\Phi^\eta(z)dzdz'
  \\
  &= \iint \F^{\mathrm{sing},\eta_{*}}(x,z')\,
    \(g(z')-g(z'-z)\)\,\Phi^\eta(z)\,dz\,dz'.
\end{align*}
Taylor-expanding $g$ to second order, the first order terms vanish and we obtain 
\begin{align*}
	\F_{<\eta}^{\mathrm{sing},\eta_{*}}*g(x)
	=-\int_{0}^{1}(1-\tau)\iint\F^{\mathrm{sing},\eta_{*}}(x,z')\, z^{\otimes 2}: \nab^{\otimes 2} g(z'-\tau z) \,\Phi^\eta(z)\,dz\,dz'd\tau.
\end{align*}
Set 
\begin{align*}
	I_5=\left|\int_{0}^{1}(1-\tau)\iint\chi(|x-z'|)\F^{\mathrm{sing},\eta_{*}}(x,z')z^{\otimes 2}: \nab^{\otimes 2} g(z'-\tau z) \,\Phi^\eta(z)dzdz'd\tau\right|.
\end{align*}
Hence,
\begin{align}\nonumber
|	\F_{<\eta}^{\mathrm{sing},\eta_{*}}*g(x)|&	\lesssim \int_{0}^{1}\iint \Big(1+\frac{1}{|x-z'|^s}\Big)|z|^2\langle z'-\tau z\rangle^{-\gamma} \Phi^\eta(z)dz dz'd\tau \, \|\xg \nab^2g\|_{L^\infty}\mathbf{1}_{s<d}\\&\nonumber+
\int_{0}^{1}\iint |z|^2\langle z'-\tau z\rangle^{-\gamma} \Phi^\eta(z) dz dz'd\tau \|\xg \nab^2g\|_{L^\infty}\mathbf{1}_{s\geq d}
+	I_5\mathbf{1}_{s\geq d}\\&\lesssim \eta^2 \|\xg \nab^2g\|_{L^\infty}+	I_5\mathbf{1}_{s\geq d}.\label{q18}
\end{align}
Now we estimate $I_5$ in the case $s\geq d$. By \eqref{assumpF2'}, we have  
\begin{align}\nonumber
\left|	\int\chi(|x-z'|)\F^{\mathrm{sing},\eta_{*}}(x,z') dz'\right|&\lesssim\int_0^2 (\rho^{d-1}+\rho^{d-s}(\mathbf{1}_{s<d+1}+\mathbf{1}_{\rho\geq \eta_{*}/2}\mathbf{1}_{s\geq d+1})) d\rho\\&\lesssim \mathcal{C}(\eta_{*}).\label{q19}
\end{align}Thus, 
\begin{align*}
	I_5&\lesssim \int_{0}^{1}(1-\tau)\iint_{|x-z'|<2}|\F^{\mathrm{sing},\eta_{*}}(x,z')||z|^2 |\nab^{2} g(z'-\tau z)-\nab^{\otimes 2} g(x-\tau z)| \Phi^\eta(z)dzdz'd\tau\\&+
\mathcal{C}(\eta_{*})	\int_{0}^{1}(1-\tau)\int |z|^2 |\nab^{2} g(x-\tau z)| \Phi^\eta(z)\,dzd\tau\\&\lesssim
\int_{0}^{1}(1-\tau)\iint_{|x-z'|<2}(1+\frac{1}{(|x-z'|+\eta_{*}\mathbf{1}_{s\geq d+1})^s})|z|^2|x-z'|\langle x-\tau z\rangle^{-\gamma}\Phi^\eta(z)dzdz'd\tau\  \|\xg \nab^3g\|_{L^\infty}\\&+
\mathcal{C}(\eta_{*})	\int_{0}^{1}(1-\tau)\int |z|^2 \langle x-\tau z\rangle^{-\gamma} \Phi^\eta(z)\,dzd\tau  \|\xg \nab^2g\|_{L^\infty}.
\end{align*}
This implies 
\begin{align}\label{q13}
	I_5\lesssim \eta^2	\mathcal{C}(\eta_{*})	\left(  \|\xg \nab^2g\|_{L^\infty}+\|\xg \nab^3g\|_{L^\infty}\right).
\end{align}
Together with \eqref{q9} and \eqref{q18}  this yields \eqref{proF3}.
\smallskip

\noindent

\textbf{Step 5: Proof of \eqref{proF4}.} We first recall that, by \eqref{q3},
\begin{equation}\label{reg-part-proF4}
\sum_{j=0,1}	\left|\nabla^j\F_{<\eta}^{\mathrm{reg}}*g(x)\right|
	\lesssim
	\eta^2 \|\langle \cdot\rangle^\gamma g\|_{L^\infty}.
\end{equation}
It remains to estimate the singular contribution. We write
\begin{align}
	\nabla \F_{<\eta}^{\mathrm{sing},\eta_*}*g(x)
	&=
	\iint
	\nabla_x \F^{\mathrm{sing},\eta_*}(x,z')
	\big(g(z')-g(z'-z)\big)
	\Phi^\eta(z)\,dz\,dz' .
\end{align}
Using Taylor's formula and the cancellation of the first-order term, we obtain
\begin{align}
	\nabla \F_{<\eta}^{\mathrm{sing},\eta_*}*g(x)
	&=
	-\int_0^1(1-\tau)
	\iint
	\nabla_x \F^{\mathrm{sing},\eta_*}(x,z')
	z^{\otimes 2}:\nabla_{z'}^2 g(z'-\tau z)
	\Phi^\eta(z)\,dz\,dz'\,d\tau
	\label{sing-decomp-proF4}
\end{align}
By \eqref{assumpF2G} and \eqref{assumpF>=d+1} with $\ep=\eta_{*}$, we get 
\begin{align}\nonumber
&|	\nabla \F_{<\eta}^{\mathrm{sing},\eta_*}*g(x)|
	\lesssim 
\left|	\int_0^1(1-\tau)
	\iint
 \F^{\mathrm{sing},\eta_*}(x,z')
	z^{\otimes 2}:\nabla_{z'}^3 g(z'-\tau z)
	\Phi^\eta(z)\,dz\,dz'\,d\tau\right|\\&\quad+ 	\int_0^1
	\iint
	\left(
	1
	+
	\frac{\mathbf 1_{\{s<d+1\}}}{|x-z'|^{s'}}
	+
	\frac{\mathbf 1_{\{s\geq d+1\}}}{(|x-z'|+\eta_*)^{s-1}}
	\right)
	|z|^2
	\langle z'-\tau z\rangle^{-\gamma}
	\Phi^\eta(z)\,dz\,dz'\,d\tau
	\|\langle \cdot\rangle^\gamma\nabla^2 g\|_{L^\infty}.
	\label{first-sing-proF4}
\end{align}
Arguing exactly as in the proof of \eqref{proF3}, the second term on the right-hand side of
\eqref{first-sing-proF4} is bounded by
\begin{equation}\label{first-sing-bound-proF4}
	\eta^2 \mathcal C(\eta_*)
	\|\langle \cdot\rangle^\gamma\nabla^2 g\|_{L^\infty}.
\end{equation}

It remains to control the first term on the right-hand side of
\eqref{first-sing-proF4}. Again, the
argument is the same as in the proof of \eqref{proF3}. More precisely,
\begin{align}
	&\left|
	\int_0^1(1-\tau)
	\iint
	\F^{\mathrm{sing},\eta_*}(x,z')
	z^{\otimes 2}:\nabla_{z'}^3 g(z'-\tau z)
	\Phi^\eta(z)\,dz\,dz'\,d\tau
	\right|
	\nonumber \\
	&\qquad\lesssim
	\eta^2\mathcal C(\eta_*)
	\left(
	\|\langle \cdot\rangle^\gamma\nabla^3 g\|_{L^\infty}
	+
	\mathbf 1_{\{s\geq d\}}
	\|\langle \cdot\rangle^\gamma\nabla^4 g\|_{L^\infty}
	\right).
	\label{second-sing-bound-proF4}
\end{align}
Indeed, in the case $s\geq d$, the localized contribution
\begin{align}
	\left|
	\int_0^1(1-\tau)
	\iint
	\chi(|x-z'|)
	\F^{\mathrm{sing},\eta_*}(x,z')
	z^{\otimes 2}:\nabla_{z'}^3 g(z'-\tau z)
	\Phi^\eta(z)\,dz\,dz'\,d\tau
	\right|
\end{align}
is estimated by the same  cancellation argument used
for $I_5$ in \eqref{q13}. This is the only place where the additional
$\|\langle \cdot\rangle^\gamma\nabla^4 g\|_{L^\infty}$ term is needed.

Combining \eqref{reg-part-proF4},
\eqref{first-sing-bound-proF4}, and \eqref{second-sing-bound-proF4}, we obtain
\begin{align}
	\left|\nabla \F_{<\eta}^{\eta_*}*g(x)\right|
	&\lesssim
	\eta^2
	\|\langle \cdot\rangle^\gamma g\|_{L^\infty}
	+
	\eta^2\mathcal C(\eta_*)
	\|\langle \cdot\rangle^\gamma\nabla^2 g\|_{L^\infty}
	\nonumber \\
	&\quad
	+
	\eta^2\mathcal C(\eta_*)
	\left(
	\|\langle \cdot\rangle^\gamma\nabla^3 g\|_{L^\infty}
	+
	\mathbf 1_{\{s\geq d\}}
	\|\langle \cdot\rangle^\gamma\nabla^4 g\|_{L^\infty}
	\right).
\end{align}
This proves \eqref{proF4}.
 The proof is complete.
\end{proof}
\begin{lem}[Gradient bound for the Coulomb truncation]
	\label{lem:gradient-Fless-coulomb}
	Assume $s=d-1$. Then, for every $\eta\in(0,1]$ and every
	$x\neq y$,
	\begin{equation}\label{proF1-gradient-coulomb}
		\left|
		\nabla_x\F_{<\eta}(x,y)
		\right|
		\lesssim
		\eta^2
		+
		\frac1{|x-y|^{s+1}}
		\min\left\{
		1,\frac{\eta}{|x-y|}
		\right\}^2.
	\end{equation}
\end{lem}

\begin{proof}
	Set $r:=|x-y|$. The regular part is bounded by \eqref{q3} with
	$m=1$.
	
	If $r\leq10\eta$, the pointwise singular-kernel assumption and
	\eqref{proF0} with $m=1$ give
	\begin{equation}
		\left|
		\nabla_x\F_{<\eta}^{\mathrm{sing}}(x,y)
		\right|
		\lesssim
		r^{-s-1}+1,
	\end{equation}
	which is the desired estimate in this regime.
	
	Assume now that $r>10\eta$. Apply the proof of
	\eqref{proF1}, with the kernel $\F^{\mathrm{sing}}$ replaced by
	$\nabla_x\F^{\mathrm{sing}}$. The required pointwise bounds follow
	from the mixed-derivative assumptions:
	\begin{equation}
		\left|
		\nabla_y^2\nabla_x
		\F^{\mathrm{sing}}(x,y)
		\right|
		\lesssim
		1+r^{-s-3}.
	\end{equation}
	The localized contribution is treated by the same coarea
	decomposition as $I_4$, now using \eqref{assumpF2'} with $m=1$:
	\begin{equation}
		\left|
		\nabla_x
		\dashint_{|z|=\rho}
		\F^{\mathrm{sing}}(x,x+z)\,dz
		\right|
		\lesssim
		1+\rho^{-s}.
	\end{equation}
	Thus, the estimates corresponding to
	\eqref{z9} and \eqref{I4-final} give
	\begin{equation}
		\left|
		\nabla_x
		\F_{<\eta}^{\mathrm{sing}}(x,y)
		\right|
		\lesssim
		\eta^2
		+
		\frac{\eta^2}{r^{s+3}},
		\qquad r>10\eta.
	\end{equation}
	Combining the two regimes proves
	\eqref{proF1-gradient-coulomb}.
\end{proof}
\begin{lem}\label{lem:integralestimates}  	For any $\eta\in[\eta_*,1)$, the following estimates hold uniformly in $x\in\mathbb R^d$.

\smallskip
\noindent\emph{Integral bounds on the kernel.} We have
	\begin{equation}
	\label{le230}  \int_{\mr^d} \frac{| \F_{\eta}^{\eta_*}(x,y)|}{\langle y\rangle^\gamma}\, dy\lesssim
		\langle x\rangle +\begin{cases}0
 & \text{if } s< d,\\
	|\log\eta| & \text{if } s= d,\\
	\eta^{d-s} & \text{if } d<s<d+1,\\
	\log(2\eta/\eta_*)\,\eta^{-1} & \text{if } s= d+1,\\
	\eta_*^{-s+d+1}\,\eta^{-1} & \text{if } s> d+1,
	\end{cases}
	\end{equation}
	\begin{equation}
	\label{le230'}  \int_{\mr^d} \frac{| \F_{\eta}^{\mathrm{sing},\eta_{*}}(x,y)|}{\langle y\rangle^\gamma}\, dy\lesssim
1+\begin{cases}0
		& \text{if } s< d,\\
		|\log\eta| & \text{if } s= d,\\
		\eta^{d-s} & \text{if } d<s<d+1,\\
		\log(2\eta/\eta_*)\,\eta^{-1} & \text{if } s= d+1,\\
		\eta_*^{-s+d+1}\,\eta^{-1} & \text{if } s> d+1,
	\end{cases}
\end{equation}
		
	\begin{equation}
	\label{le231}  \int_{\mr^d} \frac{|\nab_x \F_{\eta}^{\eta_*}(x,y)|}{\langle y\rangle^\gamma}\, dy\lesssim
	\begin{cases}
	1& \text{if } s< d-1,\\
	|\log\eta| & \text{if } s= d-1,\\
	\eta^{d-1-s} & \text{if } d-1<s<d+1,\\
	\log(2\eta/\eta_*)\,\eta^{-2} & \text{if } s= d+1,\\
	\eta_*^{-s+d+1}\,\eta^{-2} & \text{if } s> d+1.
	\end{cases}
	\end{equation}
	For $m=2,3$, we have
	\begin{equation}
	\label{le232}
	\int_{\mr^d} \frac{|\nab_x^{m} \F_{\eta}^{\eta_*}(x,y)|}{\langle y\rangle^\gamma}\, dy\lesssim
	\begin{cases}
	\eta^{d-s-m} & \text{if } s<d+1,\\
	\eta^{-m-1}\log(2\eta/\eta_*) & \text{if } s=d+1,\\
	\eta^{-m-1}\eta_*^{-s+d+1} & \text{if } s>d+1.
	\end{cases}
	\end{equation}

\smallskip
\noindent\emph{Difference estimates.} Assume in addition that $\eta,\eta'\in(\eta_*,1)$ and $\eta\sim\eta'$. Then
		\begin{equation}
	\label{le234}  \int_{\mr^d}|	(\F_{\eta}^{\eta_*}-\F_{\eta'}^{\eta_*} )(x,y)| \frac{dy}{\langle y\rangle^\gamma}\lesssim
	\begin{cases}
	\eta^{d-s} & \text{if } s<d+1,\\
	\log(2\eta/\eta_*)\,\eta^{-1} & \text{if } s= d+1,\\
	\eta_*^{-s+d+1}\,\eta^{-1} & \text{if } s> d+1,
	\end{cases}
	\end{equation}
\begin{equation}
	\label{le233}
	\int_{\mr^d}|\nab_x 	(\F_{\eta}^{\eta_*}-\F_{\eta'}^{\eta_*} )(x,y)| \frac{dy}{\langle y\rangle^\gamma}\lesssim
	\begin{cases}
	1 & \text{if } s\leq d-1,\\
	\eta^{d-1-s} & \text{if } d-1<s<d+1,\\
	\log(2\eta/\eta_*)\,\eta^{-2} & \text{if } s= d+1,\\
	\eta_*^{-s+d+1}\,\eta^{-2} & \text{if } s> d+1.
	\end{cases}
	\end{equation}

\smallskip
\noindent\emph{Diagonal and convolution bounds.} For every $\eta,\eta'\in[\eta_*,1)$, we have
\begin{align}
\label{226}
&	\left| \F_{\eta}^{\eta_*}(x,x)\right|\lesssim \frac{1+\mathcal C(\eta,\eta_*)}{\eta^{\min\{s-1,d\}}}+\langle x\rangle,	\\
&\label{226b} 	\left| \F_{\eta}^{\eta_*}(x,x)-\F_{\eta'}^{\eta_*}(x,x)\right|\lesssim \frac{1+\mathcal C(\eta,\eta_*)}{\eta^{\min\{s-1,d\}}}+\frac{1+\mathcal C(\eta',\eta_*)}{\eta'^{\min\{s-1,d\}}}+1,\\
&\label{226c} 		\left| \nab_x(\F_{\eta}^{\eta_*}(x,x))\right|\lesssim \frac{1}{\eta^s}+1~~\text{if}~s<d,
\\ & \label{227} \|\nab_x^m \F_{\eta}^{\eta_*}*\mu\|_{L^\infty}\lesssim \mathcal{C}(\eta_*)( \A(\mu)+\mathbf{1}_{s\geq d, m=3}	\|\langle \cdot\rangle^\gamma\nabla^4 
\mu\|_{L^\infty}), \quad  m=1,2,3.
\end{align}
\end{lem}
\begin{proof} 	\noindent\textbf{Proof of  \eqref{le230}--\eqref{le232}}. These estimates follow from the pointwise bound \eqref{proF0} by integration in $y$ against $\langle y\rangle^{-\gamma}$: observe that
$$
\int\frac{dy}{(\eta+|x-y|)^{\min\{s,d+1\}+m} \langle y\rangle^{\gamma}}\lesssim\begin{cases}
1&\text{when}\  \min\{s,d+1\}+m<d\\ 
 |\log \eta| & \text{when} \ \min\{s, d+1\}+m=d\\
 \eta^{d-\min\{s,d+1\}-m} &  \text{when } \ \min\{s,d+1\}+m>d.\end{cases}$$ 
  The additional terms involving $\eta_*$ come from the cutoff correction when $s\geq d+1$. 
  
\medskip

\noindent\textbf{Proof of \eqref{226}--\eqref{226c}.}
We write the diagonal value as
\[
\F_\eta^{\eta_*}(x,x)
= \int_{\mathbb{R}^d} \F^{\mathrm{sing},\eta_*}(x,x+z)\,\Phi^\eta(z)\,dz
+ \F^{\mathrm{reg}}_\eta(x,x).
\]
By \eqref{assumpFre}, the regular part satisfies, for all $0<\eta,\eta'\le 1$,
\begin{align}\label{q50}
	|\F^{\mathrm{reg}}_\eta(x,x)| &\lesssim 1+|x|,\\ \label{q51}
	|\nab_x(\F^{\mathrm{reg}}_\eta(x,x))| &\lesssim 1,\\
	|\F^{\mathrm{reg}}_\eta(x,x)-\F^{\mathrm{reg}}_{\eta'}(x,x)| &\lesssim 1 .\label{q52}
\end{align}
It remains to estimate the singular contribution.

By the spherical average assumption \eqref{assumpF2'}, for $m=0,1$ we have
\[
\left|
\nab_x^m \dashint_{|z|=\rho}
\F^{\mathrm{sing},\eta_*}(x,x+z)\mathcal H^{d-1}(z)
\right|
\lesssim
\rho^{-s+1-m}
\Big(
\mathbf{1}_{s<d+1}
+
\mathbf{1}_{\rho\ge \eta_*/2}\mathbf{1}_{s\ge d+1}
\Big)
+1 .
\]
Integrating this bound against the radial mollifier $\Phi^\eta$ in polar coordinates gives
\begin{align}\nonumber
	\left|
	\int_{\mathbb{R}^d}
	\F^{\mathrm{sing},\eta_*}(x,x+z)\Phi^\eta(z)\,dz
	\right|
	&\lesssim
	\int_0^\infty
	\left[
	\rho^{-s+1}
	\Big(
	\mathbf{1}_{s<d+1}
	+
	\mathbf{1}_{\rho\ge \eta_*/2}\mathbf{1}_{s\ge d+1}
	\Big)
	+1
	\right]
	\rho^{d-1}\Phi^\eta(\rho)\,d\rho  \\
	&\lesssim
	\frac{1+\mathcal C(\eta,\eta_*)}{\eta^{\min\{s-1,d\}}}
	+1. \label{q53}
\end{align}
Similarly, in the range $s<d$, applying the same estimate with $m=1$ yields
\begin{align}\nonumber
	\left|
	\nab_x
	\left(
	\int_{\mathbb{R}^d}
	\F^{\mathrm{sing},\eta_*}(x,x+z)\Phi^\eta(z)\,dz
	\right)
	\right|
	&\lesssim
	\int_0^\infty
	\left(\rho^{-s}+1\right)
	\rho^{d-1}\Phi^\eta(\rho)\,d\rho  \\
	&\lesssim\eta^{-s}+1. \label{q54}
\end{align}
Here we used the elementary scaling estimate
\begin{align*}
	\int_\varepsilon^\infty
	\rho^{-s+m}\rho^{d-1}\Phi^\eta(\rho)\,d\rho
	&=
	C_d\eta^{-s+m}
	\int_{\varepsilon/\eta}^\infty
	t^{d-1-s+m}
	e^{-t^2/2}\,dt.
\end{align*}
Combining the above estimates for the regular and singular parts proves
\eqref{226}, \eqref{226b}, and \eqref{226c}.\\
\medskip
\noindent\textbf{Proof of \eqref{227}.} Let $m=1,2,3.$
Since $|\mu(y)|\lesssim \A(\mu)\,\langle y\rangle^{-\gamma-1}$
(by the definition \eqref{defA} of $\A$).  By \eqref{assumpF2G}, we have 
\begin{align*}
		\|\nab_x^m \F_\eta^{\eta_*}*\mu\|_{L^\infty}\lesssim  	\| \F_\eta^{\eta_*}*(\nab_x^m\mu)\|_{L^\infty}+ \int_{\mathbb{R}^d}\left( 1+\frac{\mathbf{1}_{s<d+1}}{|x-y|^{s'}}+\frac{\mathbf{1}_{s\geq d+1}}{(|x-y|+\eta_{*})^{s-1}}\right)\sum_{k=0}^{m-1}|\nab_y^k \mu|(y) dy 
\end{align*}
So, 
\begin{align*}
	\|\nab_x^m \F_\eta^{\eta_*}*\mu\|_{L^\infty}\lesssim  	\| \F_\eta^{\eta_*}*(\nab_x^m\mu)\|_{L^\infty}\mathbf{1}_{s\geq d}+ \mathcal{C}(\eta_{*})\A(\mu).
\end{align*}
By \eqref{q19}, we get 
\begin{align*}
	 &	\| \F_\eta^{\eta_*}*(\nab_x^m\mu)\|_{L^\infty}\mathbf{1}_{s\geq d}\lesssim \mathcal{C}(\eta_{*})\A(\mu)+\mathbf{1}_{s\geq d}\sup_x\int_{|x-y|<1} |\F_\eta^{\eta_*}(x,y)| |\nab^m\mu(x)-\nab^m\mu(y)|dy\\&\lesssim
	 	\mathcal{C}(\eta_{*})\A(\mu)+\mathbf{1}_{s\geq d}	\|\langle \cdot\rangle^\gamma\nabla^{m+1} \mu\|_{L^\infty}\sup_x\int_{|x-y|<1} (1+\frac{1}{(|x-y|+\eta_{*}\mathbf{1}_{s\geq d+1})^s}) |x-y|\langle y\rangle^{-\gamma}dy
	 	\\&\lesssim
	 	\mathcal{C}(\eta_{*})(A(\mu)+\mathbf{1}_{s\geq d}	\|\langle \cdot\rangle^\gamma\nabla^{m+1} \mu\|_{L^\infty})
\end{align*}
This implies \eqref{227}.\\

\medskip
\noindent
\textbf{Proof of \eqref{le234}.}
We first dispose of the regular part. By \eqref{q3}, for any
$\eta\sim \eta'$, we have
\begin{equation}\label{reg-diff-le234}
	\sum_{m=0}^{1}
	\int_{\mathbb R^d}
	\left|
	\nabla_x^m
	\big(
	\F_{\eta}^{\mathrm{reg}}
	-
	\F_{\eta'}^{\mathrm{reg}}
	\big)(x,y)
	\right|
	\frac{dy}{\langle y\rangle^\gamma}
	\lesssim
	\eta^2 .
\end{equation}
It is therefore enough to estimate the singular contribution. We distinguish three cases: $s<d$, $s=d$, and $s>d$.\\
We first consider the case $s<d$. Since
\[
\F_{\eta}^{\mathrm{sing},\eta_*}
-
\F_{\eta'}^{\mathrm{sing},\eta_*}
=
\F_{<\eta'}^{\mathrm{sing},\eta_*}
-
\F_{<\eta}^{\mathrm{sing},\eta_*},
\]
we may apply the truncated kernel estimate \eqref{proF1} to each term. In this
regime, \eqref{proF1} gives, for $\xi\in\{\eta,\eta'\}$,
\begin{equation}\label{trunc-bound-slessd}
	\left|
	\F_{<\xi}^{\mathrm{sing},\eta_*}(x,y)
	\right|
	\lesssim
	\xi^2
	+
	\frac1{|x-y|^s}
	\min\left\{
	1,\frac{\xi}{|x-y|}
	\right\}^2 .
\end{equation}
Consequently,
\begin{align}
	&\int_{\mathbb R^d}
	\left|
	\big(
	\F_{\eta}^{\mathrm{sing},\eta_*}
	-
	\F_{\eta'}^{\mathrm{sing},\eta_*}
	\big)(x,y)
	\right|
	\frac{dy}{\langle y\rangle^\gamma}
	\nonumber \\
	&\qquad\lesssim
	\sum_{\xi\in\{\eta,\eta'\}}
	\int_{\mathbb R^d}
	\left[
	\xi^2
	+
	\frac1{|x-y|^s}
	\min\left\{
	1,\frac{\xi}{|x-y|}
	\right\}^2
	\right]
	\frac{dy}{\langle y\rangle^\gamma}.
	\label{sing-diff-slessd}
\end{align}
Splitting the singular integral at $|x-y|=\xi$, and using the decay provided
by the weight $\langle y\rangle^{-\gamma}$ at infinity, we obtain
\begin{align}
	\int_{\mathbb R^d}
	\frac1{|x-y|^s}
	\min\left\{
	1,\frac{\xi}{|x-y|}
	\right\}^2
	\frac{dy}{\langle y\rangle^\gamma}
	&\lesssim
	\int_0^\xi \rho^{d-1-s}\,d\rho
	+
	\xi^2\int_\xi^1 \rho^{d-3-s}\,d\rho
	+
	\xi^2
	\notag
	\\
	&\lesssim
	\xi^{d-s}+\xi^2
	\lesssim
	\xi^{d-s}.
	\label{radial-est-slessd}
\end{align}
Here we used the standing assumption $s\ge (d-\frac54)_+$, which implies
$d-s<2$ in the regime $s<d$, and hence $\xi^2\le \xi^{d-s}$ for
$0<\xi\le1$.
 Since $\eta\sim\eta'$, this yields
\begin{equation}\label{le234-slessd}
	\int_{\mathbb R^d}
	\left|
	\big(
	\F_{\eta}^{\mathrm{sing},\eta_*}
	-
	\F_{\eta'}^{\mathrm{sing},\eta_*}
	\big)(x,y)
	\right|
	\frac{dy}{\langle y\rangle^\gamma}
	\lesssim
	\eta^{d-s}.
\end{equation}
This proves \eqref{le234} when $s<d$.

For $s>d$, the desired estimate follows directly from the integral bound
\eqref{le230'}. Hence it remains only to treat the borderline case $s=d$.

When $s=d$, the direct use of \eqref{le230} gives a logarithmic loss of the
form $1+|\log\eta|$. Such a loss is too large for closing the Gronwall
argument on the time scale $[0,\eta_*^{s+1-d}T_0]$. We therefore exploit the
cancellation structure of $\F^{\eta_*}$.

Let
\[
\eta'' := \frac1{10}\min\{\eta,\eta'\}.
\]
By the semigroup property of the heat kernel, we may write
\begin{align}\label{semigroup-borderline}
	\big(
	\F_{\eta}^{\mathrm{sing},\eta_*}
	-
	\F_{\eta'}^{\mathrm{sing},\eta_*}
	\big)(x,y)
	=
	\int_{\mathbb R^d}
	\left[
	\F_{\sqrt{\eta^2-(\eta'')^2}}^{\mathrm{sing},\eta_*}(x,y+z)
	-
	\F_{\sqrt{(\eta')^2-(\eta'')^2}}^{\mathrm{sing},\eta_*}(x,y+z)
	\right]
	\Phi^{\eta''}(z)\,dz .
\end{align}
On the support of
$1-\chi(|x-y-z|)$, we have $|x-y-z|\geq1$. Hence,
by \eqref{proF0},
\begin{equation}
	\left|
	\F_\varepsilon^{\mathrm{sing}}(x,y+z)
	\right|
	\lesssim1,
	\qquad
	\varepsilon\in(0,1].
\end{equation}
Consequently, the nonlocalized contribution is bounded uniformly in
$\eta$, $\eta'$, and $\eta_*$ after integration against
$\langle y\rangle^{-\gamma}$ and the Gaussian. Thus,
\begin{align}
	&\int_{\mathbb R^d}
	\left|
	\big(
	\F_{\eta}^{\mathrm{sing},\eta_*}
	-
	\F_{\eta'}^{\mathrm{sing},\eta_*}
	\big)(x,y)
	\right|
	\frac{dy}{\langle y\rangle^\gamma}
	\nonumber \\
	&\qquad\lesssim
	1
	+
	\int_{\mathbb R^d}
	\left|
	\int_{\mathbb R^d}
	\left[
	\F_{\sqrt{\eta^2-(\eta'')^2}}^{\mathrm{sing},\eta_*}(x,y+z)
	-
	\F_{\sqrt{(\eta')^2-(\eta'')^2}}^{\mathrm{sing},\eta_*}(x,y+z)
	\right]
	\Phi^{\eta''}(z)
	\chi(|x-y-z|)
	\,dz
	\right|
	\frac{dy}{\langle y\rangle^\gamma}.
	\label{borderline-split}
\end{align}
For the localized term, we subtract and add
$\Phi^{\eta''}(x-y)$. This gives
\begin{align}
	&\int_{\mathbb R^d}
	\left|
	\int_{\mathbb R^d}
	\left[
	\F_{\sqrt{\eta^2-(\eta'')^2}}^{\mathrm{sing},\eta_*}(x,y+z)
	-
	\F_{\sqrt{(\eta')^2-(\eta'')^2}}^{\mathrm{sing},\eta_*}(x,y+z)
	\right]
	\Phi^{\eta''}(z)
	\chi(|x-y-z|)
	\,dz
	\right|
	\frac{dy}{\langle y\rangle^\gamma}
	\nonumber \\
	&\qquad\leq
	\int_{\mathbb R^d}
	\left|
	\int_{\mathbb R^d}
	\left[
	\F_{\sqrt{\eta^2-(\eta'')^2}}^{\mathrm{sing},\eta_*}(x,y+z)
	-
	\F_{\sqrt{(\eta')^2-(\eta'')^2}}^{\mathrm{sing},\eta_*}(x,y+z)
	\right]
	\chi(|x-y-z|)
	\,dz
	\right|
	\Phi^{\eta''}(x-y)
	\frac{dy}{\langle y\rangle^\gamma}
	\nonumber \\
	&\qquad\quad
	+
	\int_{\mathbb R^d}
	\int_{\mathbb R^d}
	\left|
	\F_{\sqrt{\eta^2-(\eta'')^2}}^{\mathrm{sing},\eta_*}(x,y+z)
	-
	\F_{\sqrt{(\eta')^2-(\eta'')^2}}^{\mathrm{sing},\eta_*}(x,y+z)
	\right|
	\nonumber \\
	&\qquad\qquad\qquad\qquad
	\times
	\left|
	\Phi^{\eta''}(z)-\Phi^{\eta''}(x-y)
	\right|
	\chi(|x-y-z|)
	\,dz\,
	\frac{dy}{\langle y\rangle^\gamma}.
	\label{borderline-subtract}
\end{align}
The reason for subtracting and adding $\Phi^{\eta''}(x-y)$ is that the
constant part can be treated using the cancellation property
\eqref{assumpF2'}, while the difference
$\Phi^{\eta''}(z)-\Phi^{\eta''}(x-y)$ gains a factor $|x-y-z|$. This gain
removes the logarithmic divergence of the borderline kernel
$|x-y-z|^{-d}$.\\

We estimate the two terms separately. First, by the cancellation assumption
\eqref{assumpF2'}, uniformly for $\varepsilon\in(0,1)$,
\begin{equation}\label{localized-cancel-bound}
	\left|
	\int_{\mathbb R^d}
	\F_\varepsilon^{\mathrm{sing},\eta_*}(x,y+z)
	\chi(|x-y-z|)
	\,dz
	\right|
	\lesssim
	1 .
\end{equation}
Therefore the first term on the right-hand side of
\eqref{borderline-subtract} is bounded by a constant.\\
For the second term, using \eqref{proF1} with $s=d$, we have
\begin{equation}\label{borderline-diff-kernel}
	\left|
	\F_{\sqrt{\eta^2-(\eta'')^2}}^{\mathrm{sing},\eta_*}(x,y+z)
	-
	\F_{\sqrt{(\eta')^2-(\eta'')^2}}^{\mathrm{sing},\eta_*}(x,y+z)
	\right|
	\lesssim
	\eta^2
	+
	\frac1{|x-y-z|^d}
	\min\left\{
	1,\frac{\eta}{|x-y-z|}
	\right\}^2 .
\end{equation}
Moreover, the Gaussian Lipschitz bound gives
\begin{equation}\label{gaussian-difference-le234}
	\left|
	\Phi^{\eta''}(z)-\Phi^{\eta''}(x-y)
	\right|
	\lesssim
	(\eta'')^{-d-1}
	|x-y-z|
	\left[
	\exp\left(-\frac{|z|^2}{4(\eta'')^2}\right)
	+
	\exp\left(-\frac{|x-y|^2}{4(\eta'')^2}\right)
	\right].
\end{equation}
Since $\chi(|x-y-z|)$ restricts to a bounded neighborhood of the singularity,
the extra factor $|x-y-z|$ in \eqref{gaussian-difference-le234} removes the
borderline logarithmic divergence in \eqref{borderline-diff-kernel}. Hence
\begin{align}\nonumber
	&\int_{\mathbb R^d}
	\int_{\mathbb R^d}
	\left|
	\F_{\sqrt{\eta^2-(\eta'')^2}}^{\mathrm{sing},\eta_*}(x,y+z)
	-
	\F_{\sqrt{(\eta')^2-(\eta'')^2}}^{\mathrm{sing},\eta_*}(x,y+z)
	\right|\\&\quad\quad\quad\quad\quad\quad\quad\quad\quad\quad\times
	\left|
	\Phi^{\eta''}(z)-\Phi^{\eta''}(x-y)
	\right|
	\chi(|x-y-z|)
	\,dz\,
	\frac{dy}{\langle y\rangle^\gamma}\nonumber
\\&\lesssim \nonumber\eta+\eta^{-d-1}
\int_{\mathbb R^d}
\int_{|x-y-z|\leq 1}
	\frac1{|x-y-z|^{d-1}}
\min\left\{
1,\frac{\eta}{|x-y-z|}
\right\}^2\\&\quad\quad\quad\quad\quad\quad\quad\quad\quad\quad\times\nonumber
	\left[
\exp\left(-c\frac{|z|^2}{\eta^2}\right)
+
\exp\left(-c\frac{|x-y|^2}{\eta^2}\right)
\right]
\,dz\,
\frac{dy}{\langle y\rangle^\gamma}
\\&\lesssim \eta+\eta^{-1}
\int_{|z|\leq 1}
\frac1{|z|^{d-1}}
\min\left\{
1,\frac{\eta}{|z|}
\right\}^2dz\lesssim 1.
	\label{borderline-second-term}
\end{align}
Indeed, in the last line we used the change of variables $w=x-y-z$ and the
Gaussian integration in the remaining variable. Finally,
\begin{align*}
	\eta^{-1}
	\int_{|w|\le 1}
	\frac1{|w|^{d-1}}
	\min\left\{
	1,\frac{\eta}{|w|}
	\right\}^2
	dw
	&\lesssim
	\eta^{-1}
	\left(
	\int_0^\eta d\rho
	+
	\eta^2\int_\eta^1 \rho^{-2}\,d\rho
	\right)
	\lesssim 1 .
\end{align*}
Combining \eqref{borderline-split}--\eqref{borderline-second-term}, we obtain
\begin{equation}\label{borderline-final-le234}
	\int_{\mathbb R^d}
	\left|
	\big(
	\F_{\eta}^{\mathrm{sing},\eta_*}
	-
	\F_{\eta'}^{\mathrm{sing},\eta_*}
	\big)(x,y)
	\right|
	\frac{dy}{\langle y\rangle^\gamma}
	\lesssim
	1 .
\end{equation}
This proves \eqref{le234} in the borderline case $s=d$.\\
Together with \eqref{reg-diff-le234}, the estimates above prove
\eqref{le234} for all $s\neq d$ and also for the critical case $s=d$.
\smallskip

\noindent
\textbf{Proof of \eqref{le233}.}
	When $s\neq d-1$, the bound \eqref{le233} follows from \eqref{le231} and the fact that $\eta\sim\eta'$.
	It remains to treat the borderline $s=d-1$, where $\nabla_x\F_\eta^{\eta_*}=\nabla_x\F_\eta$ is just barely non-integrable. The idea is to again use the near-antisymmetry of $\F$, while the $\nabla_y$ part is transferred to the Gaussian via the semigroup property, reducing the problem to an estimate on $\F$ which is integrable.
	We write
\begin{align*}
		&\int_{\mr^d}|	\nab_x(\F_{\eta}-\F_{\eta'} )(x,y)| \frac{dy}{\langle y\rangle^\gamma}\\
		&\quad\leq 	\int_{\mr^d}|	\nab_x(\F_{\eta}-\F_{\eta'} )(x,y)+\nab_y(\F_{\eta}-\F_{\eta'})(x,y)| \frac{dy}{\langle y\rangle^\gamma}\\&
		\quad\quad+	\int_{\mr^d}|	\nab_y(\F_{\eta}-\F_{\eta'} )(x,y)| \frac{dy}{\langle y\rangle^\gamma}.
\end{align*}
The first term can be bounded using \eqref{assumpFre} and \eqref{assumpF2G}. To estimate the second term, we bound
\begin{align*}
	|	\nab_y(\F_{\eta}-\F_{\eta'} )(x,y)|\leq  \int|\F_{\eta/2}-\F_{\eta''}|(x,y+z)\,|\nab_z\Phi^{c_1\eta}(z)|\,dz
\end{align*}
for some  $c_1>0$ and $\eta''\sim \eta$. By \eqref{proF1}, this is bounded by
\begin{align*}
&\sum_{\xi\in\{\eta/2,\, \eta''\}}	\iint\Big(\xi^2+\frac{1}{|x-y+z|^{d-1}}\min\!\Big(1,\frac{\xi}{|x-y+z|}\Big)^2\Big)|\nab_z\Phi^{c_1\eta}(z)|\,dz\,\frac{dy}{\langle y\rangle^\gamma}\\&\quad\quad\quad\quad\quad\quad\quad\lesssim \eta+	\iint\frac{1}{|y+z|^{d-1}}\min\!\Big(1,\frac{\eta}{|y+z|}\Big)^2\eta^{-1}\Phi^{2c_1\eta}(z)dz\,\frac{dy}{\langle y\rangle^\gamma}\lesssim 
1.
\end{align*} 
Here we used that $|\nabla \Phi^{c_2\eta}|\lesssim \eta^{-1}\Phi^{2c_2\eta}$ and, 
\begin{align*}
	\int\frac{1}{|y+z|^{d-1}}\min\!\Big(1,\frac{\eta}{|y+z|}\Big)^2\frac{dy}{\langle y\rangle^\gamma}\lesssim 	\int\frac{1}{|y|^{d-1}}\min\!\Big(1,\frac{\eta}{|y|}\Big)^2dy\sim \eta.
\end{align*}
	This concludes the proof of \eqref{le233}.
\end{proof}
\subsection{Dual transport estimates}\label{sec:dual}
To close the argument we will need to propagate $( W^{1,1})^*$ bounds forward in time.
The following two lemmas provide the necessary Grönwall-type estimates for the error $\omega=\tilde\mu_{N,\eta}$ in dual Sobolev norms, in the sub-Coulomb and super-Coulomb regimes respectively, using that $\omega$ will solve 
 \eqref{eqfordiffeq}.
The proof relies on  testing $\omega=\tilde \mu_{N,\eta}$ against the solution of the \emph{backward} transport equation, exploiting the regularity of $\F*\mu$ to control the resulting commutator, and closing by Grönwall's lemma.

The first result concerns existence and estimates for that backward equation.
\begin{lem}\label{transesle}
	Assume that $s\le d-1$, and let $t_1\in[0,T]$.  Define the dual kernel by $\F_{\rm dual}(x,y):=\F(y,x)$. Let $\mu\in C([0,T];C^2(\mathbb{R}^d))$ be a prescribed function satisfying $$\sup_{t\in [0,T]}\A(\mu^t)<\infty.$$ 
 Then, for any $\sigma\in W^{1,1}(\mathbb{R}^d)$ and $d+1< \gamma_1\leq \gamma-1,$ there exists a unique solution $\overline f$ to
	\begin{equation}\label{eq:dualb}
		\partial_t 	\overline{f}
		= -\langle x\rangle^{-\gamma_1}\F_{\rm dual}\ast(\mu\nabla( \langle \cdot \rangle^{\gamma_1}	\overline{f}))
		-\F\ast\mu\,\nabla 	\overline{f}-\gamma_1 (\F\ast\mu \cdot x\langle x\rangle^{-2})\overline{f} ,
		\quad
		\overline{f}(t_1)=\sigma
	\end{equation}
	such that
	\begin{equation}\label{fW11}
		\sup_{t\in[0,t_1]}\|\overline{f}(t)\|_{W^{1,1}}
		\le C\, e^{Ct_1\A_{t_1}}\,\|\sigma\|_{W^{1,1}}.
	\end{equation}
where $\A_{t_1}=\max_{[0,t_1]}\A(\mu^t)$.
\end{lem}

\begin{proof} We first establish the a priori estimate \eqref{fW11} for smooth
solutions. The construction of the solution in
$C([0,t_1];W^{1,1}(\mathbb R^d))$ will be given at the end of the
proof.
We first record the bounds on the coefficients. For $n=1,2$,
\eqref{assumpF2G} gives
\begin{align*}
	|	\nabla_x^n(\F^{\mathrm{sing}}*\mu)(x)|\lesssim
	|	\F^{\mathrm{sing}}*(\nabla^n\mu)(x)|+	\int_{\mathbb R^d}
	\left(
	1+\frac1{|x-y|^{s'}}
	\right)
	\sum_{k=0}^{n-1}|\nabla^k\mu(y)|\,dy.
\end{align*}
Since $s\leq d-1<d$, $s'<d$, and $\gamma>d$, the assumptions
\eqref{assumpFre}, \eqref{assumpFsing}, and \eqref{assumpF2G}
imply
\begin{equation}\label{w3}
	\sum_{n=1}^{2}
	\|\nabla_x^n(\F*\mu)\|_{L^\infty}
	+
	\|\langle\cdot\rangle^{-1}\F*\mu\|_{L^\infty}
	\lesssim
	\A_t.
\end{equation}
We now derive \eqref{fW11} under the additional assumption
\begin{equation}
	\overline f
	\in
	C\bigl([0,t_1];L^1(\mathbb R^d)\cap C^2(\mathbb R^d)\bigr).
\end{equation}
	Moreover, for $\square\in\{\mathrm{sing},\mathrm{reg}\}$, we may write
\begin{align*}
	\F_{\rm dual}^{\square}*
	\bigl(\mu\nabla(\langle\cdot\rangle^{\gamma_1}\overline f)\bigr)(x)
	&=
	\int_{\mathbb R^d}
	\G^\square(x,y)\langle y\rangle^{\gamma_1}\overline f(y)\,dy,
\end{align*}
where
\begin{equation}\label{defGdual}
	\G^\square(x,y)
	:=
	-\operatorname{div}_y\F^\square(y,x)\,\mu(y)
	+
	\F^\square(y,x)\cdot\nabla\mu(y).
\end{equation}

We begin with the regular part. By the mixed-derivative bounds in
\eqref{assumpFre}, together with
$d+1<\gamma_1\leq\gamma-1$, we obtain, for $n=0,1$,
\begin{equation}\label{qq1}
	|\nabla_x^n\G^{\mathrm{reg}}(x,y)|
	\lesssim
	\langle x\rangle^{1-n}
	\langle y\rangle^{-\gamma+1}\A_t.
\end{equation}
Consequently, for $n=0,1$,
\begin{align}\label{qq4}
	&\Bigl\|
	\langle\cdot\rangle^{-\gamma_1}
	\nabla^n
	\F_{\rm dual}^{\mathrm{reg}}*
	\bigl(\mu\nabla(\langle\cdot\rangle^{\gamma_1}\overline f)\bigr)
	\Bigr\|_{L^1}
	\nonumber\\
	&\qquad\lesssim
	\Bigl\|
	\langle\cdot\rangle^{-1}
	\nabla^n
	\F_{\rm dual}^{\mathrm{reg}}*
	\bigl(\mu\nabla(\langle\cdot\rangle^{\gamma_1}\overline f)\bigr)
	\Bigr\|_{L^\infty}
	\lesssim
	\A_t\|\overline f\|_{L^1}.
\end{align}

We next treat the singular part. If $s<d-1$, then \eqref{assumpFsing}
directly gives
\begin{equation}\label{qq5}
	\Bigl\|
	\langle\cdot\rangle^{-\gamma_1}
	\F_{\rm dual}^{\mathrm{sing}}*
	\bigl(\mu\nabla(\langle\cdot\rangle^{\gamma_1}\overline f)\bigr)
	\Bigr\|_{L^1}
	\lesssim
	\A_t\|\overline f\|_{L^1}.
\end{equation}
The endpoint case $s=d-1$ requires a duality argument. Since
$\gamma_1>d+1>\gamma_0$, using \eqref{assumpFsing} we reduce the estimate to
controlling
\begin{align*}
	\int_{\mathbb R^d}
	\left|
	\int_{\mathbb R^d}
	\operatorname{div}_y\F^{\mathrm{sing}}(y,x)\,
	\mu(y)\langle y\rangle^{\gamma_1}\overline f(y)\,dy
	\right|
	\frac{dx}{\langle x\rangle^{\gamma_0}} .
\end{align*}
By duality, this quantity is bounded by
\begin{align*}
	&\sup_{\|g\|_{L^\infty}\leq 1}
	\left|
	\int_{\mathbb R^d}
	g(x)
	\int_{\mathbb R^d}
	\operatorname{div}_y\F^{\mathrm{sing}}(y,x)\,
	\mu(y)\langle y\rangle^{\gamma_1}\overline f(y)\,dy
	\frac{dx}{\langle x\rangle^{\gamma_0}}
	\right|
	\\
	&\quad =
	\sup_{\|g\|_{L^\infty}\leq 1}
	\left|
	\int_{\mathbb R^d}
	\mu(x)\langle x\rangle^{\gamma_1}\overline f(x)\,
	\operatorname{div}_x
	\F^{\mathrm{sing}}*
	\bigl(g\langle\cdot\rangle^{-\gamma_0}\bigr)(x)
	\,dx
	\right|
	\\
	&\quad \leq
	\sup_{\|g\|_{L^\infty}\leq 1}
	\|\mu\langle\cdot\rangle^{\gamma_1}\overline f\|_{L^1}
	\,
	\Bigl\|
	\operatorname{div}_x
	\F^{\mathrm{sing}}*
	\bigl(g\langle\cdot\rangle^{-\gamma_0}\bigr)
	\Bigr\|_{L^\infty}.
\end{align*}
Using \eqref{coudet}, this is bounded by $
\A_t\|\overline f\|_{L^1}.$
Hence \eqref{qq5} holds for every $s\leq d-1$.

It remains to estimate the gradient of the singular contribution. By
\eqref{assumpF2G}, we obtain
\begin{align*} &\Bigl\| \langle\cdot\rangle^{-\gamma_1} \nabla_x \F_{\rm dual}^{\mathrm{sing}}* \bigl(\mu\nabla(\langle\cdot\rangle^{\gamma_1}\overline f)\bigr) \Bigr\|_{L^1} \\ &\qquad\lesssim \int_{\mathbb R^d} \int_{\mathbb R^d} \left(1+\frac{1}{|x-y|^{s'}}\right) \left| \nabla_y\bigl(\mu(y)\langle y\rangle^{\gamma_1}\overline f(y)\bigr) \right| dy\, \frac{dx}{\langle x\rangle^{\gamma_0}} \\ &\qquad\quad+ \int_{\mathbb R^d} \left| \int_{\mathbb R^d} \operatorname{div}_y\F^{\mathrm{sing}}(y,x)\, \nabla_y\bigl(\mu(y)\langle y\rangle^{\gamma_1}\overline f(y)\bigr) \,dy \right| \frac{dx}{\langle x\rangle^{\gamma_0}} . \end{align*} The first term is bounded by $\A_t\|\overline f\|_{W^{1,1}}$. Thus \begin{align*} &\Bigl\| \langle\cdot\rangle^{-\gamma_1} \nabla_x \F_{\rm dual}^{\mathrm{sing}}* \bigl(\mu\nabla(\langle\cdot\rangle^{\gamma_1}\overline f)\bigr) \Bigr\|_{L^1} \\ &\qquad\lesssim \A_t\|\overline f\|_{W^{1,1}} + \int_{\mathbb R^d} \left| \int_{\mathbb R^d} \operatorname{div}_y\F^{\mathrm{sing}}(y,x)\, \nabla_y\bigl(\mu(y)\langle y\rangle^{\gamma_1}\overline f(y)\bigr) \,dy \right| \frac{dx}{\langle x\rangle^{\gamma_0}} . \end{align*}
The last term is estimated by the same duality argument as above, giving
\[
\int_{\mathbb R^d}
\left|
\int_{\mathbb R^d}
\operatorname{div}_y\F^{\mathrm{sing}}(y,x)\,
\nabla_y\bigl(\mu(y)\langle y\rangle^{\gamma_1}\overline f(y)\bigr)
\,dy
\right|
\frac{dx}{\langle x\rangle^{\gamma_0}}
\lesssim
\A_t\|\overline f\|_{W^{1,1}}.
\]
Therefore,
\[
\Bigl\|
\langle\cdot\rangle^{-\gamma_1}
\nabla_x
\F_{\rm dual}^{\mathrm{sing}}*
\bigl(\mu\nabla(\langle\cdot\rangle^{\gamma_1}\overline f)\bigr)
\Bigr\|_{L^1}
\lesssim
\A_t\|\overline f\|_{W^{1,1}}.
\]
Combining this estimate with \eqref{qq4} and \eqref{qq5}, we obtain
\begin{equation}\label{qq6}
	\sum_{n=0,1}
	\Bigl\|
	\langle\cdot\rangle^{-\gamma_1}
	\nabla^n
	\F_{\rm dual}*
	\bigl(\mu\nabla(\langle\cdot\rangle^{\gamma_1}\overline f)\bigr)
	\Bigr\|_{L^1}
	\lesssim
	\A_t\|\overline f\|_{W^{1,1}} .
\end{equation}
We now derive the $W^{1,1}$ estimate. For $\varepsilon>0$, define
	\begin{align*}
		&	\phi_\ep(r):=\frac{r}{\sqrt{r^2+\ep^2}},\\&
		\overline\phi_\ep(r):=\sqrt{r^2+\ep^2}-\ep.
	\end{align*}
Note that
\begin{equation}
	0\leq
	\overline\phi_\varepsilon(r)
	\leq |r|,
	\qquad
	|\phi_\varepsilon(r)|\leq1,
	\qquad
	\overline\phi_\varepsilon(r)\longrightarrow|r|,~~\text{as}\varepsilon\to0.
\end{equation}
		Using $\phi_\varepsilon(\overline f)$ as a test function in
	\eqref{eq:dualb}, together with \eqref{w3} and \eqref{qq6}, gives
	\begin{equation}\label{L1eps}
		\left|
		\partial_t
		\|\overline\phi_\varepsilon(\overline f)\|_{L^1}
		\right|
		\lesssim
		\A_t\|\overline f\|_{L^1}.
	\end{equation}
	Similarly, using
	$\partial_{x_j}\phi_\varepsilon(\partial_{x_j}\overline f)$ as a test
	function, for $j=1,\ldots,d$, yields
	\begin{equation}\label{W11eps}
		\left|
		\partial_t
		\|\overline\phi_\varepsilon(\partial_{x_j}\overline f)\|_{L^1}
		\right|
		\lesssim
		\A_t\|\overline f\|_{W^{1,1}} .
	\end{equation}
	Summing \eqref{W11eps} over $j$, adding \eqref{L1eps}, applying Grönwall's
	inequality, and finally letting $\varepsilon\to0$, we obtain \eqref{fW11}.\\
	It remains to construct the solution. Define, for
	$f\in L^1([0,T],W^{1,1}(\mathbb R^d))$,
	\begin{align}
		\mathcal K_t f
		:=-\langle x\rangle^{-\gamma_1}\F_{\rm dual}\ast(\mu\nabla( \langle \cdot \rangle^{\gamma_1}f))-\gamma_1 (\F\ast\mu \cdot x\langle x\rangle^{-2})f.
	\end{align}
	The estimates \eqref{w3} and \eqref{qq6}, initially proved for smooth
	functions and then extended by density, imply
	\begin{equation}\label{bounded-K-dual}
		\|\mathcal K_tf\|_{W^{1,1}}
		\lesssim
		\A_t\|f\|_{W^{1,1}}.
	\end{equation}
	Let $U(t,s)$ denote the backward evolution operator associated with
$\partial_tu+\F*\mu\cdot\nabla u=0.$
	By \eqref{w3}, the vector field $\F*\mu$ is globally Lipschitz in space
	and has at most linear growth. It therefore generates a global
	characteristic flow, and
	\begin{equation}\label{transport-propagator-bound}
		\|U(t,s)g\|_{W^{1,1}}
		\leq
		\exp\left(
		C\int_t^s\A_\tau\,d\tau
		\right)
		\|g\|_{W^{1,1}},
		\qquad
		0\leq t\leq s\leq t_1.
	\end{equation}
	Equation \eqref{eq:dualb} is equivalent to the integral equation
	\begin{equation}\label{dualb-duhamel}
		\overline f(t)
		=
		U(t,t_1)\sigma
		-
		\int_t^{t_1}
		U(t,\tau)\mathcal K_\tau\overline f(\tau)\,d\tau.
	\end{equation}
	By \eqref{bounded-K-dual} and
	\eqref{transport-propagator-bound}, the right-hand side of
	\eqref{dualb-duhamel} defines a contraction on
	$C([t_1-\delta,t_1];W^{1,1})$ when $\delta>0$ is sufficiently small.
	Iterating this construction over finitely many successive intervals
	gives
	\begin{equation}
		\overline f\in
		C\bigl([0,t_1];W^{1,1}(\mathbb R^d)\bigr).
	\end{equation}
	The solution satisfies \eqref{eq:dualb} in the distributional sense.
	Uniqueness follows by applying \eqref{fW11} to the difference of two
	solutions with the same terminal datum. This completes the proof.
\end{proof}

\begin{lem}[Dual transport: sub-Coulomb case]\label{Le-dual1}
Assume $s\le d-1$. Let $\mu\in C([0,T];C^2(\mathbb{R}^d))$ be a prescribed function satisfying $$\sup_{t\in [0,T]}\A(\mu^t)<\infty.$$
Let $\omega$ solve
\begin{equation}\label{eq:omega}\partial_t\omega
  =
  -\operatorname{div}((\F\ast\omega )\mu+(\F\ast\mu) \omega)
  +\mathbf R.
\end{equation}
Then, for any $t\in (0,T]$ and $d+1<\gamma_1\leq \gamma-1$
\begin{equation}\label{omega_est}
	\|\langle \cdot \rangle^{\gamma_1}\omega\|_{L^\infty([0,t],(W^{1,1})^*)}
	\lesssim \left(\|\langle\cdot\rangle^{\gamma_1}\omega(0)\|_{(W^{1,1})^*} +\|\langle \cdot\rangle^{\gamma_1}\mathbf R \|_{L^1([0,t],(W^{1,1})^*)}\right) e^{Ct\A_t}
\end{equation}
where $C>0$ depends only on $d$ and $s$.
\end{lem}
\begin{proof}
\textbf{Step 1: Backward equation.}
Fix $t_1\in[0,T]$, and let $f$ solve the backward equation
\begin{equation}\label{eq:dual}
  \partial_t f
  = -\F_{\rm dual}\ast(\mu\nabla f)
    -(\F\ast\mu)\nabla f,
  \qquad
  f(t_1)=\langle x\rangle^{\gamma_1}\sigma.
\end{equation}
with $\sigma\in W^{1,1}$.\\
Set
\begin{align*}
	\overline f:=\langle x\rangle^{-\gamma_1}f.
\end{align*}
Then $\overline f$ satisfies \eqref{eq:dualb}. By Lemma \ref{transesle}, there exists a unique solution $f$, and $\overline f$ satisfies the estimate \eqref{fW11}.
\medskip\\
\noindent\textbf{Step 2: Duality argument.}
Using \eqref{eq:omega} and an integration by parts, we compute \begin{align*}
	\partial_t\!\int\omega(t)\,f(t) \,dx
	&= \int ( \omega(t) \partial_tf(t)+\partial_t\omega(t) f(t))dx
	\\&= \int  \(\omega(t) \partial_t f(t)+\left(  -\operatorname{div}((\F\ast\omega )\mu+(\F\ast\mu) \omega)
	+\mathbf R\right) f(t)\)dx\\&
	= \int \omega
	\Bigl[
	\partial_t f\
	+ \F_{\rm dual}\ast\bigl(\mu\nabla f\bigr)
	+(\F\ast\mu)\nabla f
	\Bigr]dx+\int \mathbf R f(t)dx\\&=\int \mathbf R f(t)dx.
\end{align*}
Integrating over $[0,t_1]$, we obtain
\begin{align*}
	\int\langle x\rangle^{\gamma_1}\omega(t_1)\sigma dx
	=\int\langle x\rangle^{\gamma_1}\omega(0)\overline{f}(0) dx+\int_{0}^{t_1}\int \langle x\rangle^{\gamma_1}\mathbf R \overline{f}(\tau)dxd\tau.
\end{align*}
Therefore,
\begin{align*}
	\left|	\int\langle x\rangle^{\gamma_1}\omega(t_1)\sigma dx\right|
	\lesssim \left(\|\langle\cdot\rangle^{\gamma_1}\omega(0)\|_{(W^{1,1})^*} +\|\langle\cdot\rangle^{\gamma_1}\mathbf R \|_{L^1([0,t_1],(W^{1,1})^*)}\right) \|\overline{f}\|_{L^\infty([0,t_1],W^{1,1})}.
\end{align*}
By \eqref{fW11}, it follows that
\begin{align*}
	\left|	\int\langle x\rangle^{\gamma_1}\omega(t_1)\sigma dx\right|
	\lesssim \left(\|\langle\cdot\rangle^{\gamma_1}\omega(0)\|_{(W^{1,1})^*} +\|\langle\cdot\rangle^{\gamma_1}\mathbf R \|_{L^1([0,t_1],(W^{1,1})^*)}\right) e^{Ct_1\A_{t_1}}\,\|\sigma\|_{W^{1,1}}.
\end{align*}
Since this holds for every $\sigma\in W^{1,1}(\mathbb R^d)$, we deduce that
\begin{align*}
\|\langle\cdot\rangle^{\gamma_1}\omega(t_1)\|_{(W^{1,1})^*}
	\lesssim \left(\|\langle\cdot\rangle^{\gamma_1}\omega(0)\|_{(W^{1,1})^*} +\|\langle\cdot\rangle^{\gamma_1}\mathbf R \|_{L^1([0,t_1],(W^{1,1})^*)}\right) e^{Ct_1\A_{t_1}}
\end{align*}
for any $t_1\in [0,T]$.
This proves \eqref{omega_est}.
\end{proof}

We next turn to the case $s >d-1$.
\begin{lem}\label{transesle2} 
	Assume that $s> d-1$, and let $t_1\in[0,T_0]$. Let $\mu\in C([0,T];C^2(\mathbb{R}^d))$ be a prescribed function satisfying $$\sup_{t\in [0,T]}\A(\mu^t)<\infty.$$ Then, for any $\sigma\in W^{1,1}(\mathbb{R}^d)$ and $d+1<\gamma_1\leq \gamma-1$, there exists a unique solution $\overline f$ to
	\begin{equation}\label{eq:dualb2}
		\partial_t 	\overline{f}
		= -\langle x\rangle^{-\gamma_1}\F^{\eta_*}_{\eta_*,\mathrm{dual}}\ast(\mu\nabla( \langle\cdot\rangle^{\gamma_1}	\overline{f}))
		-\F^{\eta_*}_{\eta_*}\ast\mu\,\nabla 	\overline{f}-\gamma_1 (\F^{\eta_*}_{\eta_*}\ast\mu\cdot x\langle x\rangle^{-2})\overline{f} ,
		\qquad
		\overline{f}(\eta_*^{s-d+1}t_1)=\sigma.
	\end{equation}
in $[0,\eta_*^{s-d+1}t_1]$
	such that
	\begin{equation}\label{fW12}
		\sup_{t\in[0,t_1]}\|\overline{f}(\eta_*^{s-d+1}t)\|_{W^{1,1}}
		\le Ce^{Ct_1\A_{\eta_*^{s-d+1}t_1}}\|\sigma\|_{W^{1,1}}.
	\end{equation}
\end{lem}

\begin{proof} 	The proof is similar to that of Lemma \ref{transesle}. By \eqref{226} and \eqref{227}, we have
	\begin{align}\label{w4}
		\sum_{n=1,2}	\|\nab_x^n(\F^{\eta_*}_{\eta_*}\ast\mu^t)\|_{L^\infty}+\|\langle\cdot\rangle^{-1}\F^{\eta_*}_{\eta_*}\ast\mu^t\|_{L^\infty}\lesssim 	\mathcal C(\eta_*)\A_t.
	\end{align}
Moreover, by \eqref{le232},
	\begin{align}
		\|\langle \cdot \rangle^{-\gamma_1}\F^{\eta_*}_{\eta_*,\mathrm{dual}}\ast(\mu^t \nabla( \langle\cdot\rangle^{\gamma_1}	\overline{f}(t)))\|_{W^{1,1}}\lesssim  	 \eta^{d-1-s}_*\A_t \|\overline{f}\|_{W^{1,1}}.
	\end{align}
Therefore
	\begin{equation}\label{fW13}
	\sup_{t\in[0,\eta_*^{s-d+1}t_1]}\|\overline{f}(t)\|_{W^{1,1}}
	\le Ce^{Ct_1\A_{\eta_*^{s-d+1}t_1}}\|\sigma\|_{W^{1,1}}.
\end{equation}
	This yields \eqref{fW12}.
\end{proof}
The proof of the  following lemma follows the same lines as 
that of Lemma \ref{Le-dual1} modulo a  time-rescaling.
\begin{lem}[Dual transport: super-Coulomb case]\label{Le-dual2}
Assume $s>d-1$. Let $\mu\in C([0,T];C^2(\mathbb{R}^d))$ be a prescribed function satisfying $$\sup_{t\in [0,T]}\A(\mu^t)<\infty.$$
Let $\omega$ satisfy
\begin{equation}\label{eq:omega-super}
  \partial_t\omega
  = -\operatorname{div}\bigl(
  (\F^{\eta_*}_{\eta_*}\ast\omega) \mu
    +(\F^{\eta_*}_{\eta_*}\ast\mu)\omega
  \bigr)
  + \mathbf R
\end{equation}
on the time interval $[0,T_0\eta_*^{s-d+1}]$.
Then, for any $t\in [0,T_0]$ and $d+1<\gamma_1\leq \gamma-1$,
\begin{equation}\label{omega-super-est}
  \|\langle\cdot\rangle^{\gamma_1}\omega(t \eta_*^{s-d+1})\|_{(W^{1,1})^*}
  \lesssim
\left(  \|\langle\cdot\rangle^{\gamma_1}\omega(0)\|_{(W^{1,1})^*}
  + \|\langle\cdot\rangle^{\gamma_1}\mathbf R\|_{L^1([0,t\eta_*^{s-d+1}],\,(W^{1,1})^*)}\right) e^{Ct\A_{\eta_*^{s-d+1}t}}.
\end{equation}
\end{lem}

\section{The time-evolution of the metric}
In \S\ref{sec:eqmollified} we derive the evolution equation satisfied by the mollified error $\tilde\mu_{N,\eta} := \mu_{N,\eta} - \mu$, identifying all the error terms $\R^1$--$\R^{12}$ that arise from the discretization. In the following subsections, we start controlling these error terms.

\subsection{The equation for the mollified quantity}\label{sec:eqmollified}
The first step is to show that the mollified empirical measure $\mu_{N,\eta}$ solves the mean-field equation, up to error terms. The bulk of the work will then consist in estimating these error terms.

We define
\begin{equation}\label{tildefneta}
	\tilde\mu_{N,\eta}
	\;:=\;
	\mu_{N,\eta}-\mu,
\end{equation}
where $\mu_{N,\eta}$ is defined as in \eqref{fneta}. The function $\mu$ denotes here
\begin{equation}\label{choixmu}
\begin{cases}
\text{the solution
of the mean-field equation \eqref{mflimit} when $s\le d-1$}\\
\text{the solution of the
regularized equation \eqref{mflimitsup} when $s\in (d-1,d)$}\\
\text{  $\mu\equiv \mu^0$ when $s\geq d$.}\end{cases}\end{equation}

We prove the following algebraic decomposition.
\begin{lem}\label{lem2.1}
Let $\eta\ge \eta_0 \geq\eta_*>0$ be such that $\eta= 2^n \eta_0$ for some integer $n\ge 0$. For any $X_N$, and any $\mu$, letting $\tilde \mu_{N,\eta}= \mu_{N,\eta}-\mu$ and using the same notation as in \eqref{defFeta}, we have 
	\begin{equation}\label{3.3}\operatorname{div} ((\F^{\eta_*}*\mu) \mu) -\frac1{N^{2}} \sum_{i=1}^N 	\sum_{j\neq i} \F^{\eta_*}(x_i,x_j)\cdot \nab \Phi_{x_i}^\eta =	 -\F^{\eta_*} \star 	\tilde \mu_{N, \eta}   \cdot \nab \mu -	\F^{\eta_*} \star 	 \mu_{N, \eta}\cdot \nab \tilde \mu_{N,\eta}   + \sum_{k=1}^{7} \R_{N, \eta}^k,
	\ee
		where, letting $\kappa_0$ be as in \eqref{defc0},
		\begin{align} \label{defr2}&
	 \R_{N, \eta}^1:=   -\frac1{N^2}\sum_{i\neq j} \F^{\eta_*}_{<\eta_0}(x_i,x_j) \cdot \nab  \pex(x),
\\
\label{defr1'} &
\R_{N,\eta}^2=  - \frac1{N^2}\sum_{i\neq j} \big(\F^{\eta_*}_{<\eta}(x_i,x_j) -\F^{\eta_*}_{<\eta_0}(x_i,x_j) \big) \cdot \nab  \pex(x),
\\
\label{defr3} &
	  \R_{N, \eta}^3:= -\frac1{N}  \sum_{i=1}^N  \big( \F^{\eta_*}_{\kappa_0\eta} *  \tilde\mu_{N,\eta/2}(x_i) - \F^{\eta_*}_{\kappa_0\eta} *  \tilde\mu_{N,\eta/2}(x)\big)  \cdot  \nab \pex(x),
\\
\label{defr1}& \R_{N, \eta}^4:=  \frac1{N^2} \sum_{i=1}^N  \F^{\eta_*}_\eta(x_i,x_i) \cdot\nab\pex(x),
	  \\  \label{defr4} &\R_{N, \eta}^5:= -\frac1{N}  \sum_{i=1}^N  \big( \F^{\eta_*}_{\kappa_0\eta} * \mu(x_i) - \F^{\eta_*}_{\kappa_0\eta} *  \mu(x)-\nab_x\F^{\eta_*}_{\kappa_0 \eta} *  \mu(x) (x_i-x)\big) \cdot   \nab \pex(x),\\
	  \label{defr5}
	  & \R_{N, \eta}^6=- \frac1{N}  \sum_{i=1}^N \nab_x\F^{\eta_*}_{\kappa_0 \eta} *  \mu(x) (x_i-x)\cdot  \nab \pex(x)
	  +\mu \operatorname{div}(\F^{\eta_*}_{\kappa_0\eta}*\mu),
	\\ \label{defr6}
	& \R_{N,\eta}^7=  \mu\operatorname{div} ( \F^{\eta_*}_{<\kappa_0\eta}*\mu).
	  \end{align}

	In particular, if $X_N^t$ solves \eqref{ode} and $\mu^t$ is as in \eqref{choixmu}, then
	\begin{align}
		\label{eqfordiffeq}
		\partial_t	\tilde \mu_{N, \eta}=	 -\F^{\eta_*} * 	\tilde \mu_{N, \eta}   \cdot \nab \mu -	\F^{\eta_*}  * 	 \mu\cdot \nab \tilde \mu_{N,\eta} -	\F^{\eta_*}  * 	 \tilde\mu_{N, \eta}\cdot \nab \tilde \mu_{N,\eta}  + \sum_{k=1}^{9} \R_{N, \eta}^k,
	\end{align}
where 
\begin{align} \label{defr70}
&	\R_{N, \eta}^{8} :=-\mathbf{1}_{s\in ( d-1,d)}  	\operatorname{div} ((\F_{<\eta_*}*\mu) \mu)-\mathbf{1}_{s\geq d}  	\operatorname{div} ((\F^{\eta_*}*\mu^0) \mu^0),\\&\label{defr100}
	\R_{N, \eta}^{9} :=  - \frac{\mathbf{1}_{s\geq d+1}}{N^2}\sum_{i\neq j}\F^{\mathrm{sing}}(x_i,x_j)\chi\(\frac{|x_i-x_j|}{\eta_*}\)\cdot \nab  \pex(x).
\end{align}
Moreover, for almost every $t\in[0,T]$, let $x^t$ be a point where the
maximum defining $\M(X_N^t,\mu^t,\eta)$ is attained, namely
\[
\M(X_N^t,\mu^t,\eta)
=
\langle x^t\rangle^\gamma
|\tilde\mu_{N,\eta}(x^t,t)|.
\]
Then

		\be \label{ptM}\partial_t \M(X_N^t, \mu^t, \eta)\le\langle x_t\rangle^\gamma |\R_{N, \eta}^1(x^t)|+ \sum_{k=2}^{9} \langle x_t\rangle^\gamma |\R_{N, \eta}^k(x^t)|+ \R_{N,\eta}^{10} + \R_{N,\eta}^{11}+ \R_{N,\eta}^{12},\ee
		with
		\be\label{defr7} \R_{N,\eta}^{10}:= \langle x_t\rangle^\gamma |\F^{\eta_*}* \tilde \mu_{N,\eta} (x^t) \cdot  \nab \mu(x^t)|,
		\ee
		\be \label{defr8}
		\R_{N,\eta}^{11} := \gamma |\langle x^t\rangle^{\gamma-2}\tilde \mu_{N,\eta}(x^t)\, x^t\cdot\F^{\eta_*}* \mu(x^t)|,\ee
		\begin{align} \label{defr12}	\R_{N,\eta}^{12} := \gamma\langle x^t\rangle^{\gamma-1}
		|\F^{\eta_*}  *	\tilde \mu_{N, \eta}(x^t)| \,|\tilde \mu_{N,\eta}(x^t) |.
		\end{align}

\end{lem}


\begin{remark}[Classification of the error terms]\label{rem:errorclass}
The error terms $\R^1$, $\R^2$ are the singular short-range errors. They carry the full singularity of $\F$ and will be the hardest to control. The term $\R^3$ is a commutator error measuring the difference between the mollified force field evaluated at the particle position $x_i$ and at the spatial variable $x$. It is the second most delicate term to control. The terms  $\R^5$, $\R^6$ come from Taylor expanding $\F^{\eta_*}_{\kappa_0\eta}*\mu$ around $x$ up to first order. 
\end{remark}

\begin{remark}\label{re-for}
	If an external force $V_N= V_N^1+V_N^2$ is added to the system \eqref{ode}, namely,
	\begin{equation}\label{ode-fo}
		\left\{
		\begin{array}{ll}
			\dot x_i=\displaystyle \frac1N\sum_{\substack{j=1\\j\neq i}}^N \F(x_i,x_j)+V_N^1(x_i)+V_N^2(x_i),\\[1ex]
			x_i(0)=x_i^0,
		\end{array}
		\right.
	\end{equation}
	then \eqref{ptM} becomes
	\begin{align}\label{w7}
		\partial_t \M(X_N^t,\mu^t,\eta)
		&\le
		\langle x^t\rangle^\gamma |\R_{N,\eta}^1(x^t)|
		+\sum_{k=2}^{9}\langle x^t\rangle^\gamma |\R_{N,\eta}^k(x^t)|
		+\R_{N,\eta}^{10}
		+\R_{N,\eta}^{11}
		+\R_{N,\eta}^{12}
		\notag\\
		&+\frac1N\sum_{i=1}^N (|V_N^1(x_i)-V_N^1(x^t)|+|V_N^2(x_i)|)|\nabla \Phi^\eta_{x_i}(x^t)|+\gamma\langle x_t\rangle^{\gamma-1}
		|V_N^1(x_t)| |\tilde \mu_{N,\eta}(x_t) |
	\end{align}
The last term can be controlled in the same way as the corresponding weighted
transport contribution, provided that $V_N^1$ has at most linear growth in $x$.
This observation will be used in the proof of the collision estimate.
\end{remark}
\begin{proof}
	First, we note that by the definitions \eqref{fneta} and \eqref{defFeta}, the convolution $\F^{\eta_*} *\mu_{N,\eta}$ can be expressed as a sum over particles:
	\begin{multline}\label{2.9}
	 \F^{\eta_*}* \mu_{N,\eta}(x)=\int \F^{\eta_*}(x,y) \mu_{N,\eta}(y)\,dy=\frac1N\sum_{ k=1}^N  \int \F^{\eta_*}(x,y)\Phi^{\eta}(y-x_k)\,dy\\=\frac1N\sum_{k=1}^N (\F^{\eta_*}(x,\cdot)\star\Phi^{\eta})(x_k)= \frac1N\sum_{k=1}^N
	 \F^{\eta_*}_\eta(x,x_k).\end{multline}

Now we rewrite the discrete interaction sum. Using the definition \eqref{fneta}, we split the force into its mollified part $\F^{\eta_*}_\eta$ and the short-range remainder $\F^{\eta_*}_{<\eta}$, then separate the evaluation at $x_i$ from the evaluation at~$x$ to obtain
	\begin{align}
 & -\frac1{N^2} \sum_{i=1}^N\sum_{j\neq i}\F^{\eta_*}(x_i,x_j) \cdot \nab \pex
		=-\frac1{N^2} \sum_{i=1}^N \sum_{j\neq i} \F^{\eta_*}_\eta(x_i,x_j)\cdot \nab  \pex - \frac1{N^2}\sum_{i\neq j} \F^{\eta_*}_{<\eta}(x_i,x_j)  \cdot \nab \pex \\
	\notag	&\quad\quad= - \frac1{N^2} \sum_{i=1}^N\sum_{j\neq i}   \F^{\eta_*}_\eta(x,x_j)\cdot \nab\pex(x)-  \frac1{N^2} \sum_{i=1}^N \sum_{j\neq i} (\F^{\eta_*}_\eta(x_i,x_j) - \F^{\eta_*}_\eta(x,x_j) )\cdot\nab\pex(x)
		\\ \notag  & \quad\quad\quad- \frac1{N^2}\sum_{i\neq j} \F^{\eta_*}_{<\eta}(x_i,x_j) \cdot \nab  \pex(x).\end{align}
		Reinserting the diagonal terms $i=j$ (which contribute $\R^4$) and using \eqref{2.9} to recognize $\F^{\eta_*}*\mu_{N,\eta}\cdot\nabla\mu_{N,\eta}$, we may write
		\begin{align*}
&	-\frac1{N^2} \sum_{i=1}^N\sum_{j\neq i}\F^{\eta_*}(x_i,x_j) \cdot \nab \pex=	  - \F^{\eta_*} *\mu_{N,\eta}   \cdot \nab\mu_{N,\eta}+ \frac1{N^2} \sum_{i=1}^N  \F^{\eta_*}_\eta(x,x_i) \cdot\nab\pex(x)\\
		&\quad\quad\quad-\frac1{N}  \sum_{i=1}^N  \big( \F^{\eta_*} * \mu_{N,\eta}(x_i) - \F^{\eta_*} * \mu_{N,\eta}(x)\big)  \cdot  \nab \pex(x)
		\\ & \quad\quad\quad+ \frac1{N^2} \sum_{i=1}^N  (\F^{\eta_*}_\eta(x_i,x_i) - \F^{\eta_*}_\eta(x,x_i) )\cdot \nab\pex(x) - \frac1{N^2}\sum_{i\neq j} \F^{\eta_*}_{<\eta}(x_i,x_j) \cdot \nab  \pex(x)
		\\
		&\quad\quad = - \F^{\eta_*} * \mu_{N,\eta}  \cdot   \nab\mu_{N,\eta}- \frac1{N}  \sum_{i=1}^N  \big( \F^{\eta_*} *  \mu_{N,\eta}(x_i) - \F^{\eta_*} *  \mu_{N,\eta}(x)\big)  \cdot  \nab \pex(x)
		\\ & \quad\quad\quad+  \frac1{N^2} \sum_{i=1}^N  \F^{\eta_*}_\eta(x_i,x_i) \cdot\nab\pex(x) - \frac1{N^2}\sum_{i\neq j} \F^{\eta_*}_{<\eta}(x_i,x_j) \cdot \nab  \pex(x).
	\end{align*}
Next, we use the semigroup decomposition \eqref{defc0} to replace $\F^{\eta_*}*\mu_{N,\eta}$ by $\F^{\eta_*}_{\kappa_0\eta}*\mu_{N,\eta/2}$, and then split $\mu_{N,\eta/2}=\tilde\mu_{N,\eta/2}+\mu$ to separate the error from the smooth background. This produces the commutator term $\R^3$ (from $\tilde\mu_{N,\eta/2}$) and the smooth terms $\R^5$, $\R^6$ (from the Taylor expansion of $\F^{\eta_*}_{\kappa_0\eta}*\mu$ around~$x$). Using \eqref{defc0} and noting that $\F^{\eta_*}* (\Phi^{\kappa_0\eta} \star \mu_{N,\eta/2} )= \F^{\eta_*}_{\kappa_0\eta} * \mu_{N,\eta/2}$,
 we may write
	\begin{align*}& \frac1{N}  \sum_{i=1}^N  \big( \F^{\eta_*}* \mu_{N,\eta}(x_i) - \F^{\eta_*}*  \mu_{N,\eta}(x)\big)  \cdot  \nab \pex(x)
\\ &\quad\quad =\frac1{N}  \sum_{i=1}^N  \big( \F^{\eta_*}_{\kappa_0\eta} * \mu_{N,\eta/2}(x_i) - \F^{\eta_*}_{\kappa_0\eta}*  \mu_{N,\eta/2}(x)\big)  \cdot  \nab \pex(x)\\
&\quad\quad =  \frac1{N}  \sum_{i=1}^N  \big( \F^{\eta_*}_{\kappa_0\eta} * \tilde\mu_{N,\eta/2}(x_i) - \F^{\eta_*}_{\kappa_0\eta}* \tilde\mu_{N,\eta/2}(x)\big)  \cdot  \nab \pex(x) \\
&\quad\quad\quad+\frac1{N}  \sum_{i=1}^N  \big( \F^{\eta_*}_{\kappa_0\eta} * \mu(x_i) - \F^{\eta_*}_{\kappa_0\eta}*  \mu(x)\big)  \cdot  \nab \pex(x).
\end{align*}
	Inserting into the above and using $\div ((\F^{\eta_*}* \mu) \mu)= \mu \,\div (\F^{\eta_*}*\mu)+\nab \mu \cdot \F^{\eta_*}*\mu$, we obtain \eqref{3.3}.

	Now we pass from the algebraic identity to the evolution equation. We compute using the ODE \eqref{ode} and
	$\frac{d}{dt}\pex=-\dot{x}_i\cdot\nab \Phi^{\eta}_{x_i}(x)$ that
	\begin{align*}
		\partial_t \mu_{N,\eta} &=  -\frac1{N^2} \sum_{i=1}^N\sum_{j\neq i}\F(x_i,x_j) \cdot \nab \pex\\&=-\frac1{N^2} \sum_{i=1}^N\sum_{j\neq i}\F^{\eta_*}(x_i,x_j) \cdot \nab \pex+\R_{N, \eta}^{9}.
	\end{align*}
	Subtracting $\partial_t\mu$ (given by \eqref{choixmu}) and applying \eqref{3.3}, we obtain \eqref{eqfordiffeq}. The additional error terms $\R^8$ and $\R^9$ account respectively for the mismatch between $\F$ and $\F^{\eta_*}$ in the continuum equation, and for the cutoff regularization when $s\geq d+1$.

	Next, to derive the differential inequality for $\M$, we multiply \eqref{eqfordiffeq} by the weight $\xg$ and use the product rule to rewrite the transport terms. We immediately deduce that
	\begin{multline}\label{eqfordiffww}
		\partial_t\big(\xg	\tilde \mu_{N, \eta}\big)
		=	 -\F^{\eta_*} * 	\tilde \mu_{N, \eta}   \cdot (\xg \nab \mu) -	 \F^{\eta_*} * 	 \mu_{N, \eta}\cdot \nab (\xg \tilde \mu_{N,\eta} ) \\+\gamma x \langle x\rangle^{\gamma-2}\tilde \mu_{N,\eta} \cdot \F^{\eta_*} * 	 \mu_{N, \eta}+ \sum_{m=1}^{9} \xg \R_{N, \eta}^m.
\end{multline}
By \cite[Lemma~B.1]{ConstantinTarfuleaVicol2015}
(see also the proof of Lemma~5 in \cite{ConstantinIgnatovaNguyen2025}),
the function
\[
t\mapsto \M(X_N^t,\mu^t,\eta)
\]
is absolutely continuous on $[0,T]$. Moreover, for almost every
$t\in[0,T]$, there exists a point $x_t\in\mathbb{R}^d$ such that
\begin{equation}\label{xt3}
	\langle x_t\rangle^\gamma
	|\tilde \mu_{N,\eta}(x_t,t)|
	=
	\M(X_N^t,\mu^t,\eta),
\end{equation}
and
\begin{align}\label{diff-max}
	\frac{d}{dt}\M(X_N^t,\mu^t,\eta)
	=
	\partial_t
	\left[
	\langle x\rangle^\gamma
	|\tilde \mu_{N,\eta}(x,t)|
	\right]_{x=x_t}.
\end{align}
	Since $x_t$ is a maximum point of $x\mapsto \langle x\rangle^\gamma|\tilde\mu_{N,\eta}(x,t)|$, we have the first-order optimality condition
\be \label{cmax}\nab \big(\langle x \rangle^\gamma |\tilde \mu_{N,\eta}(x, t)|\big)\big\vert_{x=x_t}=0,\ee
which cancels the transport terms $\F^{\eta_*}*\mu\cdot\nabla(\xg\tilde\mu_{N,\eta})$ and $\F^{\eta_*}*\tilde\mu_{N,\eta}\cdot\nabla(\xg\tilde\mu_{N,\eta})$ in \eqref{eqfordiffww} at $x=x_t$. The remaining terms give
\begin{multline} \label{ptM'}\partial_t \M(X_N^t, \mu^t, \eta)\le \Big|  \F^{\eta_*} * 	\tilde \mu_{N, \eta}   \cdot (\xg \nab \mu) \Big|\vert_{x=x_t} +\gamma \big|\langle x_t\rangle^{\gamma-2}\tilde \mu_{N,\eta}(x_t)\, x_t\cdot\F^{\eta_*} * 	 \mu(x_t)\big|\\+\gamma\langle x_t\rangle^{\gamma-1}
	|\F^{\eta_*}  * 	\tilde \mu_{N, \eta}(x_t)| \,|\tilde \mu_{N,\eta}(x_t) |+ \sum_{k=1}^{9}  \langle x_t\rangle^{\gamma} |\R_{N, \eta}^k(x_t)|,\end{multline}
which implies the result.
		\end{proof}


\subsection{Control of error terms}\label{sec:error}
		
We now turn to estimating the remainder terms $\R_{N,\eta}^k$ for $k \ge 2$, in terms of the metric~$\M$ and the regularity quantity~$\A$.
The most singular term $\R^1_{N,\eta}$ requires a separate and more delicate treatment and is set aside for now.

\smallskip

We set 
\begin{align}\label{defU}
    \M(t,\eta)&:=\M(\XN^t, \mu^t, \eta)\\
  \U(t,\eta)&:= \M(t,\eta)+ \A(\mu^t),\label{defMshort}
\end{align}
 and will use the fact that $1\lesssim \A(\mu)$ since $\int \mu=1$.

Unless there is ambiguity, we will drop the time dependence and write $\M(\eta)$, $\U(\eta)$ and $\A(\mu)$. These 
 shorthands will be used throughout. 
The goal is to bound every error term by a product of $\M$ (at various scales) with $\U$ or $\A$, so that the resulting inequality can be closed by a Grönwall argument in Section~\ref{sec:mainproof}.

Note that \eqref{triangleU} implies that the full empirical measure is controlled by 
$\langle x\rangle^\gamma |\mu_{N,\eta}(x)| \lesssim \U(\eta)$. We now proceed to bounding the error terms, defined in \eqref{defr2}--\eqref{defr12}, in terms of $\M$ and $\A$, recalling the definitions \eqref{defCeta} and \eqref{defCeta'}.

\begin{prop}\label{pro23}
 Let $\eta= 2^n\eta_0\leq \frac{1}{2}$ with $\eta_0\geq\eta_*>0$.
 In the super-Coulomb regime $s>d-1$, set $\eta_{0}=\eta_*$.
 For every fixed $\eta$ and for almost every $t$, let $x^{t}\in\mathbb R^d$ be a point such that $\M(X_N^t,\mu^t,\eta) = \langle x^t\rangle^\gamma \left| \widetilde\mu_{N,\eta}(t,x^t) \right|.$ The pointwise estimates in parts \emph{(i)} and \emph{(ii)} below are evaluated at $x=x^t$. The dual estimate in part \emph{(iii)} holds globally and does not require the choice of a maximum point.
 Then for every integer $n_0\geq0$ satisfying $2^{n_0}\eta\leq 1,$ and  for each $t$, we have the following estimates. 

\smallskip\noindent\emph{(i) Weighted $L^\infty$ bound on $\R^2$--$\R^8$:} 
\begin{multline}
	\sum_{m=2}^8\xg |\R_{N,\eta}^m|
	\lesssim
	\sum_{k=0}^{n+1}\eta_0^{d-s-1}\,
	2^{k(d+1-s)_{+}-k-n}
	\bigl(1+k\,\mathbf{1}_{s=d+1}\bigr)\,
	\M(2^{k-1}\eta_{0})\,\U(\eta)	\\
	+ \Bigl(
	(\eta^{2}+\mathbf{1}_{s\geq d}\mathcal{C}(\eta_*))	(	\A(\mu)+
	\big\|\xg \nabla^4 \mu\big\|_{L^\infty}	\mathbf{1}_{s\geq d})
	+ 	\frac{1+\mathcal C(\eta,\eta_*)}
	{N\eta_0^{\min\{s,d+1\}}}
	\Bigr)\U(\eta)
	\\
	+ \mathbf{1}_{s\neq d-1}\,
	\eta^{-\min\{2,(s+1-d)_{+}\}}\,
	\bigl(1+	\mathcal C(\eta,\eta_*)\bigr)\,
	\M(\eta/2)\,\U(\eta)
	\\
	+ \mathbf{1}_{s=d-1}\,
	\Bigl(\sum_{k=0}^{n_0-1}\M(2^{k}\eta)
	+|\!\log(2^{n_0}\eta)|\,\M(2^{n_0}\eta)
	+|\!\log(2^{n_0}\eta)|\,2^{2n_0}\eta^2\A(\mu)
	\Bigr)
	\bigl(\M(\eta/2)+\eta^2\A(\mu)\bigr).
	\label{R2-8}
\end{multline}
\smallskip\noindent\emph{(ii) Weighted $L^\infty$ bound on  $\R^{10}$--$\R^{12}$:}
\begin{multline}
	\R_{N,\eta}^{10}+\R_{N,\eta}^{11} +\R_{N,\eta}^{12}
	\lesssim
	\mathbf{1}_{s< d}\, \M(\eta)\, \U(\eta)
	+ \bigl(
	\mathbf{1}_{s=d}|\!\log\eta|
	+\mathbf{1}_{s\in (d,d+1)}\eta^{d-s}
	+\eta^{-1}\mathcal C(\eta,\eta_*)
	\bigr)\,\M(\eta/2)\, \U(\eta)
	\\
	+ \eta^2\mathcal C(\eta_*)\A(\mu)\,\U(\eta).
	\label{R10-12}
\end{multline}
\smallskip\noindent\emph{(iii) Dual estimate on $\R^2$--$\R^8$:}
\begin{multline}
	\Bigl\|\langle \cdot\rangle^{\gamma'}\Bigl(
	\sum_{k=2}^{8} \R_{N, \eta}^k
	+\tilde\mu_{N,\eta}\operatorname{div} (\F^{\eta_*}* \mu )
	+\operatorname{div}(\F^{\eta_*} *  \tilde\mu_{N,\eta})\mu
	\Bigr)\Bigr\|_{(W^{1,1})^{*}}
	\\
	\lesssim
	\eta^{-\min\{2,(s+1-d)_{+}\}}\,
	\bigl(1+	\mathcal C(\eta,\eta_*)\bigr) \M(\eta/2)\,\M(\eta)\\
	+\sum_{k=0}^{n+1}\eta_0^{d-s}\,
	2^{k(d+1-s)_{+}-k}
	\bigl(1+k\,\mathbf{1}_{s=d+1}\bigr)\,
	\M(2^{k-1}\eta_{0})\,\U(\eta)
	\\
	+ \Bigl(
	(   \mathbf 1_{ s<d}\, \eta^2+\mathbf{1}_{s\geq d}	\mathcal C(\eta_*))	(	\A(\mu)+
	\big\|\xg \nabla^4 \mu\big\|_{L^\infty}	\mathbf{1}_{s\geq d})
	+ 	\frac{1+\mathcal C(\eta,\eta_*)}
	{N\eta_0^{\min\{s-1,d\}}}+\frac{1}{N}
	\Bigr)\U(\eta)\\+\eta\mathcal C(\eta_*)  \M(\eta/2)(	\A(\mu)+
	\big\|\xg \nabla^4 \mu\big\|_{L^\infty}	\mathbf{1}_{s\geq d})
	\\
	+\mathbf{1}_{s\neq d-1}\,
	\eta^{1-\min\{2,(s+1-d)_{+}\}}\,
	\bigl(1+	\mathcal C(\eta,\eta_*)\bigr)\,
	\M(\eta/2)\,\U(\eta)
	\\
	+ \mathbf{1}_{s=d-1}\,
	\Bigl(\sum_{k=0}^{n_0-1}\M(2^{k}\eta)
	+|\!\log(2^{n_0}\eta)|\,\M(2^{n_0}\eta)
	+|\!\log(2^{n_0}\eta)|\,2^{2n_0}\eta^2\A(\mu)
	\Bigr)\eta\U(\eta),
	\label{R2-8-dual}
\end{multline}
for any $0\leq \gamma'\leq \gamma-1$.
\end{prop}

\begin{proof}
The proof proceeds by estimating each $\R^k$ individually.
We first treat the most intricate term $\R^2$ in detail (Steps~1--5), then handle $\R^3$--$\R^8$ (Steps~6--11), and finally $\R^{10}$--$\R^{12}$ (Steps~12--14).

\medskip
\noindent
\textbf{Step 1: Telescoping decomposition of $\R^2$ \eqref{defr1'}.}
The key idea is to decompose the multiscale interaction $\F_{<\eta}^{\eta_*}-\F_{<\eta_0}^{\eta_*}$
as a dyadic telescoping sum.
Using \eqref{defFeta} and \eqref{defc0}, we write
\begin{align}\notag
  -\R_{N,\eta}^2
  &= \frac1{N^2}\sum_{i\neq j}
    \bigl(\F_{<\eta}^{\eta_*}(x_i,x_j)
      -\F_{<\eta_0}^{\eta_*}(x_i,x_j) \bigr)
    \cdot \nab  \pex(x)
  \\ \notag
  &= \frac1{N^2}\sum_{i\neq j}
    \bigl(\F_{\eta_0}^{\eta_*}(x_i,x_j)
      -\F_{\eta}^{\eta_*}(x_i,x_j) \bigr)
    \cdot \nab  \pex(x)
  \\[4pt] \notag
  &= \sum_{k=0}^{n-1}\frac1{N^2}\sum_{i\neq j}
    \bigl(\F_{2^{k}\eta_0}^{\eta_*}
      -\F_{2^{k+1}\eta_0}^{\eta_*}\bigr)(x_i,x_j)
    \cdot \nab  \Phi^{\eta}_{x_i}(x)
  \\[4pt] \nonumber 
  &=\sum_{k=0}^{n-1}\frac1{N}\sum_{i=1}^N
    \bigl(\F_{2^{k}\eta_0}^{\eta_*}
      -\F_{2^{k+1}\eta_0}^{\eta_*}\bigr)* \mu_N(x_i)
    \cdot \nab  \Phi^{\eta}_{x_i}(x)
  \\
  &\qquad
 +\sum_{k=0}^{n-1}\frac1{N^2}\sum_{i=1}^N
    \bigl(\F_{2^{k+1}\eta_0}^{\eta_*}
     -\F_{2^k\eta_0}^{\eta_*}\bigr)(x_i,x_i)
    \cdot \nab  \Phi^{\eta}_{x_i}(x)
  \notag
  \\[4pt] \notag
  &= \sum_{k=0}^{n-1}\frac1{N}\sum_{i=1}^N
    \bigl(\F_{\kappa_0 2^{k}\eta_0}^{\eta_*}
      -\F_{\sqrt{15} \cdot  2^{k-1}\eta_0}^{\eta_*}\bigr)
    *\mu_{N,2^{k-1}\eta_{0}}(x_i)
    \cdot \nab  \Phi^{\eta}_{x_i}(x)
  \\
  &\qquad
  +\sum_{k=0}^{n-1}\frac1{N^2}\sum_{i=1}^N
    \bigl(\F_{2^{k+1}\eta_0}^{\eta_*}
      -\F_{2^k\eta_0}^{\eta_*}\bigr)(x_i,x_i)
    \cdot \nab  \Phi^{\eta}_{x_i}(x).
  \nonumber
\end{align}
For the last equality, we applied the semigroup property~\eqref{semigroup}, which introduces the smoothed empirical measure $\mu_{N,2^{k-1}\eta_0}$, defined by~\eqref{fneta}.
We may thus write
\begin{align}\label{x2}
  \R_{N,\eta}^2
  = -\R_{N,\eta}^{2,1}-\R_{N,\eta}^{2,2}-\R_{N,\eta}^{2,3}
    -\R_{N,\eta}^{2,4}-\R_{N,\eta}^{2,5},
\end{align}
where, writing $\tilde \F_k^{\eta_*} :=\F_{\kappa_02^{k}\eta_0}^{\eta_*}-\F_{\sqrt{15}\cdot 2^{k-1}\eta_0}^{\eta_*}$ for the difference at the $k$-th dyadic scale, we set
\begin{alignat}{2}
  \R_{N,\eta}^{2,1}
    &= \sum_{k=0}^{n-1}\frac1{N}\sum_{i=1}^N
      \tilde \F_k^{\eta_*} *\tilde\mu_{N,2^{k-1}\eta_{0}}(x_i)
      \cdot \nab  \Phi^{\eta}_{x_i}(x),
    \label{def:R21}
  \\[3pt]
  \R_{N,\eta}^{2,2}
    &= \sum_{k=0}^{n-1}\frac1{N}\sum_{i=1}^N
      \bigl(\tilde \F_k^{\eta_*}*\mu(x_i)
        -\tilde \F_k^{\eta_*}*\mu(x)\bigr)
      \cdot \nab  \Phi^{\eta}_{x_i}(x),
    \label{def:R22}
  \\[3pt]
  \R_{N,\eta}^{2,3}
    &=\sum_{k=0}^{n-1}
      \tilde \F_k^{\eta_*}*\mu(x)
      \cdot \nab \tilde \mu_{N,\eta}(x),
    \label{def:R23}
  \\[3pt]
  \R_{N,\eta}^{2,4}
    &=\sum_{k=0}^{n-1}
      \tilde \F_k^{\eta_*} *\mu(x)
      \cdot \nab \mu(x),
    \label{def:R24}
  \\[3pt]
  \R_{N,\eta}^{2,5}
    &= \frac1{N^2}\sum_{i=1}^N
      \bigl(\F_{\eta}^{\eta_*}-\F_{\eta_0}^{\eta_*}\bigr)(x_i,x_i)
      \cdot \nab  \Phi^{\eta}_{x_i}(x).
    \label{def:R25}
\end{alignat}

\medskip
\noindent
\textbf{Step 2: The principal term $\R^{2,1}$.}
By definition of $\M$ in \eqref{defM} and of $\tilde \mu_{N,\eta}$ in \eqref{tildefneta}, we may bound the integrand using $|\tilde\mu_{N,2^{k-1}\eta_0}(y)|\leq\M(X_N, \mu, 2^{k-1}\eta_0)\langle y\rangle^{-\gamma}$:
\begin{multline*}
  \xg|\R_{N, \eta}^{2,1}|
 \\
  \lesssim \xg \sum_{k=0}^{n-1}\frac1{N} \sum_{i=1}^N
    \M(\XN,\mu,2^{k-1}\eta_{0})
    \int_{\mr^d}
      \bigl|\F_{\kappa_0 2^{k}\eta_0}^{\eta_*}(x_i,y)
        -\F_{\sqrt{15}2^{k-1}\eta_0}^{\eta_*}(x_i,y)\bigr|
      \frac{dy}{\langle y\rangle^\gamma}
    \,\bigl|\nab  \Phi^{\eta}_{x_i}(x)\bigr|.
\end{multline*}
The $y$-integral is estimated by the kernel bound \eqref{le234}, and the Gaussian tail gives $|\nab \Phi^\eta|\lesssim \eta^{-1}\Phi^{2\eta}$ via~\eqref{prod2exp}.
This yields
\begin{align}\notag
  \xg|\R_{N, \eta}^{2,1}|
  &\lesssim
    \sum_{k=0}^{n-1}
      \eta_0^{d-s}\,2^{k(d+1-s)_{+}-k}
      \bigl(1+k\,\mathbf{1}_{s=d+1}\bigr)\,
      \M(2^{k-1}\eta_{0})\,
      \frac1{N}\sum_{i=1}^N
        \frac{\xg}{\eta}\,\Phi^{2\eta}_{x_i}(x)
  \\[3pt]
  &\lesssim
    \sum_{k=0}^{n-1}\eta_0^{d-s-1}\,
      2^{k(d+1-s)_{+}-k-n}
      \bigl(1+k\,\mathbf{1}_{s=d+1}\bigr)\,
      \M(2^{k-1}\eta_{0})\,\U(\eta),
  \label{R13}
\end{align}
where  we used $\eta_0=\eta_*$ when $s>d-1$ and $\eta=2^n\eta_0$.\\
This estimate reveals a  dichotomy:
when $s\leq d-1$ (sub-Coulomb),  the dyadic sum converges; when $s>d-1$ (super-Coulomb), the factor $\eta_0^{d-s-1}=\eta_*^{d-s-1}\to \infty$ as $N\to\infty$, signalling that the problem is supercritical.

Similarly, replacing $|\nab\Phi^\eta|$ by $|\Phi^\eta|$ in the argument above, we obtain the dual estimate (where we gain one power of $\eta$)
\begin{align}
  \|\langle \cdot\rangle^{\gamma'}\R_{N, \eta}^{2,1}\|_{(W^{1,1})^{*}}
  &\lesssim
    \sum_{k=0}^{n-1}\eta_0^{d-s}\,
      2^{k(d+1-s)_{+}-k}
      \bigl(1+k\,\mathbf{1}_{s=d+1}\bigr)\,
      \M(2^{k-1}\eta_{0})\,\U(\eta),
  \label{R13'}
\end{align}
for any $0\leq \gamma'\leq \gamma$.\\
\medskip
\noindent
\textbf{Step 3: The Taylor term $\R^{2,2}$ and the smooth terms $\R^{2,3}$, $\R^{2,4}$.}
These terms involve $\tilde\F_k*\mu$, which is smooth and bounded.
For $\R^{2,2}$, using the mean value theorem $|\tilde\F_k*\mu(x_i)-\tilde\F_k*\mu(x)|\leq |x_i-x|\|\nab(\tilde\F_k*\mu)\|_{L^\infty}$, together with \eqref{proF4} and \eqref{prod2exp}, and summing the geometric series $\sum_k (2^k\eta_0)^2\lesssim \eta^2$, we find
\begin{align}\label{R14}
  \langle x\rangle^\gamma|\R_{N, \eta}^{2,2}|
  &\lesssim
    \eta^{2}\,\mathcal{C}(\eta_*)
   	(	\A(\mu)+
   \big\|\xg \nabla^4 \mu\big\|_{L^\infty}	\mathbf{1}_{s\geq d})\U(\eta).
\end{align}

For $\R^{2,3}$, since the weight is maximised at $x=x_t$ (i.e.\ $\nab (\xg |\tilde\mu_{N,\eta}(x)|)=0$), the gradient $\nab\tilde\mu_{N,\eta}$ only contributes through the weight factor $\gamma x\langle x\rangle^{\gamma-2}$.
By \eqref{proF3} we get
\begin{align}\label{R15}
  \xg|\R_{N, \eta}^{2,3}(x)|
  &\lesssim
    \eta^2\,
   \mathcal{C}(\eta_*)  	\A(\mu)\M(\eta).
\end{align}

For $\R^{2,4}$, which involves only $\mu$ (no $\M$), the same estimate gives
\begin{align}\label{R16}
  \xg|\R_{N, \eta}^{2,4}(x)|
  &\lesssim
    \eta^2\mathcal{C}(\eta_*)	\A(\mu)^2.
\end{align}
\medskip
\noindent
\noindent\textbf{Step 4: The diagonal term $\R^{2,5}$.}
This is the correction from the diagonal $i=j$ contribution.
Adding and subtracting the value at the point $x$, we obtain
\begin{align}
	\xg |\R_{N,\eta}^{2,5}|
	&\leq
	\xg \frac{1}{N^2}\sum_{i=1}^N
	\left|
	(\F_{\eta}^{\eta_*}-\F_{\eta_0}^{\eta_*})(x_i,x_i)
	-
	(\F_{\eta}^{\eta_*}-\F_{\eta_0}^{\eta_*})(x,x)
	\right|
	|\nab \Phi_{x_i}^{\eta}(x)|
	\nonumber \\
	&\quad
	+ \xg \frac{1}{N}
	\left|
	(\F_{\eta}^{\eta_*}-\F_{\eta_0}^{\eta_*})(x,x)
	\right|
	|\nab \tilde \mu_{N,\eta}(x)|
	\nonumber \\
	&\quad
	+ \xg \frac{1}{N}
	\left|
	(\F_{\eta}^{\eta_*}-\F_{\eta_0}^{\eta_*})(x,x)
	\right|
	|\nab \mu(x)|.
	\label{q22}
\end{align}
By \eqref{226}, \eqref{226b}, and \eqref{226c}, we have
\begin{align}
	\left|
	(\F_{\eta}^{\eta_*}-\F_{\eta_0}^{\eta_*})(x,x)
	\right|
	&\lesssim
	\frac{1+\mathcal C(\eta,\eta_*)}{\eta_0^{\min\{s-1,d\}}}
	+1,
	\label{diag-bound-1}
	\\
	\left|
	(\F_{\eta}^{\eta_*}-\F_{\eta_0}^{\eta_*})(x_i,x_i)
	-
	(\F_{\eta}^{\eta_*}-\F_{\eta_0}^{\eta_*})(x,x)
	\right|
	&\lesssim
	\mathbf{1}_{s<d}\eta_0^{-s}|x-x_i|
	\nonumber \\
	&\quad
	+\mathbf{1}_{s\ge d}
	\left(
	\frac{1+\mathcal C(\eta,\eta_*)}{\eta_0^{\min\{s-1,d\}}}
	+1
	\right).
	\label{diag-bound-2}
\end{align}
For the second term in \eqref{q22}, we use the fact that
\[
\nab\bigl(\xg |\tilde\mu_{N,\eta}|\bigr)(x)=0.
\]
The third term is bounded directly by $\A$. Combining these bounds gives
\begin{align}
	\xg |\R_{N,\eta}^{2,5}|
	&\lesssim
	\frac{1+\mathcal C(\eta,\eta_*)}
	{N\eta_0^{\min\{s,d+1\}}}
	\,\U(\eta).
	\label{R17}
\end{align}
We now derive the corresponding dual estimate. We write
\[
\R_{N,\eta}^{2,5}
=
\operatorname{div}
\left(
\frac{1}{N^2}\sum_{i=1}^N
\bigl(\F_{\eta}^{\eta_*}
-
\F_{\eta_0}^{\eta_*}\bigr)(x_i,x_i)
\Phi_{x_i}^{\eta}(x)
\right).
\]
Hence, by \eqref{226b},
\begin{align}
	\|
	\langle \cdot\rangle^{\gamma'}
	\R_{N,\eta}^{2,5}
	\|_{(W^{1,1})^*}
	&\lesssim
	\frac{1}{N}
	\left(
	\frac{1+\mathcal C(\eta,\eta_*)}
	{\eta_0^{\min\{s-1,d\}}}
	+1
	\right)
	\U(\eta),
	\label{R17'}
\end{align}
for any $0\leq \gamma'\leq \gamma$.\\
\medskip
\noindent
\textbf{Step 5: Assembly of the $\R^2$ bound.}
Combining estimates \eqref{R13}, \eqref{R14}, \eqref{R15}, \eqref{R16}, and \eqref{R17}, we obtain
\begin{multline}\label{boundR2tot}
  \langle x\rangle^\gamma|\R_{N,\eta}^2|
  \lesssim
    \sum_{k=0}^{n-1}\eta_0^{d-s-1}\,
      2^{k(d+1-s)_{+}-k-n}
      \bigl(1+k\,\mathbf{1}_{s=d+1}\bigr)\,
      \M(2^{k-1}\eta_{0})\,\U(\eta)
  \\
  + \Bigl(
      \eta^{2}\mathcal{C}(\eta_*)	(	\A(\mu)+
      \big\|\xg \nabla^4 \mu\big\|_{L^\infty}	\mathbf{1}_{s\geq d})
    + 	\frac{1+\mathcal C(\eta,\eta_*)}
    {N\eta_0^{\min\{s,d+1\}}}
  \Bigr)\U(\eta),
\end{multline}
and the analogous dual estimate (using \eqref{R13'}, \eqref{R17'})
\begin{multline}\label{boundR2tot-1}
  \|\langle \cdot\rangle^{\gamma'}\R_{N, \eta}^{2}\|_{(W^{1,1})^{*}}
  \lesssim
    \sum_{k=0}^{n-1}\eta_0^{d-s}\,
      2^{k(d+1-s)_{+}-k}
      \bigl(1+k\,\mathbf{1}_{s=d+1}\bigr)\,
      \M(2^{k-1}\eta_{0})\,\U(\eta)
  \\
  + \Bigl(
  \eta^{2}\mathcal{C}(\eta_*)	(	\A(\mu)+
  \big\|\xg \nabla^4 \mu\big\|_{L^\infty}	\mathbf{1}_{s\geq d})
  + 	\frac{1+\mathcal C(\eta,\eta_*)}
  {N\eta_0^{\min\{s-1,d\}}}+\frac{1}{N}
  \Bigr)\U(\eta),
\end{multline}
for any $0\leq \gamma'\leq \gamma$.\\
\medskip
\noindent
\begin{align*}
	 \R_{N, \eta}^3:= -\frac1{N}  \sum_{i=1}^N  \big( \F^{\eta_*}_{\kappa_0\eta} *  \tilde\mu_{N,\eta/2}(x_i) - \F^{\eta_*}_{\kappa_0\eta} *  \tilde\mu_{N,\eta/2}(x)\big)  \cdot  \nab \pex(x),
\end{align*}
\textbf{Step 6: First bound on $\R^3$ \eqref{defr3}.}
The term $\R^3$ captures the error from replacing $\F_{\kappa_0\eta}*\tilde\mu_{N,\eta/2}(x_i)$ by its value at $x_i=x$.
A direct estimate gives
\begin{equation}\label{232}
  |\xg \R_{N, \eta}^3|
  \le \frac{1}{N}\sum_{i=1}^N
    \|\nab \F_{\kappa_0\eta}^{\eta_*}  *\tilde \mu_{N,\eta/2}\|_{L^\infty}\,
    \frac{|x_i-x|^2}{2\eta^2}\,
    \Phi_{x_i}^{\eta}(x)\,\xg .
\end{equation}
Using \eqref{le231} to bound the kernel integral, we obtain
\begin{align}\nonumber
&  \left|\int_{\mr^d}
  \nab_x \F_{\kappa_0\eta}^{\eta_*} (x,y)\,\tilde \mu_{N,\eta/2}(y)\, dy
  \right|
  \le \M(\eta/2)\int_{\mr^d}
    \frac{|\nab_x \F_{\kappa_0\eta}^{\eta_*} (x,y)|}{\langle y\rangle^\gamma}\, dy
  \\[3pt] \nonumber
  &\qquad\lesssim
    \bigl(  \mathbf{1}_{s\neq d-1}\,
    \eta^{-\min\{2,(s+1-d)_{+}\}}\,
    \bigl(1+	\mathcal C(\eta,\eta_*)\bigr)
          +\mathbf{1}_{s=d-1}|\!\log\eta|
    \bigr) \M(\eta/2),
\end{align}
and by \eqref{prod2exp} and \eqref{semigroup2}, it follows that 
\begin{align}
  |\xg\R_{N, \eta}^3| 
  \lesssim
    \bigl(  \mathbf{1}_{s\neq d-1}\,
 \eta^{-\min\{2,(s+1-d)_{+}\}}\,
 \bigl(1+	\mathcal C(\eta,\eta_*)\bigr)
 +\mathbf{1}_{s=d-1}|\!\log\eta|
 \bigr)\M(\eta/2)\U(\eta).
  \label{boundR3}	
\end{align}
\medskip
\noindent
\textbf{Step 7: Improved bound on $\R^3$ in the Coulomb case $\boldsymbol{s=d-1}$.}
In the Coulomb case, the factor $|\!\log\eta|\,\M(\eta)$ in \eqref{boundR3} is too large for closing the Grönwall argument.
The idea is to perform a second-order Taylor expansion and use a dyadic decomposition of $\nab\F_{\kappa_0\eta}*\tilde\mu_{N,\eta/2}$ to trade the logarithmic cost for a sum of $\M$ at coarser scales.

Concretely, we split $\R^3$ into a second-order Taylor remainder plus the first-order term:
\begin{align} \notag
  \R_{N,\eta}^3
  &= -
    \frac1{N}  \sum_{i=1}^N  \bigl( \F_{\kappa_0\eta} *  \tilde\mu_{N,\eta/2}(x_i) - \F_{\kappa_0\eta}*  \tilde\mu_{N,\eta/2}(x)-\nab \F_{\kappa_0\eta}*  \tilde\mu_{N,\eta/2}(x)\,(x_i-x)\bigr)   \cdot \nab_x \Phi^{\eta}_{x_i}(x)
   \\[3pt]\notag
  &\quad-
    \frac1{N}  \sum_{i=1}^N \nab\F_{\kappa_0\eta}*  \tilde\mu_{N,\eta/2}(x)\, (x_i-x)   \cdot \nab_x \Phi^{\eta}_{x_i}(x)
  .
\end{align}
Applying the identity~\eqref{z15'} with $A=\nab \F_{\kappa_0\eta}*  \tilde\mu_{N,\eta/2}(x)$ to the first-order term, we rewrite
\begin{align}
  \label{r42}
  \R_{N,\eta}^3
  &=- \frac1{N}  \sum_{i=1}^N  \bigl( \F_{\kappa_0\eta} *  \tilde\mu_{N,\eta/2}(x_i) -\F_{\kappa_0\eta}*  \tilde\mu_{N,\eta/2}(x)-\nab\F_{\kappa_0\eta}*  \tilde\mu_{N,\eta/2}(x)\, (x_i-x)\bigr)  \cdot  \nab_x \Phi^{\eta}_{x_i}(x)
  \notag \\
  &\quad -
  \operatorname{div}\F_{\kappa_0\eta}*  \tilde\mu_{N,\eta/2}(x)\;  \mu_{N,\eta}(x) -
  \eta^2  \nab \F_{\kappa_0\eta}*  \tilde\mu_{N,\eta/2}(x)\,  \nab^2 \mu_{N,\eta}(x).
\end{align}
We now estimate each term of \eqref{r42}.

For the second-order remainder, using $s=d-1$ and \eqref{le232} give
\begin{align}
  \notag
  &\langle x\rangle^{\gamma}\Bigl|	\frac1{N}  \sum_{i=1}^N  \bigl( \F_{\kappa_0\eta} *  \tilde\mu_{N,\eta/2}(x_i) -\F_{\kappa_0\eta}*  \tilde\mu_{N,\eta/2}(x)-\nab\F_{\kappa_0\eta}*  \tilde\mu_{N,\eta/2}(x)\,(x_i-x)\bigr)  \cdot  \nab_x \Phi^{\eta}_{x_i}(x)\Bigr|
  \\
  &\quad
  \lesssim
    \|\nab^2\F_{\kappa_0\eta} *  \tilde\mu_{N,\eta/2}\|_{L^\infty}\,
    \frac{\langle x\rangle^{\gamma}}{N}
    \sum_{i=1}^N |x_i-x|^2\,| \nab_x \Phi^{\eta}_{x_i}(x)|
  \notag\\
  &\quad
  \lesssim (1+\eta^{-1})\,\M(\eta/2)\, \eta\,\M(t, 0, 2\eta)
  \lesssim  \M(\eta/2)\, \U(\eta).
  \label{3etr3}
\end{align}

For the divergence term, using \eqref{coudet}, we have for $\gamma_0\in (d,d+1)$
\begin{align}
  \langle x\rangle^{\gamma}
  |(\operatorname{div}\F_{\kappa_0\eta})*  \tilde\mu_{N,\eta/2}(x)\,\mu_{N,\eta}(x) |
  &\lesssim
   \|\langle \cdot\rangle^{\gamma_0}\tilde\mu_{N,\eta/2} \|_{L^\infty}\U(\eta)
 \lesssim \M( \eta/2)\,\U(\eta).
  \label{4etre3}
\end{align}
Finally we examine the last term in the right-hand side of \eqref{r42}, and use \eqref{d2m} to  write that
\begin{equation} \label{5etre3}
	|\nab\F _{\kappa_0 \eta}*  \tilde\mu_{N,\eta}(x)| \langle x\rangle^{\gamma}  |\eta^2  \nab^2 \mu_{N,\eta}(x) |\lesssim \|\nab \F _{\kappa_0 \eta}*  \tilde\mu_{N,\eta}\|_{L^\infty}(\M(t,\mu,\eta/2)+\eta^2\A(\mu)).
\end{equation}
It remains to bound the $L^\infty$ norm in the right-hand side.
In view of \eqref{defc0}, 
we have  for $\eta\leq 2^{n_0}\eta\leq \frac{1}{2}$,
\begin{align*}
	\F_{\kappa_0\eta}^{\eta_*} * \tilde \mu_{N,\eta/2}(x)
	&=
	\bigl(\F_{\kappa_0\eta}^{\eta_*}-\F_{2\eta}^{\eta_*}\bigr)
	* \tilde \mu_{N,\eta/2}(x) +
	\sum_{k=1}^{n_0}
	\bigl(\F_{2^k\eta}^{\eta_*}-\F_{2^{k+1}\eta}^{\eta_*}\bigr)
	* \tilde \mu_{N,\eta/2}(x)  \\
	&\quad+
	\F_{2^{n_0+1}\eta}^{\eta_*} * \tilde \mu_{N,\eta/2}(x).
\end{align*}
Using \eqref{defc0}, the intermediate differences can be rewritten in the
regularized form
\begin{align*}
	\F_{2^k\eta}^{\eta_*}-\F_{2^{k+1}\eta}^{\eta_*}
	=
	\bigl(\F_{\kappa_0 2^k\eta}^{\eta_*}
	-\F_{\sqrt{15}\,2^{k-1}\eta}^{\eta_*}\bigr)
	* \Phi^{2^{k-1}\eta},
	\qquad 1\leq k\leq n_0-1.
\end{align*}
Moreover, again by \eqref{defc0},
\begin{align*}
	\F_{\sqrt{15}\,2^{n_0-1}\eta}^{\eta_*}
	* \Phi^{2^{n_0-1}\eta}
	=
	\F_{2^{n_0+1}\eta}^{\eta_*}
	=
	\F_{\kappa_0 2^{n_0+1}\eta}^{\eta_*}
	* \Phi^{2^{n_0}\eta}.
\end{align*}
Therefore,
\begin{align*}
	\F_{\kappa_0\eta}^{\eta_*} * \tilde \mu_{N,\eta/2}(x)
	&=
	\bigl(\F_{\kappa_0\eta}^{\eta_*}-\F_{2\eta}^{\eta_*}\bigr)
	* \tilde \mu_{N,\eta/2}(x)  \\
	&\quad+
	\sum_{k=1}^{n_0-1}
	\bigl(\F_{\kappa_0 2^k\eta}^{\eta_*}
	-\F_{\sqrt{15}\,2^{k-1}\eta}^{\eta_*}\bigr)
	* \Phi^{2^{k-1}\eta} * \tilde \mu_{N,\eta/2}(x)  \\
	&\quad+
	\F_{\kappa_0 2^{n_0+1}\eta}^{\eta_*}
	* \Phi^{2^{n_0}\eta} * \tilde \mu_{N,\eta/2}(x).
\end{align*}
In view of \eqref{eta1eta2}, we have 
\begin{align*}
	\Phi^{2^{k-1}\eta}\star\tilde \mu_{N,\eta/2}(x)= 	\Phi^{\eta/2}\star\tilde \mu_{N,2^{k-1}\eta}(x)+\Phi^{\eta/2} \star\mu(x)-	\Phi^{2^{k-1}\eta}\star \mu(x),
\end{align*}
and using  \eqref{le231}, \eqref{le233},  \eqref{Dphi}  and \eqref{Mee}--\eqref{Meta1eta2}, we may obtain
\begin{align}\label{z11}
		\|\nab \F_{\kappa_0\eta}^{\eta_*}  *\tilde \mu_{N,\eta/2}\|_{L^\infty} &\leq  	\|\nab	\left(\F_{\kappa_0\eta}^{\eta_*}-\F_{2\eta}^{\eta_*} \right)*\tilde \mu_{N,\eta/2}\|_{L^\infty}\\
	\notag & +\sum_{k=1}^{n_0-1}	\|\nab \left(\F_{\kappa_0 2^{k}\eta}^{\eta_*}-\F_{\sqrt{15}2^{k-1}\eta}^{\eta_*} \right)* 	\Phi^{\eta/2}\star \tilde \mu_{N,2^{k-1}\eta}\|_{L^\infty}\\&\nonumber+\sum_{k=1}^{n_0-1}	\|\nab \left(\F_{\kappa_0 2^{k}\eta}^{\eta_*}-\F_{\sqrt{15}2^{k-1}\eta}^{\eta_*} \right)* (\Phi^{\eta/2} \star\mu-	\Phi^{2^{k-1}\eta}\star\mu)\|_{L^\infty} \\ \notag
	&+\| \nab \F_{ \kappa_02^{n_0+1}\eta}^{\eta_*} *\Phi^{\eta/2}\star \tilde \mu_{N,2^{n_0}\eta}\|_{L^\infty}\\&+\| \nab \F_{ \kappa_02^{n_0+1}\eta}^{\eta_*} *\Phi^{2^{n_0}\eta}\star (\Phi^{\eta/2} \star\mu-	\Phi^{2^{n_0}\eta}\star\mu)\|_{L^\infty}
 \\&
 \lesssim \sum_{k=0}^{n_0-1}\M(2^{k}\eta)+|\log(2^{n_0}\eta)|\M(2^{n_0}\eta)+|\log(2^{n_0}\eta)|2^{2n_0}\eta^2\A(\mu).\nonumber
\end{align}

The logarithmic factor now appears multiplied by $\M$ at the coarse scale $2^{n_0}\eta$, which is more favorable for closing the estimate, since $\M$ at coarse scales is expected to be small.

Combining \eqref{3etr3}--\eqref{z11} and using \eqref{d2m}, we get
\begin{align}  \label{boundR3'}
  |\xg\R_{N, \eta}^3|
 & \lesssim
    \mathbf{1}_{s\neq d-1}\,
    \eta^{-\min\{2,(s+1-d)_{+}\}}\,
  \bigl(1+	\mathcal C(\eta,\eta_*)\bigr)
    \M(\eta/2)\,\U(\eta)
\\
  &+\mathbf{1}_{s=d-1}\,
    \Bigl(\sum_{k=0}^{n_0-1}\M(2^{k}\eta)
      +|\!\log(2^{n_0}\eta)|\,\M(2^{n_0}\eta)
      +|\!\log(2^{n_0}\eta)|\,2^{2n_0}\eta^2\,\A(\mu)
    \Bigr)\bigl(\M(\eta/2)+\eta^2\A(\mu)\bigr)
  \notag\\
  &+ \mathbf{1}_{s=d-1}\,\M( \eta/2)\,\U(\eta).
 \notag
\end{align}
\smallskip
\noindent
\textbf{Step 8: Dual estimate for $\R^3$.}
By the semigroup property and the definition of the truncated kernel, we have
\begin{equation}\label{eq:G-eta-identity}
\F_{\kappa_0\eta}^{\eta_*}
* \tilde\mu_{N,\eta/2}
	=
	\F^{\eta_*}*\widetilde\mu_{N,\eta}
	+
	\F_{<\kappa_0\eta}^{\eta_*}*\mu .
\end{equation}
Indeed,
\begin{align*}
	\F_{\kappa_0\eta}^{\eta_*}*\widetilde\mu_{N,\eta/2}
	&=
	\F_{\kappa_0\eta}^{\eta_*}*\mu_{N,\eta/2}
	-
	\F_{\kappa_0\eta}^{\eta_*}*\mu
	=
	\F^{\eta_*}*\mu_{N,\eta}
	-
	\F_{\kappa_0\eta}^{\eta_*}*\mu
	\\
	&=
	\F^{\eta_*}*\widetilde\mu_{N,\eta}
	+
	\F_{<\kappa_0\eta}^{\eta_*}*\mu .
\end{align*}
Using \eqref{eq:G-eta-identity}, we rewrite
$\R_{N,\eta}^3+\operatorname{div}(\F^{\eta_*}*\widetilde\mu_{N,\eta})\mu$
in divergence form as
\begin{align*}
	\R_{N,\eta}^3
	+\operatorname{div}(\F^{\eta_*} * \tilde\mu_{N,\eta})\, \mu  
	&=
	-\operatorname{div}\Biggl(
	\frac1N \sum_{i=1}^N
	\Bigl[
	\F_{\kappa_0\eta}^{\eta_*} * \tilde\mu_{N,\eta/2}(x_i)
	-
	\F_{\kappa_0\eta}^{\eta_*} * \tilde\mu_{N,\eta/2}(x)
	\Bigr]\pex(x)
	\Biggr)
	\\
	&\quad
	-
	\operatorname{div}\bigl(\F_{\kappa_0\eta}^{\eta_*}
	* \tilde\mu_{N,\eta/2}\bigr)(x)\,
	\tilde\mu_{N,\eta}(x)
	-
	\operatorname{div}\bigl(\F_{<\kappa_0\eta}^{\eta_*} * \mu\bigr)(x)\,\mu(x).
\end{align*}
By \eqref{proF4}, we have
\begin{align*}
	\|\operatorname{div}\F_{<\kappa_0\eta}^{\eta_*} * \mu\|_{L^\infty}
	\lesssim
	\eta^2\mathcal C(\eta_*)
	\left(
	\A(\mu)
	+
	\|\xg \nabla^4 \mu\|_{L^\infty}\mathbf{1}_{s\geq d}
	\right).
\end{align*}
Next, by \eqref{le231} for $s\not= d-1$ and \eqref{coudet} in the Coulomb case
$s=d-1$, we obtain
\begin{align*}
	\|\operatorname{div}\F_{\kappa_0\eta}^{\eta_*}
	* \tilde\mu_{N,\eta/2}\|_{L^\infty}
	&\lesssim
	\mathbf{1}_{s\neq d-1}\,
	\eta^{-\min\{2,(s+1-d)_{+}\}}
	\bigl(1+\mathcal C(\eta,\eta_*)\bigr)
	\M(\eta/2)+
	\mathbf{1}_{s=d-1}
	\|\langle \cdot\rangle^{\gamma_0} \tilde\mu_{N,\eta/2}\|_{L^\infty}
	\\
	&\lesssim
	\eta^{-\min\{2,(s+1-d)_{+}\}}
	\bigl(1+\mathcal C(\eta,\eta_*)\bigr)
	\M(\eta/2).
\end{align*}
Moreover, applying the same dyadic argument as in Step~7, we also have
\begin{align*}
&	\|\nabla_x\F_{\kappa_0\eta}^{\eta_*}
	* \tilde\mu_{N,\eta/2}\|_{L^\infty}
	\lesssim
	\mathbf{1}_{s\neq d-1}\,
	\eta^{-\min\{2,(s+1-d)_{+}\}}
	\bigl(1+\mathcal C(\eta,\eta_*)\bigr)
	\M(\eta/2)
	\\
	&+
	\mathbf{1}_{s=d-1}
	\Biggl(
	\sum_{k=0}^{n_0-1}\M(2^k\eta)
	+
	|\log(2^{n_0}\eta)|\,\M(2^{n_0}\eta)
	+
	|\log(2^{n_0}\eta)|\,2^{2n_0}\eta^2\,\A(\mu)
	\Biggr).
\end{align*}
We now estimate the dual norm. Combining the divergence representation above with the preceding bounds, we obtain, for every $0\le \gamma'\le \gamma$,
\begin{align}   \label{boundR3-1}
	&\|\langle \cdot\rangle^{\gamma'}
	\bigl(
	\R_{N,\eta}^3
	+
	\operatorname{div}(\F^{\eta_*} * \tilde\mu_{N,\eta})\,\mu
	\bigr)\|_{(W^{1,1})^{*}}
	\\&\quad\quad\lesssim \nonumber
	\|\nabla_x\F_{\kappa_0\eta}^{\eta_*}
	* \tilde\mu_{N,\eta/2}\|_{L^\infty}	\sup_x(\frac1N \langle x\rangle^{\gamma'} \sum_{i=1}^N|x-x_i||\pex(x)|)\\&\quad\quad\quad\quad+
	\|\operatorname{div}\F_{\kappa_0\eta}^{\eta_*}
	* \tilde\mu_{N,\eta/2}\|_{L^\infty}\|\langle \cdot\rangle^{\gamma'}	\tilde\mu_{N,\eta}\|_{L^\infty}+
	\|\operatorname{div}\F_{<\kappa_0\eta}^{\eta_*}
	* \mu\|_{L^\infty}\|\langle \cdot\rangle^{\gamma'}	\mu\|_{L^\infty}
	\nonumber\\
	&\quad\lesssim
	\mathbf{1}_{s\neq d-1}\,
	\eta^{1-\min\{2,(s+1-d)_{+}\}}
	\bigl(1+\mathcal C(\eta,\eta_*)\bigr)
	\M(\eta/2)\U(\eta)
	\nonumber\\
	&\qquad+
	\mathbf{1}_{s=d-1}
	\Biggl(
	\sum_{k=0}^{n_0-1}\M(2^k\eta)
	+
	|\log(2^{n_0}\eta)|\,\M(2^{n_0}\eta)
	+
	|\log(2^{n_0}\eta)|\,2^{2n_0}\eta^2\,\A(\mu)
	\Biggr)\eta\U(\eta),\nonumber
	\\
	&\qquad+
	\eta^{-\min\{2,(s+1-d)_{+}\}}
	\bigl(1+\mathcal C(\eta,\eta_*)\bigr)
	\M(\eta/2)\M(\eta)
	\nonumber\\
	&\qquad+
	\eta^2\mathcal C(\eta_*)
	\left(
	\A(\mu)
	+
	\|\xg \nabla^4 \mu\|_{L^\infty}\mathbf{1}_{s\geq d}
	\right)\A(\mu).
	\nonumber
\end{align}
\smallskip
\noindent
\textbf{Step 9: The diagonal term $\R^4$ \eqref{defr1}.}
This term corresponds to the diagonal correction
$\F_{\eta}^{\eta_*}(x_i,x_i)$ arising from the full convolution.
It can be treated in the same way as $\R^{2,5}$ in Step~4. We obtain
\begin{align}
	\xg |\R^4|
	&\leq
	\xg \frac{1}{N^2}\sum_{i=1}^N
	\left|
	\F_{\eta}^{\eta_*}(x_i,x_i)
	-\F_{\eta}^{\eta_*}(x,x)
	\right|
	|\nab \Phi_{x_i}^{\eta}(x)|
	\nonumber \\
	&\quad
	+ \xg \frac{1}{N}
	\left|\F_{\eta}^{\eta_*}(x,x)
	\right|
	|\nab \tilde \mu_{N,\eta}(x)|
	+ \xg \frac{1}{N}
	\left|\F_{\eta}^{\eta_*}(x,x)
	\right|
	|\nab \mu(x)|.
	\label{q22'}
\end{align}
By \eqref{226}, \eqref{q51} and \eqref{q53}, \eqref{q54} we have
\begin{align*}
&\langle x\rangle^{-1}	\left|\F_{\eta}^{\eta_*}(x,x)
	\right|\lesssim
	\frac{1+\mathcal C(\eta,\eta_*)}{\eta_0^{\min\{s-1,d\}}}
	+1,
	\\
	\left|\F_{\eta}^{\eta_*}(x_i,x_i)
	-\F_{\eta}^{\eta_*}(x,x)
	\right|
	&\lesssim 	\left|\F_{\eta}^{\mathrm{reg}}(x_i,x_i)
	-\F_{\eta}^{\mathrm{reg}}(x,x)
	\right|+	\left|\F_{\eta}^{\mathrm{sing},\eta_*}(x_i,x_i)
	-\F_{\eta}^{\mathrm{sing},\eta_*}(x,x)
	\right|\\&\lesssim |x-x_i|+
	\mathbf{1}_{s<d}\eta^{-s}|x-x_i|
	+\mathbf{1}_{s\ge d}
	\left(
	\frac{1+\mathcal C(\eta,\eta_*)}{\eta^{\min\{s-1,d\}}}
	+1
	\right).
\end{align*}
For the second term in \eqref{q22'}, we use the fact that
\[
\nab\bigl(\xg |\tilde\mu_{N,\eta}|\bigr)(x)=0.
\]
The third term is bounded directly by $\A$. Combining these bounds gives
\begin{align}
	\xg |\R_{N,\eta}^{4}|
	&\lesssim
	\frac{1+\mathcal C(\eta,\eta_*)}
	{N\eta^{\min\{s,d+1\}}}
	\,\U(\eta).
	\label{R4}
\end{align}
We next derive the corresponding dual estimate. Since
\[
\R_{N,\eta}^{4}
=
\operatorname{div}
\left(
\frac{1}{N^2}\sum_{i=1}^N
\F_{\eta}^{\eta_*}(x_i,x_i)
\Phi_{x_i}^{\eta}(x)
\right),
\]
we get
\begin{align}
	\|
	\langle \cdot\rangle^{\gamma'}
	\R_{N,\eta}^{4}
	\|_{(W^{1,1})^*}
	&\lesssim
	\frac{1}{N}
	\left(
	\frac{1+\mathcal C(\eta,\eta_*)}
	{\eta_0^{\min\{s-1,d\}}}
	+1
	\right)
	\U(\eta),
	\label{R17''}
\end{align}
for any $0\leq \gamma'\leq \gamma-1$.  The restriction
$\gamma'\le \gamma-1$ comes from this diagonal term: the factor
$\langle x\rangle^{-1}$ is needed to absorb the possible linear growth of
$\F_{\eta}^{\eta_*}(x,x)$. \\
\smallskip
\noindent
\textbf{Step 10: The higher-order term $\R^5$ \eqref{defr4}.}
This term arises from the third-order Taylor remainder from expanding $h:=\F_{\kappa_0\eta}^{\eta_*}*\mu$ around~$x$.
Using \eqref{z15b} to control the moment, and \eqref{prod2exp}, \eqref{d2m}, \eqref{semigroup}, \eqref{227}, we obtain for any $x\in \mathbb{R}^d$
\begin{align*}
  |\xg\R_{N, \eta}^5|
  &\lesssim
    \frac1{N}  \sum_{i=1}^N
      \frac{|x_i-x|^4}{\eta^2}\,\pex(x)\,
      \|D_x^3\F_{\kappa_0\eta}^{\eta_*}  *  \mu\|_{L^\infty}
  \\
  &\quad
  + \bigl(\eta^2|\nab \mu_{N,\eta}|
    +\eta^4|\nab^3\mu_{N,\eta}|\bigr)\,
    \|D_x^2\F_{\kappa_0\eta}^{\eta_*} *  \mu\|_{L^\infty}.
\end{align*}
Indeed, we split
$D^2h(x+\tau(x_i-x))=D^2h(x)+\mathcal{O}(|x_i-x|\,\|D^3h\|_{L^\infty})$.
The $D^3h$ correction contributes
$\mathcal{O}(|x_i-x|^3\|D^3h\|_{L^\infty})$ to the remainder;
multiplied by $|\nab\pex|\lesssim|x_i-x|\eta^{-2}\pex$,
this gives the first line.

For the $D^2h(x)$ part, since $D^2h(x)$ is independent of~$x_i$,
we apply the Hermite identity \eqref{z15'} to trade
$(x_i-x)^{\otimes 2}\cdot\nab\pex$ for derivatives of $\pex$:
the leading terms are $\eta^4\nab^3\pex$ and $\eta^2\nab\pex$.
Summing $\frac{1}{N}\sum_i$ converts these into
$\eta^4\nab^3\mu_{N,\eta}$ and $\eta^2\nab\mu_{N,\eta}$,
giving the second line after multiplication by
$\|D^2\F_{\kappa_0\eta}^{\eta_*}*\mu\|_{L^\infty}$.

We conclude that 
\begin{align}\label{boundR4}
  \|\xg\R_{N, \eta}^5\|_{L^\infty}
  &\lesssim
	\mathcal C(\eta_*) \eta (\M(\eta/2)+ \eta \A(\mu))	\left(
	\A(\mu)
	+
	\|\xg \nabla^4 \mu\|_{L^\infty}\mathbf{1}_{s\geq d}
	\right).
\end{align}
\smallskip
\noindent
\textbf{Step 11: The mean-field terms $\R^6$, $\R^7$, $\R^8$.}
These terms involve the interaction of $\F*\mu$ (a smooth field) with the error $\tilde\mu_{N,\eta}$ or with~$\mu$ itself.
Their treatment is more straightforward since $\F*\mu$ is controlled by~$\A(\mu)$.

For $\R^6$ \eqref{defr5}, inserting \eqref{z15}, we get a decomposition into a term involving $\operatorname{div}(\F*\mu)\,\tilde\mu_{N,\eta}$ (controlled by 	$\A(\mu)$ and $\M(\eta)$) and a term involving $\nab\F*\mu:\nab^2\mu_{N,\eta}$ (controlled via~\eqref{d2m}):
\begin{align*}
  \R_{N,\eta}^6
  &= -\frac1{N}  \sum_{i=1}^N
    \nab_x\F_{\kappa_0\eta}^{\eta_*} *  \mu(x)\, (x_i-x)\cdot  \nab \pex(x)
    +\mu \operatorname{div}(\F_{\kappa_0\eta}^{\eta_*}*\mu)
  \\
  &= -(\operatorname{div}\F_{\kappa_0 \eta}^{\eta_*} *  \mu)\, \tilde \mu_{N,\eta}
    -\eta^2\, \nab_x\F_{\kappa_0 \eta}^{\eta_*} *  \mu(x) : \nab^2 \mu_{N,\eta}.
\end{align*}
Using \eqref{227} and \eqref{d2m}, we obtain
\begin{align}\label{boundR5}
  |\xg\R_{N,\eta}^6|
  &\lesssim
   \mathcal{C}(\eta_*)    \A(\mu)\,\bigl(\M(\eta/2)+\eta^2\A(\mu)\bigr),
\end{align}
and the dual estimate
\begin{align}\label{boundR5-1}
  &\|\langle \cdot\rangle^\gamma
    (\R_{N,\eta}^6+\tilde\mu_{N,\eta}\operatorname{div} (\F^{\eta_*}* \mu ))
  \|_{(W^{1,1})^{*}}
\lesssim
   \mathcal{C}(\eta_*)
    \A(\mu)\,\eta\bigl(\M(\eta/2)+\eta\A(\mu)\bigr).
\end{align}
For $\R^7$ (the commutator of mollification and divergence), using \eqref{proF4}, we have
\begin{align}\label{boundR7}
  \|\xg\R_{N,\eta}^7\|_{L^\infty}
  &\lesssim
    \mathcal{C}(\eta_*) \eta^2\,\A(\mu)	\left(
    \A(\mu)
    +
    \|\xg \nabla^4 \mu\|_{L^\infty}\mathbf{1}_{s\geq d}
    \right).
\end{align}
For $\R^8$ (the time derivative error), for $s \ge d$ this is a direct estimate, for $s \in (d-1,d)$ we use \eqref{227} to write  \begin{align}\label{boundR8}
 \| \langle x\rangle^\gamma
  \R_{N, \eta}^{8}\|_{L^\infty}
  &\lesssim
    \Bigl(
      \mathbf 1_{d-1\le s<d}\, \eta_*^2
      +\mathbf{1}_{s\geq d}	\, \mathcal C(\eta_*)
    \Bigr)\A(\mu)^2.
\end{align}
\medskip
\noindent
\textbf{Step 12: The transport term $\R^{10}$.}
When $s<d$, the kernel $\F$ is locally integrable and a direct bound gives $|\R^{10}|\lesssim \M(\eta)\,\A(\mu)$.

When $s\geq d$, the kernel is too singular for direct integration.
We decompose $\F^{\eta_*}*\tilde\mu_{N,\eta} = \F_{\kappa_0\eta}^{\eta_*}*\tilde\mu_{N,\eta/2} - \F_{<\kappa_0\eta}^{\eta_*}*\mu$ and use \eqref{le230} and \eqref{proF3} to obtain
\begin{align}
\label{Step12}
  \langle x\rangle^{-1}	|\F^{\eta_*}* \tilde \mu_{N,\eta} (x) |&\lesssim
    \bigl(
      \mathbf{1}_{s=d}|\!\log\eta|
      +\mathbf{1}_{s\in (d,d+1)}\eta^{d-s}
    	+\eta^{-1}\mathcal C(\eta,\eta_*)
    \bigr)\,\M(\eta/2)
\\
  &\quad\quad
  +\eta^2
      \mathbf{1}_{s\geq d}\,	\mathcal C(\eta_*)\A(\mu). \notag
\end{align}
Combining with $|\nab\mu(x_t)|\lesssim \A(\mu)\langle x_t\rangle^{-\gamma-1}$, we get
\begin{align}\label{boundr7}
  \R_{N,\eta}^{10}
  &\lesssim  \mathbf{1}_{s<d}\,\M(\eta)\, \A(\mu)
  + \bigl(
      \mathbf{1}_{s=d}|\!\log\eta|
      +\mathbf{1}_{s\in (d,d+1)}\eta^{d-s}
      	+\eta^{-1}\mathcal C(\eta,\eta_*)
    \bigr)\,\M(\eta/2)\, \A(\mu)
  \\
  &\quad
  +\eta^2
  \mathbf{1}_{s\geq d}\,	\mathcal C(\eta_*)\A(\mu)^2.\nonumber
\end{align}
\medskip
\noindent
\textbf{Step 13: The weight term $\R^{11}$.}
This term arises from the transport of the weight $\langle x\rangle^\gamma$.
Using \eqref{Step12}, we find
\begin{align}\label{boundr8}
  \R_{N,\eta}^{11}
  &= \gamma |\langle x_t\rangle^{\gamma-2}\tilde \mu_{N,\eta}(x_t)\,x_t\cdot\F^{\eta_*}* \mu(x_t)|
  \notag\\
  &\lesssim
    \bigl(1+ \mathbf{1}_{s=d}|\!\log\eta|
      +\mathbf{1}_{s\in (d,d+1)}\eta^{d-s}
      +\eta^{-1}\mathcal C(\eta,\eta_*)
    \bigr) \M(\eta)\,\U(\eta).
\end{align}
\medskip
\noindent
\textbf{Step 14: The nonlinear term $\R^{12}$.} 
This is the ``$\M\cdot\M$'' term. Using again \eqref{Step12} we find
\begin{align*}
  |\R_{N,\eta}^{12}|
  &\lesssim
    \mathbf 1_{s<d}\,\M(\eta)^2
  +   \bigl(1+ \mathbf{1}_{s=d}|\!\log\eta|
  +\mathbf{1}_{s\in (d,d+1)}\eta^{d-s}
  +\eta^{-1}\mathcal C(\eta,\eta_*)
  \bigr)\M(\eta/2)\,\M(\eta)
  \\
  &\quad
  +\eta^2
  \mathbf{1}_{s\geq d}	\mathcal C(\eta_*)\A(\mu)\M(\eta).
\end{align*}
Combining the estimates from Steps~1--14, using $\eta=2^n\eta_0$ and
$\mathcal C(\eta_*)=1$ when $s<d+1$, we obtain the claimed estimate. Here we
also use
\begin{align*}
	\mathbf{1}_{s=d-1}\,\M(\eta)\,\U(\eta)
	&\leq
	\sum_{k=0}^{n+1}
	\eta_0^{d-s-1}
	2^{k(d+1-s)_+-k-n}
	\bigl(1+k\,\mathbf{1}_{s=d+1}\bigr)
	\M(2^{k-1}\eta_0)\,\U(\eta),
\end{align*}
and
$$
\|\langle\cdot\rangle^{\gamma'}u\|_{(W^{1,1})^*}
\lesssim
\|\langle\cdot\rangle^{\gamma'}u\|_{L^\infty}.
$$
The dual estimates are stated for $0\leq\gamma'\leq\gamma-1$, due to the
loss of one weight in the diagonal estimate \eqref{R17''}. This completes
the proof.
\end{proof}

\begin{remark}\label{rem:symmetrization}
The most delicate contribution is $\R_{N,\eta}^2$, and the decomposition in \eqref{x2} is essentially optimal. To explain why a more naive approach fails, let us consider the following symmetrization.
\begin{align*}
  -	\R_{N,\eta}^2
  =&\;
    \frac12 \sum_{k=0}^{n-1}\frac{1}{N^2}\sum_{i\neq j}
      \bigl(\F_{2^k\eta_0}^{\eta_*}-\F_{2^{k+1}\eta_0}^{\eta_*}\bigr)(x_i,x_j)
      \cdot\bigl(\nabla\Phi^\eta_{x_i}(x)-\nabla\Phi^\eta_{x_j}(x)\bigr)
  \\
  &+\sum_{k=0}^{n-1}\frac{1}{N^2}\sum_{i\neq j}
    \bigl(\F_{2^k\eta_0}^{\eta_*,\mathrm{sym}}
      -\F_{2^{k+1}\eta_0}^{\eta_*,\mathrm{sym}}\bigr)(x_i,x_j)
    \,\nabla\Phi^\eta_{x_i}(x),
\end{align*}
where $\F^{\eta_*,\mathrm{sym}}(x,y):=\frac12(\F^{\eta_*}(x,y)+\F^{\eta_*}(y,x))$.
 Taylor expanding the first term  to order $m_0\geq 1$ introduces a remainder of the form
\[
  \sum_{k=0}^{n-1}\frac{1}{N^2}\sum_{i\neq j}
    \bigl|\F_{2^k\eta_0}^{\eta_*}-\F_{2^{k+1}\eta_0}^{\eta_*}\bigr|(x_i,x_j)\,
    \min\{|x_i-x_j|,\eta\}^{m_0+1}\,
    \bigl|\nabla^{m_0+2}\Phi^\eta_{x_i}(x)\bigr|.
\]
This term is \emph{uniformly bounded but not small}. By convolution estimates it behaves like
\begin{align*}
  &\sum_{k=0}^{n-1}\frac{1}{N}\sum_i
    \Bigl(
      \frac{(2^k\eta_0)^2}{(2^k\eta_0+|\cdot|)^{s+2}}\,
      \min\{|\cdot|,\eta\}^{m_0+1}
    \Bigr)
    *\mu_{N,\eta}(x_i)\,
    \eta^{-m_0-2}\,
    \Phi^{2\eta}_{x_i}(x)
  \\
  &\qquad\lesssim
    \U(\eta)^2
    \sum_{k=0}^{n-1}
    \Bigl\|
      \Bigl(
        \frac{(2^k\eta_0)^2\,\min\{|\cdot|,\eta\}^{m_0+1}}
          {(2^k\eta_0+|\cdot|)^{s+2}}	
      \Bigr)
      *(\langle\cdot\rangle^{-\gamma})
    \Bigr\|_{L^\infty}
    \eta^{-m_0-2}
  \lesssim \U(\eta)^2,
\end{align*}
for $s\leq d-1$.
Hence $\langle x\rangle^{\gamma}|\R_{N,\eta}^2|\lesssim \U(\eta)^2$,
which is not small and therefore insufficient to close the   Grönwall argument.
This illustrates why the  decomposition into $\R^{2,1}$--$\R^{2,5}$ via the semigroup property is necessary.
\end{remark}


\section{Control of small scale interactions}\label{sec:smallscale}

We now introduce a quantity $H(t)$ that controls the small-scale interactions between particles, i.e.~the regime $|x_i-x_j|\lesssim \eta_*$.
This quantity provides a bound on the most singular error terms $\R_{N,\eta}^1$ and $\R_{N,\eta}^9$, and it can  in turn be controlled via a differential inequality which controls its derivative in terms of the metric $\M$.

Define 
\be\label{condt}
  H(t)
  :=\frac{1}{N}\sup_{i\in [1,N]}
    \sum_{\substack{j=1\\j\neq i}}^N
    \frac{\chi\bigl(\frac{|x_i-x_j|}{2\eta_*}\bigr)}{|x_i-x_j|^{d-\theta}},
  \qquad
  \eta_*=N^{-\frac{1}{d}},
\ee 
with 
\begin{equation}\label{def:theta}
  \theta:=\frac{(d-1-s)_+}{3d}.
\end{equation}
Note that $\theta=0$ whenever $s\geq d-1$, so the exponent $d-\theta$ equals $d$ in both the Coulomb and super-Coulomb regimes.
The cutoff $\chi$, equals to $1$ in $B(0,1)$ and vanishing outside $B(0, 2)$ ensures that $H$ only sees pairs of particles at mutual distance $\lesssim \eta_*$; the power $d-\theta$ is chosen so that $H$ has the right homogeneity for the Grönwall estimate below (Lemma~\ref{lem31}).

By the definition of $H(t)$, we obtain the \emph{minimum separation bound}: for any $i\neq j$,
\begin{equation}\label{q10}
  \inf_{i \neq j} |x_i - x_j|
  \;\ge\;
  \min (\eta_*, \bigl( N H(t) \bigr)^{-\frac{1}{d-\theta}}).
\end{equation}

We next turn to controlling the first error term \eqref{defr2} via this small-scale interaction control.
An easier estimate that would suffice to obtain mere convergence is possible, but we give here the estimate that will provide the best convergence rate.

Let us start with a preliminary lemma.
\begin{lem}[Empirical sum estimate]\label{lem:empsum}
Let $K:\mr^d\to [0,\infty)$ be a radial kernel satisfying
\[
  K(z)\lesssim \frac{C_K}{|z|^{\alpha}}\,\mathbf{1}_{|z|\leq R}
  \qquad\text{for some }\alpha,\;R>0,\;C_K>0,
\]
and 
\begin{equation}\label{convolutiondom}
	K(z)\mathbf{1}_{|z|>R_0}
	\lesssim
	\bigl(K\,\mathbf{1}_{|\cdot|>R_0/2}\bigr)*\Phi^{R_0/2}(z),
\end{equation}
and let $\Psi(z):=K(z)\mathbf{1}_{|z|>R_0/2}$ for some $\eta_*\leq R_0\leq R$.
The following holds.
\begin{enumerate}[label=\emph{(\roman*)},leftmargin=2em]
\item \textbf{Close pairs} ($|x_i-x_j|\leq R_0$): if $0\le \alpha\le d-\theta$,
\begin{align}\label{empsum:close}
  \frac{\xg}{N^2}\sum_{i\neq j}
    K(x_i-x_j)\,\mathbf{1}_{|x_i-x_j|\leq R_0}\,\Phi^{2\eta}_{x_i}(x)
  &\lesssim
    C_K\,\eta_*^{d-\theta-\alpha}\,H(t)\,\U(\eta)
  \notag\\
  &\quad
  + C_K\sum_{k\geq 1, 2^k\leq \frac{R_0}{\eta_*}}
    (2^k\eta_*)^{d-\alpha}\,\U(2^k\eta_*)\,\U(\eta)
\end{align} where $\U$ is defined in \eqref{defU}.
When $\alpha \le 0$, this can be replaced by $C_K\,R_0^{d-\alpha}\,\U(R_0)\,\U(\eta)$.
\item \textbf{Far pairs} ($|x_i-x_j|>R_0$): if $\int\langle y\rangle^{-\gamma}\Psi(y)\,dy=:I_\Psi<\infty$,
\begin{align}\label{empsum:far}
  \frac{\xg}{N^2}\sum_{i\neq j}
    \Psi(x_i-x_j)\,\Phi^{2\eta}_{x_i}(x)
  &\lesssim
    I_\Psi\,\U(R_0)\,\U(\eta).
\end{align}
\end{enumerate}
\end{lem}
\begin{proof}
(i) For the close pairs, we split into dyadic shells.
The innermost shell ($|x_i-x_j|\leq \eta_*$) is bounded using the definition~\eqref{condt} of $H(t)$ and the minimum separation~\eqref{q10},
\begin{align*}
  \frac{\langle x\rangle^\gamma}{N^2}\sum_{i\neq j}
    \frac{C_K}{|x_i-x_j|^{\alpha}}\,
    \mathbf{1}_{|x_i-x_j|\leq \eta_*}\,
    \Phi^{2\eta}_{x_i}(x)
  \lesssim
    C_K\,\eta_*^{d-\theta-\alpha}\,
    \frac{1}{N}\sup_i\sum_{j\neq i}
      \frac{\chi(\frac{|x_i-x_j|}{2\eta_*})}{|x_i-x_j|^{d-\theta}}\,
    \frac{\xg}{N}\sum_i\Phi^{2\eta}_{x_i}(x)\\
    \lesssim \eta_*^{d-\theta-\alpha} H(t) \U(\eta)\end{align*}
    where we used \eqref{d2m}.
For each intermediate shell $|x_i-x_j|\sim 2^k\eta_*$ with $2^k\leq R_0/\eta_*,k\geq 1$, we use $\mathbf{1}_{|z|\leq r}\lesssim r^d\,\Phi^r(z)$ to bound the sum over $j$ by $(2^k\eta_*)^d\,\mu_{N,2^k\eta_*}(x_i)$, hence by $(2^k\eta_*)^d\,\U(2^k\eta_*)$ after weighting by $\langle x \rangle^\gamma$.

When $\alpha \leq 0$, the $H$-bound is not needed and the entire close-pair sum is controlled directly by $\U(R_0)$.

(ii) For the far pairs, the convolution domination \eqref{convolutiondom}
replaces the sum over $j$ by a convolution against $\mu_{N,R_0/2}$.
Then, using again \eqref{d2m}, we have
\begin{align*}\frac{\langle x \rangle^\gamma }{N^2}
\sum_{i, j} K(x_i-x_j) \indic_{|x_i-x_j|>R_0} \Phi_{x_i}^{2\eta}(x)\lesssim
  \frac{\xg}{N}\sum_i
    \bigl(K\,\mathbf{1}_{|\cdot|>R_0/2}\bigr)*\mu_{N,R_0/2}(x_i)\,
    \Phi^{2\eta}_{x_i}(x)
 \\ \lesssim
    I_\Psi\,\U(R_0/2)\,\U(\eta). 
\end{align*}
Finally, by Lemma~\ref{lem2.3} applied to the comparable scales $R_0/2$ and $R_0$, we have $\U(R_0/2)\lesssim\U(R_0),$ we obtain
\eqref{empsum:far}. The proof is complete.
\end{proof}

The next lemma finally provides a bound on $\R^1 $ (and also on $\R^9$). Note that when $s\le 1$, we do not need the help of $H$ to control $\R^1$ or any other term.

\begin{lem}\label{lemR1}
Let $\eta= 2^n\eta_0\leq \frac{1}{2}$ and $\eta_0\geq\eta_*>0$.
In the regime $s>d-1$, we fix $\eta_{0}=\eta_*$.
We have the following.

\smallskip\noindent\emph{(i) Weighted $L^\infty$ bounds on $\R^1$ and $\R^9$:}
\begin{align}\label{R1}
	\xg|\R_{N, \eta}^{1}|
	&\lesssim
	\frac{1+ \mathbf 1_{s>d+1}\eta_*^{-s+d+1}}{\eta^2}\,
	\biggl(
	\Xi(\eta,\eta_0)\eta_{0}^2\U(\eta_0)+ \mathbf{1}_{s\leq 1}\eta_{0}^{d+1-s}\U(\eta_0)
	+ \mathbf{1}_{s> 1}\,\eta_*^{d-\theta-\min\{s-1,d\}}\,H(t)
	\notag\\
	&\qquad\qquad\qquad\qquad
	+ \mathbf{1}_{s> 1}\sum_{k, 1\leq 2^k\leq \frac{\eta_{0}}{\eta_*}}
	(2^k\eta_*)^{d-\min\{s-1,d\}}\,\U(2^k\eta_*)
	\biggr)\U(\eta),
\end{align}
and
\begin{align}\label{R9}
  \xg|\R_{N, \eta}^{9}|
  &\lesssim
    \mathbf{1}_{s\geq d+1}\,
    \frac{\eta_{*}^{-s+d+1}}{\eta^2}\,
    H(t)^{\frac{s-1}{d}}\,\U(\eta),
\end{align}
where
\begin{equation}\label{def:Xi}
  \Xi(\eta,\eta_0)
  :=
  \eta^{d-s-1}\,\mathbf 1_{s<d-1}
  + \mathbf 1_{s=d-1}\log \frac{2\eta}{\eta_0}
  + \eta_0^{d-1-\min\{s,d+1\}}\,\mathbf 1_{s>d-1}.
\end{equation}

\smallskip\noindent\emph{(ii) Dual estimates:} for any $d<\gamma_1\leq \gamma-1$
\begin{align}\label{R1-dual}
  \|\langle \cdot\rangle^{\gamma_1}\R_{N, \eta}^{1}\|_{(W^{1,1})^*}
  &\lesssim	\frac{1+ \mathbf 1_{s>d+1}\eta_*^{-s+d+1}}{\eta}\,
  \biggl(
  \Xi(\eta,\eta_0)\eta_{0}^2\U(\eta_0)+\mathbf{1}_{s\leq 1}\eta_{0}^{d+1-s}\U(\eta_0)
  \notag\\
  &
  + \mathbf{1}_{s> 1}\eta_*^{d-\theta-\min\{s-1,d\}}H(t)  + \mathbf{1}_{s> 1}\sum_{k, 1\leq 2^k\leq \frac{\eta_{0}}{\eta_*}}
  (2^k\eta_*)^{d-\min\{s-1,d\}}\,\U(2^k\eta_*)
  \biggr)\U(\eta),
\end{align}
and 
\begin{align}\label{R9-dual}
  \|\langle \cdot\rangle^{\gamma_1}\R_{N, \eta}^{9}\|_{(W^{1,1})^*}
  &\lesssim
    \mathbf{1}_{s\geq d+1}\,
    \frac{\eta_{*}^{-s+d+1}}{\eta}\,
    H(t)^{\frac{s-1}{d}}\,\U(\eta).
\end{align}
\end{lem}
\begin{proof}
Recall the definitions:
\begin{align*}
  \R_{N, \eta}^1
  &= -\frac1{N^2}\sum_{i\neq j}
    \F^{\eta_*}_{<\eta_0}(x_i,x_j) \cdot \nab  \pex(x),
  \\[3pt]
  \R_{N, \eta}^{9}
  &= -\frac{\mathbf{1}_{s\geq d+1}}{N^2}\sum_{i\neq j}\F(x_i,x_j)\chi\(\frac{|x_i-x_j|}{\eta_*}\)\cdot \nab  \pex(x).
\end{align*}
\medskip\noindent
\textbf{Step 1: Symmetrization.}
Since the kernel $\F$ may not be symmetric, we symmetrize $\R^1$ by writing
\begin{align}\notag
  \R_{N,\eta}^1
  &= -\frac1{2N^2}\sum_{i\neq j}
    \F_{<\eta_0}^{\eta_*}(x_i,x_j)
    \cdot \bigl(\nab  \Phi^{\eta}_{x_i}(x)-\nab  \Phi^{\eta}_{x_j}(x)\bigr)
  \\[3pt]\notag
  &\quad
  - \frac1{2N^2}\sum_{i\neq j}
    \bigl(\F_{<\eta_0}^{\eta_*}(x_i,x_j)+\F_{<\eta_0}^{\eta_*}(x_j,x_i)\bigr)
    \cdot \nab  \Phi^{\eta}_{x_j}(x)
  \\[3pt]\notag
  &=: \R_{N, \eta}^{1,1}+\R_{N,\eta}^{1,2}.
\end{align}
Similarly, we symmetrize $\R^9$:
\begin{align}\notag
	\R_{N,\eta}^9
	&=-\frac{\mathbf{1}_{s\geq d+1}}{2N^2}\sum_{i\neq j}
	\F^{\mathrm{sing}}(x_i,x_j) \chi\(\frac{|x_i-x_j|}{\eta_*}\)
	\cdot \bigl(\nab  \Phi^{\eta}_{x_i}(x)-\nab  \Phi^{\eta}_{x_j}(x)\bigr)
	\\
	&\quad
	-\frac{\mathbf{1}_{s\geq d+1}}{2N^2}\sum_{i\neq j}
	\bigl(	\F^{\mathrm{sing}}(x_i,x_j)+	\F^{\mathrm{sing}}(x_j,x_i)\bigr)\chi\(\frac{|x_i-x_j|}{\eta_*}\)
	\cdot \nab  \Phi^{\eta}_{x_j}(x)\label{R9symmetrization}.
\end{align}
If $|x_i-x_j|\le \eta$, then by  the mean value theorem and \eqref{prod2exp},
\begin{equation}
\label{R9symmetrization2}
  \bigl|\nabla\Phi^\eta_{x_i}(x)-\nabla\Phi^\eta_{x_j}(x)\bigr|
  \le |x_i-x_j|\,|\nabla^2\Phi^\eta_{z}(x)|
  \lesssim \frac{|x_i-x_j|}{\eta^{2}}\,\Phi^{2\eta}_{x_i}(x),
\end{equation}
for some $z\in[x_i,x_j]$.
More generally, for any $x_i,x_j$, using \eqref{prod2exp} we have
\begin{equation}\label{q11}
  \bigl|\nabla\Phi^\eta_{x_i}(x)-\nabla\Phi^\eta_{x_j}(x)\bigr|
  \lesssim
  \frac{1}{\eta}\min\Bigl(\frac{|x_i-x_j|}{\eta},1\Bigr)
  \bigl(\Phi^{2\eta}_{x_i}(x)+\Phi^{2\eta}_{x_j}(x)\bigr).
\end{equation}
This replaces the $\eta^{-1}$ blow-up of $|\nab\Phi^\eta|$ by the smaller quantity $|x_i-x_j|/\eta^2$ when particles are close.

\medskip\noindent
\textbf{Step 2: The term $\R^{1,1}$ --- applying the empirical sum lemma.}
Using \eqref{proF1} and \eqref{q11}, and separating the sum into pairs with $|x_i-x_j|\leq \eta_0\le \eta$ and $|x_i-x_j|> \eta_0$ we obtain
\begin{align}\label{R11m}
  |\R_{N, \eta}^{1,1}|
  &\lesssim
      \frac1{N^2}\sum_{i\neq j}
      \frac{	1
      	+\mathbf{1}_{s>d+1}\eta_*^{-s+d+1}}{
        |x_i-x_j|^{\min\{s,d+1\}}}
      \,\frac{|x_i-x_j|}{\eta^2}\,
      \mathbf{1}_{|x_i-x_j|\leq \eta_{0}}\,
      \Phi^{2\eta}_{x_i}(x)
  \\[5pt] \notag
  &\quad+
    \frac{\eta_0^2}{N^2}\sum_{i\neq j}
    \Bigl(1
      +\frac{	1+\mathbf{1}_{s>d+1}\eta_*^{-s+d+1}}{
        |x_i-x_j|^{2+\min\{s,d+1\}}}
    \Bigr)
    \frac{1}{\eta}\min\Bigl(\frac{|x_i-x_j|}{\eta},1\Bigr)\,
    \Phi^{2\eta}_{x_i}(x)\,
    \mathbf{1}_{|x_i-x_j|> \eta_{0}}
\\ \notag
&: = Q_1+Q_2.
\end{align}
Here we have used the fact that 	$$\mathcal C(\eta_0,\eta_*)= \mathbf{1}_{s= d+1}\!\log2+\mathbf{1}_{s>d+1}\eta_*^{-s+d+1},$$ since $\eta_{0}=\eta_{*}$ 
whenever $s>d-1$.\\
For $Q_1$, we apply Lemma~\ref{lem:empsum}\,(i) with $K(z)=\eta^{-2}(1+\mathbf{1}_{s>d+1}\eta_*^{-s+d+1})\cdot|z|^{1-\min\{s,d+1\}}$, $R_0=\eta_0$, and $\alpha=\min\{s,d+1\}-1$.
For $Q_2$, we apply Lemma~\ref{lem:empsum}\,(ii) with the far-pair kernel $\Psi(z)$ given by the integrand in \eqref{R11m} restricted to $|z|>\eta_0$.
The integral $I_\Psi$ produces the function $\Xi(\eta,\eta_0)$ defined in \eqref{def:Xi}:
\begin{align*}
  I_\Psi
  &=  \frac{\eta_0^2}{\eta}\int \langle y\rangle^{-\gamma}
    \Bigl(1
      +\frac{1+ \mathbf 1_{s>d+1}\eta_*^{-s+d+1}}{
        |y|^{2+\min\{s,d+1\}}}
    \Bigr)
    \min\Bigl(\frac{|y|}{\eta},1\Bigr)\,
    \mathbf{1}_{|y|> \frac{\eta_0}{2}}\, dy 
  \\
  &\sim
    \frac{\eta_0^2}{\eta^2}\bigl(1+\mathbf 1_{s>d+1}\eta_*^{-s+d+1}\bigr)\,
    \Xi(\eta,\eta_0).
\end{align*}
We claim that
\begin{align}\label{zz}
		&\int \langle y\rangle^{-\gamma} \left[\left(1
	+\frac{1}{
		|y|^{2+\min\{s,d+1\}}
	}\right)\min\(\frac{|y|}{\eta},1\) \mathbf{1}_{\frac{|y|}{\eta_0}> \frac{1}{2}}\right] dy \\&\lesssim \mathbf{1}_{s<d-1} \eta^{d-2-s}+\mathbf{1}_{s=d-1} \eta^{-1}\log\frac{2\eta}{\eta_0}+\mathbf{1}_{s>d-1} \eta^{-1} \eta_0^{d-1-\min\{s,d+1\}}.\nonumber
\end{align}
Since we assume $s\geq d-1-\frac{1}{4}$, we have $2+\min\{s,d+1\}> d+\frac{3}{4}$.
Splitting the integral into the three regions
\[
|y|>1,\qquad \eta<|y|\le 1,\qquad \eta_0/2\le |y|\le \eta,
\]
the left-hand side of \eqref{zz} is bounded by 
\begin{align*}
&\lesssim 1+\int_{\eta<|y|\leq 1}\left(1
	+\frac{1}{
		|y|^{2+\min\{s,d+1\}}
	}\right)dy+\int_{\frac{\eta_0}{2}\leq |y|\leq \eta} \frac{|y|}{\eta} \left(1
	+\frac{1}{
		|y|^{2+\min\{s,d+1\}}
	}\right) dy\\&\lesssim 1+\eta^{d-2-\min\{s,d+1\}}+\mathbf{1}_{s<d-1} \eta^{d-2-s}+\mathbf{1}_{s=d-1} \eta^{-1}\log \frac{2\eta}{\eta_0}+\mathbf{1}_{s>d-1} \eta^{-1} \eta_0^{d-1-\min\{s,d+1\}}.
\end{align*}
This implies \eqref{zz}.\\
Combining the two contributions and using \eqref{Mee}, we obtain
\begin{align}\label{R11}
  \xg|\R_{N, \eta}^{1,1}|
    &\lesssim
  \frac{1+ \mathbf 1_{s>d+1}\eta_*^{-s+d+1}}{\eta^2}\,
  \biggl(
  \Xi(\eta,\eta_0)\eta_{0}^2\U(\eta_0)+ \mathbf{1}_{s\leq 1}\eta_{0}^{d+1-s}\U(\eta_0)
  + \mathbf{1}_{s> 1}\,\eta_*^{d-\theta-\min\{s-1,d\}}\,H(t)
  \notag\\
  &\qquad\qquad\qquad\qquad
  + \mathbf{1}_{s> 1}\sum_{k, 1\leq 2^k\leq \frac{\eta_{0}}{\eta_*}}
  (2^k\eta_*)^{d-\min\{s-1,d\}}\,\U(2^k\eta_*)
  \biggr)\U(\eta).
\end{align}
Here we have used the fact that $s>(d-\frac{5}{4})_+$, so, $s<1$ only if $d\leq 2$.
\medskip

\noindent
\textbf{Step 3: The second term $\R^{1,2}$.}
By \eqref{proF2},  absorbing some terms, we have 
\begin{align}\label{R12m}
  \xg|\R_{N, \eta}^{1,2}|
  &\lesssim
    \frac{\xg}{N^2}\sum_{i\neq j}
   \mathbf{1}_{|x_i-x_j|\leq \eta_{0}}\frac{1+ \mathbf 1_{s>d+1}\eta_*^{-s+d+1}}{|x_i-x_j|^{\min\{s-1,d\}}}
    \,\frac{1}{\eta}\, \Phi^{2\eta}_{x_j}(x)
  \notag\\
  &\quad
  + \frac{\xg}{N^2}\sum_{i\neq j}
    \Bigl(1+\frac{1+ \mathbf 1_{s>d+1}\eta_*^{-s+d+1}}{|x_i-x_j|^{2+\min\{s-1,d\}}}\Bigr)
    \,\mathbf{1}_{|x_i-x_j|>\eta_{0}}
    \,\frac{\eta_{0}^2}{\eta}\, \Phi^{2\eta}_{x_j}(x).
\end{align}
By Lemma~\ref{lem:empsum}, the close-pair contribution is bounded exactly as $\eta Q_1$ in Step~2 with $\alpha=\min\{s-1,d\}$ and $R_0=\eta_0$, and the far-pair contribution is bounded with $I_\Psi\lesssim \Xi(1,\eta_0)$.
Combining the two yields
\begin{align}\label{R12}
  \xg |\R_{N,\eta}^{1,2}|
  &\lesssim
    \frac{1+ \mathbf 1_{s>d+1}\eta_*^{-s+d+1}}{\eta}\,
    \biggl(
    \Xi(1,\eta_0)\eta_{0}^2\U(\eta_0)+ \mathbf{1}_{s\leq 1}\eta_{0}^{d+1-s}\U(\eta_0)
      + \mathbf{1}_{s> 1}\,\eta_*^{d-\theta-\min\{s-1,d\}}\,H(t)
  \notag\\
  &\qquad\qquad\qquad\qquad
      + \mathbf{1}_{s> 1}\sum_{k, 1\leq 2^k\leq \frac{\eta_{0}}{\eta_*}}
        (2^k\eta_*)^{d-\min\{s-1,d\}}\,\U(2^k\eta_*)
    \biggr)\U(\eta).
\end{align}
Combining \eqref{R11} and \eqref{R12}, and using that $\eta|\log \eta_{0}|\lesssim \log(\frac{2\eta}{\eta_{0}})$,  we obtain \eqref{R1}.\\
\medskip\noindent
\textbf{Step 4: Bound on $\R^9$.}
By  \eqref{assumpFsing}, \eqref{assumpF2}, \eqref{R9symmetrization},  \eqref{q10}, \eqref{q11}, and $\eta_{*}=N^{-1/d}$,  we have
\begin{align*}
|\R_{N, \eta}^{9}|
	&\lesssim
	\frac{\mathbf{1}_{s\geq d+1}}{N^2}\sum_{i\neq j}
	\frac{	1}{
		|x_i-x_j|^{s}}
	\,\frac{|x_i-x_j|}{\eta^2}\,
	\mathbf{1}_{|x_i-x_j|\leq 2 \eta_{0}}\,
	\Phi^{2\eta}_{x_i}(x)\\&+
		\frac{\mathbf{1}_{s\geq d+1}}{N^2}\sum_{i\neq j}
	\frac{	1}{
		|x_i-x_j|^{s-1}}
\frac{1}{\eta}
	\mathbf{1}_{|x_i-x_j|\leq 2 \eta_{0}}\,
	\Phi^{2\eta}_{x_i}(x)\\&\lesssim 	\frac{\mathbf{1}_{s\geq d+1}}{N^2}\sum_{i\neq j}
	\frac{	1}{
		|x_i-x_j|^{s-1}}
	\frac{1}{\eta^2}
	\mathbf{1}_{|x_i-x_j|\leq 2\eta_{0}}\,
	\Phi^{2\eta}_{x_i}(x)
	  \\
	&\lesssim
	\frac{\mathbf{1}_{s\geq d+1}}{N^2}\sum_{i\neq j}
	\frac{\mathbf{1}_{|x_i-x_j|\leq 2\eta_{*}}}{|x_i-x_j|^{d}}
	\,\frac{\eta_{*}^{-s+d+1}}{\eta^2}\,
	H(t)^{\frac{s-1-d}{d}}\,
	\Phi^{2\eta}_{x_i}(x)
	\\
	&\lesssim
	\mathbf{1}_{s\geq d+1}\,
	\frac{\eta_{*}^{-s+d+1}}{\eta^2}\,
	H(t)^{\frac{s-1}{d}}\,\U(\eta).
\end{align*}
This gives \eqref{R9}.

\medskip\noindent
\textbf{Step 5: dual estimates.}
The dual estimates \eqref{R1-dual} and \eqref{R9-dual} follow from writing $\R^1$ and $\R^9$ in divergence form:
\begin{align*}
  \R_{N,\eta}^1
  &=- \operatorname{div}\biggl(
    \frac1{2N^2}\sum_{i\neq j}
    \F_{<\eta_0}^{\eta_*}(x_i,x_j)\,
    \bigl(\Phi^{\eta}_{x_i}(x)-\Phi^{\eta}_{x_j}(x)\bigr)
  \biggr)
  \\
  &\quad
  - \operatorname{div}\biggl(
    \frac1{2N^2}\sum_{i\neq j}
    \bigl(\F_{<\eta_0}^{\eta_*}(x_i,x_j)+\F_{<\eta_0}^{\eta_*}(x_j,x_i)\bigr)\,
    \Phi^{\eta}_{x_j}(x)
  \biggr),
\end{align*}
and
\begin{align*}\notag
	\R_{N,\eta}^9
	&=-\operatorname{div}\biggl[\frac{\mathbf{1}_{s\geq d+1}}{2N^2}\sum_{i\neq j}
	\F^{\mathrm{sing}}(x_i,x_j) \chi\big(\frac{|x_i-x_j|}{\eta_*}\big)
	\bigl(  \Phi^{\eta}_{x_i}(x)-  \Phi^{\eta}_{x_j}(x)\bigr)  \biggr]
	\\
	&\quad
	-\operatorname{div}\biggl[ \frac{\mathbf{1}_{s\geq d+1}}{2N^2}\sum_{i\neq j}
	\bigl(	\F^{\mathrm{sing}}(x_i,x_j)+	\F^{\mathrm{sing}}(x_j,x_i)\bigr)\chi\big(\frac{|x_i-x_j|}{\eta_*}\big)
 \Phi^{\eta}_{x_j}(x)  \biggr].
\end{align*}
The divergence structure absorbs one power of~$\eta^{-1}$, producing $\eta^{-1}$ in \eqref{R1-dual}--\eqref{R9-dual} instead of $\eta^{-2}$.
This completes the proof.
\end{proof}
We next turn to the time evolution of~$H(t)$.

\begin{lem}\label{lem31} Let $s>1$.
Let $\eta_0\geq\eta_*>0$.
In the regime $s>d-1$, set $\eta_{0}=\eta_*$.
Let $X_N^t$ be a classical solution to \eqref{ode} on
$[0,T]$, and assume that no collision occurs on this interval, namely
$$
x_i(t)\neq x_j(t),
\qquad
0\le t\le T,\quad i\neq j.
$$
Define
\begin{align}\label{defS}
	S(t)
	&:= \sup_{0\leq \tau\leq t}
	(\U(\tau,\eta_*) +1)^2
	\biggl(
	\mathbf{1}_{s=d-1}\Bigl(
	\sum_{k=0}^{[\log_2(\eta_*^{-1/2})]-1}\M(\tau, 2^{k}\eta_*)
	+(\log N)\,\M(\tau,\eta_*^{1/2})
	+(2H(0)+2)^2
	\Bigr)
	\notag\\
	&\qquad\qquad\qquad\qquad
	+ \eta_*^{d-\frac{3}{2}\theta-s-1}(2H(0)+1)^{\max\{\frac{s+1}{d},1\}}
	+ 1
	+ \mathbf{1}_{s>d-1}\,\eta_*^{d-1-s}
	\biggr).
\end{align}

Assume that there exists a constant $C$, depending only on $d$ and $s$, such that
\begin{equation}\label{condS}
  H(0)\, e^{2C  S(T) T}
  + \eta_*^{\theta}\bigl( e^{2C  S(T) T}-1\bigr)
  \leq \frac{2}{3}(2H(0)+1)\, \eta_*^{-\theta/2}.
\end{equation}
Then, for all $t\le T$:
\begin{equation}\label{boundH}
  H(t)
  \le
  H(0)\,e^{2C S(t)\,t}
  +\eta_*^{\theta}\bigl(e^{2C S(t)\,t}-1\bigr).
\end{equation}
\end{lem}

\begin{proof}
\textbf{Step 1: Differential inequality for $\boldsymbol{H}$.}
Since the solution is classical and collision-free on $[0,T]$, all the
functions
$$
t\mapsto
\frac{1}{N}\sum_{\substack{j=1\\j\neq i}}^N
\frac{\chi\bigl(\frac{|x_i(t)-x_j(t)|}{2\eta_*}\bigr)}
{|x_i(t)-x_j(t)|^{d-\theta}}
$$
are $C^1$ on $[0,T]$. Therefore, differentiating \eqref{condt}, for every $i$
we have
\begin{align*}
  &\partial_t\biggl(
    \frac{1}{N}\sum_{\substack{j=1\\j\neq i}}^N
    \frac{\chi\bigl(\frac{|x_i-x_j|}{2\eta_*}\bigr)}{|x_i-x_j|^{d-\theta}}
  \biggr)
  \\
  &\qquad=
    \frac{1}{N}\sum_{\substack{j=1\\j\neq i}}^N
    \biggl[
      -(d-\theta)
      \frac{\chi\bigl(\frac{|x_i-x_j|}{2\eta_*}\bigr)\,(x_i-x_j)}{
        |x_i-x_j|^{d-\theta+2}}
      + \frac{\chi'\bigl(\frac{|x_i-x_j|}{2\eta_*}\bigr)\,(x_i-x_j)}{
        2\eta_*\,|x_i-x_j|^{d-\theta+1}}
    \biggr]
    \cdot(\dot{x}_i-\dot{x}_j).
\end{align*}

The derivative produces two contributions: one from differentiating $|x_i-x_j|^{-(d-\theta)}$ (giving $H$ back) and one from differentiating the cutoff $\chi$ (giving a term supported at $|x_i-x_j|\sim \eta_*$).
By definition of $\eta_*$ this yields
\begin{align}\notag
  \partial_t\biggl(
    \frac{1}{N}\sum_{\substack{j=1\\j\neq i}}^N
    \frac{\chi\bigl(\frac{|x_i-x_j|}{2\eta_*}\bigr)}{|x_i-x_j|^{d-\theta}}
  \biggr)
  & \lesssim \sup_{k\neq k'}\frac{|\dot{x}_k-\dot{x}_{k'}|}{|x_k-x_{k'}|}
    \Bigl(H(t) +\frac{\eta_{*}^{\theta-d}}{N}\sum_{j\neq i} \indic_{|x_i-x_j|\le 4 \eta_*}\Bigr)
    \\  \notag  &\lesssim \sup_{k\neq k'}\frac{|\dot{x}_k-\dot{x}_{k'}|}{|x_k-x_{k'}|}
    \Bigl(H(t) +\frac{\eta_{*}^{\theta-d}}{N}\sum_{j\neq i} \Phi^{4 \eta_*} (x_i-x_j) \eta_*^d \Bigr) \\ \notag
       &\lesssim \sup_{k\neq k'}\frac{|\dot{x}_k-\dot{x}_{k'}|}{|x_k-x_{k'}|}
    \Bigl(H(t) +\eta_{*}^{\theta}\,\mu_{N, 4\eta_*} (x_i) \Bigr) 
  \\ \label{q17}&\lesssim
    \sup_{k\neq k'}\frac{|\dot{x}_k-\dot{x}_{k'}|}{|x_k-x_{k'}|}
    \Bigl(H(t) +\eta_{*}^{\theta}\, \U(\eta_*)\Bigr).
\end{align}
\medskip\noindent
\textbf{Step 2: Lipschitz bound on the relative velocity.}
We now prove that
\begin{align}\label{q16}
  \sup_{i\neq j}\frac{|\dot{x}_i-\dot{x}_j|}{|x_i-x_j|}
  &\lesssim
    \mathbf{1}_{s=d-1}\Bigl(
      \sum_{k=0}^{[\log_2(\eta_*^{-1/2})]-1}\M(2^{k}\eta_*)
      +(\log N)\,\M(\eta_*^{1/2})
      +(H(t)^{1/d}+1)\bigl(H(t)+\U(\eta_*)\bigr)
    \Bigr)
  \notag\\
  &\quad
  + \eta_*^{d-\theta-s-1}H(t)^{\max\{(s+1)/d,\,1\}}
  + \bigl(1+\mathbf{1}_{s>d-1}\eta_*^{d-1-s}\bigr)\,\U(\eta_*).
\end{align}
This bound captures the competition between the singular (small-scale, $H$-dependent) and smooth (large-scale, $\M$-dependent) parts of the velocity field.

\smallskip
\emph{Proof of \eqref{q16} for $s\neq d-1$.}
By \eqref{ode} and \eqref{assumpFsing}--\eqref{assumpFre}, using \eqref{q10} when $\theta=0$ (i.e.\ $s\geq d-1$), for $i\neq j$:
\begin{align*}
  |\dot{x}_i-\dot{x}_j|
  &\leq \frac{1}{N}
    \biggl|\sum_{\substack{k=1\\k\notin\{i,j\}}}^N
    \bigl(\F(x_i,x_k)-\F(x_j,x_k)\bigr)\biggr|
    + \frac{1}{N}|\F(x_i,x_j)-\F(x_j,x_i)|
  \\
  &\lesssim |x_i-x_j|\,\biggl(
    \sum_{k'\in\{i,j\}}\frac{1}{N}\sum_{\substack{k=1\\k\neq k'}}^N
    \frac{\mathbf{1}_{|x_{k'}-x_k|\leq 1}}{|x_{k'}-x_k|^{s+1}}
    + 1
    + \frac{\mathbf{1}_{s\leq d-1}}{N|x_i-x_j|^{s}}
  \biggr)
   \\
  &\lesssim |x_i-x_j|\,\biggl(
  \sum_{k'\in\{i,j\}}\frac{1}{N}\sum_{\substack{k=1\\k\neq k'}}^N
  \frac{\mathbf{1}_{|x_{k'}-x_k|\leq 1}}{|x_{k'}-x_k|^{s+1}}
  + 1
  \biggr).
\end{align*}
For the inner sum over~$k$, we separate close and far pairs.
The close-pair contribution ($|x_{k'}-x_k|\leq \frac{3}{2}\eta_*$) is bounded  by using the definition of $H$ for $s\le d-1$ and \eqref{q10} for $s >d-1$; the far-pair contribution ($|x_{k'}-x_k|> \frac{3}{2}\eta_*$) is controlled by using as in \eqref{convolutiondom} that 
\begin{equation}\label{muconvol}
  \frac{1}{N}\sum_{\substack{k=1\\k\neq k'}}^N
  \frac{\mathbf{1}_{\frac{3}{2}\eta_*\leq |x_{k'}-x_k|\leq 1}}{|x_{k'}-x_k|^{s+1}}
  \lesssim
  \frac{\indic_{\frac{\eta_*}{2}\le |\cdot |\le 2}}{|\cdot|^{s+1}}
  \star\mu_{N,\eta_*}(x_{k'}),
\end{equation}
and using Young's convolution inequality. This gives
\begin{align}\label{q15}
  |\dot{x}_i-\dot{x}_j|
  &\lesssim |x_i-x_j|\,\Bigl(
    \eta_*^{d-\theta-s-1}\,H(t)^{\max\{(s+1)/d,\,1\}}
  \notag\\
  &\qquad\qquad\qquad
    +\bigl(1+\mathbf{1}_{s=d-1}|\!\log\eta_*|
      +\mathbf{1}_{s>d-1}\eta_*^{d-1-s}\bigr)\,\U(\eta_{*})
  \Bigr).
\end{align}
This implies \eqref{q16} for $s\neq d-1$.

\smallskip
\emph{Proof of \eqref{q16} for the critical case $s=d-1$.} Note that this and $s>1$ imply that  $d>2$. 
Here a more refined decomposition is needed.
We write
\begin{align*}
  \dot{x}_i
  &= \F_{\eta_*}*\tilde\mu_{N,\eta_*}(x_i)
    + \F_{\eta_*}*\mu(x_i)
    + \frac{1}{N}\sum_{\substack{k=1\\k\neq i}}^N \F_{<\sqrt2\eta_*}(x_i,x_k)
    - \frac{1}{N}\F_{\sqrt2\eta_*}(x_i,x_i).
\end{align*}
The smooth part $\F_{\eta_*}*\tilde\mu_{N,\eta_*}$ is controlled via $\|\nab\F_{\eta_*}*\tilde\mu_{N,\eta_*}\|_{L^\infty}$,
which we bound using the dyadic decomposition~\eqref{z11} with $\eta=\eta_*$ and $n_0=\lfloor\log_2(\eta_*^{-1/2})\rfloor$
\begin{align*}
  \|\nab\F_{\eta_*}*\tilde\mu_{N,\eta_*}\|_{L^\infty}
  &\lesssim
    \sum_{k=0}^{\lfloor\log_2(\eta_*^{-1/2})\rfloor-1}\M(2^{k}\eta_*)
    +\log(N)\,\M(\eta_*^{1/2})
 +(\log N)\eta_*\, \A(\mu)\\
 &\lesssim
 \sum_{k=0}^{\lfloor\log_2(\eta_*^{-1/2})\rfloor-1}\M(2^{k}\eta_*)
 +\log(N)\,\M(\eta_*^{1/2})
 +\U(\eta_*).
\end{align*}
By \eqref{226c} and \eqref{227} we have for any $i\not=j$
\begin{align}\label{z10}
&	\frac{1}{N}|\F_{\sqrt2\eta_*}(x_i,x_i)-\F_{\sqrt2\eta_*}(x_j,x_j)|\lesssim  \frac{1}{N\eta_*^{s-1}}+\frac{1}{N}\lesssim   \frac{|x_i-x_j|}{N\eta_*^{s}}(1+H(t)^{1/d}),\\&
|\F_{\eta_*}*\mu(x_i)-\F_{\eta_*}*\mu(x_j)|\lesssim |x_i-x_j|\A(\mu)\label{z10'},
\end{align}
where we used 
\eqref{q10} for the first inequality. \\
For the singular part ($\F_{<\sqrt2\eta_*}$), using \eqref{proF1} and $N^{-1}\eta_*^{1-d}\leq 1$, we have
\begin{align} \left|\frac{1}{N}\sum_{\substack{k=1\\k\neq i}}^N \F_{<\sqrt2\eta_*}(x_i,x_k)- \frac{1}{N}\sum_{\substack{k=1\\k\neq j}}^N \F_{<\sqrt2\eta_*}(x_j,x_k)\right|
	&\lesssim  \eta_*^2
	+ \sup_{l}\frac{1}{N}\sum_{\substack{k=1\\k\neq l}}^N
	\frac{\min\bigl(1,\frac{\eta_*}{|x_l-x_k|}\bigr)^2}{|x_l-x_k|^{d-1}}.
\end{align}
By the definition of $H(t)$, splitting between $|x_l-x_k|\ge \eta_*/2$ and $|x_l-x_k|< \eta_*/2$, we have
\begin{align}\label{z40}
 \left|\frac{1}{N}\sum_{\substack{k=1\\k\neq i}}^N \F_{<\sqrt2\eta_*}(x_i,x_k)- \frac{1}{N}\sum_{\substack{k=1\\k\neq j}}^N \F_{<\sqrt2\eta_*}(x_j,x_k)\right|
	&\lesssim\notag\eta_{*}^2+
	\eta_*\,H(t)
	+ \eta_*^2\,
	\Bigl\|\Bigl(\frac{\mathbf{1}_{|\cdot|>\eta_*/2}}{|\cdot|^{d+1}}\Bigr)
	*\mu_{N,\eta_*}\Bigr\|_{L^\infty}
	\notag\\
	&\lesssim
	\eta_*\,\bigl(H(t)+\U(\eta_*)\bigr) \\&\lesssim |x_i-x_j| (H(t)^{1/d} +1)\bigl(H(t)+\U(\eta_*)\bigr), \notag
\end{align} where we used 
\eqref{q10} for the last inequality. \\
 Combining \eqref{z10},  \eqref{z10'}  and \eqref{z40}, we deduce
\begin{align}\label{q14}
  |\dot{x}_i-\dot{x}_j|
  &\lesssim |x_i-x_j|\,\Bigl(
    \|\nab\F_{\eta_*}*\tilde\mu_{N,\eta_*}\|_{L^\infty}
    +\A(\mu)
    +(H(t)^{1/d}+1)\bigl(H(t)+\U(\eta_*)\bigr)
    \Bigr).
\end{align}
Substituting the dyadic bound for $\|\nab\F_{\eta_*}*\tilde\mu_{N,\eta_*}\|_{L^\infty}$ into \eqref{q14} proves \eqref{q16} for $s=d-1$.\\
\medskip\noindent
\textbf{Step 3: Grönwall closure.}
It follows from \eqref{q17} and \eqref{q16} that 
\begin{align*}
	&	\partial_t\biggl(
	\frac{1}{N}\sum_{\substack{j=1\\j\neq i}}^N
	\frac{\chi\bigl(\frac{|x_i-x_j|}{2\eta_*}\bigr)}{|x_i-x_j|^{d-\theta}}
	\biggr)
\\&	\lesssim 
	\left( \mathbf{1}_{s=d-1}\Bigl(
	\sum_{k=0}^{[\log_2(\eta_*^{-1/2})]-1}\M(t,2^{k}\eta_*)
	+(\log N)\,\M(t,\eta_*^{1/2})
	+(H(t)^{1/d}+1)\bigl(H(t)+\U(t,\eta_*)\bigr)
	\Bigr)\right.
	\\
	&\quad\left.
	+ \eta_*^{d-\theta-s-1}\,H(t)^{\max\{(s+1)/d,\,1\}}
	+ \bigl(1+\mathbf{1}_{s>d-1}\eta_*^{d-1-s}\bigr)\,\U(t,\eta_*)\right)
	\Bigl(H(t) +\eta_{*}^{\theta}\, \U(t,\eta_*)\Bigr).
\end{align*}
Set \begin{equation}\label{def:Tprime-H} T' := \sup \left\{ \tau\in[0,T]: \sup_{0\leq t\leq\tau}H(t) \leq (2H(0)+1)\eta_*^{-\theta/2} \right\}. \end{equation} By continuity of $H$ and the initial bound, we have $0<T'\leq T,$ and the bootstrap bound holds on $[0,T']$. If $T'<T$, then $H(T') = (2H(0)+1)\eta_*^{-\theta/2}.$ Using the definition of $S$, we obtain, for every $0\leq t\leq T'$,
\begin{align*}
  \forall t\in [0,T'],\qquad
	\partial_t\biggl(
\frac{1}{N}\sum_{\substack{j=1\\j\neq i}}^N
\frac{\chi\bigl(\frac{|x_i-x_j|}{2\eta_*}\bigr)}{|x_i-x_j|^{d-\theta}}
\biggr)
  &\leq C\,S(t)\,\bigl(H(t) +\eta_*^{\theta}\bigr).
\end{align*}
Integrating in time,  and taking the maximum over $i$, this implies that 
\begin{align*}
	\forall t\in [0,T'],\quad	H(t) &\leq H(0)+C	S(t)\int_{0}^{t}\left(H(\tau) +\eta_*^{\theta} \right)d\tau.
\end{align*}
Applying Grönwall's lemma, we obtain
\begin{align}\label{3.8}
  \forall t\in [0,T'],\qquad
  H(t) +\eta_*^{\theta}
  &\le
    \bigl(H(0) +\eta_*^{\theta}\bigr)\,
    (1+CS(t)\,t)\,e^{CS(t)\,t}
  \le
    \bigl(H(0) +\eta_*^{\theta}\bigr)\,e^{2C S(t)\, t}.
\end{align}
If \eqref{condS} holds at $T$, then \eqref{3.8} gives $H(T')<(2H(0)+1)\eta_*^{-\theta/2}$, contradicting the maximality of $T'$ unless $T'=T$.
We conclude that $T'=T$ and \eqref{3.8} holds for all $t\leq T$, proving \eqref{boundH}.\\
The argument above is therefore valid on any classical collision-free time
interval. In the applications below, the lemma is only used before the first
collision time, or on intervals where the particle system is known to be
classical and collision-free.
\end{proof}

\begin{remark}\label{rem:timescale}
When $s>d-1$, we have $\theta=0$ and $S(t)\sim \eta_*^{d-1-s}$, so in order for \eqref{condS} to hold we must impose the short-time condition
\[
  T \ll \eta_*^{s-d+1}.
\]
This reflects the fact that for super-Coulomb interactions, the singularity is too strong to be controlled globally in time.
In Theorem~\ref{supercou} we show that this timescale is optimal.
\end{remark}

\medskip


\section{Main proof of mean-field limits}\label{sec:mainproof}

To prove Theorem~\ref{th1}, we need two elementary lemmas.
The first provides sharp bounds on coupled differential inequalities with a cascading structure; the second is a Grönwall-type comparison for quadratic nonlinearities.
\begin{lem}[Cascading differential inequality]\label{Le1}
	Let $(b_n(t))_{n\ge0}$ be nonnegative absolutely continuous functions such that, for every $0\le n\le N_0$ and a.e. $t\ge0$,
	$$
	\partial_t b_n(t)
	\le
	1+\sum_{k=0}^n b_k(t),
	\qquad
	b_n(0)\le 1.
	$$
	Then, for every $0\le n\le N_0$ and every $t\ge0$,
	\begin{equation}\label{x1}
		b_n(t)
		\le
		2\exp\bigl(2\sqrt{nt}+2t\bigr).
	\end{equation}
\end{lem}

This sub-exponential dependence on $n$ is what ultimately allows the multiscale Grönwall argument to close.

\begin{proof}
	\textbf{Step 1: Duhamel-type inequality.}
We have
	$$
	\partial_t\bigl(e^{-t}b_n(t)\bigr)
	\le
	e^{-t}
	\left(
	1+\sum_{k=0}^{n-1}b_k(t)
	\right).
	$$
	Integrating from $0$ to $t$ and using $b_n(0)\le 1$, we obtain
	\begin{equation}\label{duhamel}
		b_n(t)
		\le
		2e^t
		+
		\sum_{k=0}^{n-1}
		\int_0^t e^{t-\tau}b_k(\tau)\,d\tau .
	\end{equation}
	For $n=0$, the sum is empty, and the sharper bound
	$$
	b_0(t)\le 2e^t-1
	$$
	follows directly from $\partial_t b_0\le 1+b_0$.
	
	\medskip
	\noindent
	\textbf{Step 2: Inductive exponential bound.}
	We claim that, for every $r\in(0,\frac12]$,
	\begin{equation}\label{x2p}
		b_n(t)
		\le
		2(1-r)^{-n}e^{t/r},
		\qquad
		0\le n\le N_0.
	\end{equation}
For $n=0$, this follows from $b_0(t)\le 2e^t-1\le 2e^{t/r}$.
Assume now that \eqref{x2p} holds for all $k\le n$. Then, by
\eqref{duhamel},
\begin{align*}
	b_{n+1}(t)
	&\le
	2e^t
	+
	2\sum_{k=0}^{n}
	(1-r)^{-k}
	\int_0^t e^{t-\tau}e^{\tau/r}\,d\tau
	\\
	&=
	2e^t
	+
	2\sum_{k=0}^{n}
	(1-r)^{-k}
	\frac{r}{1-r}
	\bigl(e^{t/r}-e^t\bigr)
	\\
	&<
	2e^t
	+
	2\bigl((1-r)^{-n-1}-1\bigr)
	e^{t/r}
	\\
	&<
	2(1-r)^{-n-1}e^{t/r}.
\end{align*}
Thus \eqref{x2p} holds for every $0\le n\le N_0$.\\
	\noindent
	\textbf{Step 3: Optimization in $r$.}
	Using
	$
	(1-r)^{-n}
	\le
	\exp\left(\frac{nr}{1-r}\right),
	$
	we obtain from \eqref{x2p}
	\begin{equation}\label{x3}
		b_n(t)
		\le
		2\exp\left(
		\frac{t}{r}
		+
		\frac{nr}{1-r}
		\right).
	\end{equation}	
	First, if $n=0$, then
	$
	b_0(t)\le 2e^t-1\le 2e^{2t},
	$
	which proves \eqref{x1}.\\
	Assume now that $n\ge1$. If $t=0$, then $b_n(0)\le1\le2$, so \eqref{x1} is immediate. \\
	If $0<t\le n$, choose
	$
	r
	=
	\frac{\sqrt t}{\sqrt n+\sqrt t}\le \frac12.
	$
	Then \eqref{x3} gives
	$$
	b_n(t)
	\le
	2\exp\bigl(2\sqrt{nt}+t\bigr)
	\le
	2\exp\bigl(2\sqrt{nt}+2t\bigr).
	$$
If $t>n$, choosing $r=\frac12$ in \eqref{x3}, we obtain
	$$
	b_n(t)
	\le
	2e^{2t+n}\le
	2\exp\bigl(2\sqrt{nt}+2t\bigr).
	$$
	Thus \eqref{x1} holds in all cases. The proof is complete.
\end{proof}
\begin{lem}[Quadratic Grönwall]\label{lemgronw}
Assume that $M(t)$ satisfies
\begin{equation}\label{relationM}
  \partial_t M \le C\bigl(A+ M(t)\bigr)^2,
  \qquad t \in [t_0,T],
\end{equation}
for some constants $C>0$ and $A\geq 0$. If $A+M(t_0)=0$, then $M(t)=0$ on $[t_0,T]$. Otherwise, for every $t \in \bigl[t_0,\,\min\bigl(T,\, t_0+ C^{-1}(A+M(t_0))^{-1}\bigr)\bigr]$,
\begin{equation}\label{gronw:concl}
  M(t)
  \le
  \frac{A+M(t_0)}{1-C(t-t_0)\bigl(A+M(t_0)\bigr)}- A
  = \frac{M(t_0) + CA(t-t_0)\bigl(A+M(t_0)\bigr)}{1-C(t-t_0)\bigl(A+M(t_0)\bigr)}.
\end{equation}
\end{lem}

\begin{proof} If $A+M(t_0)=0$, then $A=0$ and $M(t_0)=0$. Comparison with the solution of $y'=Cy^2$, $y(t_0)=0$, gives $M\equiv0$. Assume now that $A+M(t_0)>0$. 
Define the antiderivative
\[
  u(r):= \int_0^r \frac{ds}{C(A+ s)^2}
  = \frac{1}{CA}-\frac{1}{C(A+r)}.
\]
Since $u$ is increasing, \eqref{relationM} can be written as $\partial_t \bigl(u(M(t))\bigr)\le 1$, hence
\[
  \frac{1}{C(A+M(t_0))}-\frac{1}{C(A+M(t))} \le t- t_0,
\]
and the result follows by rearranging.
\end{proof}


\subsection{Proof of Theorem~\ref{th1-Coulomb}}\label{sec:proof-coulomb2D}
By Proposition~\ref{prop:weak-continuation-existence}, the particle system
	\eqref{ode} admits a weak particle continuation $\mu_N^t$ on $[0,T]$, in the
	sense of Definition~\ref{def-weak-particle-solution}. This continuation is
	obtained as the limit of regularized particle systems. More precisely, for
	$\kappa>0$, let $X_N^{\kappa,t}$ be the unique global classical solution of
	the regularized system associated with the force field
	$\F^{\mathrm{reg}}+\F^{\mathrm{sing},\kappa}$
	where the regularization $\F^{\mathrm{sing},\kappa}$ is defined by 
	\begin{equation}\label{weak-Fsing-kappa-cou}
	\F^{\mathrm{sing},\kappa}(x,y)
	:=
	\int_{\mathbb R^d}
	\F^{\mathrm{sing}}(x,y-z)\Phi^\kappa(z)\,dz .
\end{equation}
	 Then, along a sequence
	$\kappa_n\downarrow0$, the empirical measures $\mu_N^{\kappa_n,t}$ converge
	uniformly in time, against $W^{2,\infty}$ test functions, to a weak particle
	continuation $\mu_N^t$.
	
	It is therefore enough to establish the desired estimates uniformly in
	$\kappa$ for the regularized classical systems $X_N^{\kappa,t}$, and then to
	pass to the limit as $\kappa=\kappa_n\to0$. For simplicity of notation, we
	first carry out the quantitative estimates for a smooth singular force, with
	constants independent of the smoothing parameter. The passage to the weak
	particle continuation is given in Step 4 of the proof.


In the Coulomb case $s=d-1$, $d=1,2$, we do not need to use $H$ to control
$\R^1$ or $\R^9$. Using \eqref{R2-8}, \eqref{R10-12}, \eqref{R1}, and
\eqref{R9}, with $\eta=2^n\eta_0,\U(\eta)=\M(X_N^t,\mu^t,\eta)+\A(\mu^t),
$ we obtain
\begin{align*}
	&	\xg|\R_{N, \eta}^{1}|
	\lesssim 2^{-2n}(1+n)
(\M(\eta_0)+\A(\mu))(\M(2^n\eta_0)+\A(\mu)),\\
	& \sum_{m=2}^8\xg |\R_{N,\eta}^{m,\kappa}|
	\lesssim
	\sum_{k=0}^{n+1}
	2^{k-n}
	\M(2^{k-1}\eta_{0})(\M(2^n\eta_0)+\A(\mu))
	+ \Bigl( \eta^2
	\A(\mu)
	+ \frac{1}{N\eta_{0}^{d-1}}
	\Bigr)(\M(2^n\eta_0)+\A(\mu))
	\\ & \quad 
	+
\biggl(\sum_{k=0}^{n_0-1}\M(2^{k+n}\eta_0)
+|\!\log(2^{n_0+n}\eta_0)|\,\M(2^{n_0+n}\eta_0)
+|\!\log(2^{n_0+n}\eta_0)|\,2^{2(n_0+n)}\eta_0^2\,\A(\mu)\biggr)
\\ & \quad 
\times\bigl(\M(2^{n-1}\eta_0)+2^{2n}\eta_{0}^2\,\A(\mu)\bigr),\\
	&\R_{N, \eta}^{9}\equiv0,\\ &	\R_{N,\eta}^{10}+\R_{N,\eta}^{11} +\R_{N,\eta}^{12}
\lesssim (\M(2^n\eta_0)
+ \eta_0^22^{2n}\A(\mu))(\M(2^n\eta_0)+\A(\mu)).
\end{align*}
Thus, in view of  \eqref{ptM}, we obtain 
\begin{multline}\label{coulomb-raw} 
  \partial_t \M(2^n \eta_0)
  \lesssim
  \biggl(\sum_{k=0}^{n+1} 2^{k-n}\,\M(2^{k-1}\eta_{0})
    + \eta_{0}+2^{2n}\eta_{0}^2\,\A(\mu)\biggr)
  \bigl(\M(2^{n-1}\eta_0)+\A(\mu)\bigr)
  \\
  + (1+n)\,2^{-2n}\,
  \bigl(\M(\eta_0)+\A(\mu)\bigr)
  \bigl(\M(2^{n-1}\eta_0)+\A(\mu)\bigr)
  \\
  + \biggl(\sum_{k=0}^{n_0-1}\M(2^{k+n}\eta_0)
    +|\!\log(2^{n_0+n}\eta_0)|\,\M(2^{n_0+n}\eta_0)
    +|\!\log(2^{n_0+n}\eta_0)|\,2^{2(n_0+n)}\eta_0^2\,\A(\mu)\biggr)
  \\
  \times\bigl(\M(2^{n-1}\eta_0)+2^{2n}\eta_{0}^2\,\A(\mu)\bigr).
\end{multline}
We now set $\eta_0=2^m\eta_*$ and $n_0=\lfloor\log_2(\eta_*^{-2/3})\rfloor-n-m$ so that $2^{n_0+n} \eta_0= \eta_*^{1/3}$.
Then, for every $n\ge 0$ such that $n_0\ge 0$, we may absorb the term 
$$(1+n)2^{-2n}\M(\eta_0)$$
into the $k=1$ contribution in the sum. Moreover,
\begin{align*}
	|\log(2^{n_0+n}\eta_0)|\lesssim \log N,
\end{align*}
Therefore, for any $t\in[0,T]$, we obtain
\begin{multline}\label{coulomb}
  \partial_t \M(2^{n+m} \eta_*)
  \\
  \leq C_1
  \biggl(\sum_{k=0}^{n+1} 2^{k-n}\,\M(2^{k+m-1}\eta_*)
    + 2^{m}\eta_*
    +\bigl((n+1)2^{-2n}+2^{2(n+m)}\eta_*^2\bigr)\,\A(\mu)\biggr)
  \bigl(\M(2^{n-1+m}\eta_*)+\A(\mu)\bigr)
  \\
  + C_1\biggl(\sum_{k=0}^{\lfloor\log_2(\eta_*^{-2/3})\rfloor-n-m-1}
    \M(2^{k+n+m}\eta_*)
    +(\log N)	\M(2^{\lfloor\log_2(\eta_*^{-2/3})\rfloor}\eta_*)+(\log N)\eta_{*}^{\frac{2}{3}}\A(\mu)\biggr)\\\times
  \bigl(\M(2^{n-1+m}\eta_*)+2^{2(n+m)}\eta_*^2\,\A(\mu)\bigr).
\end{multline}
The proof now proceeds in three steps: (1) an a priori bound under a bootstrap assumption,
(2) closing the bootstrap by iterating a short-time estimate, and (3) combining with the dual estimates.

\medskip
\noindent\textbf{Step 1. A priori bound under a bootstrap assumption.}
Let $m_1$ be the fixed integer chosen below in Step~2. Define $\widetilde T$
as the supremum of all times $t_*\le T$ such that
\begin{equation}\label{cou-bootstrap}
  \sup_{t\in [0,\tilde T]}\M(t,2^{n+m_1}\eta_*)\leq 2^{-n/4}
\end{equation}
for any $n=0,\dots,\lfloor\log_2(\eta_*^{-2/3})\rfloor-m_1$.
In \eqref{coulomb}, set $m=m_1+1$ and use \eqref{cou-bootstrap}.
The key observation is that the first factor of the second line of \eqref{coulomb} is controlled by the bootstrap assumption as follows
\begin{align}\nonumber
  &\sum_{k=0}^{\lfloor\log_2(\eta_*^{-2/3})\rfloor-n-m-1}
    \M(2^{k+n+m}\eta_*)
    +(\log N)\M(2^{\lfloor\log_2(\eta_*^{-2/3})\rfloor}\eta_*)+(\log N)\eta_{*}^{\frac{2}{3}}\A(\mu)
  \\\nonumber
  &\qquad\leq
  \sum_{k=0}^{n_0-1}2^{-(k+n+1)/4}
  +(\log N)\, 2^{-\lfloor\log_2(\eta_*^{-2/3})\rfloor/4} 2^{m_1/4}+(\log N)\eta_{*}^{\frac{2}{3}}\A(\mu)
\\& \qquad \leq \frac{2^{-(n+1)/4}}{1-2^{-1/4}}+\frac{2(\log N)}{N^{\frac{1}{6d}}} 2^{m_1/4}
+\A(\mu)  \leq 2(1+\A(\mu)),\label{w5}
\end{align}for $N\geq 2^{100d+2dm_1}$,
for all $t\in [0,\tilde T]$. \\So, 
\begin{multline*}
	\partial_t \M(2^{n+m} \eta_*)
	\leq 2C_1
	\biggl(\sum_{k=0}^{n+1} 2^{k-n}\,\M(2^{k+m-1}\eta_*)
	+ 2^{m}\eta_*
	+\bigl((n+1)2^{-2n}+2^{2(n+m)}\eta_*^2\bigr)\biggr) (1+\A(\mu))^2
	\\
	+ 2C_1
	\bigl(\M(2^{n-1+m}\eta_*)+2^{2(n+m)}\eta_*^2\bigr)(1+\A(\mu))^2.
\end{multline*}
Since $\A(\mu^t)\le \A_T$ on $[0,\tilde T]$, controlling 
$\M(2^{n-1+m}\eta_*)$ by the $n$-th term in the sum, 
we obtain for all $t\in[0,\tilde T]$
\begin{equation}\label{coulomb-step1}
  \partial_t \M(2^{n+m_1+1} \eta_*)
  \leq 5C_1
  \biggl(\sum_{k=0}^{n} 2^{k-n}\,\M(2^{k+m_1+1}\eta_*)
    + 2^{m_1+1}\eta_*+(n+1)2^{-2n}+2^{2(n+m_1+1)}\eta_*^2\biggr)
  (1+\A_T)^2.
\end{equation}
Define, for $n\ge 0$,
\[
  b_n(t):=2^n\,\M\bigl(\tfrac{t}{C_2},\, 2^{n+m_1+1} \eta_*\bigr),
\]
with $C_2\geq 1$ to be chosen.
Then for every $t\in[0,C_2\tilde T]$ and every $n=0,\dots,\lfloor\log_2(\eta_*^{-2/3})\rfloor-m_1-1$, and $N$ large enough,
\begin{align*}
  C_2\,\partial_t b_n(t)
  &\leq 10C_1
    \biggl(\sum_{k=0}^{n} b_k(t)+
     1 + 2^{m_1+1+n}\eta_*+2^{3n+2(m_1+1)}\eta_*^2\biggr)
    (1+\A_T)^2
  \\
  &\leq 100C_1(1+\A_T)^2
    \biggl(\sum_{k=0}^{n} b_k(t)+1\biggr).
\end{align*}
Choosing
\begin{equation}\label{def:C2-cou}
  C_2:=100C_1(1+\A_T)^2,
\end{equation}
we obtain
\[
  \partial_t b_n(t)\leq 1+\sum_{k=0}^{n} b_k(t).
\]
By \eqref{Cou-initial},
\begin{align*}
	  b_n(0)=2^n\,\M\bigl(0, 2^{n+m_1+1} \eta_*\bigr)\leq 2^{-(n+m_1+1)} C\leq 2^{-(m_1+1)} C\leq 1
\end{align*}
Hence, $b_n(0)\le 1$ for all admissible $n$, provided that\begin{equation}\label{q70}
	m_1\geq \log_2 C.
\end{equation}
Hence, by Lemma~\ref{Le1},
\[
  b_n(t)\leq 	2\exp\bigl(2\sqrt{nt}+2t\bigr).
\]
Returning to $\M$, this means that
\begin{equation}\label{q30}
  \M(t, 2^{n+m_1+1} \eta_*)
  \leq2^{-n+1}\exp\bigl(2\sqrt{C_2 n t}+2C_2 t\bigr),
\end{equation}
for $t\in [0,\tilde T]$ and $n=0,\dots,\lfloor\log_2(\eta_*^{-2/3})\rfloor-m_1-1$.\\
Observe that the expected optimal estimate is
$$ \M(t, 2^{n+m_1+1} \eta_*)
\lesssim2^{-n},$$
therefore, the bound in \eqref{q30} is almost optimal, the only losses being  the exponential term
$$\exp\bigl(2\sqrt{C_2 n t}+2C_2 t\bigr).$$
\medskip
\noindent\textbf{Step 2. Closing the bootstrap for large $m_1$.}
Set
\[
  \mathbf{n}:=\bigl\lfloor\log_2(\eta_*^{-2/3})\bigr\rfloor,
  \qquad
  m\le \bigl\lfloor\log_2(\eta_*^{-1/10})\bigr\rfloor,
\]
and define the multiscale metric
\begin{equation}\label{def:Mmn}
  \M_{m,\mathbf{n}}(t):= \sum_{k=0}^{\mathbf{n}-m} 2^{k/2}\, \M(t, 2^{k+m}\eta_*).
\end{equation}
A useful consequence of this definition is the monotonicity relation
\begin{equation}\label{re:Mmn}
		\M_{m+1,\mathbf n}(t)
		=
		\sum_{k=0}^{\mathbf n-m-1} 2^{k/2}	\M\bigl(t,2^{k+m+1}\eta_*\bigr) \le
		2^{-1/2}
		\M_{m,\mathbf n}(t).
\end{equation}
Observe that $\M(t,2^{n+m}\eta_*)\le 2^{-n/2}\M_{m,\mathbf{n}}(t)$ for any $0\le n\le \mathbf{n}-m$ and that by \eqref{Meta1eta2} we have 
\begin{equation*}
  \M(t, 2^{m-1}\eta_*)
  \leq
  C_0\bigl(\M(t, 2^{m}\eta_*)+\A(\mu^t)\bigr)
  \leq C_0\bigl(\M_{m,\mathbf{n}}(t)+\A(\mu^t)\bigr).
\end{equation*} Multiplying the relations 
\eqref{coulomb} by $2^{n/2}$ and summing over $n$ (and simplifying using $2^{\mathbf{n}+m}\leq  \eta_*^{-\frac{2}{3}-\frac{1}{10}}$) we then deduce that 
for all $t\in[0,\tilde T]$,
since $\mathsf{M}(t,2^{\mathbf{n}}\eta_{*})\le 2^{-\mathbf{n}/2+m/2}\M_{m,\mathbf{n}}(t)$,
it follows from \eqref{coulomb} that
\begin{multline}\partial_t 	\M_{m,\mathbf{n}}(t)\\\leq 2C_1 \sum_{n=0}^{\mathbf{n}-m}2^{\frac{n}{2}}\left(\sum_{k=0}^{n+1} 2^{k-n}\M(2^{k+m-1}\eta_{*})+ 2^{m}\eta_{*}+((n+1)2^{-2n}+2^{2(n+m)}\eta_{*}^2)\A(\mu^t)\right)(	\M_{m,\mathbf{n}}(t)+\A(\mu^t))
	\\+C_1\sum_{n=0}^{\mathbf{n}-m}2^{\frac{n}{2}}\left(\sum_{k=0}^{[\log(\eta_{*}^{-2/3})]-n-m-1}\M(2^{k+n+m}\eta_{*})+\log(N)\M(2^{\lfloor\log_2(\eta_*^{-2/3})\rfloor}\eta_*)+(\log N)\eta_{*}^{\frac{2}{3}}\A(\mu^t)\right)\\\quad\quad\quad\times\left(\M(2^{n-1+m}\eta_{*})+2^{2(n+m)}\eta_{*}^2\A(\mu^t)\right).
\end{multline}
Since \begin{align*}
&	\sum_{n=0}^{\mathbf{n}-m}2^{\frac{n}{2}}	((n+1)2^{-2n}+2^{2(n+m)}\eta_{*}^2)\leq 3+2^{2-\frac{m}{2}}2^{\frac{5}{2}\mathbf{n}}\eta_{*}^2\leq 5,\\&
	\sum_{n=0}^{\mathbf{n}-m}2^{\frac{n}{2}}2^{m}\eta_{*}\leq 2^{\frac{\mathbf{n}}{2}}2^{\frac{m}{2}+1}\eta_{*}\leq \eta_*^{1/2}\leq 1,
\end{align*}
we obtain 
\begin{equation}\label{q24}
  \partial_t \M_{m,\mathbf{n}}(t)
  \leq 100(C_0+1)C_1\,
  \bigl(\M_{m,\mathbf{n}}(t)+1+\A_T\bigr)^2.
\end{equation}
Let $C:=100(C_0+1)C_1$.
By Lemma~\ref{lemgronw}, for $t \in [t_0,\min(\tilde T, t_0+ C^{-1}(1+\A_T+\M_{m,\mathbf{n}}(t_0))^{-1})]$:
\begin{equation*}
  \M_{m,\mathbf{n}}(t)
  \le
  \frac{\M_{m,\mathbf{n}}(t_0) + C(1+\A_T)(t-t_0)\bigl(1+\A_T+\M_{m,\mathbf{n}}(t_0)\bigr)}{
    1-C(t-t_0)\bigl(1+\A_T+\M_{m,\mathbf{n}}(t_0)\bigr)}.
\end{equation*}
In particular, if $\M_{m,\mathbf{n}}(t_0)\le 1$, then
\begin{equation}\label{z}
  \sup_{\tau\in [0,\min\{\tilde T-t_0,\,\tau_0\}]}\M_{m,\mathbf{n}}(t_0+\tau)\leq \frac{5}{4}
\end{equation}
for
\begin{equation}\label{def:tau0}
  \tau_0:=\frac{1}{100C(1+\A_T)^2}.
\end{equation}
Since $\tau_0$ is independent of $m$ and $\mathbf{n}$, we may choose the initial scale $m_0$ uniformly satisfying  $m_0\geq\log_2 C.$ Indeed, by \eqref{Cou-initial},
\begin{align}\label{q71}
	\M_{m_0+1,\mathbf n}(0)\le
	C\sum_{k=0}^{\mathbf n-m_0-1}
	2^{k/2}2^{-(k+m_0+1)} \le
	C 2^{-m_0}
	< 1,
\end{align}
provided $m_0$ is chosen sufficiently large, for instance
$$
m_0=\bigl\lceil \log_2(C+2)\bigr\rceil .
$$
Therefore, the bootstrap estimate gives
$$
\sup_{t\in [0,\min\{\widetilde T,\tau_0\}]}
\M_{m_0+1,\mathbf n}(t)
\le \frac54 .
$$
Consequently, using \eqref{re:Mmn}, we obtain
\[
  \M_{m_0+2,\mathbf{n}}(\min\{\tilde T,\tau_0\})\leq \tfrac{5}{4}\times 2^{-1/2}<1.
\]
If $\tau_0<\tilde T$, we iterate \eqref{z} on successive time intervals of length $\tau_0$.
After
\[
  m^*:=1+\Bigl\lfloor \frac{\tilde T}{\tau_0}\Bigr\rfloor
  =1+\bigl\lfloor 10^4(C_0+1)C_1(1+\A_T)^2\tilde T\bigr\rfloor
\]
steps, we obtain
\begin{equation}\label{q25}
  \sup_{[0,\tilde T]}\M_{m_0+1+m^*,\mathbf{n}}(t)< 1.
\end{equation}
We can take
$$
m_1:=m_0+1+m^* .
$$
By the choice of parameters, $m_1$ satisfies \eqref{q70}. In particular,
$m_1$ is an admissible scale in the bootstrap definition above. Since
$m_0$, $T$, and $\A_T$ are fixed independently of $N$, the integer $m_1$ is
also independent of $N$.\\ Moreover, \eqref{q25} implies that $\M(t, 2^{n+m_1}\eta_*)<2^{-n/4}$ for any $n \le \lfloor \log_2(\eta_*^{-2/3})\rfloor-m_1$, contradicting the maximality of $\tilde T $ in
 \eqref{cou-bootstrap}, unless $\tilde T =T$. We conclude that $\tilde T=T$ and that \eqref{cou-bootstrap} and \eqref{q25} hold with $\tilde T=T$.

\medskip
\noindent\textbf{Step 3. Conclusion.}
With the above choice of $m_1$, and 
$  C_2=100C_1(1+\A_T)^2,
$
the estimate \eqref{q30} holds.
Moreover, using
\[
  \sup_{n\leq m_1+1}\M(t, 2^{n} \eta_*)
  \leq C_0\, 2^{2m_1}\bigl(\M(t, 2^{m_1+2} \eta_*)+\A(\mu^t)\bigr),
\]
we infer that, for $t\in[0,T]$ and $n=0,\dots,\mathbf{n}$,
\begin{equation}\label{concl-cou-M}
  \M(t, 2^{n} \eta_*)
  \leq2^{-n}\,
  \exp\bigl[C(1+\A_T)\sqrt{nt}+C(1+\A_T)^2(1+T)\bigr].
\end{equation}
We now pass from dyadic scales to arbitrary scales. Let
$\eta\in[\eta_*,\eta_*^{1/3}]$ and choose $n$ such that
$
2^n\eta_*
\le
\eta
<
2^{n+1}\eta_* .
$
By the scale comparison estimate \eqref{Meta1eta2},
$$
\M(X_N^t,\mu^t,\eta)
\lesssim
\M(X_N^t,\mu^t,2^n\eta_*)
+
\eta^2\A(\mu^t).
$$
Using \eqref{concl-cou-M}, the relation $2^n\eta_*\sim\eta$, and the fact that
$\eta^2\lesssim \eta_*/\eta$ for $\eta\le\eta_*^{1/3}$, we obtain
$$
\M(X_N^t,\mu^t,\eta)
\lesssim
\frac{\eta_*}{\eta}
\exp\left[
C(1+\A_T)\sqrt{t\log\frac{\eta}{\eta_*}}
+
C(1+\A_T)^2(1+T)
\right].
$$
This proves \eqref{conclth1-cou} for all
$\eta\in[\eta_*,\eta_*^{1/3}]$.\\
We now derive the corresponding $\mathsf{W}$-estimate. Using \eqref{eqfordiffeq}, \eqref{R2-8-dual}, \eqref{R1-dual}, and \eqref{R9-dual} with
\[
\eta_0=\eta_*,
\qquad
\eta=2^{\lfloor\log_2(\eta_*^{-1/2})\rfloor}\eta_*\sim \eta_*^{1/2},
\qquad
n_0=0,
\]
we deduce that
$$\omega:=\tilde{\mu}_{N,2^{\bigl\lfloor\log_2(\eta_*^{-1/2})\bigr\rfloor}\eta_{*}}$$
satisfies the equation
\begin{equation*}
  \partial_t \omega
  = -\operatorname{div}\bigl((\F\ast\omega )\mu+(\F\ast\mu)\omega\bigr)
    -(\F\ast\omega )\nab\omega+ \mathbf R,
\end{equation*}
with remainder term satisfying
\begin{align*}
	 \|\langle \cdot\rangle^{\gamma-1}\mathbf R (t)\|_{(W^{1,1})^*}&	\lesssim	\sum_{k=0}^{\bigl\lfloor\log_2(\eta_*^{-1/2})\bigr\rfloor+1}\eta_*\,
	2^{k}
	\M(2^{k-1}\eta_{*})(\M(\eta)+\A(\mu))
	+ 
	 \eta_{*}(\A(\mu)+1)(\M(\eta)+\A(\mu))
	\\&\quad\quad
	+ 
	\Bigl(|\log N|\,\M(\eta)
	+|\log N|\eta_*\A(\mu)
	\Bigr)\eta_{*}^{1/2}\bigl(\M(\eta/2)+\eta^2\A(\mu)\bigr).
\end{align*}
Hence, by \eqref{concl-cou-M},
\begin{equation*}
	\|\langle \cdot\rangle^{\gamma-1}\mathbf R (t)\|_{(W^{1,1})^*}
	\lesssim
	\eta_*(\log N)
	\exp\bigl[C(1+\A_T)\sqrt{(\log N)\,t}+C(1+\A_T)^2(1+T)\bigr].
\end{equation*}
Moreover, since $s\leq d-1$
\begin{align*}
		\|\langle \cdot\rangle^{\gamma-1} 	(\F\ast\omega(t) )\nab\omega (t)\|_{(W^{1,1})^*}&\leq  \| \operatorname{div}(\langle \cdot\rangle^{\gamma-1} 	(\F\ast\omega(t) ))\omega (t)\|_{L^\infty}+\| \langle \cdot\rangle^{\gamma-1} 	(\F\ast\omega(t) )\omega (t)\|_{L^\infty}
		\\&\leq\left(\| \langle \cdot\rangle^{-1} 		\operatorname{div}\F\ast\omega(t) )\|_{L^\infty}+\gamma\| \langle \cdot\rangle^{-1} 		(\F\ast\omega(t) )\|_{L^\infty}\right) \| \langle \cdot\rangle^\gamma 	\omega (t)\|_{L^\infty}.
\end{align*}
Using \eqref{assumpFsing}, \eqref{assumpFre}, and \eqref{coudet}, this yields
\begin{align*}
	\|\langle \cdot\rangle^{\gamma-1} 	(\F\ast\omega(t) )\nab\omega (t)\|_{(W^{1,1})^*}&\lesssim \| \langle \cdot\rangle^\gamma 	\omega (t)\|_{L^\infty}^2=  \M(t, 2^{\bigl\lfloor\log_2(\eta_*^{-1/2})\bigr\rfloor}\eta_*)^2.
\end{align*}
Applying again \eqref{concl-cou-M}, we obtain
\begin{align*}
	\|\langle \cdot\rangle^{\gamma-1} 	(\F\ast\omega(t) )\nab\omega (t)\|_{(W^{1,1})^*}&\lesssim \eta_{*}(\log N)
	\exp\bigl[C(1+\A_T)\sqrt{\log(N)t}+C(1+\A_T)^2(1+T)\bigr]
\end{align*}
Therefore, Lemma~\ref{Le-dual1}, and in particular \eqref{omega_est} with $\gamma_1=\gamma-1$,  yields
	\begin{align*}\mathsf{W}( X_N^t, \mu^t,2^{\bigl\lfloor\log_2(\eta_*^{-1/2})\bigr\rfloor}\eta_{*}) &\lesssim  \exp \big(CT\A_T\big) \mathsf{W}( X_N^0, \mu^0,2^{\bigl\lfloor\log_2(\eta_*^{-1/2})\bigr\rfloor}\eta_{*})\\&\quad+\eta_*(\log N) \exp\left[C(1+\A_T)\sqrt{t \log N}+C(1+\A_T)^2(1+T)\right].
\end{align*}
Finally, arguing as in \eqref{Meta1eta2}, we have
\begin{align*}
&\mathsf{W}( X_N^t, \mu^t,\eta_{*}^{1/2})\lesssim	\mathsf{W}( X_N^t, \mu^t,2^{\bigl\lfloor\log_2(\eta_*^{-1/2})\bigr\rfloor}\eta_{*}) +\eta_{*}\A_t,\\&
	\mathsf{W}( X_N^t, \mu^t,2^{\bigl\lfloor\log_2(\eta_*^{-1/2})\bigr\rfloor}\eta_{*}) 	\lesssim  \mathsf{W}( X_N^t, \mu^t,\frac{1}{2}\eta_{*}^{1/2})+\eta_{*}\A_t.
\end{align*}
Combining the preceding estimates yields \eqref{conclth1-wn-cou}. Together
with the already proved estimate \eqref{conclth1-cou}, this completes the
proof of Theorem~\ref{th1-Coulomb} under the auxiliary smoothness assumption on $\F^{\mathrm{sing}}$.\\

\noindent
{\bf Step 4. Passage to the weak particle continuation.} It remains to detail the approximation argument that justifies the smoothness assumption. For
$\kappa\in(0,N^{-100d}]$, let
$$
X_N^{\kappa,t}
=
(x_1^\kappa(t),\ldots,x_N^\kappa(t))
$$
be the unique global classical solution of the regularized system
\begin{equation}\label{weak-reg-ODE-cou}
	\left\{
	\begin{array}{ll}
		\displaystyle
		\dot{x}_i^\kappa(t)
		=
		\frac1N
		\sum_{\substack{j=1\\ j\neq i}}^N
		\bigl(\F^{\mathrm{reg}}+\F^{\mathrm{sing},\kappa}\bigr)
		\bigl(x_i^\kappa(t),x_j^\kappa(t)\bigr),
		\\[1ex]
		x_i^\kappa(0)=x_i^0,
	\end{array}
	\right.
\end{equation}
where  $\F^{\mathrm{sing},\kappa}$ is as in \eqref{weak-Fsing-kappa-cou}.
For each fixed $\kappa>0$, the vector field
$\F^{\mathrm{reg}}+\F^{\mathrm{sing},\kappa}$ is smooth, locally Lipschitz,
and has at most linear growth. Hence \eqref{weak-reg-ODE-cou} admits a unique
global classical solution. Moreover, the regularized singular force satisfies
the structural assumptions used above, with constants independent of
$\kappa$ up to replacing $C$ by $2C$.

Let
$$
\mu_N^{\kappa,t}
:=
\frac1N\sum_{i=1}^N\delta_{x_i^\kappa(t)}
$$
be the associated empirical measure. Let also $\mu^\kappa$ be the solution of
the regularized mean-field equation
\begin{equation}\label{mukappa-eq-cou}
	\partial_t\mu^\kappa
	=
	-\operatorname{div}
	\Bigl(
	\bigl((\F^{\mathrm{reg}}+\F^{\mathrm{sing},\kappa})
	*\mu^\kappa\bigr)\mu^\kappa
	\Bigr),
	\qquad
	\mu^\kappa|_{t=0}=\mu^0 .
\end{equation}
By the local well-posedness and stability estimates, together with a continuation argument around the solution $\mu$, there exist $\kappa_0>0$ and a constant $\overline{\A}_T = C\bigl(T,\A_T(\mu)\bigr),$ depending only on $T$, $\A_T(\mu)$, and the structural constants, such that, for every $0<\kappa\leq\kappa_0$, the solution $\mu^\kappa$ exists on $[0,T]$ and \begin{equation}\label{AT-uniform-cou} \sup_{0<\kappa\leq\kappa_0} \A_T(\mu^\kappa) \leq \overline{\A}_T. \end{equation} Moreover, 
\begin{equation}\label{mukappa-conv-cou}
	\sum_{\ell=0}^2
	\bigl\|
	\langle\cdot\rangle^\gamma
	\nabla^\ell(\mu^\kappa-\mu)
	\bigr\|_{L^\infty([0,T]\times\mathbb R^d)}
	\longrightarrow 0,
\end{equation}
as $\kappa\to0$.

Applying the smooth-case estimate just proved to the pair
$(X_N^{\kappa,t},\mu^{\kappa,t})$, we obtain constants independent of
$\kappa$ such that, for every $t\in[0,T]$ and
$\eta\in[\eta_*,\eta_*^{1/3}]$,
\begin{align}\label{M-kappa-cou}
	\M(X_N^{\kappa,t},\mu^{\kappa,t},\eta)
	&\le
	\frac{\eta_*}{\eta}
	\exp\left[
	C'(1+\A_T(\mu^\kappa))\sqrt{t\log\frac{\eta}{\eta_*}}
	+
	C'(1+\A_T(\mu^\kappa))^2(1+T)
	\right],
\end{align}
and
\begin{align}\label{W-kappa-cou}
	\mathsf W(X_N^{\kappa,t},\mu^{\kappa,t},\eta_*^{1/2})
	&\le
	e^{C'(T+1)\A_T(\mu^\kappa)}
	\mathsf W\left(
	X_N^0,\mu^0,\frac12\eta_*^{1/2}
	\right)
	\notag\\
	&\quad
	+
	\eta_*(\log N)
	\exp\left[
	C'(1+\A_T(\mu^\kappa))\sqrt{t\log N}
	+
	C'(1+\A_T(\mu^\kappa))^2(1+T)
	\right].
\end{align}
By Proposition~\ref{prop:weak-continuation-existence}, there exist a sequence
$\kappa_n\downarrow0$ and a weak particle continuation $\mu_N^t$ on $[0,T]$
such that
$$
\mu_N^{\kappa_n,t}\rightharpoonup \mu_N^t
$$
uniformly in time against $W^{2,\infty}$ test functions. Passing to a further
subsequence if necessary, we may also assume that
\begin{equation}\label{AT-subsequence-cou}
	\lim_{n\to\infty}\A_T(\mu^{\kappa_n})
	=
	\liminf_{\kappa\to0}\A_T(\mu^\kappa)
	\le
	\A_T .
\end{equation}

We first pass to the limit in the $\M$-estimate. For each fixed
$x\in\mathbb R^d$ and $\eta>0$, take the test function
$$
\phi_{x,\eta}(y)
:=
\langle x\rangle^\gamma\Phi^\eta(x-y).
$$
Then $\phi_{x,\eta}\in W^{2,\infty}(\mathbb R^d)$, and hence, by
Definition~\ref{def-weak-particle-solution},
$$
\langle x\rangle^\gamma\mu_{N,\eta}^{\kappa_n,t}(x)
=
\int_{\mathbb R^d}\phi_{x,\eta}(y)\,d\mu_N^{\kappa_n,t}(y)
\longrightarrow
\int_{\mathbb R^d}\phi_{x,\eta}(y)\,d\mu_N^t(y)
=
\langle x\rangle^\gamma\mu_{N,\eta}^t(x),
$$
uniformly in $t\in[0,T]$ for this fixed $x$. Together with
\eqref{mukappa-conv-cou}, this gives, for every fixed $x$,
$$
\langle x\rangle^\gamma
\bigl(
\mu_{N,\eta}^{\kappa_n,t}(x)-\mu^{\kappa_n,t}(x)
\bigr)
\longrightarrow
\langle x\rangle^\gamma
\bigl(
\mu_{N,\eta}^{t}(x)-\mu^{t}(x)
\bigr).
$$
Taking the supremum in $x$ after passing to the limit yields the lower
semicontinuity estimate
\begin{equation}\label{M-lsc-cou}
	\M(\mu_N^t,\mu^t,\eta)
	\leq
	\liminf_{n\to\infty}
	\M
	\bigl(
	X_N^{\kappa_n,t},\mu^t,\eta
	\bigr).
\end{equation}
Similarly, by the dual definition of $\mathsf W$ and the same convergence of
test functions,
\begin{equation}\label{W-lsc-cou}
	\mathsf W(\mu_N^t,\mu^t,\eta)
	\leq
	\liminf_{n\to\infty}
	\mathsf W
	\bigl(
	X_N^{\kappa_n,t},\mu^t,\eta
	\bigr).
\end{equation}
at every scale $\eta$ used below.

Passing to the limit in \eqref{M-kappa-cou} and \eqref{W-kappa-cou}, using
\eqref{mukappa-conv-cou}, \eqref{AT-subsequence-cou}, \eqref{M-lsc-cou}, and
\eqref{W-lsc-cou}, gives
\begin{align}
	\M(\mu_N^t,\mu^t,\eta)
	&\le
	\frac{\eta_*}{\eta}
	\exp\left[
	C'(1+\A_T)\sqrt{t\log\frac{\eta}{\eta_*}}
	+
	C'(1+\A_T)^2(1+T)
	\right],
\end{align}
for every $\eta\in[\eta_*,\eta_*^{1/3}]$, and
\begin{align}
	\mathsf W(\mu_N^t,\mu^t,\eta_*^{1/2})
	&\le
	e^{C'(T+1)\A_T}
	\mathsf W\left(
	X_N^0,\mu^0,\frac12\eta_*^{1/2}
	\right)
	\notag\\
	&\quad
	+
	\eta_*(\log N)
	\exp\left[
	C'(1+\A_T)\sqrt{t\log N}
	+
	C'(1+\A_T)^2(1+T)
	\right].
\end{align}
Here, after passing to the limit, $\M$ and $\mathsf W$ are understood through
the continued empirical measure $\mu_N^t$. This proves
\eqref{conclth1-cou} and \eqref{conclth1-wn-cou}, and completes the proof of
Theorem~\ref{th1-Coulomb}.

\begin{remark}
If we consider Coulomb-type interactions cut off at scale $N^{-1/d}$ or even
\[
  \F(x,y)\sim \frac{1}{|x-y|\,(|x-y|+N^{-1/d})^{d-2}},
\]
examining the structure of the singular terms $\R_{N,\eta}^{1,1}$ and $\R_{N,\eta}^{9}$ shows that the proof carries through and yields long-time mean-field convergence.
This extends the results of \cite{LP2017}, which only treat interactions cut off at scale $\eta \ge N^{-1/d+\varepsilon}$.
\end{remark}

\medskip
\subsection{Proof of Theorem~\ref{th1}}\label{sec:proofth1}
We are  in the case $s<d-1$.
\begin{proof}
	We are in the case $s<d-1$. By Lemma~\ref{lem:local-classical-ode}, for each
	fixed $N$ the particle system admits a unique classical solution on a maximal
	collision-free interval
	$
	[0,T_{\mathrm{coll}}^N).
	$
	All estimates below are first derived on
	$
	[0,\min\{T,T_{\mathrm{coll}}^N\}).
	$
	The uniform bound on $H$ obtained in Step~3 prevents collisions, and therefore
	implies $T_{\mathrm{coll}}^N>T$. Hence the classical solution extends uniquely
	to the whole interval $[0,T]$.\\
\noindent\textbf{Step 1. Time evolution of the multiscale distance.}
As before, set $\eta=2^n\eta_0$.
Using \eqref{R2-8}, \eqref{R10-12}, \eqref{R1}, and \eqref{R9} in the regime $s<d-1$, we obtain 
\begin{align*}
 & \sum_{m=2}^8\xg |\R_{N,\eta}^m|
  \lesssim
    \sum_{k=0}^{n+1}\eta_0^{d-s-1}\, 2^{k(d-s)-n}\,\M(2^{k-1}\eta_{0})\,\U(\eta)
    + \(\frac{1}{N\eta_{0}^s}
    + \eta^2\A(\mu)+	\M(\eta/2)\)\U(\eta)
  \\
 & \R_{N,\eta}^{10}+\R_{N,\eta}^{11} +\R_{N,\eta}^{12}
  \lesssim
    \M(\eta/2)\,\U(\eta)+\eta^2\,\A(\mu)\,\U(\eta),
  \\
 & \xg|\R_{N, \eta}^{1}|
  \lesssim
    \frac{1}{\eta^2}\Bigl(
      \eta_{0}^2\,\eta^{d-s-1}\,\U(\eta_0)
      + \mathbf{1}_{s> 1}\,\eta_0^{d+1-\theta-s}\,H(t)
  \\
  &\qquad\qquad
      + \mathbf{1}_{s>1}\sum_{1\leq 2^k\leq \frac{\eta_{0}}{\eta_*}}
        (2^{k} \eta_*)^{d+1-s}\U(2^{k} \eta_*)
    \Bigr)\U(\eta),
  \\
&  |\R_{N, \eta}^{9}|=0.
\end{align*}
Using \eqref{ptM}, the identity $\U(\eta)=\M(\eta)+\A(\mu)$,
and restricting to $2^n\le \eta_*^{-2/3}$, we obtain
\begin{multline}\label{q20}
  \partial_t \M(2^n \eta_0)
  \lesssim
  \biggl(\sum_{k=0}^{n-1}\eta_0^{d\theta}\, 2^{k-n}\,\M(2^{k-1}\eta_{0})
    + \frac{1}{N\eta_{0}^s}
    + \M(2^{n-1} \eta_0)+2^{2n}\eta_0^2\,\A(\mu)\biggr)
  \bigl(\M(2^n\eta_0)+\A(\mu)\bigr)
  \\
  + 2^{-2n}\eta_{0}^{d\theta}
  \biggl(\M(\eta_0)+\A(\mu)
    + \mathbf{1}_{s> 1}\Bigl(H(t)
      +\A(\mu)
      +\sum_{1\leq 2^k\leq \frac{\eta_{0}}{\eta_*}} \M(2^{k} \eta_*)
    \Bigr)\biggr)
  \bigl(\M(2^n\eta_0)+\A(\mu)\bigr),
\end{multline}
where we used $\eta_0^{d-s-1}\, 2^{k(d-s)-n}\leq \eta_0^{(d-s-1)/3}\, 2^{k-n}=\eta_0^{d\theta}\, 2^{k-n}$ and 
\begin{equation*}
	\sum_{1\leq 2^k\leq \frac{\eta_{0}}{\eta_*}}
	(2^{k} \eta_*)^{d+1-s}\lesssim   \eta_0^{d+1-s}.
\end{equation*}

We set $\eta_0=2^m\eta_*$ with $2^m\le \eta_*^{-1/10}$.
Then for all $0\le n\le\lfloor\log_2(\eta_*^{-2/3})\rfloor-m$ and $t\in[0,\min\{T,T_{\mathrm{coll}}^N\})$:
\begin{multline}\label{q21}
  \partial_t \M(2^{n+m} \eta_*)
  \\
  \lesssim
  \biggl(\sum_{k=0}^{n-1}\eta_*^{d\theta/2}\, 2^{k-n}\,\M(2^{m+k-1}\eta_*)
    + \frac{1}{N\eta_*^s}
    + \M(2^{m+n-1} \eta_*)+2^{2(n+m)}\eta_*^2\,\A_T\biggr)
  \bigl(\M(2^{n+m}\eta_*)+\A_T\bigr)
  \\
  + 2^{-2n}\eta_*^{d\theta/3}
  \biggl(\M(2^m\eta_*)+\A_T
    + \mathbf{1}_{s> 1}\Bigl(H(t)
      +\A_T+\sum_{k\leq m} \M(2^{k} \eta_*)
    \Bigr)\biggr)
  \bigl(\M(2^{n+m}\eta_*)+\A_T\bigr).
\end{multline}

\medskip
\noindent\textbf{Step 2. Bootstrap estimates.} Let $m_1\le \log\log N$ be the fixed integer chosen below. W Define
$\widetilde T$ as the supremum of all times $t_*\le \min\{T,T_{\mathrm{coll}}^N\}$ such that
\begin{align}
  &\sup_{[0,t_*]}\M(t,2^{n+m_1}\eta_*)\leq 1,
  \label{cou-bootstrap-sub}
  \\
  &\sup_{[0,t_*]}H(t)\leq 1,
  \label{cou-bootstrap-sub-H}
\end{align}
for any $n=0,\dots,\lfloor\log_2(\eta_*^{-2/3})\rfloor-m_1$. By Lemma~\ref{lem:local-classical-ode} and the initial assumptions
\eqref{Cou-initial0sub} and \eqref{condt02}, the bootstrap bounds hold for a
positive time; hence $\widetilde T>0$ for $N$ sufficiently large.\\
By \eqref{Mee}, for some constant $C\ge 1$:
\begin{equation}\label{q26}
  \sup_{\ell\leq m_1}\M(t,2^{\ell} \eta_*)
  \leq C\,2^{dm_1}\bigl(\M(t,2^{m_1} \eta_*)+\A_T\bigr)
  \leq C\,2^{dm_1}(1+\A_T),
\end{equation}
for all $t\in[0,\tilde T]$.

We return to \eqref{q21} and set $m=m_1+1$. The terms at scales $2^{m_1+k}\eta_*$, $k\ge0$, are controlled by the bootstrap assumption \eqref{cou-bootstrap-sub}. On the other hand, the contribution of the lower scales is estimated using \eqref{q26}; indeed,
$$
\sum_{\ell\le m_1}\M(t,2^\ell\eta_*)
\le
C m_1 2^{dm_1}(1+\A_T),
\qquad t\in[0,\widetilde T].
$$
Together with \eqref{cou-bootstrap-sub-H}, this gives, for every $t\in[0,\widetilde T]$ and every
$n=0,\dots,\lfloor\log_2(\eta_*^{-2/3})\rfloor-m_1-1$,
\begin{multline}\label{q27}
  \partial_t \M(2^{n+m_1+1} \eta_*)
  \\
  \lesssim
  \biggl(\sum_{k=0}^{n-1}\eta_*^{d\theta/2}\, 2^{k-n}\,\M(2^{m_1+k}\eta_*)
    + \frac{1}{N\eta_*^s}
    + \M(2^{m_1+n} \eta_*)+2^{2n+2(m_1+1)}\eta_*^2\,\A_T\biggr)
  (1+\A_T)
  \\
  + m_12^{dm_1}2^{-2n}\,\eta_*^{d\theta/3}\,(1+\A_T)^2.
\end{multline} 
We now use a standard continuation argument. Fix $T'\in(0,\tilde T]$ be the supremum of $t\le \tilde T$ such that 
\begin{equation}\label{q23}
  \M(t,2^{n+m_1+1} \eta_*)\leq e^{C_2 t}\,2^{-n},
\end{equation}
for $0\leq t\leq T'$ and $n=0,\dots,\lfloor\log_2(\eta_*^{-2/3})\rfloor-m_1-1$.
By \eqref{Cou-initial0sub} we have $T'>0$ (but possibly depending on $N$).
Combining \eqref{q23} with \eqref{q26} and \eqref{q27}:
\begin{multline*}
  \partial_t \M(2^{n+m_1+1} \eta_*)
  \lesssim
  \biggl(\sum_{k=0}^{n-1}\eta_*^{d\theta/2}2^{-n}
    + \frac{1}{N\eta_*^s}
    + 2^{-n}+2^{2(n+m_1+1)}\eta_*^2\biggr)
  (1+\A_T)^2\,e^{C_2 t}
  \\
  + m_12^{dm_1}2^{-2n}\,\eta_*^{d\theta/3}\,(1+\A_T)^2.
\end{multline*}
Since $m_1\le \log\log N$, $ \frac{1}{N\eta_*^s}\leq 2^{-n}$,  the last term can be absorbed, and
\[
  \partial_t \M(2^{n+m_1+1} \eta_*)
  \leq \tfrac{C}{4}\,(1+\A_T)^2\,e^{C_2 t}\,2^{-n}.
\]
Integrating and using \eqref{Cou-initial0sub} with $C_02^{-m_1}\leq \frac{1}{2}$ we get
\begin{align*}
  \M(t,2^{n+m_1+1} \eta_*)
  &\leq \M(0, 2^{n+m_1+1} \eta_*)+\frac{C(1+\A_T)^2}{4C_2}\,e^{C_2 t}\,2^{-n}
  \\
  &\overset{\eqref{Cou-initial0sub}}\leq
    \Bigl(\frac{1}{2}+\frac{C(1+\A_T)^2}{4C_2}\,e^{C_2 t}\Bigr)2^{-n}.
\end{align*}
Choosing $C_2:=C(1+\A_T)^2$, we get
\[
  \M(t,2^{n+m_1+1} \eta_*)\leq \tfrac{3}{4}\,e^{C_2 t}\,2^{-n}.
\]
Since the last estimate is strict with respect to the defining bound \eqref{q23}, the standard continuation argument implies that $T'=\widetilde T$. Consequently,
\begin{equation}\label{q23b}
  \M(t,2^{n+m_1+1} \eta_*)\leq e^{C(1+\A_T)^2 t}\,2^{-n},
\end{equation}
for $t\in [0,\tilde T]$ and $n=0,\dots,\lfloor\log_2(\eta_*^{-2/3})\rfloor-m_1-1$.\\
\medskip
\noindent\textbf{Step 3. Proof of \eqref{cou-bootstrap-sub}--\eqref{cou-bootstrap-sub-H}  and $ \tilde T=T$.}
Set
$\mathbf{n}:=\lfloor\log_2(\eta_*^{-2/3})\rfloor$,
$m\le \lfloor\log_2(\eta_*^{-1/10})\rfloor$,
and define
\[
  \M_{m,\mathbf{n}}(t):= \sum_{k=0}^{\mathbf{n}-m} 2^{k/2}\, \M(t, 2^{k+m}\eta_*).
\]
By \eqref{condt02}, and \eqref{Cou-initial0sub},  we have $\tilde T>0$ and
\begin{equation}\label{q28prime}
  H(0)\leq \eta_*^{\theta/4}.
\end{equation}
As in \eqref{q24}, it follows from \eqref{q21} that
$$\partial_t \M_{m,\mathbf{n}}(t)\leq C'\bigl(\M_{m,\mathbf{n}}(t)+1+\A_T\bigr)^2,$$
for $t\in [0,\tilde T]$. Moreover, as in \eqref{q71}, by \eqref{Cou-initial0sub}, we have
$$\M_{m_0+1,\mathbf{n}}(0)\le 1,~~m_0=\bigl\lceil \log_2(C+2)\bigr\rceil .$$
Applying Lemma~\ref{lemgronw}, and iterating the argument on time intervals as in the proof of Theorem~\ref{th1-Coulomb}, we obtain, for $N$ large enough,
\begin{equation*}
  \sup_{[0,\tilde T]}\M_{m_0+1+m^*,\mathbf{n}}(t)\leq \frac{11}{10},
  \qquad
  m^*=1+\bigl\lfloor 100 C'(1+\A_T)^2T\bigr\rfloor.
\end{equation*}
We can take
$$
m_1:=m_0+2+m^*=m_0+3+\bigl\lfloor 100C'(1+\A_T)^2T\bigr\rfloor  .
$$
By the choice of parameters, $m_1$ satisfies $C_02^{-m_1}\leq \frac{1}{2}$. Moreover, $m_1$ is fixed independently of $N$. Hence, for $N$ sufficiently
large, $m_1\le \log\log N$ and $m_1$ is an admissible scale in the bootstrap
definition above.\\
Then using the monotonicity relation \eqref{re:Mmn}, we get
\begin{equation}\label{q29}
  \sup_{[0,\tilde T]}\M(t,2^{n+m_1}\eta_*)\leq \frac{11}{10}2^{-1/2}.
\end{equation}
It remains to improve the bound on $H$. 
Moreover, by \eqref{q26}, $$\sup_{m\leq m_1}\sup_{[0,\tilde T]}\M(t,2^{m} \eta_*)\leq C\,2^{dm_1}(1+\A_T).$$
Combining this lower-scale bound with \eqref{cou-bootstrap-sub} and Lemma~\ref{lem31}, we obtain, for $t\in[0,\tilde T]$,
\begin{align*}
  S(\tilde T)\lesssim \sup_{0\leq \tau\leq \tilde T}\bigl(\M(\tau,\eta_*) +\A_T+1\bigr)^2
    \bigl(\eta_*^{3\theta}(2H(0)+1)+1\bigr)
  \lesssim 2^{2dm_1}(\A_T+1)^4.
\end{align*}
Since  $T$ is bounded independently of $N$, for $N$ large enough, we have
$$H(0)\, e^{2C S(\tilde T) T}+ \eta_*^{\theta}\bigl( e^{2C S(\tilde T) T}-1\bigr)
\leq \frac{2}{3}(2H(0)+1)\, \eta_*^{-\theta/2}.$$
Applying Lemma~\ref{lem31} and using \eqref{q28prime}, we deduce
\[
  \sup_{[0,\tilde T]}H(t)
  \le
  H(0)\,e^{2C S(\tilde T)T}
  +\eta_*^{\theta}\bigl(e^{2C S(\tilde T)T}-1\bigr)
  \overset{\eqref{q28prime}}\leq \frac{1}{2}.
\]
Thus both bootstrap assumptions are improved strictly on $[0,\widetilde T]$:
by \eqref{q29},
$$
\sup_{0\le t\le\widetilde T}
\M\bigl(t,2^{n+m_1}\eta_*\bigr)<1,
$$
for every admissible $n$, and also
$$
\sup_{0\le t\le\widetilde T}H(t)\le\frac12<1.
$$
By the definition of $\widetilde T$ and the standard continuation argument,
this is impossible unless
$$
\widetilde T=\min\{T,T_{\mathrm{coll}}^N\}.
$$
We now rule out the possibility that $\widetilde T=T_{\mathrm{coll}}^N<T$.
Indeed, if $T_{\mathrm{coll}}^N\le T$, then by
\eqref{maximal-collision} some pair  distance tends to zero as
$t\uparrow T_{\mathrm{coll}}^N$. Since the cutoff in the definition of $H$ is
equal to $1$ near the origin, this would force $H(t)\to\infty$, contradicting
the uniform bound above. Therefore
$
T_{\mathrm{coll}}^N>T,
$
and hence $\widetilde T=T$.

In particular, the particle system is classical and collision-free on $[0,T]$,
the bootstrap bounds \eqref{cou-bootstrap-sub}--\eqref{cou-bootstrap-sub-H}
hold on $[0,T]$, and \eqref{conclth1-con} follows.\\
\medskip
\noindent\textbf{Step 4. Conclusion.}
It remains to derive the desired bounds from the estimates above. By \eqref{q23b}, with $\widetilde T=T$, we have
$$
\M\bigl(t,2^{n+m_1+1}\eta_*\bigr)
\le
e^{C(1+\A_T)^2T}2^{-n},
$$
for every $t\in[0,T]$ and every admissible $n$. Equivalently, after relabeling the dyadic index, this gives
$$
\M\bigl(t,2^n\eta_*\bigr)
\le
e^{C(1+\A_T)^2T}2^{m_1+1}2^{-n},
$$
for all $n\ge m_1+1$ with $n\le \mathbf n$.

For the remaining lower scales $0\le n\le m_1$, we use \eqref{Mee}. Hence there exists $C'>0$ such that
\begin{equation}\label{xx1}
	\M\bigl(t,2^n\eta_*\bigr)
	\le
	e^{C'(1+\A_T)^2(T+1)}2^{-n},
\end{equation}
for every $t\in[0,T]$ and every $n=0,\dots,\mathbf n$. \\
As in Step~3 of the proof of Theorem~\ref{th1-Coulomb}, the dyadic estimate
\eqref{xx1} extends to arbitrary scales. More precisely, for every
$\eta\in[\eta_*,\eta_*^{1/3}]$,
\begin{equation*}
	\M\bigl(t,\eta\bigr)
	\lesssim \frac{\eta_{*}}{\eta}
	e^{C'(1+\A_T)^2(T+1)}.
\end{equation*}
Since $\A_T\gtrsim1$, the desired bound \eqref{conclth1} follows.\\

We now derive the corresponding $\mathsf{W}$-estimate. The argument is the same as in
the proof of \eqref{conclth1-wn-cou}, so we only record the source estimate.
Set
$$
\omega:=\tilde{\mu}_{N,2^{\bigl\lfloor\log_2(\eta_*^{-1/2})\bigr\rfloor}\eta_{*}}.
$$
Then $\omega$ satisfies the linearized equation around $\mu$ with source
$\mathbf R$:
\begin{equation}\label{eq:omega-th1}
	\partial_t\omega
	=
	-\operatorname{div}\bigl((\F*\omega)\mu+(\F*\mu)\omega\bigr)
	+
	\mathbf R,
\end{equation}
where $\mathbf R$ consists of the dual remainders together with the nonlinear
source term
$$
-\operatorname{div}\bigl((\F*\omega)\omega\bigr).
$$
Using \eqref{R2-8-dual}, \eqref{R1-dual}, the remaining dual estimates, and
the bound \eqref{xx1}, we obtain
\begin{equation}\label{eq:R-th1-dual-source}
	\|\langle\cdot\rangle^{\gamma-1}\mathbf R(t)\|_{(W^{1,1})^*}
	\lesssim
	\eta_*
	\exp\bigl[C(1+\A_T)^2(1+T)\bigr].
\end{equation}
Indeed, at the mesoscopic scale $2^{\bigl\lfloor\log_2(\eta_*^{-1/2})\bigr\rfloor}\eta_{*}\sim\eta_*^{1/2}$,
\eqref{xx1} gives $\M(t,2^{\bigl\lfloor\log_2(\eta_*^{-1/2})\bigr\rfloor}\eta_{*})\lesssim \eta_*^{1/2}
e^{C(1+\A_T)^2(1+T)}$, and hence the nonlinear source contributes at most
$O(\eta_*)$ in the dual norm.

Applying Lemma~\ref{Le-dual1} with $\gamma_1=\gamma-1$ gives
\eqref{conclth1-wn}. This completes the proof.
\end{proof}

\subsection{The case $s=d-1$ and $d \ge 3$}\label{sec:critical-high-d}

For $d \ge 3$, the method used for the two-dimensional Coulomb case does not directly apply:
we lose the crucial fact that $|x-y|\,\F(x,y)$ is bounded (which was used for $s\le 1$ in Lemma~\ref{lemR1} to control $\R^1$).
Moreover, we cannot close a good estimate on~$H$ when defining it as in \eqref{condt} with $\theta=0$.
Finally, the difference is that the argument is local in time; however, the resulting lifespan $T_0$ is independent of $N$, and we can only reason over one timestep.


\subsection{Proof of Theorem~\ref{th2}}\label{sec:proof-th2}
By Lemma~\ref{lem:local-classical-ode}, for each fixed $N$ the particle
system admits a unique classical solution on a maximal collision-free interval
$
[0,T_{\mathrm{coll}}^N).
$
All estimates below are first derived on
$
[0,\min\{T,T_{\mathrm{coll}}^N\}).
$
The uniform bound on $H$ obtained below prevents collisions and therefore
extends the classical solution to the whole time interval under consideration.
We therefore focus on the quantitative estimates. Let $\eta_0=\eta_*$ and $\eta=2^n\eta_0$.
Using \eqref{R2-8}, \eqref{R10-12}, \eqref{R1}, and \eqref{R9} in the Coulomb regime $s=d-1$, we obtain
\begin{align*}
  & \sum_{m=2}^8\xg |\R_{N,\eta}^m|
  \lesssim
    \sum_{k=0}^{n+1} 2^{k-n}\,\M(2^{k-1}\eta_{0})\,\U(\eta)
    + \Bigl(\frac{1}{N\eta_{0}^{d-1}}+\eta^2\,\A(\mu)\Bigr)\U(\eta)
  \\
  &\quad
  + \biggl(\sum_{k=0}^{n_0-1}\M(2^{k}\eta)
    +|\!\log(2^{n_0}\eta)|\,\M(2^{n_0}\eta)
    +|\!\log(2^{n_0}\eta)|\,2^{2n_0}\eta^2\,\A(\mu)\biggr)
  \\
  &\qquad\times\bigl(\M(\eta/2)+\eta^2\,\A(\mu)\bigr),
  \\[3pt]
 & \R_{N,\eta}^{10}+\R_{N,\eta}^{11} +\R_{N,\eta}^{12}
  \lesssim \M(\eta/2)\,\U(\eta)+\eta^2\,\A(\mu)\,\U(\eta),
  \\[3pt]
 & \xg|\R_{N, \eta}^{1}|
  \lesssim
    \frac{\eta_{0}^{2}}{\eta^2}\,
    \bigl(\log(2\eta/\eta_{0})\,\U(\eta_0) +H(t)\bigr)\,\U(\eta),
  \\[3pt]
 & |\R_{N, \eta}^{9}|=0.
\end{align*}
Using \eqref{ptM}, the identity $\U(\eta)=\M(\eta)+\A(\mu)$, and choosing
$$
n_0:=\bigl\lfloor\log_2(\eta_0^{-2/3})\bigr\rfloor-n,
$$
we obtain, for $0\le n\le \lfloor\log_2(\eta_0^{-2/3})\rfloor$ and $t\in [0,\min\{T,T_{\mathrm{coll}}^N\})$
\begin{multline}\label{434}
  \partial_t \M(2^n\eta_0)
  \\
  \leq C_1
  \biggl(\sum_{k=0}^{n} 2^{k-n}\,\M(2^{k-1}\eta_{0})
    +(\eta_{0}+2^{2n}\eta_0^2)(1+\A(\mu))
    +(n+1)\,2^{-2n}\bigl(\U(\eta_0) +H(t)\bigr)\biggr)
  \\
  \times\bigl(\M(2^n\eta_0)+\A(\mu)\bigr)
  \\
  +C_1\biggl(\sum_{k=0}^{\lfloor\log_2(\eta_{0}^{-2/3})\rfloor-n-1}
    \M(2^{n+k}\eta_0)
    +\log(N)\,\M(2^{\lfloor\log_2(\eta_0^{-2/3})\rfloor}\eta_0)
    +\log(N)\,\eta_0^{2/3}\,\A(\mu)\biggr)
  \\
  \times\bigl(\M(2^{n-1}\eta_0)+2^{2n}\eta_{0}^2\,\A(\mu)\bigr),
\end{multline}
where $C_1$ depends only on the fixed parameters.\\
\medskip
\noindent\textbf{Estimate on $H(0)$.}
By the definition of $H$, together with \eqref{condt01} and \eqref{Cou-initial-high}, we have
\begin{align*}
  H(0)
  &\leq \frac{1}{N}\sup_{i\in [1,N]}\sum_{\substack{j=1\\j\neq i}}^N
    \frac{\mathbf{1}_{C^{-1}\eta_{0}\leq |x_i^0-x_j^0|\leq 2\eta_{0}}}{|x_i^0-x_j^0|^{d}}
  \\
  &\lesssim \sup_{x}
    \frac{\mathbf{1}_{\frac{1}{2C}\eta_{0}\leq |\cdot|\leq 4\eta_{0}}}{|\cdot|^{d}}
    *\mu_{N,\eta_{0}}^0(x)
  \lesssim \log(2+C),
\end{align*}
hence
\begin{equation}\label{q34}
  H(0)\leq C'\log(2+C).
\end{equation}

\medskip
\noindent\textbf{Step 1. A priori bound under bootstrap assumptions.} 
Fix an integer $m_0$ satisfying \begin{equation}\label{def:m0} m_0 \geq \left\lceil \log_2 \left( 1+ \max_{0\leq n\leq n_0} 2^{n/4}\M(0,2^n\eta_0) \right) \right\rceil. \end{equation} We shall increase $m_0$ later, if necessary, depending only on the constants appearing in the bootstrap argument. In particular, \begin{equation}\label{initial-M-bootstrap} \M(0,2^n\eta_0) \leq 2^{m_0-n/4}, \qquad 0\leq n\leq n_0. \end{equation}

Let $\widetilde T$ be the supremum of all times
$
t_*\le \min\{T,T_{\mathrm{coll}}^N\}
$
such that
\begin{align}
  &\sup_{[0,t_*]}\M(t, 2^{n}\eta_0)\leq 2^{m_0+10-n/4},
  \label{q31}
  \\
  &\sup_{[0,t_*]}H(t) \leq 2(H(0)+1),
  \label{q32}
\end{align}
for any $0\leq n\leq \lfloor\log_2(\eta_{0}^{-2/3})\rfloor$.\\
By Lemma~\ref{lem:local-classical-ode} and the initial assumptions
\eqref{Cou-initial-high} and \eqref{condt01}, the bootstrap bounds hold for a
positive time. Hence $\widetilde T>0$ for $N$ sufficiently large.\\
As \eqref{w5}, using \eqref{q31}--\eqref{q32} 
\begin{align*}
	\sum_{k=0}^{\lfloor\log_2(\eta_{0}^{-2/3})\rfloor-n-1}
	\M(2^{n+k}\eta_0)
	+(\log N)\,\M(2^{\lfloor\log_2(\eta_0^{-2/3})\rfloor}\eta_0)
	+(\log N)\,\eta_0^{2/3}\,\A(\mu)\leq 2^{m_0+12}(1+\A_T),
\end{align*}
for $N\geq 10$.\\
Then, using \eqref{q31}--\eqref{q32} and \eqref{q34} in \eqref{434}:
\begin{equation}\label{434-bootstrap}
  \partial_t \M(2^n\eta_0)
  \leq C_1\,
  \biggl(\sum_{k=1}^{n} 2^{k-n}\,\M(2^{k-1}\eta_{0})
    +\eta_{0}+2^{2n}\eta_0^2+(n+1)2^{-2n}\biggr)
  2^{2m_0+30}(1+\A_{T})^2,
\end{equation}
where $C_1$ depends on $C$.
Define
\[
  b_n(t):=2^{n-m_0}\, \M\bigl(\tfrac{t}{C_2},\, 2^n\eta_0\bigr).
\]
Then for $t\in[0,C_2 \widetilde T]$ and $0\leq n\leq \lfloor\log_2(\eta_{0}^{-2/3})\rfloor$:
\[
  \partial_t b_n(t)
  \leq\frac{C_1\,2^{2m_0+100}(1+\A_{T})^2}{C_2}
  \biggl(\sum_{k=0}^{n} b_k(t)+1\biggr),
  \qquad b_n(0)\leq 1.
\]
Choosing $C_2:= C_1\,2^{2m_0+100}(1+\A_{T})^2$,
Lemma~\ref{Le1} yields
\[
b_n(t)\leq 	2\exp\bigl(2\sqrt{nt}+2t\bigr).
\]
Returning to $\M$, for $t\in[0,\widetilde T]$ and $0\le n\le \lfloor\log_2(\eta_{0}^{-2/3})\rfloor$:
\begin{equation}\label{q33}
  \M(t, 2^n\eta_0)
  \leq 2^{-n+m_0+1}\,\exp\bigl(2\sqrt{C_2 nt}+2C_2 t\bigr).
\end{equation}
As in Step~3 of the proof of Theorem~\ref{th1-Coulomb}, the dyadic estimate
\eqref{q33} extends to arbitrary scales. More precisely, for every
$\eta\in[\eta_*,\eta_*^{1/3}]$ and $t\in[0,\widetilde T]$
\begin{equation}\label{q33'}
	\M\bigl(t,\eta\bigr)
	\lesssim \frac{\eta_{*}}{\eta}
	2^{m_0+1}\,\exp\bigl(2\sqrt{C_2 t\log(\frac{\eta}{\eta_{*}})}+2C_2 t\bigr).
\end{equation}
\medskip

\noindent\textbf{Step 2. Closing the bootstrap and choosing the lifespan.}
We now prove that the bootstrap assumptions \eqref{q31}--\eqref{q32} hold on a time interval $[0,T_0]$ with $T_0>0$ independent of~$N$.
Define
\[
  \M_{\rm tot}(t):= \sum_{k=0}^{\lfloor\log_2(\eta_{0}^{-2/3})\rfloor} 2^{k/2}\, \M(t, 2^{k}\eta_{0}).
\]
By \eqref{Cou-initial-high}, 
\begin{align}\label{w6}
 \M_{\rm tot}(0)\leq 2C.
\end{align}
By \eqref{q32} and \eqref{q34},  we have 
\begin{equation}\label{q28b}
  \sup_{[0,\widetilde T]}H(t)\leq 2(H(0)+1)\leq  2C'\,\log(2+C).
\end{equation}
Arguing as in \eqref{q24}, we have
$$\partial_t \M_{\rm tot}(t)\leq C'' 2^{2m_0}\bigl(1+\M_{\rm tot}(t)+\A_{T}\bigr)^2.$$
By Lemma~\ref{lemgronw} and \eqref{w6}, there exists $\varepsilon_0>0$, independent of $N$, such that on the interval $[0,T_1]$ with 
$$T_1:=\min\{ \frac{\ep_0}{2^{2m_0}(1+\A_{T})^2},\widetilde T\},$$ we have
\begin{equation}\label{q35}
  \sup_{[0,T_1]}\M_{\rm tot}(t)\leq 4C.
\end{equation}
In particular, choosing $m_0$ large enough so that $4C\le 2^{m_0+9}$, we obtain
\begin{equation}
	\sup_{[0,T_1]}\M(t, 2^{n}\eta_0)\leq  2^{m_0+9-n/4},
\end{equation}
for any $0\leq n\leq \lfloor\log_2(\eta_{0}^{-2/3})\rfloor$.
This is a strict improvement of \eqref{q31}.\\
It remains to improve the bound on $H$. Using \eqref{q35}, \eqref{q28b}, and \eqref{q34}, we have
\[
  \sup_{[0,T_1]}S(t)
  \lesssim 2^{3m_0}(\A_{T}+1)^2(2H(0)+2)^2
  \lesssim 2^{5m_0}(\A_{T}+1)^2.
\]
Thus, there exists $\varepsilon_0'\in(0,\varepsilon_0]$, independent of $N$, such that if $$T_2= \min\bigl\{\frac{\varepsilon_0'}{2^{5m_0}(\A_{T}+1)^2},\,T_1\bigr\},$$
then \eqref{condS} holds on $[0,T_2]$, and
$$
H(0)e^{2CS(t)t}
+
e^{2CS(t)t}-1
\leq
\frac65(H(0)+1),
\qquad 0\le t\le T_2.
$$
Applying Lemma~\ref{lem31}, we deduce
\[
  H(t)\le H(0)\,e^{2C S(t)t}+e^{2C S(t)t}-1\leq \frac{6}{5}(H(0)+1)
  \qquad \forall\,t\leq T_2.
\]
This strictly improves \eqref{q32}.

Therefore both bootstrap inequalities are strictly improved on the time
interval on which the above argument applies.

We now choose the final lifespan. Set
\begin{equation}\label{T0-th2}
	T_0:=
	\min\left\{
	T,\,
	\frac{\varepsilon_0'}{2^{5m_0+1}(1+\A_T)^2}
	\right\}.
\end{equation}
Then $T_0>0$ is independent of $N$.

We claim that $\widetilde T\ge T_0$. Suppose, by contradiction, that
$\widetilde T<T_0$. Then, by the definition of $T_0$, we have
$$
\widetilde T
<
\frac{\varepsilon_0}{2^{2m_0}(1+\A_T)^2},
\qquad
\widetilde T
<
\frac{\varepsilon_0'}{2^{5m_0}(1+\A_T)^2}.
$$
Hence $T_1=\widetilde T$ and $T_2=\widetilde T$ in the preceding argument.
Thus both bootstrap bounds \eqref{q31}--\eqref{q32} are strictly improved on
$[0,\widetilde T]$.

By the definition of $\widetilde T$, this strict improvement is impossible
unless the maximal classical interval ends there, namely unless
$
\widetilde T=T_{\mathrm{coll}}^N .
$
However, this alternative cannot occur. Indeed, if
$T_{\mathrm{coll}}^N\le T_0$, then by \eqref{maximal-collision} some pair
distance tends to zero as $t\uparrow T_{\mathrm{coll}}^N$. Since the cutoff in
the definition of $H$ is equal to $1$ near the origin, this would force
$H(t)\to\infty$, contradicting the uniform bound \eqref{q32}. Therefore
$T_{\mathrm{coll}}^N>T_0,$ and 
$\widetilde T\ge T_0.$
Consequently, the particle system is classical and collision-free on
$[0,T_0]$, and the bootstrap assumptions \eqref{q31}--\eqref{q32} hold on
$[0,T_0]$.\\
Finally, applying \eqref{q33'} on $[0,T_0]$ and using
$T_0\le 2^{-5m_0}(1+\A_T)^{-2}$, we obtain the estimate
\eqref{conclth2}. \\It remains to derive the dual estimate \eqref{conclth2-wn}. This follows by
the same dual-source argument as in the proofs of \eqref{conclth1-wn-cou} and
\eqref{conclth1-wn}. Set
$$\omega:=\tilde{\mu}_{N,2^{\bigl\lfloor\log_2(\eta_*^{-1/2})\bigr\rfloor}\eta_{*}}.$$ Then the equation for $\omega$ has a source term $\mathbf R$ consisting of the
dual remainders and the quadratic nonlinear term
$$
-\operatorname{div}\bigl((\F*\omega)\omega\bigr).
$$
By \eqref{R2-8-dual}, \eqref{R1-dual}, the remaining dual estimates, and
\eqref{conclth2}, this source satisfies
$$
\|\langle\cdot\rangle^{\gamma-1}\mathbf R(t)\|_{(W^{1,1})^*}
\lesssim 	\eta_*(\log N)
\exp\bigl[C(1+\A_{T})\sqrt{(\log N)t}+C(1+\A_{T})^2(1+T)\bigr].
$$
Lemma~\ref{Le-dual1}, applied with $\gamma_1=\gamma-1$, then yields
\eqref{conclth2-wn}. This completes the proof.


\medskip
\begin{remark}[A collision-energy estimate in the Coulomb case]
	In the Coulomb case $s=d-1$, $d\ge3$, the quantity
	\begin{equation}\label{collision-energy}
		\int_0^T
		\frac1{N^2}
		\sum_{i\neq j}
		|x_i(t)-x_j(t)|^{2-d}\,dt
	\end{equation}
	is the natural integrability condition for the singular part in the
	symmetrized weak formulation, since
	$$
	|\nabla\psi(x_i)-\nabla\psi(x_j)|
	|\F^{\mathrm{sing}}(x_i,x_j)|
	\lesssim
	\|\nabla^2\psi\|_{L^\infty}|x_i-x_j|^{2-d}.
	$$
	
	For the model attractive Coulomb force
	$
	\F^{\mathrm{sing}}(x,y)
	=
	-\frac{x-y}{|x-y|^d},
	$
	this condition follows from a virial estimate. Indeed, setting
	$
	I(t):=\frac1N\sum_{i=1}^N |x_i(t)|^2,
	$
	the antisymmetry of $\F^{\mathrm{sing}}$ gives
	$$
	\frac{d}{dt}I(t)
	+
	\frac1{N^2}\sum_{i\neq j}|x_i-x_j|^{2-d}
	\lesssim
	1+I(t),
	$$
	where the right-hand side comes from the regular part of the force. Hence
	Grönwall's inequality yields
	$$
	\int_0^T
	\frac1{N^2}
	\sum_{i\neq j}|x_i(t)-x_j(t)|^{2-d}\,dt
	\lesssim_T
	1+I(0),
	$$
	on every classical collision-free time interval.
	
Thus, in the exact antisymmetric Coulomb case, the integrated collision energy
is finite before collision. This suggests that \eqref{collision-energy} is a
natural admissibility condition for weak continuations after collision, since
it is precisely the integrability needed to define the singular part in the
symmetrized weak formulation. However, this estimate alone does not provide
the stability and multiscale estimates required to prove the mean-field limit
up to the blow-up time of \eqref{mflimit}. We therefore do not pursue this
question here.
\end{remark}

\begin{remark}\label{re-for2}
	Consider the system \eqref{ode-fo} from Remark \ref{re-for}, with external forces $V_N^1(x_i)+V_N^2(x_i)$ satisfying
\begin{equation}
|V_N^2(x_i)|+\langle x_i-x\rangle^{-1}	|V_N^1(x_i)-V_N^1(x)|+\langle x\rangle^{-1}	|V_N^1(x)|\lesssim \eta_*^{d-s},
	\end{equation}
for any $x\in \mathbb{R}^d$. Then
\begin{align*} \frac1N\sum_{i=1}^N (|V_N^1(x_i)-V_N^1(x)|+|V_N^2(x_i)|)|\nabla \Phi^\eta_{x_i}(x)|+\gamma\langle x_t\rangle^{\gamma-1}
		|V_N^1(x_t)| |\tilde \mu_{N,\eta}(x_t) |
		\lesssim
		\frac{\eta_*^{d-s}}{\eta}\,\U(\eta),
	\end{align*}
for any $x\in \mathbb{R}^d$. 
	As a consequence, the additional force term can be treated as an error term, and \eqref{434} still holds. Therefore, the conclusion of Theorem~\ref{th2} remains valid in this setting.
\end{remark}


\subsection{The case $s>d-1$: proof of Theorem~\ref{th3'}}\label{sec:proof-supercoulomb}

By Lemma~\ref{lem:local-classical-ode}, for each fixed $N$ the particle
system admits a unique classical solution on a maximal collision-free
physical-time interval
$
[0,T_{\mathrm{coll}}^N).
$
We introduce the corresponding maximal rescaled time
$
T_{\mathrm{coll,res}}^N
:=
\eta_0^{d-s-1}T_{\mathrm{coll}}^N,
$
so that the physical time $\tau=\eta_0^{s+1-d}t$ belongs to
$[0,T_{\mathrm{coll}}^N)$ precisely when
$t\in[0,T_{\mathrm{coll,res}}^N)$.

All estimates below are first derived on the rescaled interval
$
[0,\min\{T,T_{\mathrm{coll,res}}^N\}).
$
The uniform bound on $H$ obtained below prevents collisions and therefore
extends the classical solution to the physical time interval constructed
below.

We introduce
$$\A'(\mu)=\A(\mu)+
\big\|\xg \nabla^4 \mu^0\big\|_{L^\infty}	\mathbf{1}_{s\geq d}.$$
Let $\eta_0=\eta_*$ and $\eta=2^n\eta_0$.
Using \eqref{R2-8}, \eqref{R10-12}, \eqref{R1}, and \eqref{R9} in the super-Coulomb regime $s>d-1$:
\begin{align*}
&  \sum_{m=2}^8\xg |\R_{N,\eta}^m| 
  \lesssim
    \sum_{k=0}^{n+1}\eta_0^{d-s-1}\, 2^{k(d+1-s)_{+}-k-n}
    (1+k\,\mathbf{1}_{s=d+1})\,\M(2^{k-1}\eta_{0})\,\U(\eta)
  \\
  &\quad
  + \Bigl(
    (\mathbf 1_{s<d}\,\eta^2+\mathbf{1}_{s\geq d}
      +\mathbf{1}_{s=d+1}|\!\log\eta_{0}|
      +\mathbf{1}_{s>d+1}\eta_0^{d-s+1})\A'(\mu)
    + \frac{1+\mathbf{1}_{s=d+1}n}{N\eta_{0}^s}
  \Bigr)\U(\eta)\\&\quad+\left(\mathbf{1}_{s<d+1}
  + \mathbf{1}_{s=d+1}|\log\eta_0|
  + \mathbf{1}_{s>d+1}\eta_0^{-s+d+1}\right) \M(\eta/2)\A'(\mu)
  \\
  &\quad
  + \eta^{-\min\{2,s+1-d\}}
    (1+\mathbf{1}_{s=d+1}n
      +\mathbf{1}_{s>d+1}\eta_{0}^{-s+d+1})\,
    \M(\eta/2)\,\U(\eta),
\end{align*}
\begin{align*}
&  \R_{N,\eta}^{10}+\R_{N,\eta}^{11} +\R_{N,\eta}^{12}
  \lesssim
    \mathbf{1}_{s< d}\,\M(\eta)\,\U(\eta)
  \\
  &\quad
  + \bigl(\mathbf{1}_{s=d}|\log\eta|
    +\mathbf{1}_{s\in (d,d+1)}\eta^{-s+d}
    +	\eta^{-1}(\mathbf{1}_{s=d+1}n
    +\mathbf{1}_{s>d+1}\eta_{0}^{-s+d+1})\bigr)
    \M(\eta/2)\,\U(\eta)
  \\
  &\quad
  + \eta^2\bigl(
    \mathbf 1_{s<d+1}
    + \mathbf 1_{s=d+1}|\!\log\eta_0|
    + \mathbf 1_{s>d+1}\eta_0^{-s+d+1}
  \bigr)\A(\mu)\,\U(\eta),\\
  &\xg|\R_{N, \eta}^{1}|
  \lesssim
    \frac{\eta_0^{-s+d+1}}{\eta^2} 
    \bigl(\U(\eta_0) +H(t)\bigr)\,\U(\eta),
  \\
&\langle x\rangle^\gamma|\R_{N, \eta}^{9}|
  \lesssim
    \mathbf{1}_{s\geq d+1}\,
    \frac{\eta_{0}^{-s+d+1}}{\eta^2}\,
    H(t)^{(s-1)/d}\,\U(\eta).
\end{align*}
Here we have used the fact that 
\begin{align*}
	&	\frac{1+\mathcal C(\eta,\eta_*)}
	{N\eta_0^{\min\{s,d+1\}}}\sim 
\frac{	1+\mathbf{1}_{s=d+1}n}{N\eta_{0}^s},\\&
(1+ \mathbf 1_{s>d+1}\eta_*^{-s+d+1})\left(\Xi(\eta,\eta_0)\eta_{0}^2\U(\eta_0)+ \mathbf{1}_{s\leq 1}\eta_{0}^{d+1-s}\U(\eta_0)
+ \mathbf{1}_{s> 1}\,\eta_*^{d-\theta-\min\{s-1,d\}}\,H(t)\right.\\&\quad\quad\left.+ \mathbf{1}_{s> 1}\sum_{k, 1\leq 2^k\leq \frac{\eta_{0}}{\eta_*}}
(2^k\eta_*)^{d-\min\{s-1,d\}}\,\U(2^k\eta_*)\right)\lesssim \eta_0^{d+1-s} \bigl(\U(\eta_0) +H(t)\bigr).
\end{align*}
Restricting to $2^n\le\eta_0^{-2/3}$ yields
\begin{multline}\label{q36-raw}
  \partial_t \M(2^n\eta_0)
  \lesssim
  \sum_{k=0}^{n+1}\eta_0^{d-s-1}\, 2^{k(d+1-s)_{+}-k-n}
    (1+k\,\mathbf{1}_{s=d+1})\,\M(2^{k-1}\eta_{0})\,
    \bigl(\M(2^n\eta_0)+\A'(\mu)\bigr)
  \\
  + \eta_0^{-s+d-1}\,
    (\eta_0+2^{-2n})\,
    \bigl(1+\A'(\mu) +H(t)^{\max\{(s-1)/d,\,1\}}\bigr)\,
    \bigl(\M(2^n\eta_0)+\A'(\mu)\bigr).
\end{multline}
Here we used
\begin{align*}
	&(2^n\eta_0)^{-\min\{2,(s+1-d)_{+}\}}(1+(1+n)\mathbf{1}_{s=d+1}+\mathbf{1}_{s>d+1}\eta_{0}^{-s+d+1})\\&+ \left(\mathbf{1}_{s<d+1}
	+ \mathbf{1}_{s=d+1}|\log\eta_*|
	+ \mathbf{1}_{s>d+1}\eta_*^{-s+d+1}\right)\\&+
	\bigl(\mathbf{1}_{s=d}|\log\eta|
	+\mathbf{1}_{s\in (d,d+1)}\eta^{-s+d}
	+\log(2\eta/\eta_{0})\,\eta^{-1}\,\mathbf{1}_{s=d+1}
	+\eta_{0}^{-s+d+1}\,\eta^{-1}\,\mathbf{1}_{s>d+1}\bigr)
\\&\quad\quad\quad\quad\quad\quad\quad\quad\quad\lesssim \eta_0^{d-s-1}\, 2^{n(d+1-s)_{+}-2n}(1+(n+1)\mathbf{1}_{s=d+1}),
\end{align*}
and 
\begin{align*}
	(\mathbf 1_{s<d}\,\eta^2+\mathbf{1}_{s\geq d}
	+\mathbf{1}_{s=d+1}|\log\eta_{0}|
	+\mathbf{1}_{s>d+1}\eta_0^{d-s+1})
	+ \frac{1}{N\eta_{0}^s}\lesssim \eta_0^{-s+d-1}\,
	(\eta_0+2^{-2n}).
\end{align*}
We introduce the \emph{rescaled metric}
\[
  \widetilde{\M}(t, 2^n\eta_0):=\M(\eta_{0}^{s+1-d}\,t,\, 2^n\eta_0),
\]
so that it suffices to estimate $\widetilde{\M}$ on $[0,T_0]$.\\
By \eqref{q34},
\begin{equation}
	H(0)\lesssim 2^{m_0}\log(2+C^{-1}).
\end{equation}
Moreover,
\begin{equation}
	\partial_t\widetilde{\M}(t,\eta)
	=
	\eta_0^{s+1-d}
	\bigl(\partial_t\M\bigr)(\eta_0^{s+1-d}t,\eta).
\end{equation}
Multiplying \eqref{q36-raw} by $\eta_0^{s+1-d}$ and evaluating all
physical-time quantities at $\eta_0^{s+1-d}t$, we obtain the rescaled
differential inequality
\begin{multline}\label{q36}
  \partial_t \widetilde{\M}(2^n\eta_0)
  \lesssim
  \sum_{k=0}^{n+1} 2^{k(d+1-s)_{+}-k-n}
    (1+k\,\mathbf{1}_{s=d+1})\,\widetilde{\M}(2^{k-1}\eta_{0})\,
    \bigl(\widetilde{\M}(2^n\eta_0)+\widetilde{\A}_{T}\bigr)
  \\
  + (\eta_0+2^{-2n})\,
    \bigl(1+\widetilde{\A}_{T}
      +H(\eta_{0}^{s+1-d}t)^{\max\{(s-1)/d,\,1\}}\bigr)\,
    \bigl(\widetilde{\M}(2^n\eta_0)+\widetilde{\A}_{T}\bigr),
\end{multline}
for
$$
0\le t<\min\{T,T_{\mathrm{coll,res}}^N\},
\qquad
2^n\le\eta_0^{-2/3},
$$
where
$$
\widetilde{\A}_T
:=
\sup_{0\le \tau\le \eta_0^{s+1-d}T}\A'(\mu^\tau).
$$ The factor
$\eta_0^{s+1-d}$ coming from the time rescaling cancels exactly the
prefactor $\eta_0^{d-s-1}$ appearing in \eqref{q36-raw}. Thus no constants
or exponents are changed in passing from \eqref{q36-raw} to \eqref{q36}.\\
\medskip
\noindent\textbf{Step 1. A priori bound under bootstrap assumptions.} Let $\widetilde T$ be the supremum of all rescaled times
$
t_*\le \min\{T,T_{\mathrm{coll,res}}^N\}
$
such that
\begin{align}
	&\sup_{0\le t\le t_*}
	\widetilde{\M}(t,2^n\eta_0)
	\le
	2^{m_0+10-n/4},
	\label{q31p}
	\\
	&\sup_{0\le t\le t_*}
	H(\eta_0^{s+1-d}t)
	\le
	2\bigl(H(0)+1\bigr),
	\label{q32p}
\end{align}
for every
$
0\le n\le \bigl\lfloor\log_2(\eta_0^{-2/3})\bigr\rfloor .
$
By Lemma~\ref{lem:local-classical-ode} and the initial assumptions, the
bootstrap bounds hold for a positive rescaled time. Hence
$\widetilde T>0$ for $N$ sufficiently large.\\
Then, using \eqref{q31p}--\eqref{q32p} and \eqref{q36}:
\begin{multline}\label{q36p}
  \partial_t \widetilde{\M}(2^n\eta_0)
  \leq C_1\sum_{k=0}^{n} 2^{k(d+1-s)_{+}-k-n}
    (1+k\,\mathbf{1}_{s=d+1})\,\widetilde{\M}(2^{k}\eta_{0})\,
    2^{m_0}(1+\widetilde{\A}_{T_0})
  \\
  + C_1\,2^{-n}\,
    2^{(1+\max\{(s-1)/d,\,1\})m_0}\,(1+\widetilde{\A}_{T_0})^2.
\end{multline}
As in the proof of \eqref{q29}, a standard continuation argument gives
\begin{equation}\label{q37}
  \widetilde{\M}(t, 2^n\eta_0)\leq e^{C_2 t}\,2^{m_0+10-n},
\end{equation}
for all $t\in[0,\widetilde T]$ and $0\le n\le\lfloor\log_2(\eta_0^{-2/3})\rfloor$, where
$$C_2:=C_1\,2^{10+(1+\max\{(s-1)/d,\,1\})m_0}(1+\widetilde{\A}_{T})^2.$$
As in Step~3 of the proof of Theorem~\ref{th1-Coulomb}, the dyadic estimate \eqref{q37} extends to arbitrary scales. More precisely, for every
$\eta\in[\eta_*,\eta_*^{1/3}]$ and $t\in[0,T_0]$
\begin{equation}\label{q37'}
  \widetilde{\M}\bigl(t,\eta\bigr)
	\lesssim \frac{\eta_{*}}{\eta} e^{C_2 t}\,2^{m_0+10}.
\end{equation}
\medskip
\noindent\textbf{Step 2. Closing the bootstrap and choice of $T_0$.}
Define
\[
  \widetilde{\M}_{\rm tot}(t):= \sum_{k=0}^{\lfloor\log_2(\eta_{0}^{-2/3})\rfloor} 2^{k/2}\, \widetilde{\M}(t, 2^{k}\eta_{0}).
\]
Assuming \eqref{q32p}, and arguing as in \eqref{q24}, we obtain
$$
\partial_t\widetilde{\M}_{\rm tot}(t)
\le
C\,2^{\max\{(s-1)/d,1\}m_0}
\bigl(1+\widetilde{\M}_{\rm tot}(t)+\widetilde{\A}_T\bigr)^2,
$$
for $0\le t\le\widetilde T$.

By Lemma~\ref{lemgronw}, there exists $\varepsilon_0>0$, independent of $N$,
such that on the interval
$$
T_1
:=
\min\left\{
\frac{\varepsilon_0}
{2^{\max\{(s-1)/d,1\}m_0}(1+\widetilde{\A}_T)^2},
\widetilde T
\right\},
$$
we have
\begin{equation}\label{q35p}
	\sup_{0\le t\le T_1}\widetilde{\M}_{\rm tot}(t)
	\le
	2^{m_0+2}.
\end{equation}
Consequently,
$$
\sup_{0\le t\le T_1}
\widetilde{\M}(t,2^n\eta_0)
\le
2^{m_0+9-n/4},
$$
for every admissible $n$. This strictly improves \eqref{q31p} on
$[0,T_1]$.

It remains to improve the bound on $H$. Using \eqref{q35p}, we have
$$
\sup_{0\le \tau\le \eta_0^{s+1-d}T_1}S(\tau)
\lesssim
\eta_0^{d-1-s}
2^{(4+(s+1)/d)m_0}
(1+\widetilde{\A}_T)^2 .
$$
Thus there exists $\varepsilon_0'\in(0,\varepsilon_0]$, independent of $N$,
such that if
$$
T_2
:=
\min\left\{
\frac{\varepsilon_0'}
{2^{(4+(s+1)/d)m_0}(1+\widetilde{\A}_T)^2},
T_1
\right\},
$$
then \eqref{condS} with $\theta=0$ holds on the physical interval
$[0,\eta_0^{s+1-d}T_2]$, and
$$
H(0)e^{2CS(\tau)\tau}
+
e^{2CS(\tau)\tau}-1
\le
\frac65\bigl(H(0)+1\bigr),
\qquad
0\le \tau\le \eta_0^{s+1-d}T_2.
$$
Indeed, the factor $\eta_0^{d-1-s}$ in the bound for $S$ is exactly
compensated by the length $\eta_0^{s+1-d}T_2$ of the physical time interval.
Applying Lemma~\ref{lem31}, we obtain
$$
H(\eta_0^{s+1-d}t)
\le
\frac65\bigl(H(0)+1\bigr),
\qquad
0\le t\le T_2.
$$
This strictly improves \eqref{q32p}.\\

We now choose the final rescaled lifespan
\begin{equation}\label{T0'}
	T_0
	:=
	\min\left\{
	T,\,
	\frac{\varepsilon_0'}
	{2^{(4+(s+1)/d)m_0+1}(1+\widetilde{\A}_T)^2}
	\right\}.
\end{equation}
Then $T_0>0$ is independent of $N$.

We claim that $\widetilde T\ge T_0$. Suppose by contradiction that
$\widetilde T<T_0$. Then, by the definition of $T_0$, the preceding argument
gives $T_1=\widetilde T$ and $T_2=\widetilde T$. Hence both bootstrap bounds
\eqref{q31p}--\eqref{q32p} are strictly improved on $[0,\widetilde T]$.

By the definition of $\widetilde T$, this is impossible unless
$
\widetilde T=T_{\mathrm{coll,res}}^N.
$
This alternative cannot occur. Indeed, if
$\widetilde T=T_{\mathrm{coll,res}}^N<T_0$, then the physical collision time
satisfies
$
T_{\mathrm{coll}}^N
=
\eta_0^{s+1-d}\widetilde T
<
\eta_0^{s+1-d}T_0.
$
By \eqref{maximal-collision}, some pairwise distance tends to zero as
$\tau\uparrow T_{\mathrm{coll}}^N$. Since the cutoff in the definition of
$H$ is equal to $1$ near the origin, this would force $H(\tau)\to\infty$,
contradicting \eqref{q32p}. Therefore
$
T_{\mathrm{coll,res}}^N>T_0,$ and $
\widetilde T\ge T_0.
$
Consequently, the particle system is classical and collision-free on the
physical interval
$
[0,\eta_0^{s+1-d}T_0],
$
and the bootstrap bounds \eqref{q31p}--\eqref{q32p} hold on the rescaled
interval $[0,T_0]$.\\
Finally, applying \eqref{q37'} on $[0,T_0]$ gives
\eqref{M_bound_supercoulomb} in the rescaled time variable, hence on the
physical time interval
$
0\le \tau\le \eta_0^{s+1-d}T_0.
$\\

It remains to derive the dual estimate \eqref{W_bound_supercoulomb}. This follows by the same dual-source argument as in the proofs of \eqref{conclth1-wn-cou} and \eqref{conclth1-wn}, but using Lemma~\ref{Le-dual2}. Set
$$
\omega
:=
\widetilde\mu_{N,2^{\lfloor\log_2(\eta_*^{-1/2})\rfloor}\eta_*}.
$$
In the rescaled time variable, the equation for $\omega$ has a source term
$\mathbf R$ consisting of the dual remainders, including the contribution of
$\R^9$, together with the quadratic nonlinear term
$$
-\operatorname{div}\bigl((\F^{\eta_*}*\omega)\omega\bigr).
$$
By \eqref{R2-8-dual}, \eqref{R1-dual}, \eqref{R9-dual}, the remaining dual
estimates, and \eqref{M_bound_supercoulomb}, this source satisfies
$$
\|\langle\cdot\rangle^{\gamma-1}\mathbf R\|_{L^1([0,T_0\eta_*^{s-d+1}],(W^{1,1})^*)}
\lesssim
\eta_*
\exp\bigl[
C(1+\A'_{T})^2(1+T)
\bigr]
$$
on the rescaled time interval. Applying Lemma~\ref{Le-dual2} with
$\gamma_1=\gamma-1$ gives \eqref{W_bound_supercoulomb}. This completes the
proof.
\section{Finite-time collision results}\label{sec:blowupf}
Consider the first-order particle system
\begin{equation}\label{1stode}
	\left\{
	\begin{aligned}
		\dot{x}_i &= \frac1N\sum_{j\neq i}\F(x_i,x_j),\\
		x_i(0) &= x_i^0.
	\end{aligned}
	\right.
\end{equation}
We recall the attraction assumption on the singular part of the kernel: there exist
constants $c_0\geq 2$ and $\varepsilon_0>0$ such that
\begin{align}\label{att-recall}
	c_0^{-1}|x-y|^{-s+1}
	\leq
	-\F^{\mathrm{sing}}(x,y)\cdot (x-y)
	\leq
	c_0 |x-y|^{-s+1}
\end{align}
for all $0<|x-y|\leq \varepsilon_0$.
The proof proceeds by constructing an initial configuration with a distinguished pair of particles that collide in finite time, while the remaining particles remain well-separated and controlled by the mean-field estimates of the preceding sections.

\begin{proof}[Proof of Theorems~\ref{cou} and~\ref{supercou}]
\textbf{Step 1: Choice of initial configuration.}
Let
\begin{equation}\label{ini1}
\mu_0(x)=\Phi^{c'^2}(x),
\qquad c'\in[2,4].
\end{equation}
and construct the initial particle positions as follows.
\begin{itemize}[leftmargin=2em]
\item Fix $X_N^0=(x_1^0,x_2^0,\dots,x_N^0)$ with $x_2^0=0$.
\item Particles $x_2^0,\dots,x_N^0$ lie on a uniform grid:
\begin{equation}\label{2Nass1}
  (x_{j}^0)_{j=2,\dots,N}\subset (3 \eta_*)\mathbb{Z}^d,
\end{equation}
such that
$\mu_{[2,N],\eta}^0(x):= \frac{1}{N-1}\sum_{j=2}^{N}\Phi^\eta_{x_{j}^0}(x)$
satisfies
\begin{equation}\label{2Nass2}
  \sup_{x\in \mathbb{R}^d}\langle x\rangle^{3d}\,
  \bigl|\mu_{[2,N],\eta}^0(x)-\mu_0(x)\bigr|
  \leq \frac{C\eta_*}{\eta}
\end{equation}
for any $\eta\in [N^{-1/d},\,N^{-1/(3d)}]$.
\item The distinguished particle $x_1^0$ is placed at a small distance from $x_2^0$:
\begin{equation}\label{x1-init}
  x_1^0=\varepsilon\, \eta_*\,e_1=(\varepsilon\, \eta_*,0,\dots,0),
\end{equation}
with $0<\varepsilon\ll 1$.
\end{itemize}
We shall later apply the mean-field estimates to the background subsystem
$X_{[3,N]}^0=(x_3^0,\dots,x_N^0)$. We therefore record that the same initial approximation estimate remains valid
after removing the single particle $x_2^0$ from the grid configuration. For
consistency with the normalization of the particle system, we define
$$
\mu_{[3,N],\eta}^0
:=
\frac{1}{N-2}\sum_{j=3}^N \Phi_{x_j^0}^{\eta}.
$$
Then, after increasing the constant $C$ if necessary, we have
$$
\sup_{x\in\mathbb R^d}
\langle x\rangle^{3d}
\bigl|
\mu_{[3,N],\eta}^0(x)-\mu_0(x)
\bigr|
\le
C\frac{\eta_*}{\eta},
\qquad
\eta\in [N^{-1/d},N^{-1/(3d)}].
$$
A direct computation gives
\begin{equation}\label{H-init-blowup}
  \frac{1}{N}\sup_{i\in [1,N]}\sum_{\substack{j=1\\j\neq i}}^N
  \frac{\chi\bigl(\frac{|x_i-x_j|}{2\eta_*}\bigr)}{|x_i^0-x_j^0|^{d}}
  =\frac{1}{N}\,\frac{1}{|x_{12}^0|^d}
  =\frac{1}{\varepsilon^d}.
\end{equation}
	Let $\mu^t$ denote the corresponding limiting density on
$[0,\eta_*^{s+1-d}T^0]$, namely:
\begin{itemize}[leftmargin=2em]
	\item the solution to the limiting equation \eqref{mflimit}, with initial datum
	$\mu^0=\mu_0$, when $s=d-1$;
	\item the solution to the regularized limiting equation \eqref{mflimitsup}
	when $d-1<s<d$;
	\item the stationary density $\mu^t\equiv\mu_0$ when $s\geq d$.
\end{itemize}
In all cases we have
\begin{equation}\label{A-mu-bound}
	\sup_{t\in[0,\eta_*^{s+1-d}T^0]}\A(\mu^t)\lesssim 1.
\end{equation}

\medskip
\noindent\textbf{Step 2: Decomposition of the dynamics.}
We decompose the system into the interacting pair $(x_1,x_2)$ and the background particles $(x_3,\dots,x_N)$ as follows
\begin{equation}\label{pair-system}
  \left\{\begin{aligned}
    \dot{x}_1 &= \frac1N \F(x_{1},x_2) +\frac1N\sum_{j=3}^N \F(x_{1},x_j), \\
    \dot{x}_2 &= \frac1N \F(x_{2},x_1) + \frac1N\sum_{j=3}^N  \F(x_{2},x_j), \\
    \dot{x}_i &= \frac1N\sum_{\substack{j=3\\j\neq i}}^N  \F(x_{i},x_j)
      + V_N^1(x_i)+ V_N^2(x_i), \qquad i=3,\dots,N,
  \end{aligned}\right.
\end{equation}
where the perturbation generated by the pair is given by
\begin{align*}
	V_N^1(x)
	&:=
	\frac1N
	\bigl(
	\F^{\mathrm{reg}}(x,x_1)
	+
	\F^{\mathrm{reg}}(x,x_2)
	\bigr),
	\\
	V_N^2(x)
	&:=
	\frac1N
	\bigl(
	\F^{\mathrm{sing}}(x,x_1)
	+
	\F^{\mathrm{sing}}(x,x_2)
	\bigr),
\end{align*}
for any $x\not\in \{x_1,x_2\}$.\\
Throughout the proof, for $1\le N_1\le N_2\le N$, we write
$$
X_{[N_1,N_2]}^t:=(x_{N_1}^t,\dots,x_{N_2}^t),
\qquad
x_{ij}^t:=x_i^t-x_j^t.
$$
\medskip
\noindent\textbf{Case 1: the borderline case $s=d-1$.}
Assume the bootstrap conditions: for $t\in [0,T]$ with $T\leq T^0$,
\begin{align}
  &0<|x_{12}^t|\leq 2\varepsilon\, \eta_*, \label{a3} \\
  &|x_{1i}^t|,\;|x_{2i}^t|\geq 2\eta_*, \qquad i=3,\dots,N. \label{a4}
\end{align}
The bootstrap \eqref{a4} and \eqref{assumpFre}, \eqref{assumpFsing} ensures that
\begin{align}
\label{tail1}
&\langle x\rangle^{-1}|V_N^1(x)|+\langle x-y\rangle^{-1}|V_N^1(x)-V_N^1(y)|\leq 10CN^{-1},\\&
|V_N^2(x_i)|\leq 10C\eta_{*},~~i=3,...,N.\label{tail1'}
\end{align}
	Therefore, thanks to Remarks~\ref{re-for}
and~\ref{re-for2}, Theorem~\ref{th2} applies to the subsystem $(x_3,\ldots,x_N)$ with external
force $V_N^1+V_N^2$. 

Consequently, there exists $T_0\in(0,T]$, independent of
$\varepsilon$ and $N$, such that for all $t\in[0,T_0]$ and all
$\eta\in[\eta_*,\eta_*^{1/3}]$,
\begin{equation}\label{q38}
	\M(X_{[3,N]}^t,\mu^t,\eta)
	\leq
	C'\frac{\eta_*}{\eta}
	\exp\bigl(C\sqrt{t\log(\eta/\eta_*)}\bigr),
\end{equation}
and
\begin{equation}\label{q40}
	\inf_{\substack{i\neq j\\ i,j\geq 3}}
	|x_i^t-x_j^t|
	\geq
	\frac12\eta_*.
\end{equation}
In particular, 
\begin{equation*}
	\inf_{i\neq j}|x_i(t)-x_j(t)|
	=
	|x_{12}(t)|,
	\qquad
	0\leq t<T_0.
\end{equation*}
By Lemma~\ref{lem:background-force-coulomb}, the estimates
\eqref{a10}--\eqref{a9'} hold on $[0,T_0]$.
We now estimate the relative motion of the distinguished pair. By \eqref{a9},
\begin{align*}
	\biggl|
	\frac{d}{dt}x_{12}
	-
	\frac1N
	\bigl(
	\F^{\mathrm{sing}}(x_1,x_2)
	-
	\F^{\mathrm{sing}}(x_2,x_1)
	\bigr)
	\biggr|
	&\lesssim
	|x_{12}|
	+
	\frac1N
	\bigl|
	\F^{\mathrm{reg}}(x_1,x_2)
	-
	\F^{\mathrm{reg}}(x_2,x_1)
	\bigr|
	\\
	&\lesssim
	|x_{12}|
	+
	\frac1N
	|x_{12}|
	\sup_{x,y}|(\nabla_x,\nabla_y)\F^{\mathrm{reg}}(x,y)|.
\end{align*} 
Hence, by \eqref{assumpFre},
\begin{equation}\label{z31}
	\biggl|
	\frac{d}{dt}x_{12}
	-
	\frac1N
	\bigl(
	\F^{\mathrm{sing}}(x_1,x_2)
	-
	\F^{\mathrm{sing}}(x_2,x_1)
	\bigr)
	\biggr|
	\lesssim |x_{12}|.
\end{equation}
Multiplying by $x_{12}|x_{12}|^{d-2}$ gives
\begin{align}\label{z30p}
	\biggl|
	\frac{d}{dt}|x_{12}|^d
	-
	\frac{d}{N}
	|x_{12}|^{d-2}
	\bigl(
	\F^{\mathrm{sing}}(x_1,x_2)\cdot x_{12}
	+
	\F^{\mathrm{sing}}(x_2,x_1)\cdot x_{21}
	\bigr)
	\biggr|
	\lesssim |x_{12}|^d.
\end{align}
Using the attraction assumption \eqref{att-recall} and the bootstrap bound
\eqref{a3}, we infer
\begin{align}\label{z30pp}
	-C N|x_{12}|^d+2dc_0^{-1}
	\leq
	-\frac{d}{dt}\bigl(N|x_{12}|^d\bigr)
	\leq
	C N|x_{12}|^d+2dc_0.
\end{align}
Equivalently,
\begin{align*}
	e^{-Ct}N|x_{12}(0)|^d
	-\frac{2dc_0}{C}(1-e^{-Ct})
	\leq
	N|x_{12}(t)|^d
	\leq
	e^{Ct}N|x_{12}(0)|^d
	-\frac{2dc_0^{-1}}{C}(e^{Ct}-1).
\end{align*}
Since $N|x_{12}(0)|^d=\varepsilon^d$, choosing
$0<\varepsilon
\leq
\frac{1}{100d(C+c_0)^{10}}$
ensures the existence of a collision time $T_*$ satisfying
\begin{equation}\label{Tstar-coulomb}
	c_0^{-1}\frac{\varepsilon^d}{4d}
	\leq
	T_*
	\leq
	c_0\frac{\varepsilon^d}{d},
\end{equation}
such that $|x_{12}(t)|>0$ for $t<T_*$ and $|x_{12}(T_*)|=0$.
Moreover, on $[0,\min\{T_0,T_*\}]$,
\begin{align}
	dc_0^{-1}
	&\leq
	-\frac{d}{dt}\bigl(N|x_{12}|^d\bigr)
	\leq
	4dc_0,
	\label{z30a}
	\\
	|x_{12}(t)|
	&\leq
	\sqrt{2}\varepsilon\eta_*.
	\label{z30ap}
\end{align}
We also record that the colliding pair remains in a fixed bounded region up
to the collision time. Indeed, by \eqref{pair-system}, \eqref{a9}, and
\eqref{z30ap},
$$
|\dot x_1(t)|+|\dot x_2(t)|
\lesssim
1+|x_1(t)|+|x_2(t)|+\frac1{N|x_{12}(t)|^{d-1}} .
$$
Using \eqref{z30a}, we obtain
$$
\int_0^{T_*}\frac{dt}{N|x_{12}(t)|^{d-1}}
\leq c_0
|x_{12}(0)|
\leq  c_0
\varepsilon\eta_* .
$$
By Grönwall's inequality, and using the boundedness of the initial positions,
we deduce
$$
\sup_{0\le t<T_*}\bigl(|x_1(t)|+|x_2(t)|\bigr)\le C,
$$
after increasing $C$ if necessary.\\
	It remains to verify that the colliding pair stays separated from the
background particles. From \eqref{a9}, for $k=1,2$ and $i=3,\ldots,N$, we have 
\begin{align*}
	| \dot{x}_{ik}|&\leq  \frac{1}{N}\left|\sum_{\substack{j\not= i}}^N  \F^{\mathrm{reg}}(x_{i},x_j)-\sum_{\substack{j\not= k}}^N  \F^{\mathrm{reg}}(x_{k},x_j)\right|+\frac{1}{N}(|\F^{\mathrm{sing}}(x_{1},x_2)|+|\F^{\mathrm{sing}}(x_{2},x_1)|)\\&+ |V_N^2(x_i)|+ \left|\frac1N\sum_{j=3, ~j\neq i}^N  \F^{\mathrm{sing}}(x_{k},x_j)-\frac1N\sum_{\substack{j=3\\j\neq i}}^N  \F(x_{i},x_j)\right|+\frac1N| \F^{\mathrm{sing}}(x_{k},x_i)|.
\end{align*}
By \eqref{assumpFre}, \eqref{assumpFsing}, \eqref{tail1'} and \eqref{a9'}, we get 
\begin{equation}\label{a16}
	|\dot{x}_{ki}|
	\lesssim
	\frac1{N|x_{12}|^{d-1}}
	+
	|x_{ki}|+\eta_*.
\end{equation}
By Grönwall's inequality, using \eqref{z30a} and
$|x_{12}(0)|=\varepsilon\eta_*$, we obtain, for
$t\leq\min\{T_*,T_0\}$,
\begin{align*}
	|x_{ki}(t)|
	&\geq
	e^{-CT_*}|x_{ki}(0)|
	-
	CT_*\eta_*
	-
	C\int_0^{T_*}\frac{d\tau}{N|x_{12}(\tau)|^{d-1}}
	\\
	&\geq
	e^{-CT_*}|x_{ki}(0)|
	-
	CT_*\eta_*
	+
Cc_0
	\int_0^{T_*}\partial_\tau |x_{12}(\tau)|\,d\tau
	\\
	&\geq
	\bigl((3-\ep)e^{-CT_*}-CT_*-Cc_0\varepsilon\bigr)\eta_*.
\end{align*}
Thus, after decreasing $\varepsilon>0$ if necessary,
\begin{equation}\label{a2}
	|x_{ki}^t|
	\geq
	\frac83\eta_*,
	\qquad
	t\leq\min\{T_0,T_*\},
	\quad k=1,2,\quad i=3,\ldots,N.
\end{equation}
Together with \eqref{z30ap}, estimate \eqref{a2} strictly improves the
bootstrap assumptions \eqref{a3}--\eqref{a4}. Since $T_0$ is independent of
$\varepsilon$ and $T_*\lesssim \varepsilon^d$, we choose $\varepsilon>0$
small enough so that $T_*<T_0$. Thus the collision occurs before the
background estimates cease to be valid, and the bootstrap remains valid up to
the collision time $T_*$. Moreover, 
$$
|x_{12}(t)|>0 \quad \text{for } t<T_*,
\qquad
|x_{12}(T_*)|=0,
$$
and  integrating
\eqref{z30a} from $t$ to $T^*$ and  using
$N=\eta_*^{-d}$, we obtain
\begin{equation}
	|x_{12}(t)|\sim
	\eta_*(T^*-t)^{1/d}
	\sim
	\varepsilon\eta_*
	\left(
	1-\frac{t}{T^*}
	\right)^{1/d}.
\end{equation}
This proves Theorem~\ref{cou}.
\medskip

\noindent\textbf{Case 2: the super-Coulomb case $s>d-1$.} 
The argument is parallel but with modified scaling.
Assume the bootstrap condition: for $t\in [0,T]$ with $T\leq \eta_*^{s+1-d}$,
\begin{align}
  &0<|x_{12}^t|\leq 2\varepsilon\, \eta_*, \label{a5} \\
  &|x_{1i}^t|,\;|x_{2i}^t|\geq 2\eta_*, \qquad i=3,\dots,N. \label{a6}
\end{align}
The singular tail generated by the pair satisfies
\begin{equation}\label{tail2}
	\sup_{i=3,\dots,N}|V_N^2(x_i)|
	\lesssim
	\eta_*^{d-s}
	=
	\eta_*^{d-s-1}\eta_* .
\end{equation}
Applying Theorem~\ref{th3'} to $X_{[3,N]}=(x_3,\dots,x_N)$ (in view of   Remarks \ref{re-for}, \ref{re-for2} ) there exists $T_0\in(0,T^0]$, independent of $\varepsilon$ and $N$, such that for $t\in[0,T_0]$ and $\eta\in[\eta_*,\eta_*^{1/3}]$:
\begin{equation}\label{q41}
	\M\bigl(
	X_{[3,N]}^{\eta_*^{s+1-d}t},
	\mu^{\eta_*^{s+1-d}t},
	\eta
	\bigr)
	\le
	C\frac{\eta_*}{\eta}e^{Ct},
\end{equation}
and
\begin{equation}\label{q40p}
  \inf_{\substack{i \neq j\\i,j\geq 3}}|x_i^{\eta_*^{s+1-d}t}-x_j^{\eta_*^{s+1-d}t}|
  \geq \tfrac{1}{2}\,\eta_*.
\end{equation}
In particular, 
\begin{equation*}
	\inf_{i\neq j}|x_i(\eta_*^{s+1-d}t)-x_j(\eta_*^{s+1-d}t)|
	=
	|x_{12}(\eta_*^{s+1-d}t)|,
	\qquad
	0\leq t<T_0.
\end{equation*}
By Lemma~\ref{lem:background-force-supercoulomb}, the estimates
\eqref{a13}--\eqref{a15} hold on the rescaled interval $[0,T_0]$.
The relative dynamics of the pair is therefore governed by
\begin{equation}\label{z31pp}
	\biggl|
	\frac{d}{d\tau}x_{12}(\tau)
	-
	\frac1N
	\bigl(
	\F^{\mathrm{sing}}(x_1,x_2)
	-
	\F^{\mathrm{sing}}(x_2,x_1)
	\bigr)
	\biggr|
	\lesssim
	\eta_*^{d-s-1}|x_{12}(\tau)|.
\end{equation}
Using \eqref{att-recall} and the bootstrap bound \eqref{a5}, we obtain
\begin{align}\label{z30cpp}
	-C\eta_*^{d-s-1}N|x_{12}|^{s+1}
	+2(s+1)c_0^{-1}
	\le
	-\frac{d}{d\tau}\bigl(N|x_{12}|^{s+1}\bigr)
	\le
	C\eta_*^{d-s-1}N|x_{12}|^{s+1}
	+2(s+1)c_0 .
\end{align}
Introduce the rescaled time variable $\tau=\eta_*^{s+1-d}t$ and set
$$
Y(t):=
N\eta_*^{d-s-1}
|x_{12}(\eta_*^{s+1-d}t)|^{s+1}.
$$
Then $Y(0)=\varepsilon^{s+1}$, and \eqref{z30cpp} gives
$$
-CY(t)+2(s+1)c_0^{-1}
\le
-Y'(t)
\le
CY(t)+2(s+1)c_0 .
$$
Therefore, after choosing $\varepsilon>0$ sufficiently small, there exists
$T_*$ satisfying
\begin{equation}\label{Tstar-supercoulomb}
	c_0^{-1}\frac{\varepsilon^{s+1}}{4(s+1)}
	\le
	T_*
	\le
	c_0\frac{\varepsilon^{s+1}}{s+1},
\end{equation}
such that
$$
|x_{12}(\eta_*^{s+1-d}t)|>0
\quad\text{for }t<T_*,
\qquad
|x_{12}(\eta_*^{s+1-d}T_*)|=0.
$$
Since $T_0$ is independent of $\varepsilon$ and $N$, we decrease
$\varepsilon$ further, if necessary, so that $T_*<T_0$.
Moreover, for $0\le t\le T_*$,
\begin{align}
	(s+1)c_0^{-1}
	&\le
	-Y'(t)
	\le
	4(s+1)c_0,
	\label{z30c}
	\\
	|x_{12}(\eta_*^{s+1-d}t)|
	&\le
	\sqrt{2}\varepsilon\eta_* .
	\label{z30cp}
\end{align}
Similarly, in the super-Coulomb case the pair remains uniformly bounded up to
the collision time. Indeed, using \eqref{z30c},
$$
\int_0^{\eta_*^{s+1-d}T_*}
\frac{d\tau}{N|x_{12}(\tau)|^s}
\leq  c_0
|x_{12}(0)|
\leq c_0
\varepsilon\eta_* .
$$
Moreover, by the equation for the pair, the background estimates, and the
bootstrap separation, we have, for
$0\le t\le \eta_*^{s+1-d}T_*$,
$$
|\dot x_1(t)|+|\dot x_2(t)|
\lesssim 1+
\eta_*^{-s+d}+|x_1(t)|+|x_2(t)|+\frac1{N|x_{12}(t)|^{s}} .
$$
Since
$
\eta_*^{d-s}\,\eta_*^{s+1-d}T_*
\lesssim
\eta_*T_*
\lesssim
\eta_*,
$
Grönwall's inequality gives
$$
\sup_{0\le t<T_*}
\bigl(
|x_1(\eta_*^{s+1-d}t)|
+
|x_2(\eta_*^{s+1-d}t)|
\bigr)
\le C.
$$
Finally, we control the separation between the collapsing pair and the
background. As \eqref{a16}, we can obtain from \eqref{a13}, \eqref{a14} and \eqref{a15} that 
\[
|\dot{x}_{ki}|
\lesssim
\frac1{N|x_{12}|^s}
+
\eta_*^{d-1-s}\bigl(|x_{ki}|+\eta_*\bigr),
\qquad k=1,2,\quad i=3,\ldots,N.
\]
Applying Grönwall's inequality on the rescaled interval and using
\eqref{z30c}, we find
\begin{align*}
	|x_{ki}(\eta_*^{s+1-d}t)|
	&\geq
	e^{-CT_*}|x_{ki}(0)|
	-
	CT_*\eta_*
	-
	C\int_0^{\eta_*^{s+1-d}T_*}
	\frac{d\tau}{N|x_{12}(\tau)|^s}
	\\
	&\geq
	e^{-CT_*}|x_{ki}(0)|
	-
	CT_*\eta_*
	+
	Cc_0
	\int_0^{\eta_*^{s+1-d}T_*}
	\partial_\tau |x_{12}(\tau)|\,d\tau
	\\
	&\geq
	\bigl(
	(3-\ep)e^{-CT_*}
	-
	CT_*
	-
	Cc_0\varepsilon
	\bigr)\eta_*.
\end{align*}
Therefore, after decreasing $\varepsilon>0$ if needed,
\begin{equation}\label{a2p}
	|x_{ki}^{\eta_*^{s+1-d}t}|
	\ge
	\frac83\eta_*,
	\qquad
	0\le t\le \min\{T_0,T_*\},
	\quad k=1,2,\quad i=3,\ldots,N.
\end{equation}
Together with \eqref{z30cp}, estimate \eqref{a2p} strictly improves the
bootstrap assumptions \eqref{a5}--\eqref{a6} on the rescaled time interval.
Since $T_0>0$ is independent of $\varepsilon$ and $N$, while
$T_*\lesssim \varepsilon^{s+1}$, we decrease $\varepsilon>0$ further, if
necessary, so that $T_*<T_0$. Thus the bootstrap remains valid up to the
collision time. 
Moreover, 
$$
|x_{12}(\eta_*^{s+1-d}t)|>0
\quad \text{for } t<T_*,
\qquad
|x_{12}(\eta_*^{s+1-d}T_*)|=0,
$$
and  integrating
\eqref{z30c} from $t$ to $T^*$ and  using
$N=\eta_*^{-d}$, we obtain
\begin{equation}
\left|
x_{12}\bigl(\eta_*^{s+1-d}t\bigr)
\right|\sim
\eta_*(T^*-t)^{1/(s+1)}\sim
\varepsilon\eta_*
\left(
1-\frac{t}{T^*}
\right)^{1/(s+1)}.
\end{equation}
This proves Theorem~\ref{supercou}.
\end{proof}
\begin{lem}[Background force estimates in the Coulomb case]\label{lem:background-force-coulomb}
	Assume $s=d-1$. Suppose that the bootstrap assumptions \eqref{a3}--\eqref{a4}
	hold on $[0,T_0]$, and that the background subsystem $X_{[3,N]}$ satisfies
	\eqref{q38}--\eqref{q40}. Then, for all $t\in[0,T_0]$, the following estimates hold.
	
	First, 
	\begin{align}
		\frac1N
		\Bigl|
		\sum_{\substack{j=3\\j\neq i}}^N
		\nabla_{x_i}\F^{\mathrm{sing}}(x_i,x_j)
		\Bigr|
		&\lesssim 1,	\qquad i=3,\dots,N,
		\label{a10}
		\\
		\frac1N
		\sum_{\substack{j=3\\j\neq i}}^N
		\frac{1}{|x_{ij}|^{d+1}}
		&\lesssim \eta_*^{-1},	\qquad i=1,\dots,N,
		\label{a10'}
	\end{align}
	Finally,
	\begin{equation}\label{a9}
		\left|
		\frac1N\sum_{j=3}^N \F(x_1,x_j)
		-
		\frac1N\sum_{j=3}^N \F(x_2,x_j)
		\right|
		\lesssim |x_{12}|,
	\end{equation}
	and, for $k=1,2$ and $i=3,\dots,N$,
	\begin{align}\label{a9'}
		\left|
		\frac1N
		\sum_{\substack{j=3\\j\neq i}}^N
		\F^{\mathrm{sing}}(x_k,x_j)
		-
		\frac1N
		\sum_{\substack{j=3\\j\neq i}}^N
		\F^{\mathrm{sing}}(x_i,x_j)
		\right|
		\lesssim |x_{ik}|.
	\end{align}
\end{lem}
\begin{proof}
	We first prove \eqref{a10'}. By \eqref{convolutiondom}, for every $p>d$,
	\begin{equation}\label{a100}
			\frac{1}{(\eta_*+|z|)^p}
		\sim
		\int
		\frac{1}{(\eta_*+|z-y|)^p}
		\Phi^{\eta_*}(y)\,dy .
	\end{equation}
	Hence, using \eqref{q38}, we obtain
	\begin{align}\label{a12}
		\frac1N
		\sum_{j=3}^N
		\frac{1}{(\eta_*+|x-x_j|)^p}
		&\lesssim
		\int_{\mathbb R^d}
		\frac{1}{(\eta_*+|x-y|)^p}
		\mu_{[3,N],\eta_*}^t(y)\,dy
		\lesssim
		\eta_*^{d-p}.
	\end{align}
	Taking $p=d+1$, and using the separation estimates \eqref{a4} and \eqref{q40},
	gives \eqref{a10'}.
	
	We next prove \eqref{a10}. Decompose
	$
	\F^{\mathrm{sing}}
	=
	\F^{\mathrm{sing}}_{2\eta_*}
	+
	\F^{\mathrm{sing}}_{<2\eta_*}.
	$
	For $i=3,\dots,N$, we write
	\begin{align*}
		\frac1N
		\Bigl|
		\sum_{\substack{j=3\\j\neq i}}^N
		\nabla_{x_i}\F^{\mathrm{sing}}(x_i,x_j)
		\Bigr|&\le
		\frac1N
		\Bigl|
		\sum_{j=3}^N
		\nabla_{x_i}\F^{\mathrm{sing}}_{2\eta_*}(x_i,x_j)
		\Bigr|
		+
		\frac1N
		\Bigl|
		\nabla_x\F^{\mathrm{sing}}_{2\eta_*}(x_i,x_i)
		\Bigr|
		\\
		&\qquad\quad
		+
		\frac1N
		\Bigl|
		\sum_{\substack{j=3\\j\neq i}}^N
		\nabla_{x_i}\F^{\mathrm{sing}}_{<2\eta_*}(x_i,x_j)
		\Bigr|.
	\end{align*}
	The first term is controlled by
	$$
	\frac1N
	\Bigl|
	\sum_{j=3}^N
	\nabla_{x_i}\F^{\mathrm{sing}}_{2\eta_*}(x_i,x_j)
	\Bigr|
	\lesssim
	\bigl\|
	\nabla_x\F^{\mathrm{sing}}_{\sqrt3\eta_*}
	*
	\mu_{[3,N],\eta_*}^t
	\bigr\|_{L^\infty}.
	$$
	As in the proof of \eqref{z11}, we get
	\begin{align*}
		\bigl\|
		\nabla_x\F^{\mathrm{sing}}_{\sqrt3\eta_*}
		*
		\mu_{[3,N],\eta_*}^t
		\bigr\|_{L^\infty}
		&\lesssim
		\sum_{\ell=0}^{n_0-1}
		\M(X_{[3,N]}^t,\mu^t,2^\ell\eta_*)
		+
		\bigl|\log(2^{n_0}\eta_*)\bigr|
		\M(X_{[3,N]}^t,\mu^t,2^{n_0}\eta_*)
		\\
		&\quad
		+
		\bigl|\log(2^{n_0}\eta_*)\bigr|
		2^{2n_0}\eta_*^2\A(\mu^t).
	\end{align*}
	Choosing
	$
	n_0:=
	\bigl\lfloor \log_2(\eta_*^{-1/10})\bigr\rfloor
	$
	and using \eqref{q38}, we obtain
	\begin{equation}\label{a11}
		\bigl\|
		\nabla_x\F^{\mathrm{sing}}_{\sqrt3\eta_*}
		*
		\mu_{[3,N],\eta_*}^t
		\bigr\|_{L^\infty}
		\lesssim 1 .
	\end{equation}
	The self-interaction term satisfies, by \eqref{assumpF2'},
	$$
	\frac1N
	\sup_x
	\bigl|
	\nabla_x\F^{\mathrm{sing}}_{2\eta_*}(x,x)
	\bigr|
	\lesssim
	\frac{1}{N\eta_*^{d-1}}
	\lesssim \eta_* .
	$$
	Finally, by \eqref{proF1-gradient-coulomb} with $\eta=2\eta_*$,
	\begin{align*}
		\frac1N
		\Bigl|
		\sum_{\substack{j=3\\j\neq i}}^N
		\nabla_{x_i}\F^{\mathrm{sing}}_{<2\eta_*}(x_i,x_j)
		\Bigr|
		&\lesssim
		\eta_*^2
		+
		\frac1N
		\sum_{\substack{j=3\\j\neq i}}^N
		\frac{\eta_*^2}{|x_{ij}|^{d+2}} .
	\end{align*}
	Taking $p=d+2$ in \eqref{a12} gives
	$
	\frac1N
	\sum_{\substack{j=3\\j\neq i}}^N
	\frac{\eta_*^2}{|x_{ij}|^{d+2}}
	\lesssim 1 .
	$
	Thus \eqref{a10} follows.
	
	We now prove \eqref{a9}. The regular part is Lipschitz by \eqref{assumpFre}, hence
	$$
	\left|
	\frac1N\sum_{j=3}^N
	\F^{\mathrm{reg}}(x_1,x_j)
	-
	\frac1N\sum_{j=3}^N
	\F^{\mathrm{reg}}(x_2,x_j)
	\right|
	\lesssim |x_{12}|.
	$$
	For the singular part, we again write
	$\F^{\mathrm{sing}}=\F^{\mathrm{sing}}_{2\eta_*}
	+\F^{\mathrm{sing}}_{<2\eta_*}$. By \eqref{a11},
	$$
	\left|
	\frac1N\sum_{j=3}^N
	\F^{\mathrm{sing}}_{2\eta_*}(x_1,x_j)
	-
	\frac1N\sum_{j=3}^N
	\F^{\mathrm{sing}}_{2\eta_*}(x_2,x_j)
	\right|
	\lesssim |x_{12}|.
	$$
	For the truncated part, \eqref{proF1-gradient-coulomb}, \eqref{a12} with
	$p=d+2$, and the separation \eqref{a4} imply
	\begin{align*}
		&\left|
		\frac1N\sum_{j=3}^N
		\F^{\mathrm{sing}}_{<2\eta_*}(x_1,x_j)
		-
		\frac1N\sum_{j=3}^N
		\F^{\mathrm{sing}}_{<2\eta_*}(x_2,x_j)
		\right|
\lesssim
		|x_{12}|
		\left(
		\eta_*^2
		+
		\frac1N
	\sum_{k=1,2}	\sum_{j=3}^N
		\frac{\eta_*^2}{|x_{kj}|^{d+2}}
		\right)
		\lesssim |x_{12}|.
	\end{align*}
	This proves \eqref{a9}.\\
	It remains to prove \eqref{a9'}. Fix $k=1,2$ and $i=3,\dots,N$. Decomposing again into the smoothed and truncated parts gives
	\begin{align*}
		\left|
		\frac1N
		\sum_{\substack{j=3\\j\neq i}}^N
		\F^{\mathrm{sing}}(x_k,x_j)
		-
		\frac1N
		\sum_{\substack{j=3\\j\neq i}}^N
		\F^{\mathrm{sing}}(x_i,x_j)
		\right|
	&\le
		\left|
		\frac1N
		\sum_{\substack{j=3\\j\neq i}}^N
		\F^{\mathrm{sing}}_{2\eta_*}(x_k,x_j)
		-
		\frac1N
		\sum_{\substack{j=3\\j\neq i}}^N
		\F^{\mathrm{sing}}_{2\eta_*}(x_i,x_j)
		\right|
		\\
		&\quad
		+
		\frac1N
		\sum_{\substack{j=3\\j\neq i}}^N
		\left|
		\F^{\mathrm{sing}}_{<2\eta_*}(x_k,x_j)
		\right|
		+
		\frac1N
		\sum_{\substack{j=3\\j\neq i}}^N
		\left|
		\F^{\mathrm{sing}}_{<2\eta_*}(x_i,x_j)
		\right| .
	\end{align*}
	The smoothed term is bounded by \eqref{a11}:
	$$
	|x_{ik}|
	\bigl\|
	\nabla_x\F^{\mathrm{sing}}_{\sqrt3\eta_*}
	*
	\mu_{[3,N],\eta_*}^t
	\bigr\|_{L^\infty}+
	\frac1N
	\left(
	\left|
	\F_{2\eta_*}^{\mathrm{sing}}(x_k,x_i)
	\right|
	+
	\left|
	\F_{2\eta_*}^{\mathrm{sing}}(x_i,x_i)
	\right|
	\right)
	\lesssim |x_{ik}|.
	$$
	For the truncated terms, \eqref{proF1} gives
	\begin{align*}
		&\frac1N
		\sum_{\ell\in\{i,k\}}
		\sum_{\substack{j=3\\j\neq i}}^N
		\left|
		\F^{\mathrm{sing}}_{<2\eta_*}(x_\ell,x_j)
		\right|\lesssim
		\eta_*^2
		+
		\frac1N
		\sum_{\ell\in\{i,k\}}
		\sum_{\substack{j=3\\j\neq i}}^N
		\frac{1}{|x_{\ell j}|^{d-1}}
		\min\!\Bigl(1,\frac{\eta_*}{|x_{\ell j}|}\Bigr)^2 .
	\end{align*}
	Using \eqref{q38} and the convolution estimate above, this last term is
	bounded by $C\eta_*$. Since $i=3,\dots,N$ and $k=1,2$, the bootstrap separation
	\eqref{a4} implies $\eta_*\lesssim |x_{ik}|$. Therefore the truncated
	contribution is bounded by $C|x_{ik}|$, and \eqref{a9'} follows.
\end{proof}

\begin{lem}[Background force estimates in the super-Coulomb case]
	\label{lem:background-force-supercoulomb}
	Assume $s>d-1$. Suppose that the bootstrap assumptions \eqref{a5}--\eqref{a6}
	hold on the physical time interval $[0,\eta_*^{s+1-d}T_0]$, and that the
	background subsystem satisfies \eqref{q41}--\eqref{q40p}. Then, for every
$t\in[0,\eta_*^{s+1-d}T_0]$,  the following
	estimates hold.\\
	First, for every $i=1,\dots,N$,
	\begin{align}\label{a13}
		\frac1N\sum_{\substack{j=3\\j\neq i}}^N \frac{1}{|x_{ij}|^{s+1}}\lesssim\eta_{*}^{d-s-1}.
	\end{align}
	Consequently,
	\begin{equation}\label{a14}
		\biggl|\frac1N\sum_{j=3}^N \F(x_{1},x_j)-\frac1N\sum_{j=3}^N  \F(x_{2},x_j)\biggr|
		\lesssim  \eta_{*}^{d-s-1}|x_{12}|, 
	\end{equation}
	and 
	\begin{align}\label{a15}
		\left|\frac1N\sum_{\substack{j=3\\j\neq i}}^N  \F^{\mathrm{sing}}(x_{k},x_j)-\frac1N\sum_{\substack{j=3\\j\neq i}}^N  \F^{\mathrm{sing}}(x_{i},x_j)\right| \lesssim \eta_{*}^{d-s-1}|x_{ik}|
	\end{align}
	for $k=1,2$, $i=3,..,N$.
\end{lem}
\begin{proof}
	We first prove \eqref{a13}. Taking $p=s+1>d$ and using \eqref{q41}, \eqref{a100}, we obtain, uniformly in
	$x\in\mathbb R^d$ and $t\in[0,\eta_*^{s+1-d}T_0]$,
	\begin{align}\label{super-convolution-bound-phys}
		\frac1N
		\sum_{j=3}^N
		\frac{1}{(\eta_*+|x-x_j^t|)^{s+1}}
		\lesssim
		\int_{\mathbb R^d}
		\frac{1}{(\eta_*+|x-y|)^{s+1}}
		\mu_{[3,N],\eta_*}^t(y)\,dy
	\lesssim
		\eta_*^{d-s-1}.
	\end{align}
	Using the separation estimates \eqref{a6} and \eqref{q40p}, we may
	replace $\eta_*+|x_\ell^t-x_j^t|$ by $|x_\ell^t-x_j^t|$ in the relevant sums,
	for $\ell=i$ with $i=1,\dots,N$. This
	proves \eqref{a13}.
	
	We next use the elementary estimate
	\begin{align}\label{two-point-force-bound-phys}
		|\F^{\mathrm{sing}}(x,y)-\F^{\mathrm{sing}}(x',y)|
		\lesssim
		|x-x'|
		\left(
		1+\frac{1}{|x-y|^{s+1}}
		+\frac{1}{|x'-y|^{s+1}}
		\right),
	\end{align}
	which follows from \eqref{assumpFsing}--\eqref{assumpFre}. Indeed, when
	$|x-x'|$ is smaller than the relevant distances, this follows from the
	mean-value theorem and the derivative bound on $\F^{\mathrm{sing}}$. In the complementary
	case, the pointwise bound on $\F$ gives the same estimate.
	
	Applying \eqref{two-point-force-bound-phys} with $x=x_1$,
	$x'=x_2$, and $y=x_j$, we get
\begin{align*}
	\biggl|\frac1N\sum_{j=3}^N \F(x_{1},x_j)-\frac1N\sum_{j=3}^N  \F(x_{2},x_j)\biggr|&\lesssim |x_{12}|+	\biggl|\frac1N\sum_{j=3}^N \F^{\mathrm{sing}}(x_{1},x_j)-\frac1N\sum_{j=3}^N  \F^{\mathrm{sing}}(x_{2},x_j)\biggr|\\&\lesssim 
	|x_{12}|+	\frac1N\sum_{j=3}^N|x_{12}|\left(\frac{1}{|x_{1j}|^{s+1}}+\frac{1}{|x_{2j}|^{s+1}}+1\right).
\end{align*}
	By the estimate already proved, the sum is bounded by
	$C\eta_*^{d-s-1}$. Since $s>d-1$ and $\eta_*\le1$, we also have
	$\eta_*^{d-s-1}\ge1$. Therefore
	$$
	\left|
	\frac1N\sum_{j=3}^N \F(x_1,x_j)
	-
	\frac1N\sum_{j=3}^N \F(x_2,x_j)
	\right|
	\lesssim
	\eta_*^{d-s-1}|x_{12}|.
	$$
	This proves \eqref{a14}.
	
	Finally, we prove \eqref{a15}. Fix $k=1,2$ and $i=3,\dots,N$. Applying the
	same two-point estimate to the singular part only, we obtain
	\begin{align*}
		&\left|
		\frac1N
		\sum_{\substack{j=3\\j\neq i}}^N
		\F^{\mathrm{sing}}(x_k,x_j)
		-
		\frac1N
		\sum_{\substack{j=3\\j\neq i}}^N
		\F^{\mathrm{sing}}(x_i,x_j)
		\right|\lesssim
		|x_{ik}|
		\left[
		1+
		\frac1N
		\sum_{\substack{j=3\\j\neq i}}^N
		\left(
		\frac{1}{|x_{kj}|^{s+1}}
		+
		\frac{1}{|x_{ij}|^{s+1}}
		\right)
		\right].
	\end{align*}
	The sums are bounded by $C\eta_*^{d-s-1}$ by \eqref{a13}. Hence
	$$
	\left|
	\frac1N
	\sum_{\substack{j=3\\j\neq i}}^N
	\F^{\mathrm{sing}}(x_k,x_j)
	-
	\frac1N
	\sum_{\substack{j=3\\j\neq i}}^N
	\F^{\mathrm{sing}}(x_i,x_j)
	\right|
	\lesssim
	\eta_*^{d-s-1}|x_{ik}|.
	$$
	This proves \eqref{a15}.
\end{proof}

\appendix

\noindent\section{Strong and weak solutions of the ODE system}\label{app1}
\subsection{Strong solutions of the ODE system}\label{app1-Strong}

We record the elementary local well-posedness fact used in the proofs of
Theorems~\ref{th1}, \ref{th2}, and~\ref{th3'}.
\

\begin{lem}[Maximal classical collision-free solution]\label{lem:local-classical-ode} 	Assume that $\F$ satisfies \eqref{assumpFre}--\eqref{assumpF2'} and is smooth away from the diagonal.
	Assume that the initial configuration is pairwise distinct, namely
	$$
	x_i^0\neq x_j^0,
	\qquad i\neq j.
	$$
	Then the particle system \eqref{ode} admits a unique classical solution on a
	maximal collision-free interval
	$
	[0,T_{\mathrm{coll}}^N),
	$
where the time $T_{\mathrm{coll}}^N=T_{\mathrm{coll}}^N(X_N^0)$ may depend on
$N$ and on the initial configuration. Moreover, either
$T_{\mathrm{coll}}^N=\infty$, or
\begin{equation}\label{maximal-collision}
	\liminf_{t\uparrow T_{\mathrm{coll}}^N}
	\min_{i\neq j}|x_i(t)-x_j(t)|=0.
\end{equation}
Finally, on every compact subinterval
$
[0,T']\subset [0,T_{\mathrm{coll}}^N),
$
one has
$$
\inf_{0\le t\le T'}\min_{i\neq j}|x_i(t)-x_j(t)|>0 .
$$
\end{lem}

\begin{proof}
	Let
	$
	\Omega_N
	:=
	\Bigl\{
	X=(x_1,\dots,x_N)\in(\mathbb R^d)^N:
	x_i\neq x_j \text{ for all } i\neq j
	\Bigr\}.
	$
	This is an open subset of $(\mathbb R^d)^N$. For $X\in\Omega_N$, define
	$$
	B_i(X)
	:=
	\frac1N
	\sum_{\substack{j=1\\j\neq i}}^N
	\F(x_i,x_j),
	\qquad i=1,\dots,N.
	$$
	Then \eqref{ode} can be written as
	$
	\dot X=B(X).
	$
	Since $\F$ is smooth away from the diagonal, the map
	$$
	B:\Omega_N\to(\mathbb R^d)^N
	$$
	is locally Lipschitz. Indeed, if $K\Subset\Omega_N$ is compact, then there
	exists $\delta_K>0$ such that
	$$
	|x_i-x_j|\ge \delta_K,
	\qquad X\in K,\quad i\neq j.
	$$
	On the set $\{|x-y|\geq\delta_K\}$, the first derivatives of $\F$
	are bounded. Hence $B$ is Lipschitz on $K$.
	
	Since the initial configuration is collision-free, $X_N^0\in\Omega_N$.
	Therefore the Cauchy--Lipschitz theorem gives a unique classical solution as
	long as the trajectory remains in $\Omega_N$. We denote by
	$[0,T_{\mathrm{coll}}^N)$ the maximal such interval.
	
	If $T_{\mathrm{coll}}^N<\infty$ and
$
		\liminf_{t\uparrow T_{\mathrm{coll}}^N}
		\min_{i\neq j}|x_i(t)-x_j(t)|>0,$
	then the trajectory remains a positive distance away from the diagonal. By
	the at-most-linear growth of the force, the particles also remain bounded on
	$[0,T_{\mathrm{coll}}^N)$. Hence $X_N(t)$ remains in a compact subset of
	$\Omega_N$, where $B$ is Lipschitz. The Cauchy--Lipschitz theorem then extends
	the solution beyond $T_{\mathrm{coll}}^N$, contradicting maximality. Thus
	\eqref{maximal-collision} holds.
	
	The last assertion follows from the continuity of the finitely many functions
	$t\mapsto |x_i(t)-x_j(t)|$ on compact subintervals of
	$[0,T_{\mathrm{coll}}^N)$.
\end{proof}
\subsection{Weak solutions of the ODE system in the two-dimensional Coulomb case}
\label{app1-weak}

In this subsection, we introduce a weak notion of particle solution for the
one- and two-dimensional Coulomb cases appearing in
Theorem~\ref{th1-Coulomb}. The particle ODE system \eqref{ode} may lead to
collisions in finite time. Although one may sometimes continue the particle
trajectories after the first collision time, such a continuation is not
canonical in general and may depend on the chosen approximation. We therefore
formulate the continuation at the level of empirical measures.

The main point is that, thanks to the almost anti-symmetric structure of the
singular interaction, the singularity can be handled in a weak, symmetrized
form. This gives compactness of regularized empirical measures and hence weak
continuations beyond possible collision times. We do not claim uniqueness of
these weak continuations for fixed $N$.

Throughout this subsection, we are in the Coulomb case
$$
s=d-1,\qquad d=1,2.
$$
In particular,
\begin{equation}\label{coulomb-weak-bound}
	|\F^{\mathrm{sing}}(x,y)|\,|x-y|
	\lesssim 1,
	\qquad
	0<|x-y|\le 1 .
\end{equation}
The vector field is not covered by the DiPerna--Lions theory, since in
general
$
\operatorname{div}_x\F(x,y)\notin L^\infty .
$
The key point is instead the cancellation coming from the almost
anti-symmetry of $\F^{\mathrm{sing}}$.

Indeed, for any smooth test function $\phi$, one formally has
\begin{align}\label{weak-symmetrization}
	\sum_{i\neq j}
	\nabla\phi(x_i)\cdot\F^{\mathrm{sing}}(x_i,x_j)
	&=
	\frac12
	\sum_{i\neq j}
	\bigl(
	\nabla\phi(x_i)-\nabla\phi(x_j)
	\bigr)
	\cdot
	\F^{\mathrm{sing}}(x_i,x_j)
	\notag\\
	&\quad
	+
	\frac12
	\sum_{i\neq j}
	\nabla\phi(x_i)\cdot
	\bigl(
	\F^{\mathrm{sing}}(x_i,x_j)
	+
	\F^{\mathrm{sing}}(x_j,x_i)
	\bigr).
\end{align}
The first term is bounded by \eqref{coulomb-weak-bound}, since
$$
|\nabla\phi(x_i)-\nabla\phi(x_j)|
\le
\|\nabla^2\phi\|_{L^\infty}|x_i-x_j|.
$$
The second term is controlled by the almost anti-symmetry assumption
\eqref{assumpF2}. This cancellation is the basic mechanism allowing us to pass
to the limit in the regularized dynamics.

We use the regularized dynamics introduced in the proof of
Theorem~\ref{th1-Coulomb}; see \eqref{weak-reg-ODE-cou} and
\eqref{weak-Fsing-kappa-cou}. Thus, for $\kappa\in(0,N^{-100d}]$,
$X_N^{\kappa,t}$ denotes the corresponding global classical solution and
$$
\mu_N^{\kappa,t}
=
\frac1N\sum_{i=1}^N\delta_{x_i^\kappa(t)},
$$
denotes the associated empirical measure.
\begin{defi}[Weak particle continuation by vanishing regularization]
	\label{def-weak-particle-solution}
	Let $T>0$. A narrowly continuous curve
	$$
	\mu_N^t\in C\bigl([0,T];\mathcal P(\mathbb R^d)\bigr)
	$$
	is called a weak particle continuation of \eqref{ode} on $[0,T]$ if there
	exist a sequence $\kappa_n\downarrow0$ and regularized particle solutions
	$X_N^{\kappa_n,t}$ of \eqref{weak-reg-ODE-cou} such that, for every
	$\phi\in W^{2,\infty}(\mathbb R^d)$,
	\begin{equation}\label{weak-conv-muN}
		\sup_{t\in[0,T]}
		\left|
		\int_{\mathbb R^d}\phi(x)\,d\mu_N^{\kappa_n,t}(x)
		-
		\int_{\mathbb R^d}\phi(x)\,d\mu_N^t(x)
		\right|
		\longrightarrow 0 .
	\end{equation}
	Equivalently,
	\begin{equation}\label{weak-conv-particle-observable}
		\sup_{t\in[0,T]}
		\left|
		\frac1N\sum_{i=1}^N\phi(x_i^{\kappa_n}(t))
		-
		\int_{\mathbb R^d}\phi(x)\,d\mu_N^t(x)
		\right|
		\longrightarrow 0 .
	\end{equation}
	Here narrow continuity means continuity with respect to the weak topology of
	probability measures, or equivalently continuity against bounded continuous
	test functions.
\end{defi}

The definition does not require the existence of pointwise limiting
trajectories $x_i(t)$ after collisions. If no collision occurs on $[0,T]$,
then the weak continuation is simply the empirical measure associated with
the classical ODE solution:
$$
\mu_N^t
=
\frac1N\sum_{i=1}^N\delta_{x_i(t)} .
$$
After collisions, however, the labels of the particles may no longer be
canonically determined, and the empirical measure is the natural object to
continue.

\begin{prop}[Existence of weak particle continuations]
	\label{prop:weak-continuation-existence}
	Assume that $d=1$ or $d=2$, $s=d-1$, and that $\F$ satisfies the assumptions
	used in Theorem~\ref{th1-Coulomb}. Let $X_N^0$ be a collision-free initial
	configuration. Then, for every $T>0$, there exists a weak particle
	continuation $\mu_N^t$ of \eqref{ode} on $[0,T]$ in the sense of
	Definition~\ref{def-weak-particle-solution}. Such a continuation is obtained
	as a subsequential limit of the regularized dynamics
	\eqref{weak-reg-ODE-cou} as $\kappa\downarrow0$.
\end{prop}

\begin{proof}
	For every fixed $\kappa>0$, the regularized vector field in
	\eqref{weak-reg-ODE-cou} is smooth and has at most linear growth.
	Hence, the regularized system has a unique global classical solution
	$X_N^{\kappa,t}$ on $[0,T]$.	Set $
		\mu_N^{\kappa,t}
		:=
		\frac1N\sum_{i=1}^N\delta_{x_i^\kappa(t)}.$	\\
	\medskip
	\noindent
	\textbf{Step 1. Uniform moment bound.}
	Define $
		M_1^\kappa(t)
		:=
		\int_{\mathbb R^d}\langle x\rangle
		\,d\mu_N^{\kappa,t}(x)
		=
		\frac1N\sum_{i=1}^N\langle x_i^\kappa(t)\rangle.$
	Differentiating along the flow, using \eqref{assumpFre} for the regular part and the regularized version of the symmetrization \eqref{weak-symmetrization} for the singular part, we obtain \begin{equation} \frac{d}{dt}M_1^\kappa(t) \leq C\bigl(1+M_1^\kappa(t)\bigr), \end{equation} where $C$ is independent of $\kappa$. Therefore, by Grönwall's inequality,
	\begin{equation}\label{uniform-first-moment-weak}
		\sup_{\kappa>0}
		\sup_{0\leq t\leq T}
		\int_{\mathbb R^d}\langle x\rangle
		\,d\mu_N^{\kappa,t}(x)
		\leq
		C_T.
	\end{equation}
	In particular, the family
	$\{\mu_N^{\kappa,t}\}_{\kappa>0,\;t\in[0,T]}$ is uniformly tight.
	
	\medskip
	\noindent
	\textbf{Step 2. Compactness for a countable determining family.}
	Let
$\{\phi_q\}_{q\geq1}
		\subset C_c^\infty(\mathbb R^d)$
	be a countable convergence-determining family for
	$\mathcal P(\mathbb R^d)$, chosen dense in $C_c(\mathbb R^d)$
	for uniform convergence on compact sets.
	
	For every $q\geq1$, set
	$	I_q^\kappa(t)
		:=
		\int_{\mathbb R^d}\phi_q(x)
		\,d\mu_N^{\kappa,t}(x).$
	Arguing as in \eqref{uniform-first-moment-weak}, using the symmetrization
	\eqref{weak-symmetrization}, we obtain
	\begin{equation}\label{countable-observable-equicont}
		\sup_{\kappa>0}
		\left\|
		I_q^\kappa
		\right\|_{W^{1,\infty}(0,T)}
		\leq
		C_T\|\phi_q\|_{W^{2,\infty}},
		\qquad q\geq1.
	\end{equation}
	By Arzelà--Ascoli and a diagonal extraction, there exist a sequence
	$\kappa_n\downarrow0$ and continuous functions $I_q:[0,T]\to\mathbb R$
	such that
$\sup_{t\in[0,T]}
		\left|
		I_q^{\kappa_n}(t)-I_q(t)
		\right|
		\longrightarrow0 $
	for every $q\geq1$.
	
	The uniform tightness \eqref{uniform-first-moment-weak} and the fact
	that $\{\phi_q\}_{q\geq1}$ is convergence determining show that there
	exists a unique narrowly continuous curve 
$\mu_N\in C\bigl([0,T];\mathcal P(\mathbb R^d)\bigr)$
	such that $
		I_q(t)
		=
		\int_{\mathbb R^d}\phi_q(x)\,d\mu_N^t(x)$
	for every $q\geq1$ and $t\in[0,T]$.
	
	\medskip
	\noindent
	\textbf{Step 3. Extension to all $W^{2,\infty}$ test functions.}
Let $\phi\in W^{2,\infty}(\mathbb R^d)$ and choose $\chi_R\in C_c^\infty(\mathbb R^d)$ such that $\chi_R=1$ on $B_R$. For every $\varepsilon>0$, choose $q$ such that $ \|\phi\chi_R-\phi_q\|_{L^\infty} \leq\varepsilon. $ Using \eqref{uniform-first-moment-weak}, we obtain
	\begin{align*}
		&
		\sup_{t\in[0,T]}
		\left|
		\int_{\mathbb R^d}\phi\,d\mu_N^{\kappa_n,t}
		-
		\int_{\mathbb R^d}\phi\,d\mu_N^t
		\right|\leq
		2\varepsilon
		+
		\sup_{t\in[0,T]}
		\left|
		\int_{\mathbb R^d}\phi_q
		\,d\bigl(
		\mu_N^{\kappa_n,t}-\mu_N^t
		\bigr)
		\right|
		+
		\frac{2C_T}{R}\|\phi\|_{L^\infty}.
	\end{align*}
	Letting first $n\to\infty$, then $\varepsilon\to0$, and finally
	$R\to\infty$, we obtain
	\begin{equation}
		\sup_{t\in[0,T]}
		\left|
		\int_{\mathbb R^d}\phi\,d\mu_N^{\kappa_n,t}
		-
		\int_{\mathbb R^d}\phi\,d\mu_N^t
		\right|
		\longrightarrow0.
	\end{equation}
	Thus \eqref{weak-conv-muN} holds for every
	$\phi\in W^{2,\infty}(\mathbb R^d)$.
	
	The curve $\mu_N^t$ is therefore a weak particle continuation in
	the sense of Definition~\ref{def-weak-particle-solution}. By
	construction, it is selected by the vanishing-regularization
	sequence $\kappa_n\downarrow0$. No uniqueness with respect to the
	choice of regularization sequence is asserted.
\end{proof}

We now record the distributional identity associated with the selected
vanishing-regularization continuation. This identity is not used as an
independent definition of a weak solution: the continuation
$\mu_N^t$ is defined by Definition~\ref{def-weak-particle-solution} through
the convergence of the regularized particle systems.

Let $\mu_N^t$ be obtained along a sequence $\kappa_n\to0$. For every 
	$\psi\in
	C^1\bigl([0,T];W^{2,\infty}(\mathbb R^d)\bigr)$
compactly supported in space, the regularized empirical measure satisfies
\begin{align}\label{weak-form-muNkappa}
	&\int_{\mathbb R^d}\psi(T,x)\,d\mu_N^{\kappa,T}(x)
	-
	\int_{\mathbb R^d}\psi(0,x)\,d\mu_N^{\kappa,0}(x)
	-
	\int_0^T
	\int_{\mathbb R^d}
	\partial_t\psi(t,x)\,d\mu_N^{\kappa,t}(x)\,dt
	\notag
	\\
	&=
	\int_0^T
	\int_{\mathbb R^d}
	\nabla\psi(t,x)\cdot
\left[	\bigl(\F^{\mathrm{reg}}*\mu_N^{\kappa,t}\bigr)(x)-\frac{1}{N}\F^{\mathrm{reg}}(x,x)\right]
	\,d\mu_N^{\kappa,t}(x)\,dt
	+
	\mathcal S_N^\kappa[\psi],
\end{align}
where
\begin{align}\label{def-singular-functional-kappa}
	\mathcal S_N^\kappa[\psi]
	&:=
	\frac1{2N^2}
	\int_0^T
	\sum_{\substack{i,j=1\\ i\neq j}}^N
	\Biggl\{
	\bigl(
	\nabla\psi(t,x_i^\kappa(t))
	-
	\nabla\psi(t,x_j^\kappa(t))
	\bigr)
	\cdot
	\F^{\mathrm{sing},\kappa}
	\bigl(x_i^\kappa(t),x_j^\kappa(t)\bigr)
	\notag
	\\
	&\qquad\qquad\qquad
	+
	\nabla\psi(t,x_i^\kappa(t))
	\cdot
	\Bigl[
	\F^{\mathrm{sing},\kappa}
	\bigl(x_i^\kappa(t),x_j^\kappa(t)\bigr)
	+
	\F^{\mathrm{sing},\kappa}
	\bigl(x_j^\kappa(t),x_i^\kappa(t)\bigr)
	\Bigr]
	\Biggr\}\,dt .
\end{align}
The symmetrized estimates above imply
\begin{equation}\label{singular-functional-uniform-bound}
	\left|
	\mathcal S_N^\kappa[\psi]
	\right|
	\leq
	C_T
	\|\psi\|_{L^\infty(0,T;W^{2,\infty}(\mathbb R^d))},
\end{equation}
where $C_T$ is independent of $\kappa$. In particular, the limiting
functional defined below is linear and bounded on the test class.

Passing to the limit in the nonsingular terms of
\eqref{weak-form-muNkappa}, we define the singular interaction functional
associated with the selected continuation by
\begin{align}\label{def-singular-functional}
	\mathcal S_N[\psi]
	&:=
	\int_{\mathbb R^d}\psi(T,x)\,d\mu_N^T(x)
	-
	\int_{\mathbb R^d}\psi(0,x)\,d\mu_N^0(x)
	-
	\int_0^T
	\int_{\mathbb R^d}
	\partial_t\psi(t,x)\,d\mu_N^t(x)\,dt
	\notag
	\\
	&\quad
	-
	\int_0^T
	\int_{\mathbb R^d}
	\nabla\psi(t,x)\cdot
	\bigl(\F^{\mathrm{reg}}*\mu_N^t\bigr)(x)
	\,d\mu_N^t(x)\,dt .
\end{align}
Equivalently,
\begin{equation}\label{singular-functional-limit}
	\mathcal S_N[\psi]
	=
	\lim_{n\to\infty}
	\mathcal S_N^{\kappa_n}[\psi].
\end{equation}
Therefore,
\begin{equation}\label{singular-functional-bound}
	|\mathcal S_N[\psi]|
	\leq
	C_T
	\|\psi\|_{L^\infty(0,T;W^{2,\infty}(\mathbb R^d))}.
\end{equation}
The selected continuation consequently satisfies the distributional identity
\begin{align}\label{weak-form-muN}
	&\int_{\mathbb R^d}\psi(T,x)\,d\mu_N^T(x)
	-
	\int_{\mathbb R^d}\psi(0,x)\,d\mu_N^0(x)
	-
	\int_0^T
	\int_{\mathbb R^d}
	\partial_t\psi(t,x)\,d\mu_N^t(x)\,dt
	\notag
	\\
	&=
	\int_0^T
	\int_{\mathbb R^d}
	\nabla\psi(t,x)\cdot
\left[	\bigl(\F^{\mathrm{reg}}*\mu_N^{t}\bigr)(x)-\frac{1}{N}\F^{\mathrm{reg}}(x,x)\right]
	\,d\mu_N^t(x)\,dt
	+
	\mathcal S_N[\psi].
\end{align}
We emphasize that \eqref{weak-form-muN} is a consequence of the
selection-by-regularization procedure and is not, by itself, an independent
characterization of $\mu_N^t$. At the present level of generality, the
singular interaction cannot necessarily be expressed directly in terms of
$\mu_N^t$ after collisions. The functional $\mathcal S_N$ records the limit
of the total symmetrized singular interaction along the selected
regularization sequence.\\
Moreover, if the limiting particle dynamics is collision-free on $[0,T]$,
namely if
$$
\delta_T
:=
\inf_{0\le t\le T}\inf_{i\neq j}|x_i(t)-x_j(t)|
>0,
$$
then the weak continuation coincides with the classical empirical measure
$$
\mu_N^t=\frac1N\sum_{i=1}^N\delta_{x_i(t)}.
$$
In this case the singular functional reduces to the usual singular-force
contribution:
\begin{align}\label{singular-functional-classical}
	\mathcal S_N[\psi]
	=
	\int_0^T
	\int_{\mathbb R^d}
	\nabla\psi(t,x)
	\cdot
	\bigl(\F^{\mathrm{sing}}*\mu_N^t\bigr)(x)
	\,d\mu_N^t(x)\,dt,
\end{align}
where the convolution is understood with the diagonal terms removed.
Equivalently, the same quantity can be written in the symmetrized form
\eqref{weak-symmetrization}. Indeed, the uniform separation $\delta_T>0$
keeps all particle pairs away from the singularity, so the regularized forces
converge uniformly to the unregularized singular force along the trajectories.
\begin{remark}[Non-uniqueness at fixed $N$ and collision rules]
	\label{rem-fixed-N-nonunique}
	At this level of generality, we do not claim that the weak continuation
	$\mu_N^t$ is unique for fixed $N$. The symmetrized estimates above control the
	singular interaction and yield compactness, but they do not compare two
	different regularized particle flows after collisions. In particular,
	different vanishing regularizations may in principle select different weak
	continuations.
	
	This issue is related to the absence of a prescribed collision rule at the
	particle level. For instance, in the classical two-dimensional Euler
	point-vortex system, the interaction is Hamiltonian and given by the rotated
	Coulomb field. In the same-sign case, the logarithmic Hamiltonian is conserved,
	which prevents collisions; see, for example,
	\cite{MarchioroPulvirenti1994}. This mechanism is special to the Hamiltonian
	point-vortex structure and is not available in the present setting.
	
	In this work, we do not impose any collision rule. Instead, we regularize the
	dynamics and pass to subsequential limits of the empirical measures. This
	yields weak particle continuations beyond possible collision times, but does
	not imply uniqueness of the continuation for fixed $N$. The proof of
	Theorem~\ref{th1-Coulomb} shows that every weak continuation selected by this
	vanishing-regularization procedure satisfies the same quantitative estimate of convergence
	to the unique mean-field solution $\mu^t$.
\end{remark}

\section{Proof of short-time existence for the mean-field equation}\label{app:existence}

\begin{lem}[Local existence for the mean-field equation]
	\label{1stpde}
	Let $\gamma>d$ and assume that $s\leq d-1$. Suppose that the force
	$\F$ satisfies \eqref{assumpFre}-\eqref{coudet}, and let $\mu_0$ satisfy
	\begin{equation}\label{def:bracket-mu}
		[\mu_0]
		:=
		\sum_{0\leq|\beta|\leq3}
		\left\|
		\langle x\rangle^{\gamma+1}
		\partial_x^\beta\mu_0
		\right\|_{L^\infty(\mathbb R^d)}
		<\infty.
	\end{equation}
	There exists a constant $c_0>0$, depending only on $d$, $\gamma$,
	$s$, and the constants in \eqref{assumpF2G}, such that, for every $0<T\leq T_*,$ $T_*:=c_0/(1+[\mu_0]),$
	the equation
	\begin{equation}\label{mflimit-local}
		\partial_t\mu
		=
		-\operatorname{div}\bigl((\F*\mu)\mu\bigr),
		\qquad
		\mu|_{t=0}=\mu_0,
	\end{equation}
has a unique solution on $[0,T_*]$ such that \begin{equation}\label{local-solution-regularity} \mu\in C\bigl([0,T_*];W^{2,\infty}(\mathbb R^d)\bigr) \cap L^\infty\bigl([0,T_*];W^{3,\infty}(\mathbb R^d)\bigr), \end{equation} and \begin{equation}\label{local-solution-estimate} [\mu]_{T_*} := \sup_{0\leq t\leq T_*} \sum_{0\leq|\beta|\leq3} \left\| \langle x\rangle^{\gamma+1} \partial_x^\beta\mu^t \right\|_{L^\infty(\mathbb R^d)} \leq 2[\mu_0]. \end{equation} Uniqueness holds in the class of solutions satisfying \eqref{local-solution-regularity} for which the quantity $[\mu]_{T_*}$ is finite.
\end{lem}

\begin{proof}
\textbf{Step 1: Mapping estimates.} By \eqref{coudet} and \eqref{assumpF2G}, we bound for any $\omega_0\in L^\infty(\mathbb{R}^d)$:
\begin{equation}\label{z4}
  \sum_{n=0}^{3}\|\nabla_x^n \operatorname{div} \F *\omega_0\|_{L^\infty}
  +\|\langle x\,\mathbf{1}_{n=0}\rangle^{-1}\nabla_x^n \F *\omega_0\|_{L^\infty}
  \lesssim [\omega_0].
\end{equation}
\medskip
\textbf{Step 2: Fixed-point argument.}
Define $\mathbf{S}: V_{T}\to L^\infty(\mathbb{R}^d\times [0,T])$ by $\mathbf{S}(\omega)=\mu$ where
$\partial_t \mu= -\operatorname{div}\bigl((\F *\omega)\mu\bigr)$,
with
\[
  V_T:=\Bigl\{\omega\in L^\infty(\mathbb{R}^d\times [0,T]):
    [\omega]_T\leq 2[\mu_0],\;\omega|_{t=0}=\mu_0\Bigr\}.
\]
Since $\F *\omega\in L^\infty([0,T],\mathrm{Lip}(\mathbb{R}^d))$ and
$\langle x\rangle^{-1} \F *\omega\in L^\infty([0,T]\times\mathbb{R}^d)$,
standard transport theory ensures $\mathbf{S}$ is well-defined.

We prove, for $\omega_j\in V_T$, $\mu^j=\mathbf{S}(\omega_j)$, $j=1,2$:
\begin{align}
  [\mu^1]_T &\leq [\mu_0]+C_0\, T\, [\mu^1]_T\,[\omega_1]_T,
  \label{z2}
  \\
  [\mu^1-\mu^2]_T &\leq C_0\,T\,
    \bigl([\omega_1]_T+[\omega_2]_T+[\mu^1]_T+[\mu^2]_T\bigr)
    \bigl([\omega_1-\omega_2]_T+[\mu^1-\mu^2]_T\bigr).
  \label{z3}
\end{align}

\emph{Proof of \eqref{z2}.}
Set $\tilde{\mu}=\langle x\rangle^{\gamma+1}\mu^1$ and let $\chi_R$ be a cutoff with
$\mathbf{1}_{B_R}\leq \chi_R\leq \mathbf{1}_{B_{2R}}$,
$|\nabla\chi_R|\lesssim R^{-1}\,\mathbf{1}_{B_{2R}}\sim \langle x\rangle^{-1}\,\mathbf{1}_{B_{2R}}$.
Then
\[
  \partial_t(\tilde{\mu}\chi_R)
  =- \F *\omega^1\,\nabla_x(\tilde{\mu}\chi_R)
  - \bigl(\operatorname{div}\F *\omega^1\,\chi_R
    -\F *\omega^1\,\nabla_x\chi_R
    -(\gamma+1)\F *\omega^1\cdot x\langle x\rangle^{-2}\chi_R\bigr)\,\tilde{\mu}.
\]
By the maximum principle and letting $R\to \infty$:
\begin{align}\label{z1}
  \|\langle \cdot\rangle^{\gamma+1}\mu^1\|_{L^\infty(\mathbb{R}^d\times [0,T])}
  &\leq \|\langle \cdot\rangle^{\gamma+1}\mu_0\|_{L^\infty}
    +C\,T\,[\mu^1]_T\,[\omega_1]_T.
\end{align}
The same argument applies to derivatives $\partial^\beta_x\mu^1$ for $1\leq |\beta|\leq 2$.
For the third derivatives, we use a difference quotient $\delta^\varepsilon_i f(x):=f(x)-f(x-\varepsilon\,e_i)$
with $\varepsilon\in(0,1)$.
Setting $\overline{\mu}=\langle x\rangle^{\gamma+1}\delta^\varepsilon_i \partial_x^\beta\mu^1$ with $|\beta|=2$,
applying the maximum principle, estimating the commutator by \eqref{z4}, and letting $\varepsilon\to 0$:
\[
  \|\langle \cdot\rangle^{\gamma+1}\partial_{x_i} \partial_x^\beta\mu^1\|_{L^\infty(\mathbb{R}^d\times [0,T])}
  \leq \|\langle \cdot\rangle^{\gamma+1}\partial_{x_i} \partial_x^\beta\mu_0\|_{L^\infty}
    +C\,T\,[\mu^1]_T\,[\omega_1]_T.
\]
Summing over all derivatives proves \eqref{z2}.

\emph{Proof of \eqref{z3}.}
For $\mu=\mu^1-\mu^2$ and $\omega=\omega_1-\omega_2$:
\begin{equation}
	\label{difference-equation-fixed-point}
	 \partial_t \mu
	= -\operatorname{div}\bigl(\F *\omega_1\,\mu\bigr)
	- \operatorname{div}\bigl((\F *\omega)\,\mu_2\bigr).
\end{equation}
We differentiate \eqref{difference-equation-fixed-point} only up to
order two. For every multi-index $\beta$ with $|\beta|\leq2$, the
highest derivative of $\mu^2$ appearing in $\partial_x^\beta
	\operatorname{div}\bigl((\F*\omega)\mu^2\bigr)$
has order at most three and is therefore controlled by $[\mu^2]_T$.
Applying the same weighted maximum-principle and commutator argument
as above, together with the corresponding estimates following from
\eqref{z4}, yields
\begin{align*}
	[\mu]_{T,2}
	\leq
	C_0T
	\bigl(
	[\omega_1]_T+[\omega_2]_T
	+[\mu^1]_T+[\mu^2]_T
	\bigr)
	\bigl(
	[\omega]_{T,2}+[\mu]_{T,2}
	\bigr).
\end{align*}
This proves \eqref{z3}.

\medskip
\noindent
Choosing $T=\frac{1}{100C_0[\mu_0]}$, the Banach fixed-point theorem guarantees
a unique fixed point $\mu\in V_T$, completing the proof.
\end{proof}

\begin{proof}[Proof of Lemma~\ref{Le-localwp}]
	Introduce the rescaled unknown
		$\widetilde\mu(t,x)
		:=
		\mu(\eta_*^{s+1-d}t,x).$
	Then $\widetilde\mu$ satisfies
	\begin{equation}\label{mflimitsupb}
		\partial_t\widetilde\mu
		=
		-\operatorname{div}
		\bigl(
		(G_{\eta_*}*\widetilde\mu)\widetilde\mu
		\bigr),
		\qquad
		\widetilde\mu|_{t=0}=\mu^0,
	\end{equation}
	where
	\begin{equation}\label{defG-rescaled}
		G_{\eta_*}(x,y)
		:=
		\eta_*^{s+1-d}\F_{\eta_*}(x,y).
	\end{equation}
	
	We decompose
	\begin{equation}
		G_{\eta_*}
		=
		G_{\eta_*}^{\mathrm{sing}}
		+
		G_{\eta_*}^{\mathrm{reg}},
	\end{equation}
	where
	\begin{equation}
		G_{\eta_*}^{\mathrm{sing}}
		:=
		\eta_*^{s+1-d}\F_{\eta_*}^{\mathrm{sing}},
		\qquad
		G_{\eta_*}^{\mathrm{reg}}
		:=
		\eta_*^{s+1-d}\F^{\mathrm{reg}}.
	\end{equation}
	By \eqref{assumpFre}--\eqref{coudet}, the singular cancellation
	estimates and the bounds for the regular part yield
	\begin{align}\label{uniform-G-mapping}
		&
		\left\|
		\langle x\rangle^{-1}
		G_{\eta_*}*\omega
		\right\|_{L^\infty}
		+
		\sum_{1\leq|\alpha|\leq3}
		\left\|
		\partial_x^\alpha
		(G_{\eta_*}*\omega)
		\right\|_{L^\infty}
		+
		\sum_{|\alpha|\leq3}
		\left\|
		\partial_x^\alpha
		\operatorname{div}
		(G_{\eta_*}*\omega)
		\right\|_{L^\infty}
		\leq
		C\A(\omega),
	\end{align}
	where $C$ is independent of $\eta_*$.
	
	We may therefore apply the fixed-point argument from the proof of
	Lemma~\ref{1stpde} to \eqref{mflimitsupb}. It gives a time 
		$\tau=	\frac{c_0}{1+\A(\mu^0)},$
	independent of $\eta_*$, such that \eqref{mflimitsupb} has a unique
	solution on $[0,\tau]$ satisfying
$\sup_{0\leq t\leq\tau}
		\A(\widetilde\mu^t)
		\leq
		2\A(\mu^0).$
	Returning to the original time variable, $\mu$ exists on
$0\leq t\leq
		\tau\eta_*^{s+1-d}.$
	This proves the result.
\end{proof}

\bibliographystyle{alpha}
\bibliography{Master.bib}

\end{document}